\documentclass[twoside]{article}
\usepackage{amsmath,amsthm,amsfonts,latexsym,amscd,amssymb,enumerate}
\usepackage[hypertex]{hyperref}
\usepackage{graphicx}
\usepackage[all]{xy}

\newtheorem{theorem}{Theorem}[section]
\newtheorem{lemma}[theorem]{Lemma}
\newtheorem{sublemma}[theorem]{Sublemma}

\newtheorem{proposition}[theorem]{Proposition}

\theoremstyle{definition}
\newtheorem{definition}[theorem]{Definition}

\newtheorem{set-up}[theorem]{Geometric set-up}
\newtheorem{remark}[theorem]{Remark}

\newcommand{\K}{\mathbb{K}}

\DeclareMathOperator{\End}{End}    

\DeclareMathOperator{\ind}{ind}
\DeclareMathOperator{\Ind}{Ind}

\newcommand{\forget}[1]{}

\def  \nuint {\raise10pt\hbox{$\nu$}\kern-6pt\int}

\newcommand\tn{\tilde{n}}

\newcommand\wx{\widehat{X}}

\newcommand\we{\widehat{E}}

\newcommand\Tr{\operatorname{Tr}}

\def \L{\mathcal L}
\def \F{\mathcal F}
\def \I{\mathcal I}

\newcommand\E{\mathcal E}
\newcommand\Q{\mathcal Q}

\newcommand\C{\mathcal C}

\def \L {{\cal L}}
\def \Sp {{\cal S}}
\newcommand\B{\mathcal B}

\def \J{\mathcal J}
\def \H {{\cal H}}

\def\Id{{\rm Id}}

\def \chix {\chi^0}

\newcommand\cl{\operatorname{cl}}
\newcommand\cyl{\operatorname{cyl}}

\renewcommand\Im{\operatorname{Im}}

\newcommand\D{\mathcal D}
\newcommand\Di{D\kern-6pt/}
\newcommand\cDi{{\mathcal D}\kern-6pt/}
\newcommand\spi{S\kern-6pt/}
\newcommand \cspi{\Sp\kern-6pt/}

\newcommand\CC{\mathbb C}

\def \cal {\mathcal}
\def \C {{\cal C}}
\def \K {{\cal K}}

\newcommand\KK{\mathbb K}
\newcommand\NN{\mathbb N}

\newcommand\RR{\mathbb R}
\newcommand\ZZ{\mathbb Z}

\newcommand\pa{\partial}
\newcommand\Ker{\operatorname{Ker}}

\def\tV{{\tilde V}}
\def\tN{{\tilde N}}

\def\tM{{\tilde M}}
\def\tm{{\tilde m}}

\def\maF{{\mathcal F}}

\def\A{{\mathcal A}}

\def\K{{\mathcal K}}

{\catcode`@=11\global\let\c@equation=\c@theorem}

\allowdisplaybreaks[2]
\date{}
 \usepackage{color}
\definecolor{darkgreen}{cmyk}{1,0,1,.2}
\definecolor{m}{rgb}{1,0.1,1}

\title{Eta cocycles, relative pairings and  \\the Godbillon-Vey  index theorem}

\author{Hitoshi Moriyoshi and Paolo Piazza}

\begin{document}

\maketitle

\begin{abstract}
We prove a Godbillon-Vey index formula for longitudinal Dirac operators on a foliated bundle with boundary
$(X,\F)$;
in particular, we define a {\it Godbillon-Vey eta invariant} on  $(\pa X,\F_{\pa})$, that is, a secondary
invariant for longitudinal Dirac operators on type-III foliations.
Moreover, employing the Godbillon-Vey index as a pivotal example, 
we explain a new approach to higher index theory on geometric
structures with boundary. This is heavily based on
the interplay between the absolute and relative pairings of $K$-theory and 
cyclic cohomology for an exact sequence of Banach algebras which in the present
context takes the form 
$0\to \mathbf{ \mathfrak{J}} \to \mathbf{ \mathfrak{A}} \to
\mathbf{ \mathfrak{B}} \to 0$ 
with $ \mathbf{ \mathfrak{J}}$  dense
and holomorphically closed in $C^* (X,\F)$ and   $ \mathbf{ \mathfrak{B}}$
depending only on boundary data. Of particular importance is the definition of a {\it relative}
cyclic cocycle $(\tau_{GV}^r,\sigma_{GV})$ for the pair $\mathbf{ \mathfrak{A}} \to
\mathbf{ \mathfrak{B}}$; $\tau_{GV}^r$ is a cyclic {\it cochain} on $\mathbf{ \mathfrak{A}}$ defined
through a regularization, \`a la Melrose, of the usual Godbillon-Vey cyclic cocycle $\tau_{GV}$; 
$\sigma_{GV}$ is a cyclic
cocycle on $\mathbf{ \mathfrak{B}}$, obtained through a suspension procedure involving
$\tau_{GV}$ and a specific 1-cyclic cocycle (Roe's 1-cocycle). 
We call $\sigma_{GV}$ the eta cocycle associated to $\tau_{GV}$.
The Atiyah-Patodi-Singer formula is obtained by defining a relative index class
$\Ind (D,D^\partial)\in K_* (\mathbf{ \mathfrak{A}},
\mathbf{ \mathfrak{B}})$ and establishing the equality $\langle \Ind (D),[\tau_{GV}] \rangle\,=\,\langle \Ind (D,D^\pa), [\tau^r_{GV}, \sigma_{GV}] \rangle$. The Godbillon-Vey eta invariant $\eta_{GV}$ is obtained through
the eta cocycle $\sigma_{GV}$.

\end{abstract}

%

\tableofcontents
\newpage
\section{{\bf Introduction}}\label{sec:intro}
 
 The Atiyah-Singer index theorem on closed compact manifolds is regarded
nowadays as one of the milestones of modern Mathematics. The original result has branched
into several directions, producing new ideas, new results as well as new connections 
between different fields of Mathematics and Theoretical Physics.
One of these directions consists in considering  elliptic differential  operators on the following
hierarchy of  geometric structures:
\begin{itemize}
\item fibrations and operators that are elliptic in the fiber directions; 
\item Galois $\Gamma$-coverings and $\Gamma$-equivariant elliptic operators;
\item measured foliations and operators that are elliptic along the leaves;
\item general foliations and, again, operators that are elliptic along the leaves.
\end{itemize}
One pivotal example, going through all these situations, is the one of foliated
bundles.
Let $\Gamma\to \tN\to N$ be a Galois $\Gamma$-cover of a smooth compact manifold 
without boundary $N$,
let $T$ be an oriented  compact manifold on which $\Gamma$ acts by orientation preserving diffeomorphisms. 
We can  consider the diagonal action of $\Gamma$ on $\tN\times T$ and the 
quotient space $Y:=\tN\times_\Gamma T$, which is a compact manifold,
a bundle over $N$ and carries 
a foliation $\F$. This foliation is  obtained by considering the images
of the fibers of the trivial fibration $\tN\times T\to T$ under the quotient map
$\tN\times T\to \tN\times_\Gamma T$ and is known as a {\it foliated bundle}. 
We also consider
$E\to Y$ a complex vector bundle on $Y$ and  $\widehat{E} \to \tN\times T $ the $\Gamma$-equivariant vector bundle obtained
by lifting $E$ to $\tN\times T$.
We then consider a family of elliptic differential operators $(D_\theta)_{\theta\in T}$
on the product fibration $\tN\times T\to T$, acting on the sections of  
$\widehat{E}$, and we assume that it is $\Gamma$-equivariant;
it therefore yields a leafwise differential operator $(D_L)_{L\in V/\maF}$
on $Y$, which is elliptic along the
leaves of $\F$. Notice that, if $\dim T >0$ and $\Gamma=\{1\}$ then we are in the family
situation; if $\dim T=0$ and $\Gamma\not=\{1\}$, then we are in the covering situation;
if $\dim T>0$, $\Gamma\not=\{1\}$ and $T$ admits a $\Gamma$-invariant Borel
measure $\nu$, then we are in the measured foliation situation and if $\dim T>0$, $\Gamma\not=\{1\}$ 
then we are dealing with  a  general \footnote{Typically type III} foliation. 
As an example of this latter type III situation we can consider
  $T=S^1$, $N$  a compact Riemann surface of genus $\geq 2$, $\tN={\mathbb H}^2$ 
$\tN={\mathbb H}^2$
the hyperbolic plane,
and 
 $\Gamma=\pi_1 (N)$ acting on $S^1$ by fractional linear transformations; we obtain a foliated
 bundle $(Y, \F)$, 
 where 
$Y$ is the unit tangent bundle of $N$ and $\F$ is the Anosov foliation of codimension one. 
It is known that the resulting foliation von Neumann algebra is 
 the unique hyperfinite factor of type $ \mathrm{III}_1$;
in particular $(Y, \F)$  is {\it not} measured. 

\medskip
In the first three cases, there is first of all
a {\it numeric} index: for families this is simply the integral over $T$ of the locally constant
function that associates to $\theta$ the index of $D_\theta$; for $\Gamma$-coverings
we have the  $\Gamma$-index of  Atiyah and for measured foliations we have
the measured index introduced by Connes. These last two examples involve the
definition of a von Neumann algebra endowed with a suitable trace.
The index theorems of Atiyah and Connes provide geometric formulae for these numeric
indeces.

\medskip
The numeric index, when defined, is only part of the information carried by the elliptic operator
in question. More generally one is interested in {\it higher indeces}, numbers obtained
by pairing the index class, an element in the K-theory of a suitable algebra, with
cyclic cocycles 
of degree $> 0$ defined on the same algebra. Notice that in the case of type III foliation, 
such as 
the example above, we {\it must}
consider higher indeces 
of degree$>0$ 
(indeed,  there is no trace on the foliation von Neumann algebra
and thus there is no numeric index).

\medskip
The {\it higher index problem} can be stated as the problem of 
\begin{itemize}
\item defining these higher
indeces;
\item proving explicit geometric formulae
for them, in the spirit of the original result of Atiyah and Singer;
\item studying their stability properties.
\end{itemize}
It is important to observe that geometric applications of this theory, for example to 
Novikov-type conjectures on the  (foliated) homotopy invariance  of  higher signatures
or to topological obstructions to the existence of positive scalar curvature metrics, are obtained
by combining {\it all} of these points. 
Put it differently, it might  be possible to define higher indeces 
and prove geometric formulae for them, but it might  be  difficult, or require
extra assumptions, to establish stability properties for these indeces.
This phenomenon presents itself in the following way: stability properties are obtained
by considering the index class in the K-theory of a suitable $C^*$-algebra; 
in the case of foliated bundles, which is our concern here, one considers
the foliation $C^*$-algebra $C^* (Y,\F)$ and the K-theory groups 
$K_* (C^* (Y,\F))$. Equivalently, we can consider the index class in the K-theory
of the Morita-equivalent
algebra $C(T)\rtimes_r \Gamma$.
The index class, hovewer, is typically defined in a smaller algebra $C^\infty_c (Y,\F)
\subset C^* (Y,\F)$ and
 higher indeces are easily obtained by pairing this class, call it $\Ind^c (D)$, with  cyclyc cocycles $\tau^c$
 for $C^\infty_c (Y,\F)$ \footnote{In the Morita-equivalent picture we would be considering 
the small algebra  $ C(T)\rtimes_{{\rm alg}}\Gamma\subset C(T)\rtimes_r \Gamma$}. The delicate point we have alluded to is then the following:
it might very well be possible to prove a formula for these numbers 
$\langle \Ind^c (D), \tau^c \rangle$
without connecting them
with the $C^*$-algebraic index class $\Ind (D)$, which is the index class
 showing the most interesting geometric properties. In order to achieve a {\it complete} solution
 of the higher index problem for the cocycle $\tau^c$ one is usually confronted with the task of 
 finding an {\it intermediate }
 subalgebra $\mathfrak{A}$, $C^\infty_c (Y,\F)\subset \mathfrak{A}\subset C^* (Y,\F)$ which satisfies the following crucial properties: it is big enough to be  holomorphycally closed in $C^* (Y,\F)$
and  contain representatives
of the $C^*$-index class
$\Ind (D)$ but it is small enough that the cyclic cocycle $\tau^c$ {\it extends} from
 $C^\infty_c (Y,\F)$ to $\mathfrak{A}$. Finding such an intermediate algebra
 can be a difficult task.
 
 \medskip
Connes' index theorem for $G$-proper manifolds \cite{Co}, with $G$
an \'etale groupoid, gives a very satisfactory answer to the 
computation of the pairing between the index class $\Ind^c (D)$
for the small algebra  and the  cyclic cohomology classes $[\tau^c]$ of this same algebra.
This higher  index theorem applies in particular to a foliated bundle
$\tN\times_\Gamma T$ (this is 
a $G$-proper
manifold with $G$ equal to the groupoid $T\rtimes \Gamma$).

 \medskip
 One fascinating  higher index is the so-called Godbillon-Vey  index on a codimension 1 foliation. In this case 
 Connes proves the following  \cite{connes-transverse}:
  there is an
 intermediate subalgebra $\mathcal{A}$, $C(T)\rtimes_{{\rm alg}}\Gamma \subset 
 \mathcal{A} \subset C(T)\rtimes_r \Gamma$, which is holomorphically closed and contains the index class $\Ind (D)$;
 there is  a
 cyclic 2-cocycle $\tau_{\,{\rm BT}}$ on $C(T)\rtimes_{{\rm alg}}\Gamma $ 
 (the Bott-Thurston cocycle) which is extendable
 to $\mathcal{A}$;
 the general index formula for the pairing $\langle \Ind (D), [\tau_{\,{\rm BT}}]\rangle$
 can be written down explicitly 
 and it involves the  Godbillon-Vey class of the foliation, $GV\in H^3 (Y)$. This is a {\it complete}
 solution to the higher index problem.
 For the particular 3-dimensional
 example presented above this formula reads
 \begin{equation}\label{gv-bt}
\langle \Ind (D), [\tau_{\,{\rm BT}}] \rangle
= \langle GV,[Y]\rangle
=:\mathfrak{gv}(Y,\F)\,
\end{equation}
with $[Y]$ the fundamental homology class of $Y$ and 
$\mathfrak{gv}(Y,\F)$
the Godbillon-Vey invariant of the foliation $(Y,\F)$. 
 Thus, a
purely geometric invariant of the foliation $(Y,\F)$, $\mathfrak{gv}(Y,\F)$, is in fact
a higher index. For  geometric properties of the Godbillon-Vey invariant we refer the reader to
the excellent survey of Ghys \cite{Gh}. 
It is worth recalling here the  remarkable result by Hurder and Katok \cite{HK} 
relating the Godbilloy-Vey  invariant
to
properties of the foliation von Neumann algebras;
in our case,  
this result states that 
the von Neumann algebra of the foliation
contains a nontrivial type III component if $\frak{gv}(Y,\F)\ne 0$; 
 thus the Godbillon-Vey invariant detects type III properties of the foliation von Neumann algebra.

\medskip
An alternative treatment of the fascinating index formula \eqref{gv-bt} was given by Moriyoshi-Natsume
in \cite{MN}. In this work, a Morita-equivalent {\it complete} solution to the Godbillon-Vey
index theorem is given. First of all,  
there is a cyclic 2-cocycle $\tau_{GV}$ on $C^\infty_c (Y,\F)$ which can be paired with
the index class $\Ind^c (D)$.
Next, Moriyshi and Natsume define a 
holomorphically closed subalgebra $\mathfrak{A}$, $C^\infty_c (Y,\F)\subset \mathfrak{A} \subset C^* (Y,\F)$, containing the index class $\Ind (D)$ and such that $\tau_{GV}$ extends to $\mathfrak{A}$.
The pairing $\langle \Ind (D),[\tau_{GV}]\rangle$, which is obtained in \cite{MN}
as a direct  evaluation of the functional $\tau_{GV}$, 
 is explicitly computed 
by expressing the index class through the {\it graph projection} $e_D$ associated to $D$,
considering $sD$, $s>0$ and taking the limit as $s\downarrow 0$. 
Getzler's rescaling method is 
used crucially
 in establishing the analogue of \eqref{gv-bt}:
 \begin{equation}\label{gv-mn}
\langle \Ind (D), [\tau_{\,{\rm GV}}] \rangle= \int_Y  \omega_{GV}\end{equation}
with $\omega_{GV}$ an explicit closed 3-form on $Y$ such that $[\omega_{GV}]=GV\in H^3 (Y)$.
In particular, we find once again that $\langle \Ind (D), [\tau_{\,{\rm GV}}] \rangle=\mathfrak{gv}(Y,\F)$.

\medskip
Subsequently, Gorokhovsky and Lott \cite{Go-Lo}
gave a superconnection proof of Connes'
index theorem, including an explicit  formula for the Godbillon-Vey higher index. 
See also the Appendix of \cite{Go-Lo2}.
Yet another treatment was given by Gorokhovsky in his elegant paper \cite{Go-JFA}.

\medskip
In the past 40 years this complex  circle of ideas has 
  been extended  to {\it some} geometric structures
with boundary. Let us give a short 
summary of these contributions, with an emphasis on the higher case.
First, in the case of a single manifold and
of a Dirac-type operator on it, such an index theorem is due to
Atiyah-Patodi-Singer \cite{APS1}. Assume that $D$ is an odd $\ZZ_2$-graded
Dirac operator on a compact even-dimensional manifold $M$ with boundary $\pa M=N$ acting on a $\ZZ_2$-graded
bundle of Clifford module $E$. Assume all geometric structures to be of
product type near the boundary. For simplicity, here and in what follows
always assume the boundary operator $D^{\pa}$ to be invertible.
Then the Dirac operator $D^+$ with boundary conditions 
$\{u\in C^\infty (M,E^+)\;|\; u|_{\pa M}\in \Ker \Pi_{\geq}\}$
with $\Pi_{\geq}=\chi_{[0,\infty)} (D^\pa)$, extends to a Fredholm 
operator; the index is given  by the celebrated formula of Atiyah-Patodi-Singer
\begin{equation}\label{aps}\ind_{{\rm APS}} D^+ = \int_M {\rm AS} - \frac{1}{2} \eta (D^{\pa})\,
\end{equation}
with ${\rm AS}$ the Atiyah-Singer form associated to  $M$ and $E$ and $\eta( D^{\pa})$
the eta invariant of the formally self-adjoint operator $D^{\pa}$, a spectral invariant measuring the asymmetry of the spectrum of $D^{\pa}$. The number $\eta( D^{\pa})$ should be thought of as a
{\it secondary} invariant of the boundary operator.
The Atiyah-Patodi-Singer index is also equal to the $L^2$-index on the manifold with cylindrical end
$V=((-\infty,0]\times \pa M)\cup_{\pa M} M$.
See Melrose' book  \cite{Melrose} for a thorough treatment of the APS index theorem 
from this point of view. 

\medskip
Let us move on in the hierarchy of geometric structures 
considered at the beginning of this Introduction. 
 For families of Dirac operators on manifolds
with boundary, the index theorem is due to Bismut and Cheeger (\cite{BC},\cite{BC2}) and, more
generally, to Melrose and Piazza (\cite{MPI}, \cite{MPII}). See also \cite{melrose-rochon}
for the pseudodifferential case. 
The numeric index theorem on Galois coverings of a compact
manifold with boundary was established by Ramachandran \cite{Ra},
whereas the corresponding higher index problem  was solved by 
 Leichtnam and Piazza \cite{LPMEMOIRS}, \cite{LPGAFA}, following a
conjecture of Lott \cite{LottII}. See \cite{LPFOURIER} for a survey.
The numeric index theorem on measured foliations was established by Ramachandran 
 in \cite{Ra}. See also \cite{antonini}
for the cylindrical treatment.
Finally, under a polynomial growth assumption on the group $\Gamma$,
Leichtnam and Piazza
\cite{LPETALE}
extended Connes' higher index theorem to foliated bundles with boundary, using an extension of Melrose'
$b$-calculus and the Gorokhovsky-Lott superconnection approach. For general foliations,
but always under a polynomial growth assumption, see  also the recent contribution \cite{zadeh}. 
Notice that, by a result of Plante, foliations with leaves of polynomial growth are {\it measured}.

\medskip
{\it Geometric} applications of the above results are too numerous
to be treated here.

\medskip
The structure of the (higher) index formulae in all of these contributions is precisely the
one displayed  by the classic Atiyah-Patodi-Singer index formula recalled above, see
\eqref{aps}. Thus there is a local contribution, which is  the one 
appearing in the corresponding higher index formula in the closed case, and there is a
boundary-correction term, which is a {\it higher eta invariant}. 
This higher eta invariant should be thought of as a secondary higher invariant
of the operator on the boundary (indeed, the index class
for the boundary operator is always zero). We remark that some
of the interesting geometric applications of the theory do employ 
this secondary invariant in order to tackle classification problems
\footnote{e.g: moduli spaces of metrics of positive scalar curvature;
diffeomorphism-type of closed orientable manifolds} that
cannot be treated by  ordinary higher indeces. Here we have in mind
classification problems 
 on  geometric structures {\it without} boundary: the Atiyah-Patodi-Singer higher
 index formulae enter in these investigations through the notion of {\it bordism}
 connecting two such geometric structures. See, for example \cite{BoGi},\cite{LPPSC},\cite{PS},
 \cite{PS2},\cite{ChW}.

\medskip
Now,  going back to the task of extending the Atiyah-Patodi-Singer index formula
to more general geometric structures, we make the crucial observation 
 that the polynomial growth assumption  in \cite{LPETALE}
 excludes many interesting \footnote{typically type III}
examples and 
higher indeces;  in particular it excludes the possibility of proving a Atiyah-Patodi-Singer  formula
for the Godbillon-Vey higher index.

\medskip
{\it One primary objective of this article is to prove such a result,
 thus establishing the first instance of  a higher  APS index theorem 
  on  type III foliations}. Consequently, we also 
 define  {\it a Godbillon-Vey eta invariant} on the boundary-foliation; this is
 a {\it type III eta invariant}, i.e. a type III secondary invariant for Dirac operators. 

\medskip
In tackling this specific index  problem we also develop what we believe is a new approach to 
index theory on geometric structures with boundary, heavily based
on the interplay between absolute and relative pairings. We think  that this
new method can be applied to a variety of situations.

\medskip
Notice that relative pairings in K-theory and cyclic cohomology have already been successfully 
employed in the study of geometric and topological invariants of
elliptic operators. We  wish to mention here the paper by Lesch, Moscovici and Pflaum 
\cite{lmpflaum}; in this interesting article the absolute and relative pairings  associated to a suitable
short exact sequence of algebras (this is a  short exact sequence of parameter
dependent pseudodifferential operators)
are used in order to define and study a generalization of the
divisor flow of Melrose on a closed compact manifold, see  \cite{melrose-eta} and also \cite{LPfla}.

\medskip
 Let us give a very short account of our main results. First of all, it is clear from
 the structure of the Atiyah-Patodi-Singer index formula \eqref{aps}
 that one of the basic tasks
 in the theory is to split precisely the interior contribution from the boundary contribution
 in the higher index formula.
 We look at operators on the boundary through the translation invariant operators
 on the associated infinite cyclinder; by Fourier transform these two pictures are equivalent.
 We solve the Atiyah-Patodi-Singer higher index problem on a foliated bundle
 with boundary $(X_0,\F_0)$, $X_0=\tM\times_\Gamma T$,
  by solving the associated $L^2$-problem on the associated
 foliation with cylindrical ends $(X,\F)$.
 Thus, after explaining the geometric set-up in Section \ref{sec:data} ,
  we begin by defining a short exact sequence of $C^*$-algebras
 $$0\to C^* (X,\F)\to A^* (X,\F) \to B^* (\cyl (\pa X),\F_{\cyl})\to 0\,.$$
 This  is an extension by the foliation  $C^*$-algebra  $C^* (X,\F)$ of a suitable
 algebra of {\it translation invariant operators} on the cylinder; we call it the Wiener-Hopf extension. 
 We briefly denote the Wiener-Hopf extension as $0\to C^* (X,\F)\to A^*\to B^* \to 0$.
 These $C^*$-algebras are the receptacle
 for the two $C^*$-index classes we will be working with. Thus, given a $\Gamma$-equivariant
 family of Dirac operators $(D_\theta)_{\theta\in T}$ with invertible boundary family
 $(D^{\pa}_\theta)_{\theta\in T}$ we prove that there exist an index class $\Ind (D)\in K_* (C^* (X,\F))$
  and a relative index class
 $\Ind (D,D^{\pa})\in 
K_* (A^*,B^*)\,.$
 The higher Atiyah-Patodi-Singer index problem for the Godbillon-Vey cocycle
 consists in proving that there is a well defined
 paring $\langle \Ind (D), [\tau_{GV}] \rangle$ and giving a formula for it, with
 a structure similar to the one displayed by \eqref{aps}. Now, as in the case of Moriyoshi-Natsume,
 $\tau_{GV}$ is initially defined on the  small algebra $J_c (X,\F)$ of $\Gamma$-equivariant
 smoothing kernels
 of $\Gamma$-compact support; however,
  because of the structure of the parametrix on manifolds
 with cylindrical ends, there does {\it not} exist an index class in $K_* (J_c (X,\F))$. Hence,  even {\it defining}
 the  index pairing is not obvious. We shall solve this problem by showing that
there exists a holomorphically closed intermediate subalgebra $\mathbf{\mathfrak{J}}$ containing the  index class $\Ind (D)$ but such that $\tau_{GV}$ extends. More on this in a moment.
This point involves elliptic theory on manifolds with cylindrical ends in an essential way.
  Incidentally, we develop this theory without any reference to pseudodifferential operators,
 using nothing more than the functional calculus for Dirac operators on complete manifolds.

 \smallskip
 Once the higher Godbillon-Vey index is defined,  we search for an index formula for it. Our main idea
   is to show that such a formula is a direct consequence of the equality 
   \begin{equation}\label{ab=rel}
   \langle \Ind (D), [\tau_{GV}] \rangle= \langle \Ind (D,D^{\pa}), [(\tau_{GV}^r,\sigma_{GV})] \rangle
  \end{equation}
 where on the right hand side a new mathematical object, the {\it relative} Godbillon-Vey cocycle,
  appears. 
 The relative Godbillon-Vey cocycle is built out of the usual  Godbillon-Vey
 cocycle by means of a very natural procedure. First, we proceed algebraically.
 Thus we first look at a subsequence of $0\to C^* (X,\F)\to A^*\to B^*\to 0$ made of
 small algebras, call it $0\to J_c (X,\F) \to A_c \to B_c\to 0$; $J_c (X,F)$
 are, as above, the $\Gamma$-equivariant smoothing kernels of $\Gamma$-compact support;
  $B_c$ is made of $\Gamma\times\RR$-equivariant smoothing kernels
  on the cylinder of $\Gamma\times\RR$-compact support.  The $A_c$ cyclic 2-cochain $\tau_{GV}^r$
 is obtained from $\tau_{GV}$ through a regularization \`a la Melrose. The $B_c$ cyclic 3-cocycle
 $\sigma_{GV}$ is obtained by {\it suspending} $\tau_{GV}$ on the cylinder with Roe's 
 1-cocycle. We call this $\sigma_{GV}$
 the {\it eta cocycle} associated to $\tau_{GV}$.
 One proves, but it is not quite obvious, that $(\tau_{GV}^r,\sigma_{GV})$ is a 
 relative cyclic 2-cocycle for $A_c \to B_c$. We obtain in this way a relative cyclic cohomology
 class $[\tau_{GV}^r,\sigma_{GV}]\in HC^2 (A_c, B_c).$ 
 All of this is explained in Section \ref{section:algebraic}; at the end of this section we
 also  explain how this natural procedure can be 
 extended to other higher indices, producing each time an associated eta cocycle.

\smallskip 
We remark here that for technical reasons having to do with
 the extension of these cocycles to suitable smooth subalgebras, see below, we shall
have to  consider the cyclic cocycle and the relative cyclic cocycle obtained from $\tau_{GV}$ and
 $(\tau_{GV}^r,\sigma_{GV})$ through the $S$ operation in cyclic cohomology, see  \cite{connes-ihes}:
  thus we consider
  $S^{p-1}\tau_{GV}$ and $(S^{p-1}\tau_{GV}^r, \frac{3}{2p+1}S^{p-1} \sigma_{GV})$
  obtaining in this way  a class in  $HC^{2p} (J_c)$ and a relative class in $HC^{2p} (A_c, B_c)$.
 With a small abuse of notation we still denote these cyclic $2p$-cocycles by $\tau_{GV}$
 and $(\tau_{GV}^r,\sigma_{GV})$.

 \smallskip
 Once the algebraic theory is clarified, we need to pair the  class
 $[\tau_{GV}]\in H^{2p} (J_c)$ and the relative class $[\tau_{GV}^r,\sigma_{GV}]\in HC^{2p} (A_c,B_c)$ with the corresponding index classes $\Ind (D)\in K_* (C^*(X,\F))$ and  $\Ind (D,D^{\pa})\in K_* (A^*,B^*)$. 
 To this end we construct an {\it intermediate}
 short exact subsequence $0\to \mathbf{\mathfrak{J}}  \to \mathbf{\mathfrak{A}}
 \to \mathbf{\mathfrak{B}}\to 0$ of Banach algebras, sitting half-way between 
 $0\to C^* (X,\F)\to A^*\to B^*\to 0$  and  $0\to J_c (X,\F) \to A_c \to B_c\to 0$. Much work is needed 
 in order to define such a subsequence and prove that  $$\Ind(D)\in K_* (\mathbf{\mathfrak{J}} )\cong K_* (C^*(X,\F))\,,\quad
 \Ind (D,D^\pa)\in K_* ( \mathbf{\mathfrak{A}}, \mathbf{\mathfrak{B}})\cong K_* (A^*, B^*)\,.
 $$
 Even more work is needed in order to establish that the Godbillon-Vey cyclic $2p$-cocycle $\tau_{GV}$
 and the relative cyclic $2p$-cocycle $(\tau_{GV}^r,\sigma_{GV})$
 extend for $p$ large enough from $J_c$ and $A_c\to B_c$ to $\mathbf{\mathfrak{J}}$ and $\mathbf{\mathfrak{A}}\to \mathbf{\mathfrak{B}}$, 
 thus defining  elements 
 $$[\tau_{GV}] \in HC^{2p} (\mathbf{\mathfrak{J}})\quad\text{and}\quad [\tau_{GV}^r,\sigma_{GV}]\in HC^{2p} (\mathbf{\mathfrak{A}}, \mathbf{\mathfrak{B}}).$$
 
 \smallskip
  We have now made sense of both sides of the equality \eqref{ab=rel}
 $\langle \Ind (D), [\tau_{GV}] \rangle= \langle \Ind (D,D^{\pa}), [(\tau_{GV}^r,\sigma_{GV})] \rangle
  $.
 The equality itself is proved by establishing and using the excision formula:
 if 
 $\alpha_{{\rm ex}}: K_* (\mathbf{\mathfrak{J}})\to  K_* (\mathbf{\mathfrak{A}}, \mathbf{\mathfrak{B}})$
 is the excision isomorphism, then
 $$\alpha_{{\rm ex}} ( \Ind (D)) = \Ind (D,D^{\pa})\quad\text{in}\quad  K_* (\mathbf{\mathfrak{A}}, \mathbf{\mathfrak{B}})\,.$$
 The  index formula is obtained by writing explicitly  the relative pairing
$ \langle \Ind (D,D^{\pa}), [(\tau_{GV}^r,\sigma_{GV})] \rangle$  in terms of the graph projection
 $e_D$,
 multiplying the operator $D$ by $s>0$ and taking the limit as $s\downarrow 0$.
 The final formula in the 3-dimensional case (always with an invertibility assumption
 on the boundary family)  reads:
 \begin{equation}\label{main-dim3-intro}
\langle \Ind (D), [\tau_{GV}] \rangle= 
\int_{X_0} \omega_{GV} - \eta_{GV}\,,
\end{equation}
with $\omega_{GV}$ equal, as in the closed case,  to (a representative of) 
the Godbillon-Vey class $GV$ and
 \begin{equation}\label{main-eta-intro}
 \eta_{GV}:= \frac{(2p+1)}{p!} \int_0^{\infty}\sigma_{GV} ([\dot{p_t},p_t],p_t,\dots,p_t,p_t)dt\,,
\end{equation} with $p_t:= e_{tD^{{\rm cyl}}}$ the graph projection associated to 
the cylindrical Dirac family $tD^{{\rm cyl}}$.
Observe that by Fourier transform {\it the Godbillon-Vey eta invariant} 
$\eta_{GV}$ only depends on the boundary
family $D^\pa\equiv (D^\pa_\theta)_{\theta\in T}$.
 Notice, finally, that this is  a {\it complete} solution to the Godbillon-Vey higher index problem
on foliated bundles with boundary, in the spirit of Connes and Moriyoshi-Natsume.

\medskip
The paper is organized as follows. In Section \ref{sec:data} we explain
our geometric data. Section \ref{sect:operators} is devoted to a discussion
of the operators involved in our analysis. In Section \ref{sec:algebras} we define the Wiener-Hopf
extension $0\to C^* (X,\F)\to A^*\to B^*\to 0$. In Section \ref{section:algebraic}  we 
restrict our analysis to a subsequence $0\to J_c\to A_c\to B_c\to 0$ of small dense subalgebras.
We begin
by defining the 1-eta cocycle associated to the usual trace-cocycle $\tau_0$. This is nothing
but Roe's 1-cocycle $\sigma_1$; we define a relative 0-cocycle for $A_c\to B_c$ by considering
$(\tau^r_0,\sigma_1)$, with $\tau^r_0$ the regularized trace of Melrose ($b$-trace). 
We also discuss the relation of all this with Melrose' formula for the $b$-trace
of a commutator. Next we pass to the Godbillon-Vey cocycle $\tau_{GV}$, defining the associated
eta 3-cocycle $\sigma_{GV}$ on $B_c$ and the associated {\it relative} Godbillon-Vey cocycle 
$(\tau^r_{GV},\sigma_{GV})$, with $\tau^r_{GV}$ defined via Melrose'  regularization.
In the last Subsection of this Section we also discuss briefly more general relative cocycles. In Section \ref{sec:shatten}
we construct the intermediate short exact sequence $0\to \mathbf{\mathfrak{J}}  \to \mathbf{\mathfrak{A}}
 \to \mathbf{\mathfrak{B}}\to 0$. In Section \ref{sec:index} we define the index class $\Ind D\in K_0 (C^*(X,\F))$
  and the relative 
 index class $\Ind (D,D^{\partial})\in K_0 (A^*,B^*)$ and we prove that they correspond under excision. In Section \ref{section:smoothpairings} we prove that the two Godbillon-Vey
 cocycles extend to the subalgebras in the exact sequence $0\to \mathbf{\mathfrak{J}}  \to \mathbf{\mathfrak{A}}
 \to \mathbf{\mathfrak{B}}\to 0$. We also show how to smooth-out the two $C^*$-index classes and define
 them directly in $K_* (\mathbf{\mathfrak{J}} )$ and $K_* (\mathbf{\mathfrak{A}}, \mathbf{\mathfrak{B}})$.
 In Section
 \ref{section:indextheorems} 
 we then proceed to state and prove the main result of the paper. We also make some further remarks;
 in particular we explain  how to get the classic 
 Atiyah-Patodi-Singer index theorem  from these relative-pairings arguments. 
 The proof  of the APS formula  using relative-pairings
 techniques and the Roe's 1-cocycle $\sigma_1$ was obtained by the first author in 1988 and 
 announced in \cite{AMS}. 
Long and technical proofs have been collected in a separate
 Section, Section \ref{section:proofs}. There  is one Appendix in which we 
 summarize the results from \cite{MN} needed in order to define the Godbillon-Vey cocycle $\tau_{GV}$
 in the closed case.

\medskip
The results of this paper were
announced in 
\cite{MoPi0}.

\medskip
\noindent
{\bf Acknowledgements.} Most of this work has been done while the first author was
visiting Sapienza Universit\`a di Roma and the second author was visiting Keio University
and Nagoya University. We thank the Japan Society for the Promotion of Science (JSPS), 
Grants-in-Aid for Scientific Research, 
and the 21st century COE program at Keio 
for sponsoring most of these visits. Further financial support was provided by
{\it Istituto Nazionale di Alta Matematica}, through the GNSAGA, and the {\it Ministero dell'Istruzione,
dell'Universit\`a e della Ricerca (MIUR)} through the project {\it Spazi di moduli e Teoria di Lie}.
Part of this research was also carried out while the two authors were visiting jointly
 the Chern Institute in Tianjin and  the {\it Institut de Math\'ematiques de Jussieu} 
in Paris (\'Equipe
Alg\`ebres d'Op\'erateurs). 
We thank these institutions for  hospitality and financial support.
 Finally, it is a pleasure to thank Sergio Doplicher, Sacha Gorokhovsky, Eric Leichtnam, Henri Moscovici,
Toshikazu Natsume, John Roe and Xiang Tang for helpful 
discussions. 

\section{{\bf Geometry of foliated bundles}}\label{sec:data} 
\subsection{Closed manifolds} 
We shall denote by $N$ a closed orientable compact manifold.
We consider a Galois $\Gamma$-cover  $\Gamma\to \tN \to N$, with $\Gamma$ acting on the right, and 
$T$, a smooth oriented compact manifold with a left action of $\Gamma$
which is assumed to be by orientation preserving diffeomorphisms
 and locally faithful, as in \cite{MN}, that is: 
if $\gamma \in \Gamma$ acts as the identity map on an open set in $T$, 
then $\gamma$ is the identity element in $\Gamma$. 
See also \cite{ben-pia}. 
We let $Y=\tN\times_\Gamma T$; 
thus $Y$ is the quotient of $\tN\times T$ by the $\Gamma$-action 
$$ (\tilde{n},\theta)\gamma:= (\tilde{n}\gamma,\gamma^{-1}\theta)\,.$$
$Y$ is foliated by the images under the quotient
map of the fibers of the trivial fibration $\tN\times T\to T$ and is referred to as a
{\it foliated $T$-bundle}.
We use the notation $(Y,\mathcal{F})$ when we want to stress the foliated structure of $Y$. Finally,
we consider
 $E\to Y$, a hermitian complex vector bundle on $Y$ as well as  $\widehat{E} \to \tN\times T $, the latter being
 the
  $\Gamma$-equivariant vector bundle obtained
by lifting $E$, a bundle on $Y\equiv \tN\times_\Gamma T$,  to $\tN\times T$. 

\subsection{Manifolds with boundary}
%
Let now 
$(M,g)$ be a riemannian manifold with boundary; the metric is assumed to be of product type in a collar
neighborhood $U\cong [0,2]\times \partial M$ of 
the boundary. Let $\tM$ be a Galois $\Gamma$-cover of $M$; we let $\tilde g$ be the lifted metric. We also consider  $\partial \tM$, the
boundary of $\tM$.
Let $X_0=\tM\times_\Gamma T$; this is a manifold with boundary and the boundary $\partial X_0$ is equal to $\partial \tM \times_\Gamma T$.
$(X_0,\mathcal F_0)$ denotes  the associated foliated bundle. The leaves of
$(X_0,\mathcal F_0)$ are manifolds with boundary
endowed with a product-type metric. The boundary $\partial X_0$ inherits a  foliation $\mathcal F_\partial$.
The  cylinder $\RR\times \partial X_0$ also inherits a foliation $\mathcal F_{{\rm cyl}}$, obtained by crossing 
the leaves of $\mathcal F_\partial$ with $\RR$. Similar considerations apply to the half cylinders
$(-\infty,0] \times \partial X_0$ and $[0,+\infty) \times \partial X_0 $.
We shall consider a complex hermitian vector bundle on $X_0$ and we shall assume the usual product structure
near the boundary: we adopt without further comments the identification explained, for example,
in \cite{Melrose} and adopted also in \cite{MPI}  and \cite{LPETALE}.

\subsection{Manifolds with cylindrical ends. Notation.}\label{subsection:cyl+notation}
%
We consider 
$\tV:= \tM\cup_{\partial \tM} \left(   (-\infty,0] \times \partial\tM \right),$
endowed with the extended metric and the obviously extended $\Gamma$ action along the cylindrical
end. Notice incidentally that we obtain in this way a $\Gamma$-covering
\begin{equation}\label{cylindrical-covering}
\Gamma\to \tV \to V, \quad \text{with}\quad V:= M\cup_{\partial M} \left(   (-\infty,0] \times \partial M \right).
\end{equation}
We consider $X:= \tV\times_\Gamma T$; this is a foliated bundle, with leaves manifolds with cylindrical ends.
We denote by $(X,\mathcal F)$ this foliation. Notice that $X=X_0 \cup_{\partial X_0} \left(   (-\infty,0] \times \partial X_0 \right)$;
moreover the foliation $\mathcal F$ is obtained by extending $\mathcal F_0$ on $X_0$ to $X$ via the product cylindrical
foliation $\mathcal F_{{\rm cyl}}$ on $(-\infty,0] \times \partial X_0$. We can write more suggestively:
$$ (X,\mathcal F)= (X_0,\mathcal F_0)\cup_{(\partial X_0,\mathcal F_\partial)} 
\left(   (-\infty,0] \times \partial X_0,  \mathcal F_{{\rm cyl}} \right).$$
For $\lambda>0$ we shall also consider the finite cyclinder $\tV_\lambda = \tM\cup_{\partial \tM} \left(   [-\lambda,0] \times \partial\tM \right)$
and the resulting foliated manifold $(X_\lambda,\mathcal F_\lambda)$.
Finally, with a small abuse \footnote{The abuse of notation is in  writing $\cyl (\pa X)$ for $ \RR\times \partial X_0$ whereas we should
really write $\cyl (\pa X_0)$.}, we  introduce the notation:
\begin{equation}\label{cyl-notation}
\cyl (\pa X):= \RR\times \partial X_0\;,\;\;\;
\cyl^- (\pa X):=(-\infty,0] \times \partial X_0\;\;\text{ and }\;\;\cyl^+ (\pa X):=
[0,+\infty) \times \partial X_0\,.
\end{equation}
The foliations induced on $\cyl (\pa X)$, $\cyl^\pm (\pa X)$ by $\F_{\pa}$ will be denoted by
$\F_{\cyl}$, $\F_{\cyl}^\pm$; we obtain in this way foliated bundles
 \begin{equation*}
\cyl (\pa X,\F_{\cyl} )\;,\;\;\;
(\cyl^- (\pa X),\F_{\cyl}^-)\;\;\text{ and }\;\;(\cyl^+ (\pa X),\F_{\cyl}^+).
\end{equation*}

\subsection{Holonomy groupoid}
%
We consider the 
groupoid $G:=(\tV\times\tV\times T)/\Gamma$
with $\Gamma$ acting diagonally; $G^{(0)}:= X$ and the source map and the range map
are defined  by 
$s[y,y',\theta]=[y',\theta]$, $ r[y,y',\theta]=[y,\theta]$. 
Since the action on $T$ is assumed to be locally faithful, 
we know that $(G,r,s)$
is isomorphic to the 
holonomy groupoid of the foliation $ (X,\mathcal F)$. 
In the sequel, we shall directly call 
$(G,r,s)$ the holonomy groupoid.
If $E\to X$ is a complex vector bundle on $X$, with product structure along the cylindrical end as above, then
we can consider the bundle over $G$ equal to $(s^* E)^*\otimes r^*E$; this bundle is sometime denoted
$\mathrm{END}(E)$. If $F$ is a second complex  vector bundle on $X$, we can likewise consider the bundle
$\mathrm{HOM}(E,F): =  (s^* E)^* \otimes r^*F$.
Finally, we consider  the maps
$\hat{r}, \hat{s}: \tV\times\tV\times T\to \tV\times T$, $\hat{r}(y,y',\theta)=(y,\theta)$,
 $\hat{s}(y,y',\theta)=(y',\theta)$ and, more importantly, the bundles
 $\mathrm{END}(\widehat{E}):= (\hat{s}^*\widehat{E})^*\otimes (\hat{s}^*\widehat{E})$
 and  $\mathrm{HOM}(\widehat{E}, \widehat{F}):= (\hat{s}^*\widehat{E})^*\otimes (\hat{s}^*\widehat{F})$.

\subsection{The Godbillon-Vey differential form.}\label{subsect:gv-form}
%
Following \cite{MN},   we describe the explicit representative of 
the Godbillon-Vey class as a differential form 
in terms of the modular function of  the holonomy groupoid. 
Let $X_0=\tM\times_\Gamma T$. 
In this section $X_0$ can also be a closed manifold, namely, a compact manifold without boundary. 
Assume that $T$ is one-dimensional, and 
take an arbitrary 1-form $\omega$ on $X_0$ 
defining the codimension-one foliation $\mathcal F_0$. 
Due to the integrability condition, there
exists a 1-form $\eta$  such that $d\omega = \eta \wedge \omega$.  
The Godbillon-Vey class for $\mathcal F_0$ is, by definition, 
the de Rham cohomology class  given by 
$\eta \wedge d\eta$, 
denoted by $GV$; thus $GV:=[ \eta \wedge d\eta]\in H^3_{{\rm dR}}(X_0)$.
We shall explain another description of $GV$ 
in terms of the modular function of the holonomy groupoid. 

Consider the product space $\tM\times T$, 
which is a covering of  $X_0$. 
Choose a volume form $d\theta$ on $T$; it is in general impossible to choose
$d\theta$  to be, in addition, $\Gamma$-invariant.
Then 
$d\theta$ yields a defining 1-form for 
the foliation (which is in fact a fibration) obtained by lifting the foliation $\mathcal F_0$. 
The de Rham complex on $\tM\times T$ is isomorphic to the graded tensor
product $\Omega^*(\tM)\otimes \Omega^*(T)$ and 
accordingly  the exterior differential on $\tM\times T$ 
splits  as
\begin{equation}\label{d-decomposition}
d_{\tM\times T}= d + (-1)^p d_T  
\end{equation}
on $\Omega^p(\tM)\otimes \Omega^q(T)$, 
with  $d$ and $d_T$ denoting respectively  the exterior differentials along $\tM$ and $T$. 
Let us take the volume forms  $\omega$ and  $\Omega$ respectively on $M$ and $X_0$ and take the pullbacks
$\tilde{\omega}$ and  $\tilde{\Omega}$ to 
$\tM$ and $\tM\times T$. These are $\Gamma$-invariant volume forms.
The modular function of the holonomy groupoid is defined as  the Radon-Nikodym derivative of 
the two volume forms on  $\tM\times T$:
\begin{equation}\label{modular function}
\psi = 
\frac{\tilde{\omega}\times d\mu}{\tilde{\Omega}}. 
\end{equation}
Notice that $\psi$ has values in $\RR^+$ since $\Gamma$ acts by orientation preserving
diffeomorphisms.
Set 
\begin{equation}\label{phi}
\varphi =\log \psi.
\end{equation}

\begin{proposition}[\cite{MN}, p.504]
\label{MN_p.504} 
The 3-form $\omega_{GV}= d\varphi\wedge d d_T\varphi= -d\varphi\wedge d_Td\varphi$
 is  $\Gamma$-invariant  and closed on $\tM\times T$.
The Godbillon-Vey class of $\mathcal F_0$  is represented by 
$\omega_{GV}$ in $H^3_{{\rm dR}}(X_0)$. 
\end{proposition}

\section{{\bf Operators}}\label{sect:operators}

\subsection{Equivariant families of integral operators} 
 We consider $C_c (X,\mathcal F):= C_c (G)$ and more generally
\begin{equation}\label{eq:small-cstar}
C_c (X,\mathcal F;E) := C_c (G,(s^* E)^*\otimes r^*E )\equiv C_c (G,\mathrm{END}(E))
\end{equation}
with its well known  *-algebra structure given by convolution. 
Given an additional vector bundle $F$,
we can also consider 
\begin{equation}\label{eq:small-cstar-bis}
C_c (X,\mathcal F;E,F) := C_c (G,(s^* E)^*\otimes r^*F )\equiv C_c (G,\mathrm{HOM}(E,F))
\end{equation}
which is a left module over $C_c (X,\mathcal F;E)$ and a right module over $C_c (X,\mathcal F;F).$

The $*$-algebra $C_c (X,\mathcal F)$ can also be defined  as the space of $\Gamma$-invariant continuous functions on 
$\tV\times\tV\times T$ with $\Gamma$-compact support, i.e. with support which is compact in
$(\tV\times\tV\times T)/\Gamma$. A similar description holds for $C_c (X,\mathcal F;E)$.
 Notice, in particular, that given an  element $k$ in $C_c (X,\mathcal F;E)$
there exists a $\lambda (k)\equiv \lambda>0$ such that $k$
is   identically zero outside $\tV_\lambda \times\tV_\lambda \times T\subset \tV\times\tV\times T$.

An element $k\in C_c (X,\mathcal F)$ defines in a natural way an equivariant
family of integral operators $(k (\theta))_{\theta\in T}$.

\subsection{Dirac operators}\label{subsect:dirac0} 
We begin with a closed foliated bundle $(Y,\F)$, with $Y=\tN\times_\Gamma T$.
We are also given a $\Gamma$-equivariant complex vector bundle $\we$ on $\tN\times T$,
or, equivalently, a complex vector bundle on $Y$. We assume that $\we$ has a $\Gamma$-equivariant
vertical Clifford structure. We obtain in this way a $\Gamma$-equivariant family of Dirac operators $(D_\theta)_{\theta\in T}$
that will be simply denoted by $D$. 
\footnote{Observe that this family is denoted $\tilde D$ both in \cite{MN} and \cite{ben-pia}, the symbol $D$ being employed 
for the longitudinal operator induced on the quotient $Y=\tN\times_\Gamma T$. However, in this paper we shall work
exclusively with the $\Gamma$-equivariant picture, which is why we don't use the tilde notation.}

If $(X_0,\F_0)$,
$X_0=\tM\times_\Gamma T$, is a
foliated bundle with boundary, as in the previous section, then we shall assume the relevant
geometric structures to be of product-type near the boundary.
If $(X,\F)$ is the associated foliated bundle with cylindrical ends, then we shall extend all the structure 
in a constant way along the cylindrical ends. We shall eventually assume $\tM$ to be of even dimension, the
bundle $\we$ to be $\ZZ_2$-graded and the Dirac operator to be odd and formally self-adjoint.
We denote by $D^\pa\equiv (D^\pa_\theta)_{\theta\in T}$ the boundary family defined by $D^+$.
This is a $\Gamma$-equivariant family of formally self-adjoint first order elliptic differential operators on a
complete manifold.  
We denote by $D^{{\rm cyl}}$ the operator induced by  $D^\pa\equiv (D^\pa_\theta)_{\theta\in T}$ 
on the cylindrical foliated manifold $(\cyl(\pa X),\F_{{\rm cyl}})$; $D^{{\rm cyl}}$ is  $\RR\times\Gamma$-equivariant.
We refer to \cite{MN} \cite{LPETALE} and also \cite{ben-pia} for precise definitions.

 \subsection{Pseudodifferential operators}\label{pseudo-index-closed-0}
 Let $(Y,\F)$, $Y=\tN\times_\Gamma T$, be a closed foliated bundle.
Given vector bundles $E$ and $F$ on $Y$ with lifts $\widehat{E}$, $\widehat{F}$ on $\tN\times T$,
we can define the space of  $\Gamma$-compactly supported pseudodifferential operators of order $m$,
denoted here, with a small abuse of notation,
$\Psi^m_c(G;E,F)$. 
\footnote{The abuse consists in omitting the hats in the notations. It should be added here that the
notation for this space of operators is not unique. In \cite{MN} $\Psi^m_c(G;\widehat{E},\widehat{F})$  is simply denoted
as $\Psi_\Gamma^* (\widehat{E},\widehat{F})$ whereas it
 is denoted $\Psi^*_{\rtimes,c}(\tN\times T; \widehat{E},\widehat{F})$ 
in \cite{LPETALE} with $\rtimes$ denoting equivariance
and $c$ denoting again  {\it of $\Gamma$-compact support}   }. An element   
$P\in \Psi^m_c(G;E,F)$
should be thought of as a $\Gamma$-equivariant family of psedodifferential operators, $(P(\theta))_{\theta\in T} $
with Schwartz kernel $K_P$, a distribution on $G$, of compact support. See \cite{MN}
and \cite{ben-pia} for more details.

The space 
$\Psi^\infty_c (G;E,E):= \bigcup_{m\in \ZZ} \Psi^m_c(G;E,E)$
is a filtered algebra.
Moreover, assuming $E$ and $F$ to be hermitian and assigning to $P$ its
formal adjoint $P^*= (P_\theta^*)_{\theta\in T}$ gives 
$\Psi^\infty_c (G;E,E)$
the structure of an involutive algebra; the formal adjoint of an element  
$P\in \Psi^m_c (G;E,F)$
 is in general  an alement in $ \Psi^m_c (G;F,E)$.

\section{{\bf Wiener-Hopf extensions}}\label{sec:algebras}

\subsection{Foliation $C^*$-algebras}\label{subsec:cstar}
The foliation $C^*$-algebra 
$C^*(X,\mathcal F)$ is defined as the  completion of $C_c (X,\mathcal F)$ with respect
to $\| k \|_{C^*}:= \sup_{\theta\in T} \| k (\theta) \|$, the norm
on the right hand side being equal to the $L^2$-operator norm on 
$L^2(\tV\times\theta)$. A similar definition holds for $C_c (X,\mathcal F;E)$.
For more on this foundational material see, for example, \cite{MN}
and \cite{ben-pia}.
It is proved in \cite{MN} that
$C^*(X,\mathcal F;E)$ is isomorphic to the $C^*$-algebra of compact operators of a $C(T)\rtimes \Gamma$-Hilbert module
$\mathcal E$. The Hilbert module $\mathcal E$ is obtained by completing 
$C^\infty_c (\tV\times T, \widehat{E})$, endowed with its $C(T)\rtimes_{{\rm alg}} \Gamma$-module structure
and $C(T)\rtimes_{{\rm alg}} \Gamma$-valued inner product, with respect to the $C(T)\rtimes \Gamma$-norm. 
Once again, see  \cite{MN}
and \cite{ben-pia} for details: summarizing 
\begin{equation}\label{iso-comp-cstar}
C^*(X,\mathcal F;E)\cong \KK (\mathcal E)\subset \mathcal L (\mathcal E)\,.
\end{equation}

\subsection{Foliation von Neumann algebras}\label{subsect:vonneumann}
Consider the family of Hilbert spaces $\H:= (\H_\theta)_{\theta\in T}$, with
$\H_\theta := L^2 (\tilde{V}\times \{\theta\}, E_\theta)$.
Then $C_c (\tilde{V}\times T)$ is a continuous field inside $\H$, that is, a linear subspace in the
space of measurable sections of $\H$ satisfying a certain number of properties (see 
\cite{Connes-survey}, pag 576 for the details). Let $\End (\H)$ the space of measurable families
of bounded operators $T=(T_\theta)_{\theta\in T}$, where bounded means that each
$T_\theta$ is bounded on $\H_\theta$.
Then $\End (\H)$ is a $C^*$-algebra, in fact a von Neumann algebra, equipped with the norm
$$\| T \|:= {\rm ess.} \sup \{\|T_\theta\|\,;\theta\in T\}\,$$
with $\|T_\theta\|$ the operator norm. We also denote by $\End_\Gamma (\H)$ the subalgebra
of $\End (\H)$ consisting of $\Gamma$-equivariant measurable families  of operators. This is a von Neumann algebra which is, by definition, the {\it foliation von Neumann algebra} associated to  $(X,\F)$; it is often
denoted $W^* (X,\F)$.
We set $C^*_\Gamma (\H)$ the closure of $\Gamma$-equivariant families $T=(T_\theta)_{\theta\in T}
\in \End_\Gamma (\H)$ preserving the continuous field  $C_c (\tilde{V}\times T)$. In \cite{MN}, Section 2, 
it is proved that the foliation $C^*$-algebra $C^* (X,\F)$ is isomorphic to a $C^*$-subalgebra of
$C^*_\Gamma (\H)\subset \End_\Gamma (\H)$ \footnote{The 
 $C^*$-algebra $C^*_\Gamma (\H)$ was denoted $\mathfrak{B}$ in \cite{MN}}. Notice, in particular, 
that an element in $C^* (X,\F)$ 
can be considered as a $\Gamma$-equivariant
family of operators. 

\subsection{Translation invariant operators}\label{subsec:translation}
Recall $\cyl (\pa X):=\RR\times \partial X_0\equiv (\RR\times\partial \tM)\times_\Gamma T$ with $\Gamma$ acting trivially
in the $\RR$-direction of $(\RR\times\partial \tM)$.
We consider the foliated cylinder $(\cyl (\pa X), \mathcal F_{{\rm cyl}})$ and its holonomy groupoid 
$G_{{\rm cyl}}:=((\RR\times\partial \tM)\times (\RR\times\partial \tM)\times T)/\Gamma$
(source and range maps are clear). Let $\RR$ act trivially on $T$; then $(\RR\times\partial \tM)\times (\RR\times\partial \tM)\times T$
has a $\RR\times\Gamma$-action, with $\RR$ acting by translation on itself.
We consider the *-algebra 
$B_c (\cyl (\pa X),\mathcal F_{{\rm cyl}})\equiv B_c$
\begin{equation}\label{eq:tras-algebra}
B_c := \{k\in C((\RR\times\partial \tM)\times (\RR\times\partial \tM) \times T); k 
\text{ is } \RR\times \Gamma\text{-invariant}, k \text{ has } \RR\times\Gamma\text{-compact  support}\}
\end{equation} 
The product is by convolution. An element $\ell$ in 
$B_c$  defines
a $\Gamma$-equivariant
family $(\ell(\theta))_{\theta\in T}$ of translation-invariant operators. The completion of 
 $ B_c$ with respect to the obvious $C^*$-norm (the sup over $\theta$
of the operator-$L^2$-norm of $\ell(\theta)$) gives us a $C^*$-algebra that will be denoted $B^{*} (\cyl (\pa X),\F_{{\cyl}})$
or more briefly $B^{*}$.
Notice that we can in fact define  $B^{*} (\cyl (Y),\F_{{\cyl}})$ for any foliated flat bundle
$(Y,\F)$, with 
$Y=\tN\times_\Gamma T$.

\begin{proposition}\label{prop:bstar-structure}
Let $(Y,\F)$, with 
$Y=\tN\times_\Gamma T$, a foliated flat bundle without boundary.
Let us denote by $\RR_{\Delta}$ the group $\RR$ acting diagonally by translation on 
$\RR\times\RR$. Consider the quotient group
$(\RR\times\RR)/\RR_{\Delta}$ which is isomorphic to $\RR$. Consider the quotient groupoid 
$G_{\cyl}/\RR_{\Delta}$. Then $B^{*} (\cyl (Y),\F_{{\cyl}})= C^* (G_{\cyl}/\RR_{\Delta})$ and we have the following
$C^*$-isomorphisms:
\begin{equation}\label{bstar-structure}
 C^* (G_{\cyl}/\RR_{\Delta})\cong
C^* ( (\RR\times\RR)/\RR_{\Delta})\otimes C^* (Y,\F)\cong C^* ( \RR)\otimes C^* (Y,\F)
\end{equation}
\end{proposition}

\begin{proof}
The holonomy groupoid for $(\cyl (Y),\F_{{\cyl}})$ is $G_{\cyl}=
(\RR\times\tN\times \RR\times \tN\times T)/\Gamma$; directly from
the definition we see that $B^*$ is the $C^*$-algebra for the
quotient groupoid $G_{\cyl}/\RR_{\Delta}$ which is clearly isomorphic to $(\RR\times\RR)/\RR_{\Delta} \times (\tN\times\tN\times T)/\Gamma\equiv  (\RR\times\RR)/\RR_{\Delta} \times G(Y,\F)$. From these isomorphisms we can immediately end the proof. 
\end{proof}
\begin{remark}\label{remark:b*=compact}
We can interpret $B^* (\cyl (Y),\F_{{\cyl}})$ as the compact
operators of a suitable Hilbert $C^*$-module. Consider $\RR\times\tN\times T$ with its natural 
$\Gamma\times\RR$-action; consider
 $C^\infty_c ( \RR\times\tN\times T)$; we can complete it to a 
 Hilbert $C^*$-module $\mathcal{E}_{\cyl}$ over 
$(C(T)\rtimes\Gamma)\otimes C^* \RR$. Proceeding as in \cite{MN} one can prove that
there is a $C^*$-algebra isomorphism $B^* (\cyl (Y),\F_{{\cyl}})\simeq \KK (\mathcal{E}_{\cyl})$.
In particular, we see that $B^* (\cyl (Y),\F_{{\cyl}})$ can be seen as an ideal in the
$C^*$-algebra $\mathcal{L} (\mathcal{E}_{\cyl})$.
\end{remark}

\subsection{Wiener-Hopf extensions}\label{subsec:extension}
Recall the Hilbert  $C(T)\rtimes\Gamma$-module $\mathcal{E}$ and the $C^*$-algebras $\KK(\mathcal{E})$ and $\mathcal{L}( \mathcal{E})$.
Since  the $C(T)\rtimes\Gamma$-compact operators $\KK(\mathcal{E})$ are an ideal in  $\mathcal{L}( \mathcal{E})$
 we have the classical  short exact sequence of $C^*$-algebras
$$0\to  \KK(\mathcal{E})\hookrightarrow \mathcal{L}( \mathcal{E}) \xrightarrow{\pi} \mathcal{Q} ( \mathcal{E})\rightarrow 0$$
with $\mathcal{Q} ( \mathcal{E})= \mathcal{L}( \mathcal{E})/\KK(\mathcal{E})$ the Calkin algebra.
Let $\chi_\RR^0:\RR\to \RR$ be the characteristic function of $(-\infty,0]$; let $\chi_\RR:\RR\to\RR$ be a smooth function with values in $[0,1]$ such that:
\begin{equation}\label{smooth-cut-off}
\chi_\RR (t) = 
\begin{cases}
1
\qquad & \text{ for }\;\; t\leq -\epsilon
\\
0
\qquad  & \text{ for } \;\;t\geq 0.
\end{cases}
\end{equation}
for given $\epsilon >0$.
Let $\chi$ be the smooth function induced by $\chi_\RR$ on $X$; 
when we want to exhibit the dependence on $\epsilon$ clearly, 
we shall denote it by $\chi_{\epsilon}$. Similarly, we consider  $\chi_{{\rm cyl}}$, 
the smooth function induced by $\chi_\RR$ on  $\cyl (\pa X)$. 
Finally, let  $\chi^0$ and $\chi_{{\rm cyl}}^0$ be the functions
induced by the characteristic function $\chi_\RR^0$ on $X$ and $\cyl ( \pa X)$ respectively.
For $\lambda>0$, we shall also make use of the
 the real functions 
$\chi^\lambda$ and $\chi_{{\rm cyl}}^\lambda$, induced on $X$ and   $\cyl(\pa X)$ by 
$\chi^\RR_{(-\infty,-\lambda]}$, 
the characteristic function
of $(-\infty,-\lambda]$ in $\RR$; thus $\chi^\lambda$ is equal to 0 on the interior of $X_\lambda$ and equal to 1 on 
its complement in $X$. Similarly: $\chi_{{\rm cyl}}^\lambda$ is equal to zero on $(\lambda,+\infty)\times \pa X_0$
and equal to one on $(-\infty,\lambda]\times \pa X_0$.

\begin{lemma}\label{lemma:split}
There exists a bounded linear map
\begin{equation}\label{eq:split-2}
s : B^* \to \mathcal{L}(\mathcal{E})
\end{equation}
extending $s_c : B_c \to 
\mathcal{L} ( \mathcal{E})$,  $s_c (\ell) := \chi^0 \ell \chi^0 $ \footnote{For the precise meaning
of this composition, see Subsection  \ref{subsection:split}.}.
Moreover, the composition $\rho =\pi s $ induces an {\it injective} 
$C^*$-homomorphism
\begin{equation}\label{eq:split}
\rho : B^* \to \mathcal{Q}(\mathcal{E}).
\end{equation}
\end{lemma}
See
Section \ref{section:proofs}, Subsection \ref{subsection:split} for a detailed proof of  Lemma
\ref{lemma:split}; there we also
explain why $s_c$ is well defined.
A key tool  in the proof of the Lemma is the following Sublemma, stated here for later use but
proved in Subsection \ref{subsection:split}:
\begin{sublemma}\label{sublemma:basic1}
Let $\ell\in B_c$. Then $\chi^\lambda \ell (1-\chi^\lambda)$, $(1-\chi^\lambda) \ell \chi^\lambda$
and $[\chi^\lambda,\ell]$ are all of $\Gamma$-compact support on $\cyl(\pa X)$.
\end{sublemma}
In the sequel we shall use repeatedly this simple but crucial result.

\bigskip
We now consider  $\Im \rho$ as a $C^*$-subalgebra in $ \mathcal{Q} ( \mathcal{E})$ and 
 identify it with  $B^*\equiv B^* (\cyl (\pa X), \mathcal F_{{\rm cyl}})$ via $\rho$. 
Set $$A^* (X;\mathcal F):= \pi^{-1} (\Im \rho) \;\;\text{ with }
\pi \;\text{ the projection }\;\;  
\mathcal{L}( \mathcal{E}) \rightarrow \mathcal{Q} ( \mathcal{E}).$$ 
Recalling the identification $C^*(X,\mathcal F)=\KK(\mathcal{E})$,
we thus  obtain a short exact sequence of $C^*$-algebras: 
$ 0\rightarrow C^*(X,\mathcal F)\rightarrow 
A^* (X;\mathcal F)\xrightarrow{\pi} B^* (\cyl (\pa X), \mathcal F_{{\rm cyl}}) \rightarrow 0$
 where 
  the quotient map is still denoted by $\pi$.

  \begin{definition}\label{definition:wiener-hopf}
  The short exact sequence of $C^*$-algebras
  \begin{equation}\label{eq:short-cstar}
 0\rightarrow C^*(X,\mathcal F)\rightarrow 
A^* (X;\mathcal F)\xrightarrow{\pi} B^* (\cyl (\pa X), \mathcal F_{{\rm cyl}}) \rightarrow 0
 \end{equation}
 is by definition the Wiener-Hopf extension of $B^* (\cyl (\pa X), \mathcal F_{{\rm cyl}}) $.
 \end{definition}
 
 Notice that \eqref{eq:short-cstar} splits as a short exact sequence of {\it Banach spaces}, 
since we can choose  
$s: B^* (\cyl (\pa X), \mathcal F_{{\rm cyl}}) \to A^* (X;\mathcal F)$,  the map in the statement
of Lemma \ref{lemma:split}, as a section. 
So $$A^* (X;\mathcal F) \cong C^*(X,\mathcal F)\oplus s( B^* (\cyl (\pa X), \mathcal F_{{\rm cyl}}))$$ as Banach spaces.

  There is also a linear map $t: A^* (X,\mathcal F)\rightarrow C^*(X,\mathcal F)$ which is obtained as follows: if 
 $k\in A^* (X;\mathcal F)$,  then  $k$
 is  uniquely expressed as 
 $k= a + s(\ell)$ with $a\in C^*(X,\mathcal F)$ and 
$\pi(k)=\ell\in B^* (\cyl (\pa X), \mathcal F_{{\rm cyl}})$. 
Thus,  $\pi(k)=[\chi^0 \ell \chi^0]\in\mathcal Q (\mathcal E )$ for  one (and only one)  $\ell\in B^* (\cyl (\pa X), \mathcal F_{{\rm cyl}})$ 
since $\rho$ is injective.
 We set
\begin{equation}\label{eq:t}
t(k):= k-s\pi (k) =k-\chi^0 \ell \chi^0
\end{equation}
Then  $t(k)\in C^*(X,\mathcal F)$.

\begin{remark}\label{remark:why-wiener-hopf}
 A standard Wiener-Hopf extension of  $C^* \RR$ is obtained as follows. 
Let $C^* \RR$ act on the Hilbert space $\H=L^2(\RR)$ by convolutions. 
Recall that  $\chi^0_{\RR}$ is the characteristic function of $(-\infty, 0]$ 
and denote by $\H_o$ the subspace  $L^2(-\infty, 0]\subset \H$.  
Then  the same proof of Lemma \ref{lemma:split} can be applied to prove that 
there exists an injective homomorphism $\rho_{\RR}: C^* \RR \to \Q (\H_o)$ 
with  
$\rho_{\RR}(\ell)=\pi_o(\chi^0_{\RR}\ell \chi^0_{\RR})$, 
where $\Q (\H_o)$ denotes the Calkin algebra 
and $\pi_o$  the projection from the bounded operators on $\H_o$ onto $\Q (\H_o)$. 
Set $\E_o= \pi^{-1}_o(\Im \rho_{\RR})$. 
Exactly in the same manner as before, one has a short exact sequence 
$$
 0\rightarrow \K_o \rightarrow 
\E_o
\xrightarrow{\pi_o} C^* \RR \rightarrow 0
,
$$
where $\K_o$ denotes the compact operators on $\H_o$.  It is called 
a standard Wiener-Hopf extension of  $C^* \RR$. 
What we are going to construct is a slightly larger algebra than this. 
Observe that $\Q (\H_o)$ is naturally embedded in the Calkin algebra $\Q (\H)$. 
Thus one has another injective homomorphism $\hat{\rho}_{\RR}: C^* \RR \to \Q(\H)$.  
Set $\E= \pi^{-1}(\Im \hat{\rho}_{\RR})$ with 
$\pi$ the projection onto $\Q(\H)$. It then induces an extension of 
 $C^*$-algebras:  
$$
 0\rightarrow \K \rightarrow 
\E
\xrightarrow{\pi} C^* \RR \rightarrow 0
$$
where $\K$ is the compact operators on $\H$.  
Obviously it contains the above extension. 
Now recall the definition of the Extension group ${\rm Ext} (C^* \RR)$ and 
its additive structure; see Douglas \cite{Douglas} for instance. It is easily verified that 
the second exstension is exactly the one corresponding to the sum 
of $\E_o$  and the trivial extension;  
hence the resulting extension class  is the same as that of $\E_o$. 
Therefore, the second extension deserves to be called 
a Wiener-Hopf extension too. 

Let us consider the simplest case, namely 
a foliation consists of a single leaf $X$, 
which is a complete manifold with cylindrical end. 
It turns out that 
our extension \eqref{eq:short-cstar}
is isomorphic to the second extesion tensored with 
the algebra of compact operators.  
This can be proved by observing the
property 
$B^* (\cyl (\pa X), \mathcal F_{{\rm cyl}}) \cong C^* \RR\otimes \K$  in 
 \ref{bstar-structure}. 
Thus we also call the short exact sequence \eqref{eq:short-cstar} 
the Wiener-Hopf extension of $B^* (\cyl (\pa X), \mathcal F_{{\rm cyl}}) $.
\end{remark}


\section{{\bf Relative pairings and eta cocycles: the algebraic theory}}\label{section:algebraic}

\subsection{Introductory remarks}\label{subsection:intr-remarks}
On a {\it closed} foliated bundle $(Y,\mathcal{F})$ with holonomy groupoid $G$,
the Godbillon-Vey cyclic cocycle is initially defined
on the "small" algebra $\Psi^{-\infty}_c (G,E) \subset C^*(Y,\mathcal{F};E) $  of  $\Gamma$-equivariant 
smoothing operators of $\Gamma$-compact support. With respect
to our current notation:
$$\Psi^{-\infty}_c (G,E):=C^\infty_c (G, (s^*E)^*\otimes r^*E))\,.$$
Since the index class defined using a pseudodifferential parametrix is already well defined in 
$K_* (\Psi^{-\infty}_c (G,E))$, the pairing between the 
the Godbillon-Vey cyclic cocycle and the index class is well-defined.

In a second stage, the cocycle is continuously extended to a dense holomorphically closed subalgebra $\mathfrak{A}\subset C^*(Y,\mathcal{F})$;
there are at least two reasons
for doing this. First, as already remarked in the Introduction,
it is only by going to the $C^*$-algebraic index that the
well known properties for the signature  and the spin Dirac operator  of  a metric of
positive scalar curvature hold. The second reason
for this extension rests on the structure of the index class {\it which is employed in 
the proof of  the higher index formula}, i.e. either the graph projection or 
the Wassermann projection; in both cases $\Psi^{-\infty}_c (G,E)$ is too small to contain these
particular representatives of the  index class and one is therefore forced to find
an  intermediate subalgebra $\mathfrak{A}$,
\begin{equation}\label{small-MN}
\Psi^{-\infty}_c (G,E) \subset\mathfrak{A} \subset C^*(Y,\mathcal{F};E)\,;
\end{equation}
 $\mathfrak{A}$  is big enough for  the two particular representatives of the index classe to belong to it but small enough for the Godbillon-Vey cyclic cocycle to extend;
moreover, being dense and holomorphically closed it has the  same $K$-theory as $C^*(Y,\mathcal{F};E)$.

Let now $(X,\mathcal{F})$ be a foliated bundle with cylindrical ends.
For notational simplicity, unless confusion should arise, let us not write the bundle $E$
in our algebras.
In  this section we shall select "small" subalgebras 
$J_c\subset  C^*(X,\mathcal F)$, $A_c\subset  A^*(X,\mathcal F)$, $B_c\subset B^* (\cyl (\pa X), \mathcal F_{{\rm cyl}})$,
with $J_c$ an ideal in $A_c$,
so that there is a short exact sequence
$0\rightarrow J_c\hookrightarrow A_c \xrightarrow{\pi_c} B_c \rightarrow 0$
which is a subsequence of 
 $0\rightarrow C^*(X,\mathcal F)\hookrightarrow A^* (X;\mathcal F)\xrightarrow{\pi} B^* (\cyl (\pa X), \mathcal F_{{\rm cyl}})\rightarrow 0$.
We shall then proceed to define the two relevant Godbillon-Vey cyclic cocycles and study, algebraically,
their main properties. As in the closed case, we shall eventually need to find an intermediate short exact sequence,
sitting between the two, 
call it 
$0\rightarrow \mathbf{\mathfrak{J}} \hookrightarrow  \mathbf{\mathfrak{A}} \rightarrow   \mathbf{\mathfrak{B}} \rightarrow 0$,
with constituents big enough for the index classes  to belong to them but small enough for
the two cyclic cocycles  to extend; this is quite a delicate point and it will be explained  in  Section 
\ref{sec:shatten}.
We anticipate that, in contrast with the closed case, the ideal $J_c$ in the small subsequence will be too small even for
the index class defined by a pseudodifferential parametrix. This has to do with the non-locality of the parametrix on manifolds with boundary;
it is a phenomenon that was explained in detail in \cite{LPETALE}; we shall come back to it in Section \ref{sec:index}.

\subsection{Small dense subalgebras}\label{subsect:small-dense}
Define $J_c:= C^\infty_c(X,\mathcal{F})$; see subsection \ref{subsec:cstar}. Redefine 
$$B_c:= \{k\in C^\infty((\RR\times\partial \tM)\times (\RR\times\partial \tM) \times T); k 
\text{ is } \RR\times \Gamma\text{-invariant}, k \text{ has } \RR\times\Gamma\text{-compact  support}\}
$$
see subsection \ref{subsec:translation} (we pass from continuous to smooth functions).
We now define $A_c$; consider the functions $\chi^\lambda$, $\chi^\lambda_{{\rm cyl}}$ induced on $X$ and $\cyl (\pa X)$
by the real function $\chi_{(-\infty,-\lambda] }$ (the characteristic function of the interval $(-\infty,-\lambda] $). We shall say that
 $k$ is in $A_c$ if it is a smooth function on  $\tV\times\tV\times T$ which is
$\Gamma$-invariant and 
for which there  exists $\lambda\equiv \lambda(k)>0$, such that
\begin{itemize}
\item $k-\chi^\lambda k \chi^\lambda$ is of $\Gamma$-compact support 
\item there exists $\ell \in B_c$ such that $\chi^\lambda k \chi^\lambda =  \chi^\lambda_{{\rm cyl}} \ell \chi^\lambda_{{\rm cyl}}
$ on $((-\infty,-\lambda]\times \partial \tM) \times ((-\infty,-\lambda]\times \partial \tM) \times T$
\end{itemize}

\begin{lemma}\label{lemma:ac}
$A_c$ is a *-subalgebra of $A^*(X,\mathcal{F})$. 
\end{lemma}

\begin{proof} 
Let  $k,k'\in A_c$. Write, with a small abuse of notation,  $k=a+\chi^\lambda\,\ell\,\chi^\lambda$ with  $a$ of $\Gamma$-compact support
and $\ell\in B_c$ and similarly for $k'$. Observe first of all that
if $\mu>\lambda$, so that $-\mu<-\lambda$, then $(\chi^\lambda\,\ell\,\chi^\lambda - \chi^\mu\,\ell \, \chi^\mu)$
is also of $\Gamma$-compact support (since $\ell$ if of $\RR\times\Gamma$-compact support).
Thus we can assume that $k=a+\chi^\mu\,\ell\,\chi^\mu$, $k'=a'+\chi^\mu\,\ell '\, \chi^\mu$.
We compute:
$$k k'= a a' + a\chi^\mu\ell ' \chi^\mu + \chi^\mu\ell\chi^\mu a'  + \chi^\mu\ell\chi^\mu \chi^\mu\ell' \chi^\mu \,.$$
The first summand on the right hand side is again of $\Gamma$-compact support; the second and the third summand
are also of $\Gamma$-compact support since $\ell$ and $\ell^\prime$ are of $\RR\times\Gamma$-compact support; 
the last term can be written as 
$$ \chi^\mu\ell\ell' \chi^\mu + (\chi^\mu\ell(\chi^\mu-1))((\chi^\mu -1)\ell' \chi^\mu)\,.$$
Thus $k k' -  \chi^\mu\ell\ell' \chi^\mu= a a' + (\chi^\mu\ell(\chi^\mu-1))((\chi^\mu -1)\ell' \chi^\mu)$; now, by  Sublemma
\ref{sublemma:basic1} both
$(\chi^\mu\ell(\chi^\mu-1))$ and $((\chi^\mu -1)\ell' \chi^\mu)$ are of $\Gamma$-compact support. 
Thus $k k' -  \chi^\mu\ell\ell' \chi^\mu$ is of  $\Gamma$-compact support
as required.
Finally, consider $\nu\in\RR$, $\nu>\mu$ and let $F(p,p',\theta):=  \chi^\nu  (p) (1-\chi^\mu)(p')$,
a function on $W\times W \times T$ which is $\theta$-independent. 
Since $\ell$ and $\ell'$ are of
$\RR\times\Gamma$-compact support, we can choose $\nu>\mu$ so that that ${\rm supp}\, \ell \cap {\rm supp}\, F=\emptyset$.
Thus $\chi^\nu (\chi^\mu\ell(\chi^\mu-1))= \chi^\nu \ell(\chi^\mu-1)$ is equal to zero. We conclude that for such a 
$\nu$ we do get $\chi^\nu k k'\chi^\nu = \chi^\mu\ell\ell' \chi^\mu$ and the proof is complete.
\end{proof}

We thus have:

\begin{proposition}
Let $\pi_c:=\pi |_{A_c}$. Then there is a short exact sequence of *-algebras
\begin{equation}\label{short-compact}
0\rightarrow J_c\hookrightarrow A_c\xrightarrow{\pi_c} B_c \rightarrow 0\,.
\end{equation}
\end{proposition}

\begin{remark}
Notice that the image of $A_c$ through $t|_{A_c}$ is not contained in $J_c$ since
$\chi^0$ is not even continuous. Similarly, the image of $B_c$ through $ s|_{B_c}$
is not contained in $A_c$.
\end{remark}

\begin{remark}\label{remark:b-dense}
Using the foliated $b$-calculus developed in \cite{LPETALE} and 
Melrose' indicial operator in the foliated context, it is possible to define
a slightly bigger dense subsequence.  
  We shall briefly comment
 on this in Subsection \ref{subsection:rbm-cocycle}.\end{remark}

\subsection{Relative cyclic cocycles}\label{subsection:rcc}
Let $A$ be a  
$k-$algebra over $k=\CC$. The
cyclic cohomology groups $HC^*(A)$  \cite{connes-ihes} (see also
 \cite{tsygan-cyclic}) are the cohomology groups of the complex
$(C^n_\lambda, b)$ where $ C^n_\lambda$ denotes the space of
$(n+1)-$linear functionals $\varphi$ on $A$ satisfying the
condition: $$ \varphi(a^1, a^2, \ldots, a^n,a^0)= (-1)^n
\varphi(a^0, \ldots, a^{n+1})\,,\;\; \forall a^i \in A $$ and
where $b$ is the Hochschild coboundary map given by
\begin{align*} (b \varphi) (a^0, \ldots, a^{n+1})& = \sum_{j=0}^n
(-1)^j \varphi (a^0, \ldots, a^j a^{j+1}, \ldots, a^{n+1}) +\\&
(-1)^{n+1} \varphi(a^{n+1}a^0,\ldots , a^n). \end{align*} 
Given a second unital algebra $B$ together with a surjective homomorphism
$\pi: A\to B$, one can define the relative cyclic complex 
$$C^n_\lambda(A,B):=\{(\tau,\sigma)\,:\, \tau\in C^n_\lambda (A), \sigma\in C^{n+1}_\lambda (B)\}$$
with  coboundary map given by
$$(\tau,\sigma)\longrightarrow (b\tau-\pi^*\sigma, b\sigma)\,.$$
A relative cochain $(\tau,\sigma)$
is thus a cocycle if $b\tau=\pi^* \sigma$ and $b\sigma=0$.
One obtains in this way the relative cyclic cohomology groups $HC^* (A,B)$.
If $A$ and $B$ are Fr\'echet algebra, then we can also define the topological (relative) cyclic
cohomology groups. More detailed information are given,
for example, 
in \cite{lmpflaum}.


\subsection{Roe's 1-cocycle}\label{subsection:roe}
In this subsection, and in the next two, we study a particular but important example.
We assume that $T$ is a point and that $\Gamma=\{1\}$, so that we are really considering
a compact manifold $X_0$ with boundary $\pa X_0$ and associated manifold with  cylindrical ends
$X$; we keep denoting the cylinder $\RR\times \pa X_0$ by $\cyl (\pa X)$ (thus, as before,
we omit the subscript $0$).
 The algebras appearing in the short exact sequence \eqref{short-compact} 
 are now given by 
 $J_c=C^\infty_c (X\times X)$ and 
 $$B_c=\{k\in C^\infty((\RR\times \pa X_0)\times (\RR\times \pa X_0)); k \text{ is $\RR$-invariant}, k
 \text{ has compact $\RR$-support}\}\,.$$
Finally, a smooth function $k$ on $X\times X$ is in  $A_c$ if there exists
a $\lambda\equiv \lambda(k)>0$ such that \\(i) $k-\chi^\lambda k \chi^\lambda$ is 
of compact support on $X\times X$; \\(ii) $\exists$ $\ell\in B_c$ such that  $\chi^\lambda k \chi^\lambda=
\chi^\lambda_{{\rm cyl}} \ell  \chi^\lambda_{{\rm cyl}}$ on $((-\infty,-\lambda]\times \pa X_0)\times (-\infty,-\lambda]\times \pa X_0)$\,.\\
For such a $k\in A_c$ we set
$\pi_c := \pi|_{A_c}$.
Since $k - \chi^0\ell \chi^0$ admits compact support, it  belongs to $C^*(G)$ (in this case
this is just the equal to the compact operators on $L^2 (X)$). Hence it follows  that 
 $\pi (k)= \ell $ and thus $\pi_c (k)= \ell$; so  we have the short exact sequence of $*$-algebras
$0\rightarrow J_c\hookrightarrow A_c\xrightarrow{\pi_c} B_c \rightarrow 0\,.$
 (The Wiener-Hopf short exact sequence  \eqref{eq:short-cstar} now reads as
$0\rightarrow \mathcal{K} (L^2 (X))\rightarrow 
A^* (X)\xrightarrow{\pi} B^* (\cyl (\pa X)) \rightarrow 0$.)
All of this has an obvious generalization if instead of functions we consider
sections of the bundle ${\rm END} E:=E\boxtimes E^*\to X\times X$, with $E$ a complex vector bundle on $X$.
 
We shall define below a 0-{\it relative} cyclic cocycle associated to the homomorphism $\pi_c: A_c \to B_c$.
To this end we start by
defining  a cyclic 1-cocycle $\sigma_1$ for the algebra $B_c$; this is directly inspired from work
of  John Roe
(indeed, a similarly defined 1-cocycle  plays a fundamental  role in his index
theorem on partioned manifolds \cite{roe-partitioned}). 
It should be noticed that $\sigma_1$ is in fact defined on $B_c (\cyl (Y))$, with
$Y$ any closed compact manifold.

Consider the characteristic function $\chi^\lambda_{{\rm cyl}}$, $\lambda>0$,
 induced on the cylinder $\cyl (Y)$ by the real function $\chi^\RR_{(-\infty,-\lambda]}$.
For notational convenience, unless absolutely necessary,
 we shall use the simpler notation  $\chi^\lambda$.

We define $\sigma_1^\lambda: B_c\times B_c \to \mathbb{C}$ as
\begin{equation}\label{roe-1}
\sigma^\lambda_1 (\ell_0,\ell_1):= \Tr (\ell_0 [\chi^\lambda\,,\,\ell_1])\,.
\end{equation}
 We need to show that the definition is well posed.
 
\begin{proposition}\label{prop:sigma1-ok}
The operators $[\chi^\lambda\,,\,\ell_0]$ and 
$\ell_0 [\chi^\lambda\,,\,\ell_1]$ are trace class $\forall \ell_0, \ell_1 \in B_c$
(and 
 $\Tr [\chi^\lambda\,,\,\ell_0]=0$). In particular
$\sigma^\lambda_1 (\ell_0,\ell_1)$ is well defined.

\end{proposition}
\begin{proof}
We already know, see Sublemma \ref{sublemma:basic1}, that the operator $[\chi^\lambda\,,\,\ell_1]$ is expressed by a  kernel on the cylinder
which is of compact support. Indeed, in the proof of Sublemma \ref{sublemma:basic1}, which is given
in Section \ref{section:proofs}, 
 we have explicitly written down the kernel  $\kappa$ corresponding to $[\chi^\lambda\,,\,\ell]$ as 
 \begin{equation}\label{kernel-for-commutator-bis}
 \kappa (y,s,y^\prime,s^\prime) = 
\begin{cases}
\ell (y,y^\prime,s-s^\prime)
\qquad & \text{ if }\;\; s\leq -\lambda\,,\;\;s^\prime \geq -\lambda
\\
- \ell (y,y^\prime,s-s^\prime)
\qquad & \text{ if }\;\; s^\prime\leq -\lambda\,,\;\;s \geq -\lambda\\
0 \qquad & \text{ otherwise }
\end{cases}
\end{equation}
where $y,y^\prime\in Y$, $s,s^\prime\in \RR$ and where we have used the $\RR$-invariance of $\ell$ in order to write 
$\ell (s,y,s^\prime,y^\prime)\equiv  \ell (y,y^\prime,s-s^\prime)$.
  Choose now
 a {\it smooth} compactly supported function $\varphi$ on $\cyl(Y)\times \cyl(Y)$, equal to 1 on the support of $\kappa$.
 Let  $\kappa_0$ be the smooth compactly supported kernel obtained by multiplying $\kappa$ by $\varphi$; $\kappa_0$
 is clearly trace class.
 Now,  multiplication by $\chi^\lambda$ is a bounded operator
 so the operators given by  $\chi^\lambda \kappa_0$ and  $\kappa_0 \chi^\lambda$ are also trace class. 
 Since $[\chi^\lambda,\kappa_0]= [\chi^\lambda,\ell]$, we conclude that 
 $[\chi^\lambda,\ell]$
 is trace class;
  since $\ell_0$ defines a bounded
operator, we also see immediately that the trace of $\ell_0 [\chi^\lambda\,,\,\ell_1]$ is  well defined. 
Finally, it remains to justify that $\Tr [\chi^\lambda,\ell]=0$; this is now clear, since   
$\Tr[\chi^\lambda,\ell]=\Tr[\chi^\lambda,\kappa_0]=0$. 
The Proposition is proved.
\end{proof}

\begin{proposition}\label{prop:lambda-indipendent}
The value $\Tr (\ell_0 [\chi^\lambda\,,\,\ell_1])$ is independent of $\lambda$
and will simply be denoted by $\sigma_1 (\ell_0,\ell_1)$.
The functional $\sigma_1 :B_c\times B_c\to \CC$ is a 1-cyclic cocycle. 
\end{proposition}
\begin{proof}
In order to prove the indipendence
on $\lambda$ we make crucial use of the $\RR$-invariance of $\ell_j$. 
We write $\ell_j (y,y^\prime,s,s^\prime)\equiv \ell_j (y,y^\prime,s-s^\prime)$. We 
compute:
\begin{align*}
&\sigma_1^\lambda (\ell_0,\ell_1)=\Tr (\ell_0\chi^\lambda\ell_1 - \ell_0 \ell_1 \chi^\lambda)\\
&=\int_{Y\times Y} dy\,dy^\prime \int_{\RR\times \RR} ds\,ds^\prime\left[ \ell_0 (y,y^\prime,s-s^\prime)\chi^\lambda
(s^\prime) \ell_1 (y^\prime,y,s^\prime-s)- \ell_0 (y,y^\prime,s-s^\prime)\ell_1 (y^\prime,y,s^\prime-s)
\chi^\lambda (s) \right]\\
&= \int_{Y\times Y} dy\,dy^\prime \left( \int_{\RR}ds\int_{-\infty}^{-\lambda} ds^\prime- \int_{-\infty}^{-\lambda}ds
 \int_{\RR}ds^\prime \right) \ell_0 (y,y^\prime,s-s^\prime)\ell_1 (y^\prime,y,s^\prime-s)\\
&= \int_{Y\times Y} dy\,dy^\prime \left( \int_{-\lambda}^{+\infty}ds\int_{-\infty}^{-\lambda} ds^\prime- \int_{-\infty}^{-\lambda}ds
 \int_{-\lambda}^{+\infty} ds^\prime \right) \ell_0 (y,y^\prime,s-s^\prime)\ell_1 (y^\prime,y,s^\prime-s)\\
 &= \int_{Y\times Y} dy\,dy^\prime \left( \int_{0}^{+\infty}dt\int_{-\infty}^{0} dt^\prime- \int_{-\infty}^{0}dt
 \int_{0}^{+\infty} dt^\prime \right) \ell_0 (y,y^\prime,t-t^\prime)\ell_1 (y^\prime,y,t^\prime-t)
\end{align*}
Thus $\Tr (\ell_0 [\chi^\lambda\,,\,\ell_1])$ is independent
of $\lambda$ since we have proved that $\forall \lambda$ it is equal to $\Tr (\ell_0 [\chi^0\,,\,\ell_1])$.
In particular we record that
\begin{equation}\label{sigma1-0}
\sigma_1^\lambda (\ell_0,\ell_1)=\Tr (\ell_0 [\chi^0\,,\,\ell_1])\,.
\end{equation}
We shall denote $\sigma_1^\lambda$ as $\sigma_1$.
In order to show that  $\sigma_1 $ is a  cyclic cocycle we begin by 
recalling that $\Tr [\chi^\lambda,\ell]=0$ $\forall \ell\in B_c$. Thus we have
\begin{align*}
\sigma_1 (\ell_0,\ell_1)+ \sigma_1 (\ell_1,\ell_0)&=\Tr (\ell_0 [\chi^0,\ell_1]) + \Tr ([\chi^0,\ell_0]\ell_1)\\
&= \Tr ( [\chi^0, \ell_0\ell_1]) =0
\end{align*}
proving that $\sigma_1$ is a cyclic cochain. Next we compute
\begin{align*}
b\, \sigma_1 (\ell_0,\ell_1,\ell_2)&= \Tr  \left( \ell_0 \ell_1 [\chi^0,\ell_2]) + \ell_0[\chi^0,\ell_1\ell_2] + \ell_2 \ell_0 [\chi^0,\ell_1] \right)\\
&= \Tr  \left( -\ell_0  [\chi^0,\ell_1] \ell_2 + \ell_2\ell_0[\chi^0,\ell_1] \right)\\
&= \Tr \left( [\ell_2, \ell_0 [\chi^0,\ell_1] ]\right) =0
\end{align*}

\end{proof}

\begin{remark}\label{remark:new-sigma-1}
We point out that following expression for $\sigma_1$:
\begin{equation}\label{alternative-sigma-1}
\sigma_1 (\ell_0,\ell_1)= \frac{1}{2}\Tr\left( \chi^0 [\chi^0,\ell_0][\chi^0,\ell_1] \right) \,.
\end{equation}
The proof of \eqref{alternative-sigma-1} is elementary (just apply repeatedly the fact that $1= \chi^0 + (1-\chi^0)$) 
and for the sake of brevity we omit it. 
The advantage of this new expression for $\sigma_1$ is that it makes the extension 
to certain dense subalgebras easier to deal with. (Notice, for example, that
$\sigma_1$ is now defined under the weaker assumption that $[\chi^0,\ell_j]$
is Hilbert-Schmidt.)  The right hand side of \eqref{alternative-sigma-1}
is in fact the original definition by Roe.
 \end{remark}

\subsection{Melrose' regularized integral}\label{subsection:melrose}
Recall that our immediate goal is to define a 0-{\it relative} cyclic cocycle for the homomorphism
$\pi_c : A_c \to B_c$ appearing in the short exact sequence of the previous section. Having defined
a 1-cocycle $\sigma_1$ on $B_c$ we now need to define a 0-cochain on $A_c$. Our definition will be
a simple adaptation of the definition of the $b$-trace in Melrose' $b$-calculus
(but since our algebra $A_c$ is very small, we can give a somewhat simplified treatment). Recall that
for $\lambda>0$ we are denoting by $X_\lambda$ the compact manifold obtained attaching $[-\lambda,0]\times \pa X_0$
to our manifold with boundary $X_0$. 

So, let $k\in A_c$ with $\pi_c (k)=\ell\in B_c$. Since $\ell$ is $\RR$-invariant on the cylinder $\RR\times \pa X_0$
we can write $\ell(y,y^\prime,s)$ with  $y,y^\prime\in \pa X_0$, $s\in \RR$.
Set
\begin{equation}\label{tau0}
\tau_0^r (k) := \lim_{\lambda\to +\infty} \left( \int_{X_\lambda} k(x,x) {\rm dvol}_g -\lambda \int_{\pa X_0} \ell(y,y,0) {\rm dvol}_{g_\partial} 
\right)
\end{equation}
where the superscript $r$ stands for {\it regularized}. (The $b$-superscript would be of course more appropiate;
unfortunately it gets confused with the $b$ operator in cyclic cohomology.)
It is elementary to see that the limit exists; in fact, because of the very particular definition of
$A_c$ the function 
$$\varphi (\lambda):= \int_{X_\lambda} k(x,x) {\rm dvol}_g -\lambda \int_{\pa X_0} \ell(y,y,0) {\rm dvol}_{g_\partial}$$
becomes constant for large values of $\lambda$. The proof is elementary and thus omitted.
$\tau^r_0$ defines a 0-cochain on $A_c$.
 
\begin{remark}
Notice that \eqref{tau0}  is nothing but Melrose' regularized integral, in the cylindrical language, 
for the restriction of $k$ to the diagonal of $X\times X$.
\end{remark}

We shall also need  the following 

\begin{lemma}\label{lemma:regularized}
If $k\in A_c$ then $t(k)$, which is a priori a compact operator, is in fact trace class.
Moreover
\begin{equation}\label{b-trace=composition}
\tau^r_0 (k)= \Tr (t (k))\,.
\end{equation}
\end{lemma}

\noindent
We remark once again that $t(k)$ is not an element in $J_c$.
\begin{proof}
We first need the following:
\begin{sublemma}
Let $\chi$ is the characteristic function of a mesurable set $K$  in  $X$. 
If $a\in J_c$, then  $k=\chi a \chi$ 
is of trace class and the trace is obtained as
$$
\Tr (k) =
\int_K a(x,x)dx
.
$$
\end{sublemma}
\begin{proof}
Since $a$ gives rise to a smoothing operator with compact support, it is of trace class. 
Recall that the algebra of trace class operators forms an ideal in the algebra of bounded operators. 
Thus $k$ is of trace class and we can assume that 
 $k=bc$ with $b$ and $c$  operators of Hilbert-Schmidt class. 
 Then 
$$
\Tr (k) =\langle b,c^* \rangle_{2}=
\int_{X\times X} b(x,y)c(y,x)dxdy =
\int_K a(x,x)dx
, 
$$
with $\langle \; ,\; \rangle _{2}$ denoting the inner product 
for  operators of Hilbert-Schmidt class. 
\end{proof}
Write $k=a+\chi^\lambda\,\ell\,\chi^\lambda$ with  $a \in J_c$ 
and $\ell\in B_c$ as in subsection \ref{subsect:small-dense}.
There exists a compactly supported smooth function $\sigma$ on $X$,  depending on $\ell$, 
such that 
$$
\chi^{\lambda}\ell \chi^{\lambda} - \chi^0\ell\chi^0
=
\chi^{\lambda}\sigma\ell\sigma \chi^{\lambda} - \chi^0\sigma\ell\sigma\chi^0
$$ 
since the support of $\chi^{\lambda}-\chi^0$ is compact. 
Note that we can choose the same $\ell$ in \ref{eq:t} and subsection \ref{subsect:small-dense}. 
Thus 
$t(k)=k-\chi^0\ell\chi^0 
=a+\chi^\lambda\,\sigma\ell\sigma\,\chi^\lambda-\chi^0\sigma\ell\sigma\chi^0$ 
is of trace class 
due to the sublemma above. 
Therefore, we have
$$
\Tr (t (k))=
\int_X a(x,x)dx - \int_{X_{\lambda}\setminus X_0}\ell(y,y,0)dydt
= 
\tau_0^r (k)
$$
for a sufficiently large $\lambda$. This completes the proof.
\end{proof}

\subsection{Melrose' regularized integral  and Roe's 1-cocycle define a relative 0-cocycle}
\label{subsection:relative0}

We finally consider the relative 0-cochain $(\tau_0^r,\sigma_1)$ for the pair $A_c\xrightarrow{\pi_c}B_c$.
\begin{proposition}\label{prop:c-relative-cocycle}
The relative 0-cochain $(\tau_0^r,\sigma_1)$ is a relative 0-cocycle. It thus defines an element $[(\tau_0^r,\sigma_1)]$
 in the relative group $HC^0 (A_c,B_c)$.
\end{proposition}
\begin{proof}
We need to show that $b\sigma_1=0$ and that $b\tau_0^r=(\pi_c)^* \sigma_1$. The first has already been
proved, so we concentrate on the second. We compute:
$b\tau (k,k^\prime)= \tau_0^r (k k^\prime - k^\prime k).$ Write
$k=a+\chi^\mu \ell \chi^\mu$, $k^\prime=a^\prime+\chi^\mu \ell^\prime \chi^\mu$ as we did in the proof
of Lemma \ref{lemma:ac}. Then we need to
show that
\begin{equation}\label{0-cocycle}
\tau_0^r (k k^\prime - k^\prime k)= \sigma_1 (\pi_c k,\pi_c k^\prime)=\sigma_1 (\ell,\ell^\prime)\,.
\end{equation} 
 There are several proofs of this fundamental relation. 
 One proof of \eqref{0-cocycle} employs  Melrose' formula
 for the $b$-trace of a commutator; we shall give the details in the next Subsection. 
 Here we propose
 a different proof that has the advantage of extending to more general situations.
 Following the proof of Lemma \ref{lemma:ac}, we can write
$$k k'= (a a' + a\chi^\mu\ell ' \chi^\mu + \chi^\mu\ell\chi^\mu a'   - \chi^\mu\ell(1-\chi^\mu) \ell' \chi^\mu)+ \chi^\mu\ell \ell' \chi^\mu \,.$$
Notice that the first summand is trace class; this is obvious for the first term $a a'$
and clear for the next two terms; the fourth term, viz.  $-\chi^\mu\ell(1-\chi^\mu) \ell' \chi^\mu$
is trace class because $\chi^\mu\ell(1-\chi^\mu) $ is trace class and $ \ell' \chi^\mu$ is bounded (see the proof of Proposition
\ref{prop:sigma1-ok}). 
A similar expression can be written for $k^\prime k$. Using first Lemma 
\ref{lemma:regularized} and then the definition of $t$,  we obtain easily 
\begin{align*}
\tau_0^r (k k^\prime - k^\prime k)&=\Tr (t (k k^\prime - k^\prime k))\\
&= \Tr \left( [a,a^\prime]+ [\chi^\mu \ell \chi^\mu, a^\prime] + [a,\chi^\mu \ell^\prime \chi^\mu]- \chi^\mu \ell (1-\chi^\mu)\ell^\prime
\chi^\mu + \chi^\mu\ell^\prime (1-\chi^\mu)\ell \chi^\mu \right)\\
&=\Tr (- \chi^\mu \ell (1-\chi^\mu)\ell^\prime
\chi^\mu + \chi^\mu\ell^\prime (1-\chi^\mu)\ell \chi^\mu)\\ &= \sigma_1^\mu (\ell,\ell^\prime)=\sigma_1 (\ell,\ell^\prime)
\end{align*}
The Proposition is proved.
\end{proof}

\subsection{Melrose' 1-cocycle and the relative cocycle condition via the $b$-trace formula}
\label{subsection:rbm-cocycle}

{\it The results in this  Subsection will not be used in the sequel.}

\smallskip

As we have anticipated in the previous subsection, the equation $b\tau_0^r= \pi^*_c \sigma_1$ is nothing
but a compact way of rewriting Melrose' formula for the $b$-trace of a commutator. We wish to
explain this point here. 

First of all, since it will  cost us nothing, we consider a  slightly larger  subsequence
of dense subalgebras. We hinted to this subsequence in Remark \ref{remark:b-dense};
we explain it here  for $T={\rm point}$ and $\Gamma=\{1\}$ even though it exists in the general
foliated case.
Thus, following the notations of the $b$-calculus, we set
$$
A^b_c:= \Psi^{-\infty}_{b,c}(X,E)\,,\quad B^b_c:= \Psi^{-\infty}_{b,I,c}(\overline{N_+\partial X},E|_{\pa})\,,\quad
J^b_c := \rho_{{\rm ff}}
\Psi^{-\infty}_{b,c}(X,E)$$
and consider
\begin{equation}\label{short-b}
0\,\longrightarrow \,
J^b_c\,
\longrightarrow
A^b_c\,
\xrightarrow{ \pi^b_c}
B^b_c\,
\longrightarrow \,0\,,\end{equation} with $ \pi^b_c $ equal to Melrose' indicial operator $I(\cdot)$. 
This sequence is certainly larger than the one we have defined, viz.
$0\rightarrow J_c\hookrightarrow A_c\xrightarrow{\pi_c} B_c \rightarrow 0$ (indeed
the latter corresponds  to the subsequence
of  \eqref{short-b} obtained by restricting  \eqref{short-b}  to the  sub-algebra
$\{k\in A^b_c : k-k|_{{\rm ff}} \;\text{ has support in the interior of }\; \wx\times_b \wx \},$
with ${\rm ff}$ denoting the front face of the $b$-stretched product).

\medskip
Let $\tau_0^r$ be equal to the $b$-Trace:
$\tau_0^r:={}^b \Tr$. 
 Observe that $\sigma_1$ also defines a 1-cocyle on $B_c^b$.
 We can thus consider the relative 0-cochain $(\tau_0^r,\sigma_1)$ for the homomorphism
$A^b_c\,
\xrightarrow{I(\cdot)}
B^b_c$; in order to prove that this is  a relative 0-cocycle it remains to 
 to show that $b\tau_0^r (k,k^\prime)= \sigma_1 (I(k),I(k^\prime))$, i.e.
 \begin{equation}\label{1cocycle-is-b}
 {}^b \Tr [k,k^\prime]= \Tr (I(k)[\chi^0,I(k^\prime)])
 \end{equation}
 Recall here that Melrose' formula for the $b$-trace of a commutator is
 \begin{equation}\label{melrose}
 {}^b \Tr [k,k^\prime]= \frac{i}{2\pi}\int_\RR \Tr_{\pa X} \left( \partial_\mu I(k,\mu) \circ I (k^\prime,\mu) \right) d\mu
 \end{equation}
 with 
 $\CC\ni z\to I(k,z)$ denoting the indicial family of the operator $k$, i.e. the Fourier transform 
of the indicial operator $I(k)$.

Inspired by the right hand side of \eqref{melrose} we consider an arbitrary  compact manifold
 $Y$, the algebra $B^b_c (\cyl (Y))$ and the functional
\begin{equation}\label{melrose-1-cocycle}
\mathfrak{s}_1 (\ell,\ell^\prime):=  \frac{i}{2\pi }\int_\RR \Tr_Y 
\left( \partial_\mu \hat{\ell}(\mu) \circ \hat{\ell}^\prime (\mu) \right) d\mu
\end{equation}
That this is a cyclic 1-cocyle follows by elementary arguments (it also follows from the Proposition below).
Formula \eqref{melrose-1-cocycle}
defines what should be called Melrose' 1-cocycle

\begin{proposition}\label{prop:hilbert-transf}
Roe's 1-cocycle $\sigma_1$ and Melrose 1-cocycle $\mathfrak{s}_1 $
coincide:
\begin{equation}\label{equality-with-rbm-bis}
\sigma_1 (\ell,\ell^\prime):= \Tr (\ell[\chi^0,\ell^\prime]) = \frac{i}{2\pi }\int_\RR \Tr_Y \left( \partial_\mu \hat{\ell}(\mu) \circ \hat{\ell}^\prime (\mu) \right) d\mu
 =:  \mathfrak{s}_1 (\ell,\ell^\prime)
\end{equation}
\end{proposition}

\begin{proof}
In order to prove \eqref{equality-with-rbm-bis} we shall employ the Hilbert transformation
$\mathcal{H}:L^2 (\RR)\to L^2 (\RR)$:
\begin{equation*}
\mathcal{H} (f):= \lim_{\epsilon\downarrow 0} \frac{i}{\pi}\int_{|x-y|>\epsilon} \frac{f(x)}{x-y}dy \,.
\end{equation*}
The crucial observation is that 
if we denote by $F:L^2 (\RR) \to L^2 (\RR)$ the Fourier transformation, then
\begin{equation}\label{fourier-hilbert} 
F\circ \mathcal{H} \circ F^{-1}= - F^{-1} \circ \mathcal{H} \circ F  = 1-2\chi^0_\RR 
\end{equation} 
where the right hand side denotes, as usual,
the multiplication operator. Using this, we see that 
$$ \Tr (\ell[\chi^0,\ell^\prime]) = \frac{1}{2} \int_\RR \Tr_Y (\hat{\ell}(\mu) [ \mathcal{H}, \hat{\ell}^\prime ](\mu)) \, d\mu\,.$$
Using the definition of the Hilbert transform $\H$
one checks that 
  $[\H,\hat{\ell}]$ is the integral operator with kernel function equal to $-i/\pi\, \omega(u,v)$,
  with $\omega (u,v)= (\hat{\ell}(u) - \hat{\ell}(v))/(u-v)$. 
This imples that  $$\frac{1}{2} \int_\RR \Tr_Y (\hat{\ell}(\mu) [ \mathcal{H}, \hat{\ell}^\prime ](\mu)) \, d\mu\,.$$ is equal to
$$-\frac{i}{2\pi}\int_\RR \Tr_Y \left( \hat{\ell}(\mu) \circ \partial_\mu \hat{\ell}^\prime (\mu) \right) d\mu$$
which is equal to 
 the right hand side of \eqref{equality-with-rbm-bis} once we integrate by parts.
 
\end{proof}
 
Proposition \ref{prop:hilbert-transf} and Melrose' formula imply at once the relative 0-cocyle
condition for $(\tau_0^r,\sigma_1)$:
indeed using first Proposition \ref{prop:hilbert-transf} and then Melrose' formula we get:
\begin{align*}
\sigma_1 (I(k),I(k^\prime))&:=\Tr (I(k)[\chi^0,I(k^\prime)])= \frac{i}{2\pi}\int_\RR \Tr_{\pa X} \left( \partial_\mu I(k,\mu) \circ I (k^\prime,\mu) \right) d\mu
\\
&= {}^b \Tr [k,k^\prime]= b\tau_0^r (k,k^\prime)\,.
\end{align*}
Thus $I^* (\sigma_1)=b\tau_0^r$ as required.

\smallskip
\noindent
{\bf Conclusions.} We have established the following:
\begin{itemize}
\item  
the right hand side of Melrose' formula defines a 1-cocyle $\mathfrak{s}_1$ on $B_c (\cyl (Y))$,
$Y$ any closed compact manifold;
\item  Melrose 1-cocyle $\mathfrak{s}_1$ equals Roe's 1-cocyle $\sigma_1$
\item 
Melrose' formula itself can be interpreted
as a {\it relative} 0-cocyle condition for the 0-cochain $(\tau^r_0,\mathfrak{s}_1)\equiv (\tau^r_0,\sigma_1)$. 
\end{itemize}

\subsection{Philosophical remarks}\label{subsect:strategy}
We wish to recollect the information obtained in the Subsections 
\ref{subsection:roe}, \ref{subsection:melrose}, 
\ref{subsection:relative0}
and start to explain
our approach to Atiyah-Patodi-Singer higher index theory. 

On a closed compact  orientable riemannian smooth manifold $Y$ let us consider the algebra of smoothing operators
$J_c (Y):= C^\infty (Y\times Y)$. Then the functional $J_c (Y) \ni k\rightarrow \int_Y k|_\Delta {\rm dvol}$
defines a 0-cocycle $\tau_0$ on $J_c (Y)$; indeed by Lidski's theorem the functional is nothing but
the functional analytic trace of the integral operator corresponding to $k$ and we know that
the trace vanishes on commutators; in formulae, $b\tau_0 =0$. The 0-cocycle $\tau_0$
plays a fundamental role in the proof of the Atiyah-Singer index theorem, but we leave this
aside for the time being.

Let now $X$ be a smooth orientable manifold with cyclindrical ends, obtained from a manifold 
with boundary $X_0$; let $\cyl (\pa X)=\RR\times \pa X_0$.
 We have then defined algebras $A_c (X)$, $B_c (\cyl (\pa X))$ and $J_c (X)$ fitting into a short exact sequence
 $0\to J_c (X)\to A_c (X)\xrightarrow{\pi_c}B_c (\cyl (\pa X))\to 0$. 
 
 Corresponding to the 0-cocycle $\tau_0$ in the closed
 case we can define two important 0-cocycles  on a manifold with cyclindrical ends
$X$:
 
\begin{itemize} 
\item We can consider $\tau_0$ on $J_c (X)=C^\infty_c (X\times X)$; this is well defined
and does define a 0-cocycle . 
\item Starting with the 
0-cocycle $\tau_0$ on $J_c (X)$ we
 define a {\it relative} 0-cocycle $(\tau_0^r,\sigma_1)$ for $A_c (X)\xrightarrow{\pi_c}B_c (\cyl (\pa X))$.
  The relative 0-cocycle   $(\tau_0^r,\sigma_1)$ is obtained 
 through the following two fundamental steps.
 \begin{enumerate}
 \item We  define a 0-{\it cochain} $\tau_0^r$ on $A_c (X)$ by replacing the integral with Melrose' regularized
 integral.
 \item We define a 1-cocycle $\sigma_1$ on $B_c (\cyl (\pa X))$ by 
 taking  
a {\it suspension} 
 of $\tau_0$ through the linear map $\delta(\ell):= [\chi^0,\ell]$. In other words,   $\sigma_1 (\ell_0,\ell_1)$
 is obtained from $\tau_0\equiv \Tr$ by 
 considering $(\ell_0,\ell_1)\rightarrow \tau_0 (\ell_0 [\chi^0,\ell_1])\equiv \tau_0 (\ell_0 \delta(\ell_1))$.
\end{enumerate}
  \end{itemize}


\begin{definition}
We shall also call   Roe's 1-cocycle $\sigma_1$ {\it the  eta 1-cocycle corresponding to the 
0-cocycle $\tau_0$}. 
\end{definition}

In order to justify the wording of this definition we need to show that all this has something to
do with the eta invariant and its role in the Atiyah-Patodi-Singer index formula. This will be  explained 
in Subection \ref{subsection:aps}.

\subsection{Cyclic cocycles on graded algebras endowed with commuting derivations}\label{subsect:ccgraded}

In this subsection we collect some general facts about cyclic cocycles on algebras endowed
with commuting derivations. These results will be repeatedly used in the sequel.

 The following Lemma is
obvious.

\begin{lemma}\label{lemma:existence-omega}
Let $A^0$ be an algebra and $A^1$ a bimodule over $A^0$.  
Consider $\Omega:=A^0\oplus A^1$ as a linear space;
define a  multiplication on $\Omega$   by 
$$\alpha\beta=(a_0 b_0, a_0 b_1+a_1 b_0)\;\;\text{ for }\;\;\alpha=(a_0,a_1), \beta=(b_0,b_1) \;\;\text{ in }\;\; A^0\oplus A^1\,.$$
Then $\Omega$ is a graded algebra, with the  grading on $\Omega$
defined by $\deg a_i = i$ for $a_i\in A^i$, $i=0,1$. Observe that 
 $\Omega$  is not a $\star$-algebra.
\end{lemma}

\begin{definition}\label{def:deriv}
A linear map $\delta:\Omega\to \Omega$ is called a derivation of degree $k$ if it satisfies:
\item{1)} $\delta(\alpha\beta)=(\delta \alpha)\beta + \alpha (\delta\beta)$ for $\alpha, \beta\in \Omega$;
\item{2)} $\deg(\delta\alpha)= \deg\alpha+k$
\end{definition}

Let $\delta_{\Omega}:\Omega\to\Omega$ be a derivation on $\Omega$. Suppose that $\delta_{\Omega}$ is of degree
0 and denote by $\delta:A^0\to A^0$ and $\delta^\prime: A^1\to A^1$ the restrictions  $\delta_{\Omega}|_{A^0}$,
$\delta_{\Omega}|_{A^1}$ respectively. Then the derivation property of $\delta_{\Omega}$ is equivalent to the following
three properties:
\begin{align*}
\delta(a_0 b_0) &=  (\delta a_0) b_0 + a_0 (\delta b_0)\\
\delta^\prime (a_0 b_1) &= (\delta a_0) b_1 + a_0 (\delta^\prime b_1)\\
\delta^\prime (a_1 b_0) &= (\delta^\prime a_1) b_0 + a_1 (\delta b_0)
\end{align*}
for $a_i, b_i \in A^i$.
We also observe that giving  a derivation $\delta_{\Omega}$ of degree 1 is equivalent to
giving a linear map $\delta:A^0\to A^1$ with 
$\delta (a_0 b_0)=(\delta a_0)b_0 + a_0 (\delta b_0)$ in such a way that $\delta_{\Omega} a_0=
\delta a_0$ and  $\delta_{\Omega} a_1=0$ for $a_i\in A^i$.\\
Finally, let $\omega:A^1\to \CC$ be a linear map such that
\begin{equation}\label{bimodule-trace}\omega(a_0 a_1)=\omega(a_1 a_0)\:\:\text{ for } \:\:a_i\in A^i \,.
\end{equation}
We shall call such a linear map a {\it trace map} on the bimodule $A^1$.
The following Lemma is obvious 

\begin{lemma}\label{lemma:trace-bimod}
A trace map $\omega$ on the bimodule $A^1$ extends to a trace map $\tau_0:\Omega\to\CC$
so that $\tau_0 |_{A^0}=0$ and $\tau_0 |_{A^1}=\omega$. 
\end{lemma}

Let $\delta_i$, $i=1,\dots,k$, be derivations on $\Omega$ and $\tau_0$ a trace map on $\Omega$.
We consider the following $k$-linear map on $\Omega$:
\begin{equation}\label{general-k-linear}
\tau(a_0, \dots, a_k):= \frac{1}{k!}\,\sum_{\alpha\in \mathfrak{S}_k} \,{\rm sign}(\alpha)\,
\tau_0 (a_0 \;\delta_{\alpha (1)} a_1\; \dots\;\delta_{\alpha(k)} a_k)\quad\text{for}\quad a_i\in\Omega\,.
\end{equation}

\begin{proposition}\label{prop:cyclic-commut-deriv}
Assume that \begin{enumerate}
\item the derivations $\delta_i$ are pairwise  commuting, i.e. $\delta_i \delta_j=\delta_j \delta_i$ for $1\leq i,j\leq k$;
\item $\tau_0 (\delta_i a)=0$ for $a\in\Omega$ and $1\leq i\leq k$.
\end{enumerate}
Then $\tau$ defined in \eqref{general-k-linear} gives rise to a cyclic $k$-cocycle on $\Omega$.
\end{proposition}

\begin{proof}
First we verify the cyclic condition. The second assumption  and the derivation property imply that
\begin{align*}
\tau(a_k,a_0, \dots, a_{k-1}) 
&=  \frac{1}{k!}\,\sum_{\alpha\in \mathfrak{S}_k} \,{\rm sign}(\alpha)\,
\tau_0 (a_k \;\delta_{\alpha (1)} a_0\; \dots\;\delta_{\alpha(k)} a_{k-1})\\
&=  \frac{1}{k!}\,\sum_{\alpha\in \mathfrak{S}_k} \,{\rm sign}(\alpha)\,\left( 
\tau_0 (a_k \;\delta_{\alpha (1)} a_0\; \dots\;\delta_{\alpha(k)} a_{k-1}) - \tau_0 \left( 
\delta_{\alpha(1)} (a_k a_0 \;\delta_{\alpha (2)} a_1\; \dots\;\delta_{\alpha(k)} a_{k-1}) \right)  \right)\\
&= -  \frac{1}{k!}\,\sum_{\alpha\in \mathfrak{S}_k} \,{\rm sign}(\alpha)\,
\tau_0 ((\delta_{\alpha(1)} a_k) a_0 \;\delta_{\alpha (2)} a_1\; \dots\;\delta_{\alpha(k)} a_{k-1})\\
&\:\:\:\:\:-  \frac{1}{k!}\sum_{\alpha\in \mathfrak{S}_k} \,{\rm sign}(\alpha)\,\sum_{i=1}^{k-1}\tau_0 (a_k a_0 \delta_{\alpha (2)}
a_1 \dots \delta_{\alpha (1)}  \delta_{\alpha (i+1)} a_i \dots \delta_{\alpha (k)} a_{k-1} )
\end{align*}
The second summand in the last term vanishes. 
In fact the signatures are opposite to each other for $\alpha$ and  $\alpha\circ(1,i+1)$;
thus the values cancel out due to  assumption (1). 
Observing that the signature of the cyclic permutation $(1,2,\dots,k)$ is equal to $(-1)^{k-1}$, 
the trace property implies that 
\begin{align*}
\tau(a_k,a_0, \dots, a_{k-1})
&= \frac{(-1)^k}{k!}\,\sum_{\alpha\in \mathfrak{S}_k}\tau_0 (a_0 \;\delta_{\alpha (1)} a_1\; \dots\;\delta_{\alpha(k)} a_k)\\
&= (-1)^k\tau(a_0, a_1, \dots, a_{k})
\end{align*}
Second we prove the cocycle condition. 
Due to the derivation and trace properties again we obtain
\begin{align*}
b\tau(a_0, \dots, a_{k+1})
&= \sum_{i=0}^k(-1)^i\tau(a_0, \dots , a_ia_{i+1}, \dots, a_{k+1})
+(-1)^{k+1}\tau(a_{k+1}a_0, a_1, \dots , a_{k})\\
&= \frac{(-1)^{k}}{k!}\,\sum_{\alpha\in \mathfrak{S}_k} \,{\rm sign}(\alpha)\,
\left[ \tau_0(a_0 \;\delta_{\alpha (1)} a_1\; \dots\;\left(\delta_{\alpha(k)} a_k\right)a_{k+1})
- \tau_0(a_{k+1}a_0\;\delta_{\alpha (1)} a_1\; \dots\;\delta_{\alpha (k)} a_k)
\right]\\
&=0.
\end{align*}
This complets the proof.
\end{proof}

\subsection{
The  Godbillon-Vey cyclic 2-cocycle $\tau_{GV}$}\label{subsect:absolute-taugv}
Let $(Y,\F)$, $Y=\tN\times_\Gamma T$, be a foliated bundle {\it without} boundary. We take 
directly $T=S^1$.
 Let $E\to Y$ a hermitian complex vector bundle
on $Y$.  Let $(G,s:G\to Y,r:G\to Y)$ be the holonomy groupoid
associated to $Y$, $G=(\tN\times\tN\times T)/\Gamma$. Consider again  the convolution algebra 
$\Psi^{-\infty}_c (G,E):=C^\infty_c (G, (s^*E)^*\otimes r^*E)$ 
 of equivariant smoothing
families with $\Gamma$-compact support.  
On  $\Psi^{-\infty}_c (G,E)$ there exists  a remarkable 2-cocycle, denoted $\tau_{GV}$, and 
known as the Godbillon-Vey cyclic cocycle. It was defined by Moriyoshi and Natsume in
\cite{MN}. See the Appendix \ref{appendix:absolute-taugv}
for a self-contained  summary of the theory developed in \cite{MN}.
Here we shall simply recall the very basic facts leading to the definition of $\tau_{GV}$.

Recall from the Subsection 
on the Godbillon-Vey class, Subsection \ref{subsect:gv-form},
the modular function  $\psi$ on $\tN\times T$ defined by $\tilde{\omega}\wedge d\theta=\psi \tilde{\Omega}$,
$\tilde{\omega}$ and $ \tilde{\Omega}$ denoting $\Gamma$-invariant volume forms on $\tN$ and $\tN\times T$ respectively.
%


There is a well defined  derivation $\delta_2$ on the algebra $\Psi^{-\infty}_c (G,E)$:
\begin{equation}\label{delta_2}
\delta_2 (P)=[\phi,P]\quad\text{with}\quad  \phi=\log \psi\,.
\end{equation}
We observe here that $\phi$ is not $\Gamma$-invariant nor it is compactly supported.
In fact $\phi$ is not even bounded.
Recall next the bundle $\widehat{E}^\prime$ on $\tN\times T$ introduced in \cite{MN}: this is the same as
 $\widehat{E}$
but with a new $\Gamma$-equivariant structure. See \cite{MN}
or the Appendix of this paper.
There is  a natural bijection between
$\Psi^{-\infty}_c (G,E)$ and $\Psi^{-\infty}_c (G,E^\prime).$
Using this identification we see that the space $\Psi^{-\infty}_c (G;E,E^\prime)$ can be  considered as a bimodule over $\Psi^{-\infty}_c (G,E)$.
Consider $\dot{\phi}$, the partial derivative of $\phi$ in the direction of $S^1$.
There is a well defined bimodule
derivation $\delta_1 : \Psi^{-\infty}_c (G,E)\to \Psi^{-\infty}_c (G;E,E^\prime)$:
\begin{equation}\label{delta_1}
\delta_1 (P)=[\dot{\phi},P]\quad\text{with}\quad  \phi=\log \psi\,.
\end{equation}
There is also a  linear map 
$\delta_2^\prime : \Psi^{-\infty}_c (G;E,E^\prime)\to \Psi^{-\infty}_c (G;E,E^\prime)$ 
defined in a way similar to  \eqref{delta_2}.
One can also check that
\begin{equation}\label{deriv-commute}
\delta_1 (\delta_2 (P))= \delta^\prime_2 (\delta_1 (P))
\end{equation}
Recall, finally, that there is a weight $\omega_{\Gamma}$
defined on the
algebra $\Psi^{-\infty}_c (G;E)$,
\begin{equation}\label{weight-again}
\omega_{\Gamma}(k)=
 \int_{Y(\Gamma)} \mathrm{Tr}_{(\tilde{n},\theta)} k(\tilde{n},\tilde{n},\theta)d\tilde{n}\,d\theta\,.
\end{equation}
In this formula $Y(\Gamma)$ is the fundamental domain in $\tN\times T$ for the free diagonal action
of $\Gamma$  on $\tN\times T$ and  we have restricted the kernel $k$ to $\Delta_{\tN}\times T\subset
\tN\times\tN\times T$,  $\Delta_{\tN}$ denoting the diagonal 
set in $\tN\times \tN$, $\Delta_{\tN}\times T\equiv \tN\times T$;
$\mathrm{Tr}_{(\tilde{n},\theta)}$ denotes the trace on $\End (\widehat{E}_{(\tilde{n},\theta)})$.
(If the measure on $T$ is $\Gamma$-invariant, then this weight is
a trace; however, we don't want to make this assumption here.) 

We shall be interested in the linear functional \footnote{This will not be a weight, given
that on a bimodule there is no notion of positive element} defined on the bimodule
 $\Psi^{-\infty}_c (G;E,E^\prime)$
by the analogue of \eqref{weight-again}. 
We call this linear functional on  $\Psi^{-\infty}_c (G;E,E^\prime)$
the  {\it bimodule trace}.
%
The following fundamental relation, justifying the name {\it bimodule trace}, is proved in \cite{MN}:
\begin{equation}\label{weight-commute}
k\in \Psi^{-\infty}_c (G;E,E^\prime)\,, k^\prime \in \Psi^{-\infty}_c (G;E)\equiv \Psi^{-\infty}_c (G;E^\prime) \;\;
\Rightarrow \;\;\omega_{\Gamma}(k k^\prime)=\omega_{\Gamma}( k^\prime k)
\end{equation}
It is also important to recall that the bimodule 
trace $\omega_\Gamma$ on  $\Psi^{-\infty}_c (G; E, E^\prime)$ satisfies the
following analogues of Stokes formula:

\begin{equation}\label{exact-weight}
k\in \Psi^{-\infty}_c (G;E)
\Rightarrow \;\;\omega_{\Gamma}(\delta_1 (k))=0\,,\quad k\in \Psi^{-\infty}_c (G;E,E^\prime)
\Rightarrow \;\;\omega_{\Gamma}(\delta_2^\prime (k))=0\,.
\end{equation}

\begin{definition}\label{def:gv-cocycle}
With $1=\dim T$, 
the Godbillon-Vey cyclic $2$-cocycle 
on  $\Psi^{-\infty}_c (G;E)\equiv C^\infty_c (G, (s^*E)^*\otimes r^*E)$  is defined to be 
\begin{align}\label{taugv-sum}
\tau_{GV} (a_0, a_1, a_2)&= \frac{1}{2!}\sum_{\alpha\in \mathfrak{S}_2} {\rm sign}(\alpha)\,
\omega_\Gamma (a_0 \;\delta_{\alpha (1)} a_1 \;\delta_{\alpha (2)} a_2)\\
&=\frac{1}{2}\,
\left( \omega_\Gamma (a_0\; \delta_1 a_1\;\delta_2  a_2) - \omega_\Gamma(a_0 \;\delta_2 a_1\; \delta_1 a_2) \right)
\end{align}
\end{definition}
Remark that the Morioyshi-Natsume cocycle  is equal to  twice the above cocycle.
\begin{proposition}\label{prop:cyclic-cocycle-tgv}
The $3$-functional $\tau_{GV}$ does satisfy 
\begin{equation}\label{taugv-cyclic-cocycle}
b \tau_{GV}=0\,, \quad \tau (a_0,a_1,a_2)= \tau(a_1,a_2,a_0) \,,\forall a_j\in \Psi^{-\infty}_c (G;E)
\end{equation}
\end{proposition}


%
 
 \begin{proof}
 This Proposition is of course proved in \cite{MN}.
We give a proof here using the general results of the
previous Subsection; this will serve as a guide for the more
complicated situation  that we will need to consider later.
Recall the definition of $\tau_{GV}$:
$$\tau_{GV} (a_0, a_1, a_2):=\frac{1}{2}\,
\left( \omega_\Gamma (a_0\; \delta_1 a_1\;\delta_2  a_2) - \omega_\Gamma(a_0 \;\delta_2 a_1\; \delta_1 a_2) \right)$$
where $\delta_1 a = [\dot{\phi},a]$ and $\delta_2 a = [\phi,a]$. Let $A^0$ be the algebra $\Psi^{-\infty}_c (G;E)$ and let $A^1$ be
the $A^0$-bimodule
$\Psi^{-\infty}_c (G;E,E^\prime)$ introduced above. 
Proceeding as in Subsection \ref{subsect:ccgraded},  we consider the graded algebra $\Omega$
built out of $A^0$ and $A^1$, as in Lemma \ref{lemma:existence-omega}. We denote this algebra by $\Omega(G)$. 
Then, according to the explanations given in Subsection \ref{subsect:ccgraded}, there exists extensions of our derivations to 
\begin{equation}\label{deriv-on-omega}
\delta_j : \Omega (G)\to \Omega(G),\:\:j=1,2 \quad \text{with}\quad \delta_1 a = [\dot{\phi},a], \:\:\delta_2 a = [\phi,a]
\end{equation}
with $\delta_1$ of degree 1 and $\delta_2 $ of degree 0.
Notice that we have employed the same notation for these extensions. On the other hand,  we know that
the functional
$\omega_\Gamma: \Psi^{-\infty}_c (G;E,E^\prime)\to\CC$ defined using  \eqref{weight-again} gives rise to a trace map on the bimodule $A^1$, due to property
\eqref{weight-commute}. Thus Lemma \ref{lemma:trace-bimod}, which is obvious,
implies that the there exists a trace map
$\tau_\Gamma : \Omega(G)\to \CC$ with $\tau_\Gamma (a)=\omega_\Gamma (a)$, for $a\in A^1$ and
 $\tau_\Gamma (a)=0$ for $a\in A^0$. Now \eqref{deriv-commute} shows that the derivations $\delta_1$, $\delta_2$ introduced 
 in \eqref{deriv-on-omega} commute with each other, whereas \eqref{exact-weight} implies that $\tau_{\Gamma} (\delta_j a)=0$
 for $a\in \Omega(G)$ and $j=1,2$. Thus, directly from Proposition \ref{prop:cyclic-commut-deriv}, we obtain a cyclic 2-cocycle 
 $$\tau_2 (a_0,a_1,a_2)=\frac{1}{2}\,
\left( \tau_\Gamma (a_0\; \delta_1 a_1\;\delta_2  a_2) - \tau_\Gamma(a_0 \;\delta_2 a_1\; \delta_1 a_2) \right)\,.$$
Thus $\tau_{GV}$ is also a cyclic cocycle on $A^0\equiv \Psi^{-\infty}_c (G;E)$, since it is nothing but the
restriction of $\tau_2$ to the subalgebra $A^0\subset \Omega(G)$. The Proposition is proved.
\end{proof}

\smallskip

We now go back to a foliated bundle $(X,\F)$ with cylindrical ends, with $X:= \tV\times_\Gamma T$, as in Section
\ref{sec:data}. We consider the small subalgebras introduced in Subsection \ref{subsect:small-dense}. 
The weight $\omega_\Gamma$ is still well defined on $J_c (X,\F)$;
the 2-cocycle $\tau_{GV}$ can thus be defined on $J_c (X,\F)$, giving us the 
Godbillon-Vey cyclic cocycle of $(X,\F)$.

\subsection{The eta 3-coycle $\sigma_{GV}$ corresponding to $\tau_{GV}$}\label{subsect:sigma3}
Now we apply the general philosophy explained at the end of the previous Section. 
Let $\chi^0$ be the usual characteristic function of $(-\infty,0]\times \pa X_0$ in $\cyl (\pa X)=\RR\times \pa X_0$.
Write $\cyl (\pa X)= (\RR\times\pa \tM)\times_\Gamma T$ with $\Gamma$ acting trivially on the $\RR$ factor.
Let $\cyl (\Gamma)$ be a fundamental domain for the action of $\Gamma$ on $(\RR\times\pa \tM)\times T$;
finally, 
let $\omega_\Gamma^{\,{\rm cyl}}$ be the corresponding trace map on the bimodule defined 
similarly to \eqref{weight-again}.
Recall $\delta (\ell): =[\chi^0,\ell]$;
recall that we passed from the 
0-cocycle $\tau_0\equiv \Tr$
to the 1-eta cocycle on the cylindrical algebra $B_c$ by considering $(\ell_0, \ell_1)\rightarrow \tau_0 (\ell_0 \delta (\ell_1))$.
We referred to this operation as a {\it suspension}.

We are thus led  to {\it suspend}
definition \ref{def:gv-cocycle}, thus defining the following 4-linear functional on the algebra $B_c$.

\begin{definition}\label{def:sigma3}
 The eta  functional $\sigma_{GV}$ associated to the absolute 
Godbillon-Vey 2-cocycle $\tau_{GV}$
is  given by the 4-linear functional on $B_c$
\begin{equation}\label{sigma3}
\sigma_{GV} (\ell_0, \ell_1, \ell_2,\ell_3):= \frac{1}{3 !}\,\sum_{\alpha\in \mathfrak{S}_3} \,{\rm sign}(\alpha)\,
\omega_\Gamma^{\,{\rm cyl}} (\ell_0 \;\delta_{\alpha (1)} \ell_1\; \delta_{\alpha (2)} \ell_2 \;\delta_{\alpha(3)} \ell_3)\,.
\end{equation}
with
\begin{equation}\label{3-derivations}
\delta_3 \ell:= [\chi^0,\ell]\,,\quad \delta_2 \ell:=[\phi_{\pa},\ell]\,\quad\text{and}\quad
 \delta_1 \ell :=[\dot{\phi}_{\pa},\ell]
 \end{equation}
and  $\phi_{\pa}$ equal to the restriction of the modular function to the boundary, extended in a constant
way along the cylinder.
We shall prove that  this is a cyclic 3-cocycle for the algebra $B_c (\cyl (\pa X),\F_{{\rm cyl}})$. More generally,
formula \eqref{sigma3}
 defines {\it the Godbillon-Vey eta 3-cocycle} on  $B_c (\cyl (Y),\F_{{\rm cyl}})$ with $Y=\tN\times_\Gamma T$ any closed foliated
 $T$-bundle, not necessarily arising as a boundary. Here, as usual,
 we don't write the bundle $E$ in the notation.
  In this case $\delta_2 \ell:=[\phi_Y,\ell]$ and $\delta_1 \ell :=[\dot{\phi}_{Y},\ell]$
 with $\phi_Y$ the logarithm of a modular function on $Y$ extended in a constant
way along the cylinder.

\end{definition}

We must justify the well-posedness of this definition. To this end, remark that  each sum will contain
an element of type $\delta_3 (\ell_j):=[\chi^0, \ell_j]$. This is a 
kernel of $\Gamma$-compact support (we have already justified this claim in Sublemma
\ref{sublemma:basic1}) which is, of course, not translation invariant. Since the other three operators
appearing in the composition $(\ell_0 \;\delta_{\alpha (1)} (\ell_1)\; \delta_{\alpha (2)} (\ell_2) \;\delta_{\alpha(3)} (\ell_3))$
are 
$(\RR\times \Gamma)$-equivariant and of $(\RR\times\Gamma)$-compact support, 
we can conclude easily that each term appearing in the definition of $\sigma_{GV}$,
$(\ell_0 \;\delta_{\alpha (1)} (\ell_1)\; \delta_{\alpha (2)} (\ell_2) \;\delta_{\alpha(3)} (\ell_3))$, is in fact 
of $\Gamma$-compact support.  
Indeed, recall that a kernel  that is  $\Gamma$-equivariant and of $\Gamma$-compact support,  
such as   $\delta_3 (\ell_j)=[\chi^0, \ell_j]$ above, 
can be considered as  a compactly supported function  on the holonomy groupoid $G_{\rm cyl}$ for 
$(\cyl (\pa X),\F_{{\rm cyl}})$. 
On the other hand,  a kernel that is  $(\RR\times\Gamma)$-equivariant and 
of $(\RR\times\Gamma)$-compact support  
corresponds to a compactly supported function on  $G_{\rm cyl} /\RR_{\Delta}$, 
which admits a $\RR$-compact support  once lifted to  $G_{\rm cyl}$; 
see Propositon  \ref{prop:bstar-structure}. 
We then take the convolution product of these kernels. 
A simple argument on support  implies that the resulting kernel
 corresponds to a compactly supported function on  $G_{\rm cyl}$ and hence 
the kernel itself is of $\Gamma$-compact support on 
$(\RR\times\pa \tM)\times (\RR\times\pa \tM) \times T$.

Summarizing, $\omega_\Gamma^{\,{\rm cyl}}  (\ell_0 \;\delta_{\alpha (1)} (\ell_1)\; \delta_{\alpha (2)} (\ell_2) \;\delta_{\alpha(3)} (\ell_3))$
is finite and the definition of $\sigma_{GV}$ is well posed.

In fact, we can define, as we did for $\sigma_1$, the 3-cochain $\sigma_{GV}^\lambda$  by
employing the characteristic function $\chi^\lambda$. However, one checks as in Proposition
\ref{prop:lambda-indipendent}
 that the value of 
$\sigma_{GV}^\lambda$ does not depend on $\lambda$.


%
Next  we have the important
\begin{proposition}\label{prop:sigma3-cyclic}
Let $Y=\tN\times_\Gamma T$ be  an arbitrary foliated $T$-bundle {\it without} boundary.
The eta functional  $\sigma_{GV}$ on $B_c (\cyl (Y),\F_{{\rm cyl}})$ is cyclic and is a cocycle: $b \,\sigma_{GV}=0$; it thus defines a cyclic 3-cocycle on the algebra $B_c (\cyl (Y),\F_{{\rm cyl}})$.
\end{proposition}

\begin{proof}
We wish to apply Proposition \ref{prop:cyclic-commut-deriv} as we did 
in the  proof of the  2-cyclic-cocycle property for $\tau_{GV}$, see Proposition \ref{prop:cyclic-cocycle-tgv} . 
However, we need to deal with a small complication, having to do with the fact that
 $\chi^0$ is not smooth and that $[\chi^0,\ell]$ is no longer
 translation invariant. Recall
the groupoid $G_{\cyl}:= \cyl (\tilde{N})\times \cyl(\tilde{N})\times T/\Gamma$
which is nothing but $G_Y\times \RR\times \RR$ with $G_Y$ the holonomy
groupoid for $Y=\tN\times_\Gamma T$.
Define
$$L^\infty_c (G_{\cyl})=\{k:G_{\cyl}\to \CC\;|\; k\text{ is measurable, essentially bounded and of  
$\Gamma$-compact support}\}
$$
More generally, let $E$ be a vector bundle on $Y$ with lift $\widehat{E}$
on $\tN\times T$; we pull  back $E$ to
$\cyl(Y)$ through the obvious projection obtaining a vector bundle $E_{\cyl}$.
We can then consider in a natural way
$L^\infty_c (G_{\cyl};E_{\cyl})$ and $L^\infty_c (G_{\cyl});E_{\cyl},E^\prime_{\cyl})$. 
We omit the obvious details.
Recall also 
$$B_c (G_{\cyl})\equiv B_c:=\{\ell:G_{\cyl}\to\CC\;\;|\;\; \ell \text{ is smooth, }
\RR\times\Gamma-\text{invariant and of }\;\RR\times\Gamma-\text{support}\}\,.$$
We also have $B_c (G_{\cyl};E)$ (this is the algebra on which $\sigma_{GV}$
is defined) and  $B_c (G_{\cyl};E,E^\prime)$.
We set now
\begin{align*}
A^0&:=
B_c (G_{\cyl};E)\oplus L^\infty_c (G_{\cyl};E_{\cyl})\\
A^1&:=B_c (G_{\cyl};E,E^\prime)\oplus L^\infty_c (G_{\cyl};E_{\cyl},E^\prime_{\cyl})\,.
\end{align*}
First, observe here that
$A^0$ and $A^1$ are naturally considered as subspaces in $\End (\H)$ and ${\rm Hom}(\H,\H^\prime)$ respectively, where we recall that $\H=(\H_\theta)_{\theta\in T}$, $\H_\theta=
L^2 (\cyl(\tN)\times\{\theta\}, E_{\cyl,\theta})$ and similarly for $\H^\prime$; indeed, each summand
of $A^0$, for example, is in $\End (\H)$ and the direct sum holds because of the support
conditions.
Next we observe that $A^0$ is in fact as a subalgebra of $\End (\H)$, since the product of 
$k\in B_c (G_{\cyl};E)$ and $k'\in  L^\infty_c (G_{\cyl};E_{\cyl})$ is an element in
$L^\infty_c (G_{\cyl};E_{\cyl})$. Moreover, for the same reason,
$A^1$ has a bimodule structure over $A^0$,
inherited from the one of ${\rm Hom}(\H,\H^\prime)$ over $\End (\H)$.
The direct sum $\Omega:= A^0\oplus A^1$, with the product defined in
Lemma  \ref{lemma:existence-omega}, is the graded algebra to which we 
want to apply Proposition \ref{prop:cyclic-commut-deriv}.

We can define three derivations $\delta_1$, $\delta_2$ and $\delta_3$
as in \eqref{3-derivations}.
We consider $\delta_1$ as a derivation of degree 1, mapping $A^0$ to $A^1$
and vanishing on $A^1$; we consider $\delta_2$ and $\delta_3$ as derivations
of degree 0, preserving $A^0$ and $A^1$ respectively. Notice that since $\phi_\pa$
and $\dot{\phi}_{\pa}$ are translation invariant on the cylinder, $\delta_1$ and $\delta_2$
are diagonal with respect to the direct sum decomposition of $A^0$ and $A^1$.
As far as $\delta_3$ is concerned, we remark that 
using \eqref{kernel-for-commutator-bis}
we see that $\delta_3$ maps $B_c\oplus L^\infty_c$ into $L^\infty_c$ both on $A^0$
and $A^1$.
It is clear that these three derivations are pairwise commuting. Finally, we define a 
bimodule trace map
on $A^1$ by employing the bimodule trace $\omega_\Gamma^{\cyl}$ appearing in
Definition \ref{def:sigma3}; this is well defined on $L^\infty_c (G_{\cyl};E,E^\prime)$
since elements in this space have $\Gamma$-compact support. We can then define
$\omega:A^1\to\CC$ by $\omega (\alpha)=\omega_\Gamma^{\cyl} (k)$
if $\alpha=(\ell,k)\in A^1\equiv B_c (G_{\cyl};E,E^\prime)\oplus L^\infty_c (G_{\cyl};E_{\cyl},E^\prime_{\cyl})$. We know that $\omega( \delta_j \alpha)=0$ if $j=1,2$. On the other hand,
always for $\alpha=(\ell,k)= \ell+k\in A^1$,
we have  
\begin{align*}
\omega (\delta_3 \alpha)=& \omega ([\chi^0,\ell+k])=\omega^{\cyl}_\Gamma
([\chi^0,\ell+k])\\
=&\int _{\cyl (\Gamma)} \Tr_{(\tn,s,\theta)} \big(
\chi^0 (s) (\ell (\tn,\tn,0,\theta)+ k(\tn,\tn,s,s,\theta))- ( \ell (\tn,\tn,0,\theta)+ k(\tn,\tn,s,s,\theta))\chi^0 (s)
\big) \,d\tilde{n}\,d\theta\\
= & 0
\end{align*}
Thus, we also have Stokes formula for the derivation $\delta_3$. Now we define $\tau_0$
from $\omega$ as in Lemma \eqref{lemma:trace-bimod} so that all the conditions in the hypothesis of Proposition
\ref{prop:cyclic-commut-deriv} are satisfied. Finally, we point out that
$B_c$ is a subalgebra of $A^0$: proceeding exactly as in the proof of Proposition \ref{prop:cyclic-commut-deriv}  we can now check that $\sigma_{GV}$ is indeed a cyclic 3-cocycle on $B_c$.

\end{proof}
\subsection{The relative Godbillon-Vey cyclic cocycle $(\tau^r_{GV},\sigma_{GV})$}
We now apply our strategy as in Subsection \ref{subsect:strategy}. Thus starting with the 
cyclic cocycle
$\tau_{GV}$ on $J_c (X,\F)$ 
 we first consider the 3-linear functional  on $A_c (X,\F)$ given by
 \begin{equation*}
 \psi_{GV}^r (k_0,k_1,k_2) :=  \frac{1}{2!}\sum_{\alpha\in \mathfrak{S}_2} {\rm sign}(\alpha)\,
\omega_\Gamma^r (a_0 \;\delta_{\alpha (1)} a_1 \;\delta_{\alpha (2)} a_2)
\end{equation*}
with $\omega^r_{\Gamma}$ the regularized weight corresponding to $\omega_{\Gamma}$. The definition of
$\omega^r_{\Gamma}$ is clear: consider $X=\tV\times_\Gamma T$, a free $\Gamma$-quotient of $\tV\times T$; 
consider a fundamental domain $X(\Gamma)$ for this $\Gamma$-covering.
 $X(\Gamma)$ can be taken to be 
equal to
 $F \times T$, with $F$ a fundamental domain
for the Galois covering $\Gamma\to\tV\to V$, with $\tV= \tM\cup_{\partial \tM} \left(   (-\infty,0] \times \partial\tM \right)$
and $V:= M\cup_{\partial M} \left(   (-\infty,0] \times \partial M \right)$. 
See subsection \ref{subsection:cyl+notation}.
Thus $F$ has a cylindrical end, with cross section $F_\pa$
Then, using the usual notations, 
\begin{equation}\label{regularized-weight}
\omega^r_{\Gamma} (k):= \lim_{\lambda\to +\infty} 
\left( 
 \int_{F_\lambda\times T}  k(x,x,\theta) \,dx\,d\theta\; - \;\lambda\int_{F_\pa \times T} \ell (y,y,0, \theta) dy d\theta\, \right)\quad\text{where}\quad \pi_c (k)=
 \ell
\end{equation}
and where we have used the translation invariance of $\ell$ in order to write $\ell$
as a function of $(y,y^\prime,s,\theta)$, $s\in \RR$.
Notice that, as in Subsection \ref{subsection:melrose} the function
$\phi (\lambda):=    \int_{F_\lambda\times T}  k (x,x,\theta) \,dx\,d\theta\;-\;\lambda\int_{F_\pa \times T} \ell (y,y,0, \theta) dy d\theta$
becomes constant for $\lambda>>0$.

 Next we consider the {\it cyclic} cochain associated to  $\psi_{GV}^r $:
\begin{equation}\label{regularized-gv}
\tau_{GV}^r (k_0,k_1,k_2):=   \frac{1}{3} \left( \psi_{GV}^r (k_0,k_1,k_2) +  \psi_{GV}^r (k_1,k_2,k_0)
+  \psi_{GV}^r (k_2,k_0,k_1) \right) \,.
\end{equation}

The next Proposition is crucial:

\begin{proposition}\label{prop:gv-cocycle}
The relative cyclic cochain $(\tau_{GV}^r,\sigma_{GV})\in C^2_\lambda (A_c,B_c)$ is a relative
cocycle: thus \begin{equation}\label{gv-cocycle}
b\sigma_{GV}=0\quad\text{and} \quad b\tau_{GV}^r= (\pi_c)^* \sigma_{GV}\,.
\end{equation}
\end{proposition}
We shall present a proof of  Proposition \ref{prop:gv-cocycle} in Section \ref{section:proofs}.
 For later use we also state and prove the analogue of 
  formula  \eqref{b-trace=composition}:
 
 \begin{proposition}\label{prop:regularized-via-t}
 Let $t: A^* (X,\F)\to C^* (X,\F)$ be  the section introduced in Subsection \ref{subsec:extension}. 
 If $k\in A_c\subset A^* (X,\F)$ then $t(k)$ has finite weight.
Moreover, for  the regularized weight $\omega_\Gamma^r: A_c \to \CC$
 we have
 \begin{equation}\label{regularized-via-t}
\omega^r_\Gamma = \omega_\Gamma \circ t
\end{equation}

 \end{proposition}

 \begin{proof}
 The proof is virtually identical to the one establishing \eqref{b-trace=composition}.
Write $k=a+\chi^\lambda\,\ell\,\chi^\lambda$ with  $a \in J_c$ 
and $\ell\in B_c$. 
Remark that the support of $\chi^{\lambda}-\chi^0$ is compact. 
Thus 
$t(k)=k-\chi^0\ell\chi^0 
=a+\chi^\lambda\ell\chi^\lambda-\chi^0\ell\chi^0$ 
has certainly finite weight, given that it is of $\Gamma$-compact support.
Thus,
$$
\omega_\Gamma (t (k))=
\int_{F\times T} a(x,x,\theta)\,dx\,d\theta - \int_{F_{\lambda}\times T\setminus F_0\times T}\ell(y,y,0,\theta)dy\,dt\,d\theta
= 
\omega^r_\Gamma (k)
$$
for a sufficiently large $\lambda$. This completes the proof.
 \end{proof}

\subsection{Eta cocycles}\label{subsection:eta-cocycles}

The ideas explained in the previous sections can be extended to
 general cocycles $\tau_k\in HC^k (C^\infty_c (G, (s^*E)^*\otimes r^*E))$; we simply need
 to require that these cocycles  are in the image of a suitable 
 Alexander-Spanier homomorphism since we can then replace integrals with
 regularized integrals in the passage from absolute to relative cocycles. 
 This general theory will be treated elsewhere.
Here we only want to comment on the particular case
 of Galois coverings, since this case illustrates very well the general framework.
 In this important example the  techniques of this paper can be used in order to give 
 an alternative approach to the higher index theory developed in \cite{LPMEMOIRS}, much more
 in line with the original treatment given by Connes and Moscovici in their fundamental
 article \cite{CM}. 
 
We now give a very short treatment of
this important example, assuming a certain familiarity with the seminal work of Connes and Moscovici. 
Let  $\Gamma\to \tM\to M$ be a Galois covering with boundary
and let $\Gamma\to \tV\to V$ be the associated covering with cylindrical ends. In the closed case
higher indeces for a $\Gamma$-equivariant Dirac operator on $\tM$ are obtained through Alexander-Spanier
cocycles, so we concentrate directly on these.  Let $\phi$ be an Alexander-Spanier $p$-cocycle; for simplicity
we assume that $\phi$ is the sum of decomposable elements given by the cup product of Alexander-Spanier
1-cochains: $$\phi= \sum_i \delta f^{(i)}_1 \cup \delta f^{(i)}_2 \cup \cdots \cup \delta f^{(i)}_p\;\;\text{where} 
\;\;f_j^{(i)}:\tM\to \CC\;\;\text{is continuous}.$$
Here we assume that $\delta f_j^{(i)}$, $\delta f^{(i)}_j (\tm,\tm'):=(f^{(i)}_j (\tm')-
f^{(i)}_j (\tm))$ is  $\Gamma$-invariant with respect to the diagonal action
of $\Gamma$ on $\tM\times \tM$.
This is a non-trivial assumption.
 We 
shall omit $\cup$ from the notation.   The cochain $\phi$ is a cocycle (where we recall that for an Alexander-Spanier $p$-cochain given by a continuos
function $\phi: \tM^{p+1}\to \CC$ invariant under the diagonal $\Gamma$-action, one sets $\delta\phi (x_0,x_1,\dots,
x_{p+1}):= \sum_0^{p+1} (-1)^i \phi (x_0,\dots,\hat{x}_i,\dots,x_{p+1})$). Always in the closed case we obtain
a cyclic $p$-cocycle for the convolution algebra $C^\infty_c (\tM\times_\Gamma \tM)$ by setting
\begin{equation}\label{as-cyclic}
\tau_\phi (k_0,\dots,k_p)=\frac{1}{p!} \,  \sum_{\alpha\in \mathfrak{S}_p}\, \sum_i\, {\rm sign} (\alpha) \omega_\Gamma 
(k_0 \, \delta^{(i)}_{\alpha (1)} k_1 \cdots \delta^{(i)}_{\alpha (p)} k_p)\;\;\text{ with }\;\; \delta^{(i)}_j k := [k,f^{(i)}_j]\,.
\end{equation}
Notice that $[k,f^{(i)}_j]$ is the $\Gamma$-invariant kernel whose value at 
$(\tm,\tm')$ is given by $k(\tm,\tm') \delta f^{(i)}_j (\tm,\tm')$ which is by definition $k(\tm,\tm') (f^{(i)}_j (\tm')-
f^{(i)}_j (\tm))$; $\omega_\Gamma$ is as usual given by 
$\omega_\Gamma (k)=\int_F \Tr_{\tm} k(\tm,\tm)$, with $F$ a fundamental domain for the $\Gamma$-action.

Pass now to manifolds with boundary and associated manifolds with cylindrical ends. Consider
the small subalgebras $J_c (\tV)$, $A_c (\tV)$, $B_c (\pa \tV\times \RR)$ appearing 
in the (small) Wiener-Hopf extension constructed in Subection \ref{subsect:small-dense} (just take $T=$point there).
We write briefly  $J_c$, $A_c $, $B_c$ and $0\to J_c\to A_c \xrightarrow{\pi_c}B_c\to 0$. We adopt
the notation of the previous sections.
Given $\phi$ as above, we can clearly define an {\it absolute} cyclic p-cocycle $\tau_\phi$ on $J_c$.
Next, define the $(p+1)$-linear functional $\psi^r_\phi$ on $A_c$ by replacing the integral in $\omega_\Gamma$ with Melrose' regularized
integral. Consider next 
$$\tau_\phi^r (k_0, \dots, k_p):= \frac{1}{p+1} \left( \psi^r_\phi(k_0,k_1,\dots,k_p)+ \psi^r_\phi(k_1,\dots,k_p,k_0)+\cdots+
\psi^r_\phi(k_p,k_0,\dots,k_{p-1}) \right)\,.$$
This is 
a cyclic $p$-cochain on $A_c$. Finally, introduce the new derivation $\delta^{(i)}_{p+1} (\ell):= [\chi^0,\ell]$
with $\chi^0$ the function on $\pa \tV\times \RR$ induced by the characteristic function
of $(-\infty,0]$. Then
the eta cocycle associated to $\tau_\phi$ is given by
\begin{equation}\label{eta-cm}
\sigma_\phi (\ell_0,\dots,\ell_{p+1})=\frac{1}{(p+1)!}  \sum_{\alpha\in \mathfrak{S}_{p+1}}\, \sum_i\, {\rm sign} (\alpha) \omega_\Gamma 
(\ell_0 \, \delta^{(i)}_{\alpha (1)} \ell_1 \cdots \delta^{(i)}_{\alpha (p+1)} \ell_{p+1})
\end{equation}
It should be possible to prove, using the techniques of this Section,  that this is a cyclic $(p+1)$-cocycle for $B_c$ and that $(\tau_\phi^r, \sigma_\phi)$ is a {\it relative}
cyclic $p$-cocycle for the pair $(A_c,B_c)$. $\sigma_\phi$ is, by definition, the {\it eta cocycle} corresponding to
$\tau_\phi$.

\section{{\bf Smooth subalgebras}}\label{sec:shatten}

\subsection{Summary of this Section} $\;$

\noindent
The goal of this whole Section is to define
the   subsequence $$0\rightarrow \mathbf{\mathfrak{J}} \hookrightarrow  \mathbf{\mathfrak{A}} \rightarrow   \mathbf{\mathfrak{B}} \rightarrow 0$$
of $0\to C^* (X,\F)\to A^* (X,\F)\to
B^* (\cyl (\pa X),\F_{\cyl})\to 0$
we have alluded to in the Introduction and in Subsection \ref{subsection:intr-remarks}.
Since the definitions are somewhat involved, we have decided to give here
a brief account
of the main definitions and of the main results of the whole Section; this summary will
be enough for understanding the main ideas in the proof of our main theorem.

\medskip
\noindent
{\bf Step 1.} We begin by defining Shatten-type ideals $\I_m (X,\F)\subset C^* (X,\F)$; these are for each $m\geq 1$ dense and holomorphically closed  subalgebras of  $C^* (X,\F)$. (We shall eventually fix
$m$ greater than dimension of the leaves. ) By imposing that the kernels in $\I_m (X,\F)$
define bounded operators when multiplied by a function that goes like $(1+s^2)$
on the cylindrical end, we obtain the Banach algebras $\J_m (X,\F)\subset C^* (X,\F)$;
these are still dense and holomorphically closed. 

\medskip
\noindent
{\bf Step 2.} Next we define  dense holomorphically closed subalgebras 
$\B_m  (\cyl (\pa X),\F_{\cyl})\subset B^* (\cyl (\pa X),\F_{\cyl})$ (often simply denoted $\B_m$). \\To this end
we first define ${\rm OP}^{-1}  (\cyl (\pa X),\F_{\cyl})$, the closure of
$\Psi^{-1}_{\RR,c} (G_{\cyl})\subset  B^* (\cyl (\pa X),\F_{\cyl})$ with respect
to the norm $|||P|||:= {\rm max} (\|P \|_{-n,-n-1},\|P\|_{n+1,n})$, where on the
right hand side we have the norm for operators between Sobolev spaces and where
$n$ is a fixed integer greater or equal to the dimension of the leaves.
Next we define $\D_m$ as those elements in ${\rm OP}^{-1}  (\cyl (\pa X),\F_{\cyl})$
for which (a certain closure of ) the derivation $[\chi^0,\cdot]$ has values in $\J_m$.
We then define $\D_{m,\alpha}$ as $\D_m\cap {\rm Dom} (\partial_\alpha)$ with $\pa_{\alpha}$ the closed derivation associated to the $\RR$-action $\alpha_t$ defined 
by $\alpha_t (\ell):= e^{its} \ell e^{-its}$. $\B_m$ is obtained as a subalgebra of 
$\D_{m,\alpha}$: 
$\B_m =\{\ell\in \D_{m,\alpha}\,|\, [f,\ell]\;\;\text{and}\;\;[f,[f,\ell]]\;\;\text{are bounded}\,,\;\;
\text{with} \;\;f(y,s)=\sqrt{1+s^2}\}\,.$
We endow $\B_m$ with a Banach norm and we prove that it is a dense holomorphically
closed subalgebra of $\B^*$ for each $m\geq 1$.

\medskip
\noindent
{\bf Step 3.} We define 
$\A_m (X,\F):=\{k\in A^* (X,\F);\pi (k)\in \B_m  (\cyl (\pa X),\F_{\cyl}), t(k)\in \J_m (X,\F)\}$
with $t:A^* (X,\F)\to C^* (X,\F)$ defined in \eqref{eq:t}. 
We endow $\A_m$  with a norm that makes it a Banach subalgebra
of $A^*$

\medskip
\noindent
{\bf Step 4.} We prove that $\J_m$ is an ideal in $\A_m$ and that there is for each
$m\geq 1$ a short exact
sequence of Banach algebras $0\to\J_m (X,\F)\to \A_m (X,\F)\to \B_m  (\cyl (\pa X),\F_{\cyl})\to 0$.

\medskip
\noindent
{\bf Step 5.} Recall the function $\phi$, equal to the logarithm of the modular function.
Recall the (algebraic) derivations $\delta_1:= [\dot{\phi},\;]$ and $\delta_2:= [\phi,\;]$. We define suitable closures $\overline{\delta}_1$, $\overline{\delta}_2$ of these two derivations
and we define $\mathbf{\mathfrak{J}_m}$ as $\J_m\cap {\rm Dom} (\overline{\delta}_1)
\cap {\rm Dom} (\overline{\delta}_2)$. We endow $\mathbf{\mathfrak{J}_m}$
with a Banach norm and we remark that it is a dense holomorphically closed
subalgebra of $C^* (X,\F)$. Similarly, we define suitable closures of the derivations
 $\delta_1:= [\dot{\phi}_{\pa},\;]$ and $\delta_2:= [\phi_{\pa},\;]$ on the cylinder and we
 define  $\mathbf{\mathfrak{B}_m}$ as $\B_m\cap {\rm Dom} (\overline{\delta}_1)
\cap {\rm Dom} (\overline{\delta}_2)$. We endow  $\mathbf{\mathfrak{B}_m}$
with a Banach norm
and we show that it  is a dense holomorphically closed subalgebra
of $B^*$. Finally, we define in a similar way the Banach algebra 
$\mathbf{\mathfrak{A}_m}$; this is a subalgebra of $A^*$.

\medskip
\noindent
{\bf Step 6.} We prove that  $\mathbf{\mathfrak{J}_m}$ is an ideal in  $\mathbf{\mathfrak{A}_m}$ and that there is a short exact sequence of Banach algebras 
$0\rightarrow \mathbf{\mathfrak{J}_m} \hookrightarrow  \mathbf{\mathfrak{A}_m} \rightarrow   \mathbf{\mathfrak{B}_m} \rightarrow 0$. The subsequence we are interested
in is obtained by taking $m=2n+1$ in the above sequence, with $2n$ equal to the dimension of the leaves
in $(X,\F)$.

\subsection{Shatten ideals}\label{subsect:shatten-ideals}
Let $\chi_\Gamma$ be a characteristic function for a fundamental domain of $\Gamma\to \tV\to V$.
Consider $C^\infty_c (G)=:J_c (X,\F)\equiv J_c$.

\begin{definition}\label{def:shatten}
Let $k\in J_c$ be positive and self-adjoint. The Shatten norm $||k||_m$ of $k$ is defined as
\begin{equation}\label{shatten-norm}
(||k||_m)^m := 
\sup_{\theta\in T} || \chi_\Gamma \, (k(\theta))^m \chi_\Gamma ||_1 
\end{equation}
with the $||\,\,||_1$ denoting the usual trace-norm on the Hilbert space $\mathcal{H}_\theta$.
Equivalently 
\begin{equation}\label{shatten-norm-HS}
(||k||_m)^m = \sup_{\theta\in T} || \chi_\Gamma \, (k(\theta))^{m/2}  ||^2_{HS}\,.
\end{equation}
 with 
 $|| \,\,\,||_{HS}$ denoting the usual Hilbert-Schmidt norm.
 In general, we set $||k||_m:= ||\,\,(k k^*)^{1/2}\,\,||_m$.
The Shatten norm of $k\in J_c$ is easily seen to be finite for any $m\geq 1$.
\end{definition}

\begin{proposition}\label{prop:shatten-properties}
The following properties hold:
\begin{enumerate}
\item if $1/r = 1/p \,+\,1/q$ then $|| k k^\prime ||_r \leq || k||_p\,  ||k^\prime ||_q$;\\
\item if $r\geq 1$ then 
$|| k k^\prime||_r \leq || k ||_{C^*}\, || k^\prime ||_r$;
\item if $p<q$ then    $|| k||_p\geq || k||_q$; 
\item if  $p\geq 1$ then $|| k ||_p \geq    || k ||_{C^*}$.
\end{enumerate}
\end{proposition}

\noindent
The proof of the Proposition is easily given using standard properties of the Shatten
norms on a Hilbert space.

\medskip
Consider now $\chi_\Gamma$, the characteristic function of  a fundamental
domain for $\tilde{V}$. Define a map
\begin{equation}\label{phi-p}
\phi_m :C^* (X,\F) \rightarrow \End (\H)
\end{equation}
to be given by 
$\phi_m (k) := (\chi_\Gamma |T_\theta |^m \chi_\Gamma)_{\theta\in T}$ with $m\in\NN$.
It is a continuous map (although, obviously, not a linear operator), given as the composition of
 $(T_\theta)_{\theta\in T} \rightarrow (|T_\theta |^m )_{\theta\in T}$ and left and right multiplication by $\chi_\Gamma$.
 Let $\L^1 (\H)$ be the subalgebra of $\End (\H)$ (see Subsection
 \ref{subsect:vonneumann}) consisting of measurable families  $T=(T_\theta)_{\theta\in T}$
such that $T_\theta$ is an operator of trace class for almost every $\theta$.
It is a Banach subalgebra of $\End (\H)$ with the norm

\begin{equation}\label{trace-norm}
\| T \|_1:= {\rm ess.}\sup \{\|T_\theta\|_1\,;\,\theta\in T\} 
\end{equation}
where $\| T_\theta \|_1 $ denotes the trace norm. For $m\in \NN$, $m\geq 1$ we set
\begin{equation}\label{shatten-ideal}
\mathcal{I}_m (X,\F):= \{ T\in C^* (X,\F)\;|\; \phi_m (T)\in \L^1 (\H)\}
\end{equation}
and denote by $\psi_m$ the  restriction of $\phi_m$ to $\mathcal{I}_m (X,\F)$, so that
$\psi_m : \mathcal{I}_m (X,\F)\to \L^1 (\H)$. 
We anticipate that we shall need to take a slightly smaller algebra;
this smaller algebra will be denoted $\J_m (X,\F)$.

It is clear that $\mathcal{I}_m (X,\F)$ is closed
under composition.
We can prove that the graph of $\psi_m$ is a closed subset of $C^* (X,\F)\times  \L^1 (\H)$:
indeed the graph of $\phi_m$ is  a closed subset of $C^* (X,\F)\times \End (\H)$
due to continuity, the inclusion of $C^* (X,\F)\times  \L^1 (\H)$
into $C^* (X,\F)\times \End (\H)$ is continuous and the graph 
of $\psi_m$ is the intersection of the graph of $\phi_m$
with $C^* (X,\F)\times  \L^1 (\H)$. 

\begin{proposition}\label{prop:shatten-is-ideal}
$\mathcal{I}_m(X,\F)$ is a Banach algebra, 
an ideal inside $C^*(X,\mathcal F)$ and is isomorphic to the completion of $J_c (X,\F)$ with respect to the 
$m$-Shatten norm.
In particular $\mathcal{I}_m(X,\F)$  is a   holomorphically closed dense subalgebra of  $C^* (X,\F)$.
\end{proposition}

\begin{proof}
We define a norm on $\mathcal{I}_m(X,\F)$ by considering the graph norm
associated to $\psi_m$, viz:
$$\| T \|_m := \| T \|_{C^*} + \|\psi_m (T)\|_1\,.$$ Since the graph of $\psi_m$ is closed
this is a complete Banach space. 
Moreover, by the analogue of the first inequality in Proposition  \ref{prop:shatten-properties}
(stated for elements in $\End_\Gamma (\H)$) we see that this graph norm satisfies
$\| ST \|_m \leq \| S \|_m \| T \|_m$ so that $\mathcal{I}_m (X,\F) $ is a Banach algebra.
Next observe that, obviously, $J_c (X,\F)\subset \mathcal{I}_m (X,\F) $; moreover,
 from the fourth inequality in Proposition  \ref{prop:shatten-properties}
we see that on $J_c (X,\F)$ the graph-norm and the Shatten norm introduced in Definition 
\eqref{def:shatten} are equivalent (thus the small abuse of notation); since 
$\mathcal{I}_m (X,\F) $ contains  $J_c (X,\F)$ as a dense set and it is complete by
the norm $\|\;\;\|_m$, we conclude that the completion of $J_c (X,\F)$
by the norm of Definition 
\eqref{def:shatten} is naturally isomorphic, as a Banach algebra, to $\mathcal{I}_m (X,\F) $.
The fact that
   $\mathcal{I}_m$ is an ideal in $C^* (X,\F)$
 follows  easily
from the inequality $|| k k^\prime||_m \leq || k || \,|| k^\prime ||_m$. From the ideal property
one can easily prove that $\mathcal{I}_m$ is closed under holomorphic functional calculus;
indeed if $a\in \mathcal{I}_m$ and $f$ is a holomorphic function in a neighbourhood of
${\rm spec} a$ such that $f(0)=0$ then we can write $f(z)=z g(z)$ for some holomorphic function
$g$ and thus $f(a)= a g(a)$ which therefore belongs to $\mathcal{I}_m$, given that $\mathcal{I}_m$
is an ideal.
\end{proof}




\begin{remark}\label{extended-shatten-ineq}
For the elements in the ideals $\mathcal{I}_p (X,\F)$ 
the inequalities of Proposition   \ref{prop:shatten-properties} continue to hold.
In particular, if we have $T_j\in \mathcal{I}_p (X,\F)$ for $j=1,\dots,p$, then
their composition $T_1 \cdots T_p\in \mathcal{I}_1 (X,\F)$ and the
product map
$$\mathcal{I}_p (X,\F)\times \cdots \times \mathcal{I}_p (X,\F)\to \mathcal{I}_1 (X,\F)$$
is continuous.
\end{remark}

Recall now the   weight $\omega_{\Gamma}$ defined on 
 $J_c:=C^\infty_c (G, (s^* E)^*\otimes r^*E)$:
\begin{equation}\label{weight}
\omega_{\Gamma} (k):=
\int_{X(\Gamma)} \mathrm{Tr}_p k(x,x,\theta )dx\,d\theta,
\end{equation}
where $\mathrm{Tr}_p$ denotes the trace on $\mathrm{End}(E_p)$, $p=[(x,\theta)]\in X$
and we are identifying $\End (\widehat{E}_{(x,\theta)})$ with $\End (E_{[(x,\theta)]})$.
Recall also that\begin{equation}\label{weight-again.bis-bis}
\omega_{\Gamma}(k)=
 \int_{S^1} \Tr (\sigma_\theta k(\theta) \sigma_\theta)d\theta\,
\end{equation}
with $\sigma$ a compactly supported smooth function on $\tN\times S^1$ such that $\sum_{\gamma\in \Gamma} \gamma(\sigma)^2=1$,
$\sigma_\theta:= \sigma |_{\tV\times\{\theta\}}$
and $\Tr$ denoting the usual trace functional on the Hilbert space $\H_\theta$.

\begin{proposition}\label{prop:extension-of-weight}
The weight $\omega_\Gamma$ in \eqref{weight} extends continuously from $J_c$ to 
 $\mathcal{I}_1$. In particular, if $k_0, k_1,\dots,k_p\in \mathcal{I}_{p+1}$ then 
 $\omega_{\Gamma} (k_0 k_1\cdots k_p)$ is finite.
 \end{proposition}
 \begin{proof}
 We need to prove that for an element $k\in J_c (X,\F)$ we have $|\omega_\Gamma (k)|
 \leq C \| k\|_1$. However, 
 this follows at once from 
 the following two inequalities
 $$|\int_{{\rm FD}} \Tr_x k(x,x,\theta)dx |\leq \|\chi_\Gamma k_\theta \chi_\Gamma \|_1\,,
 \quad \int_T |f(\theta)|d\theta\leq {\rm vol}(T) \sup_{\theta} |f(\theta)|\,.$$
 Thus  $|\omega_\Gamma (k)|
 \leq  {\rm vol}(T) \| k\|_1$ as stated.
 
\end{proof}

We shall now introduce the subalgebra of $C^* (X,\F)$  that will be used in the proof of our index theorem.
Consider on the cylinder $\RR\times Y$ (with cylindrical variable $s$) the functions
\begin{equation}\label{functions-for-ideal}
f_{\cyl} (s,y):= \sqrt{1+s^2} \quad \quad g_{\cyl}(s,y)=1+ s^2\,.
\end{equation}
We denote by $f$ and $g$  smooth functions on $X$ equal to  $f_{\cyl}$ and $g_{\cyl}$ on 
the open subset $(-\infty,0)\times Y$;
$f$ and $g$ are well defined up to a compactly supported function. We set
\begin{equation}\label{new-j}
\J_m (X,\F):= \{k\in\I_m \;|\: gk \;\text{and}\;kg \;\text{are bounded}\}
\end{equation}
We shall often simply write $\J_m$.

\begin{proposition}\label{prop:newj}
$\J_m$ is a subalgebra of $\I_m$ and a Banach algebra with the norm 
\begin{equation}\label{norm-of-newj}
\|k\|_{\J_m} := \| k \|_m + \|gk\|_{C^*} + \|kg\|_{C^*}\,.
\end{equation}
Moreover $\J_m$ is holomorphically closed in $\I_m$ (and, therefore, in $C^* (X,\F)$).
\end{proposition}

\begin{proof}
The subalgebra property is obvious, so we pass directly to the fact that
$\J_m$ is a Banach
algebra.
It suffices to show that multiplication by $g$ on the left and on the right
induces closed operators; namely if $k_j\to k$, $k_j g\to \ell_1$, $g k_j\to
\ell_2$ for $k_j\in \Psi^{-1}_c (G)$, then $\ell_1 = kg$ and $\ell_2 = gk$.
In fact, given $\xi\in C^\infty_c (\tV\times\{\theta\})$, one has 
$$\ell_1 (\xi)=(\lim k_j g)(\xi)=(\lim k_j ) (g\xi)= k g (\xi)$$
noting that $g \xi \in C^\infty_c (\tV\times\{\theta\})$, which proves that $\ell_1 = kg$.
Similarly one has $\ell_2= gk$. This proves that $(\J_m, \|\;\;\|_{\J_m})$ is a Banach space.
The Banach-algebra property of this norm follows easily from the Banach-algebra
property of $\|\;\;\|_m$ on $\I_m$ and $\|\;\;\|_{C^*}$ on $C^* (X,\F)$.
Finally we show that $\J_m$ is holomorphically closed in $\I_m$. To this end we need
to show that if $1+k\in\J_m^+ := \J_m + \CC\cdot 1$ is invertible in $\I_m^+$, with
$(1+k)^{-1}=1+k^\prime $ and $k^\prime\in \I_m$
then one has $k^\prime\in \J_m$. First we observe that $1=(1+k)(1+k^\prime)=
1+k + k^\prime+ k k^\prime$. Thus $k^\prime=-k-k k^\prime$. Similarly one has $k^\prime=-k-k^\prime k$
(using $1=(1+k^\prime)(1+k)$). Thus $g k^\prime=-gk - (g k) k^\prime$ and $k^\prime g=
-k g - k^\prime (k g)$.
Since the right hand sides are bounded so are  $g k^\prime$ and $k^\prime g$. The Proposition is proved.
\end{proof}

\begin{remark}
As usual, we have not included the vector bundle $E$ into the notation; however, strictly
speaking, the notation for the Shatten ideals we have  defined above
should be $\I_m (X,\F;E)$. With obvious changes we can also define
 $\I_m (X,\F;E,F)$, with $F$ a hermitian vector bundle on $X$; in particular, given $E$ on $X=\tV\times_\Gamma T$, and thus $\widehat{E}$ on $\tV\times T$, we can define $\widehat{E}^\prime$, which
 is $\widehat{E}$ but with a new $\Gamma$-equivariant structure.
 We then have $\I_m (X,\F;E,E^\prime)$. Notice that, by continuity, we have an
 isomorphism of Banach algebras   $\I_m (X,\F;E)\simeq \I_m (X,\F;E^\prime)$
 as well as continuous
$ \I_p (X,\F;E,E^\prime)\times \I_q (X,\F;E)\to \I_r (X,\F;E,E^\prime)$ and 
$\I_p (X,\F;E)\times \I_q (X,\F;E,E^\prime)\to \I_r (X,\F;E,E^\prime) $
if $ 1/r=1/p + 1/q$\,.
Moreover, the analogue of Proposition \ref{prop:extension-of-weight}
 holds for the bimodule trace $\omega_\Gamma: J_c (X,\F;E,E^\prime)\to\CC$.
 \end{remark}

\subsection{Closed derivations}\label{subsect:derivations}
 In this Subsection we give some general results on derivations; this material plays
 an important role in the sequel.
  Let in general $T:\mathcal{B}_0\to \mathcal{B}_1$ be a linear operator between Banach spaces 
 with a domain $\mathrm{Dom}(T)$ which is assumed to be dense. Denote by $G_T$ the graph
 of $T$, namely the subspace $G_T:=\{(u,Tu)\in \mathcal{B}_0\oplus \mathcal{B}_1 \,|\, u\in 
 \mathrm{Dom}(T) \}$ and consider the closure $\overline{G_T}$. Also, denote by $p$ the projection 
 $p:\mathcal{B}_0\oplus \mathcal{B}_1\to \mathcal{B}_0$
onto the first component. The following Lemma is well known:
\begin{lemma}
The followings are equivalent:
\item{1)} $\overline{G_T}$ is the graph of a linear operator $\overline{T}$, with $p(\overline{G_T})$
equal to the domain of $\overline{T}$, which is an extension of $T$;
\item{2)} set $p_T:= p | _{\overline{G_T}}$; then $\mathrm{Ker} p_T=0$;
\item{3)} for $u_i\in\mathrm{Dom}(T)$ with $u_i\to 0$ and $T u_i \to v$ one has $v=0$.
\end{lemma}
%

\begin{definition}
A linear operator $T:\mathcal{B}_0\to \mathcal{B}_1$ with dense domain $\mathrm{Dom}(T)$ is a closable operator
if one of the properties of the Lemma above is satisfied. Then $\overline{T}$ is called the closure of $T$.
\end{definition}

It is obvious that $\mathrm{Dom}(\overline{T}) (=   \mathrm{Im} \,p_T)$ becomes a Banach space if we equip it with
the graph norm
\begin{equation}\label{graphnorm}
\|u \|_T := \|u\|_0 + \|Tu\|_1\,,
\end{equation}
with $\|\:\:\|_i$ denoting the Banach norms on $\mathcal{B}_i$. It is also obvious that the closure $\overline{T}$
induces a bounded operator $\overline{T}: (\mathrm{Dom}(\overline{T}),\|\;\;\|_T )\to (\mathcal{B}_1,\|\;\;\|_1)$.

Let now $A_0$ be a Banach algebra with norm $\|\:\:\|_0$ and $A_1$ a $A_0$-bimodule with norm
$\|\:\:\|_1$.
Let $\delta:A_0\to A_1$ be a closable derivation into the bimodule $A_1$, that is: $\delta$
is a closable operator that has the derivation property
\begin{equation}\label{deriv-prop}
\delta(ab)=(\delta a) b + a (\delta b)\,, \quad\text{for}\quad a,b\in \mathrm{Dom}(\delta)\,.
\end{equation}
Denote by $\overline{\delta}: \mathrm{Dom}(\overline{\delta})\to A_1$ the closure of $\delta$.

\begin{proposition}\label{prop:extension-deriv}
Set $\mathfrak{A}:=  \mathrm{Dom}(\overline{\delta})$.
\item{1)} $\mathfrak{A}$ is a Banach algebra with respect to the graph norm;
\item{2)} $\overline{\delta}$ induces a derivation $ \mathfrak{A} \to A_1$, $\overline{\delta}(ab)=
(\overline{\delta} a)b + a (\overline{\delta} b)$, $a,b\in \mathfrak{A}$.
\end{proposition}

\begin{proof}
Let $a,b\in  \mathrm{Dom}(\overline{\delta})$. Then there exist sequences $\{a_i\}, \{b_i\}$ in  $\mathrm{Dom}(\delta)$
such that $a_i\to a$, $\delta a_i \to \overline{\delta} a$, $b_i\to b$ and $\delta b_i \to \overline{\delta} b$
in $A_0$ and $A_1$ respectively. Since $A_0$ is a Banach algebra and $A_1$ is a bimodule over $A_0$, we have
$a_i b_i\to ab$ and $\delta(a_i b_i)=(\delta a_i) b_i + a_i (\delta b_i)\to ( \overline{\delta} a) b + a (\overline{\delta} b)$,
which implies $(ab, (\overline{\delta} a) b + a (\overline{\delta} b))\in \overline{G_{\delta}}$
and $\overline{\delta}(ab)=(\overline{\delta} a) b + a (\overline{\delta} b)$ since $\overline{G_{\delta}}$ is the graph
of $\overline{\delta}$ by the previous Lemma. This proves that $ab\in \mathfrak{A}$ and hence $\mathfrak{A}$
is an algebra. Moreover $\overline{\delta}$ satisfies the derivation property. Finally, we note that 
\begin{align*}
\|ab\|_{\overline{\delta}}&=\|ab\|_0+ \|\overline{\delta} (ab)\|_1\\
&\leq \|a\|_0\,\|b\|_0 + \|\overline{\delta}a\|_1\,\|b\|_0 +\|a\|_0 \,\|\overline{\delta}b\|_1\\
&\leq ( \|a\|_0 +\|\overline{\delta}a\|_1) (\|b\|_0  +\|\overline{\delta}b\|_1)\\
&= \|a\|_{\overline{\delta}} \,\|b\|_{\overline{\delta}}\,.
\end{align*}
which proves that $\mathfrak{A}$ is a Banach algebra with respect to the graph norm of $\overline{\delta}$.
\end{proof}

We shall also need the following simple but important Lemma. First we introduce the relevant
objects.
Let $B_0$ be a subalgebra of $A_0$ endowed with a Banach algebra
norm, $\| \;\;\|_{B_0}$, satisfying $\| b_0\|_{B_0}\geq 
\| b_0\|_{A_0}$. Let $B_1\subset A_1$ be a $B_0$-bimodule with 
$\| b_1\|_{B_1}\geq 
\| b_1\|_{A_1}$. Observe that $A_1$ is then also a $B_0$-bimodule since
$$\| b_0 a_1\|_{A_1}\leq \| b_0 \|_{A_0} \| a_1 \|_{A_1}
\leq \| b_0 \|_{B_0} \| a_1 \|_{A_1}$$
and similarly $\| a_1 b_0\|_{A_1}\leq \| a_1 \|_{A_1}  \| b_0 \|_{B_0}$
for $b_0\in B_0$, $a_1\in A_1$. Then $B_1$ is a $B_0$-submodule of $A_1$
endowed with the above $B_0$-bimodule structure and moreover the inclusion
is clearly bounded.

\begin{lemma}\label{lemma:restriction}
Let $\overline{\delta}$ be a closed derivation from ${\rm Dom}(\overline{\delta})\subset A_0$
to $A_1$. 
Set
$${\rm Dom}_B:= \overline{\delta}\,^{-1} (B_1)\cap B_0
\equiv \{a\in {\rm Dom}(\overline{\delta})\cap B_0\;\;|\;\; \overline{\delta} a
\in B_1 \}\,.$$
Define $\delta_B: {\rm Dom}_B\to B_1$ as  $\delta_B (b):= \overline{\delta} (b)$.
Then $\delta_B$ is a closed derivation.
\end{lemma}

\begin{proof}
By hypothesis we know that the graph
of $\overline{\delta}$ is a closed subspace of $A_0\oplus A_1$. Then, because
of our assumptions, its
intersection with $B_0\oplus B_1$ is a closed subset of  $B_0\oplus B_1$
(indeed, it is the inverse image of the graph for the inclusion map, which is
continuous).
On the other hand, this intersection is easily seen to be the graph of $\delta_B$.
The Lemma is proved.
\end{proof}

\subsection{Shatten extensions}
Let $(Y,\F)$, $Y:=\tN\times_\Gamma T$,  be a foliated $T$-bundle {\it without} boundary; for example $Y=\pa X\equiv \pa X_0$.
Consider $(\cyl (Y), \F_{\cyl})$ the associated foliated cylinder.
Recall the function
$\chi^0_{{\rm cyl}}$ (often just $\chi^0$), the   function on the cylinder induced by the characteristic function of $(-\infty,0]$ in $\RR$.
Notice that the definition of Shatten norm also apply to $(\cyl (Y), \F_{\cyl})$, viewed as a foliated $T$-bundle with cylindrical ends.
Let $\Psi^{-p}_{\RR,c}(G_{\cyl})\equiv \Psi^{-p}_{c}(G_{\cyl}/\RR_{\Delta})$, see Proposition \ref{prop:bstar-structure}, 
be the space of $\RR\times\Gamma$-equivariant families of pseudodifferential operators
of order $-p$ on the fibration $(\RR\times \tN)\times T\to T$ with $\RR\times\Gamma$-compact support. 
Consider an element  $\ell\in \Psi^{-p}_{c}(G_{\cyl}/\RR_{\Delta})$; then we know that 
$\ell$ defines a bounded operator from the Sobolev field $\E^{k}$ to the Sobolev field $\E^{k+p}$.
See \cite{MN}, Section 3.
Let us denote, as in \cite{MN}, the operator norm of a bounded operator $L$ from $\E^k$ to $\E^j$ as
$\| L \|_{j,k}$; notice the reverse order. For a $\RR\times \Gamma$-invariant, $\RR\times\Gamma$-compactly supported pseudodifferential operator of order $(-p)$, $P$, we consider the norm
\begin{equation}\label{triple-norm}
||| P |||_p \,:= \,{\rm max}(\| P \|_{-n,-n-p}\,,\, \|P\|_{n+p,n})
\end{equation}
with $n$ a fixed integer strictly greater than $\dim N$.
We denote the closure of 
$\Psi^{-p}_{c}(G_{\cyl}/\RR_{\Delta})$ with respect to the norm $|||\,\cdot\,|||_p$
by $\operatorname{OP}^{-p}  (\cyl (Y),\F_{\cyl})$. We shall often  write $\operatorname{OP}^{-p}$.

\begin{proposition}\label{prop:b-o-sub-of-b-star}
$\operatorname{OP}^{-p} (\cyl (Y),\F_{\cyl})$ is a Banach algebra and a subalgebra of $B^* (\cyl (Y),\F_{\cyl})$
\end{proposition}

\begin{proof}
It is proved in \cite{MN}, section 3, that the norm $|||\;\cdot\;|||_p$ satisfies the Banach algebra inequality
$||| PQ |||_p  \leq |||P|||_p \; |||Q|||_p $. Thus $\operatorname{OP}^{-p}$ is indeed a  Banach algebra. 
In order to prove that $\operatorname{OP}^{-p}$ is a subalgebra of $B^*$ we need the
following
\begin{lemma}\label{lemma:(-1)-bstar}
$B^*$ coincides with the $C^*$-closure of $\Psi^{-p}_{c}(G_{\cyl}/\RR_{\Delta})$ 
\end{lemma}

\begin{proof}
Let $D$ be the Dirac operator on $(\cyl Y,\F_{\cyl})$. Applying the same arguments
as in \cite{MN} we can prove that $(D+\mathfrak{s})^{-1}$ belongs to 
$B^*$ (see the proof of Proposition \ref{prop:relative-indeces} in Subsection
\ref{subsection:proof-relative} for the details).
Given $\ell\in \Psi^{-p}_{c}(G_{\cyl}/\RR_{\Delta})$,
$p\geq 1$,  
we can write $\ell=\ell (D+\mathfrak{s})^p (D+\mathfrak{s})^{-p}$ where we know that
$\ell (D+\mathfrak{s})^p\in \Psi^{0}_{c}(G_{\cyl}/\RR_{\Delta})$ and $(D+\mathfrak{s})^{-p}\in B^*$.
Now recall from Remark \ref{remark:b*=compact} that $B^*$ is an ideal in $\L (\E_{\cyl})$; thus the above
equality proves that $\ell\in B^*$. On the other hand, obviously,
 $\Psi^{-p}_{c}(G_{\cyl}/\RR_{\Delta})$ contains
$B_c\equiv C^\infty_c (G_{\cyl}/\RR_{\Delta})$. Thus $B^*\equiv C^* (G_{\cyl}/\RR_{\Delta})$,
which is by definition the $C^*$-closure of $B_c$, is contained in the $C^*$-closure of 
$\Psi^{-p}_{c}(G_{\cyl}/\RR_{\Delta})$. Thus one has:
$$B^*\equiv C^* (G_{\cyl}/\RR_{\Delta})\subset C^*\text{-closure of }\Psi^{-p}_{c}(G_{\cyl}/\RR_{\Delta})
\subset B^*$$
proving the Proposition.
\end{proof}

Since the $C^*$-norm is dominated by the $|||\;\cdot\;|||_p$-norm, we can immediately
conclude the proof of the Proposition.
\end{proof}

\noindent
{\bf Notation:} from now until the end of this subsection we fix $p=1$ and, following
\cite{MN},  we denote the corresponding norm simply as $|||\;\cdot\;|||$.

\bigskip
Consider now the bounded linear map
 $\partial_3^{{\rm max}} : B^*\to \End_{\Gamma} \H$ given
by $\partial_3^{{\rm max}} \ell:= [\chi^0,\ell]$. Consider in $B^*$ the Banach subalgebra
$\operatorname{OP}^{-1}$ endowed with the Banach norm  $|||\,\cdot\,|||$
and consider in 
$\End_{\Gamma} \H$ the subalgebra
$\mathcal{J}_m (\cyl (Y),\F_{\cyl})$. Let  $\partial_3$ be the restriction
of $\partial_3^{{\rm max}} $ to $\operatorname{OP}^{-1}$. Since $\|\cdot\|\leq |||\,\cdot\,|||$
we see that $\partial_3$ is also bounded.
 Let $\D_m:=\{\ell\in \operatorname{OP}^{-1}\;\;|\;\; \partial_3 (\ell)\in 
\mathcal{J}_m (\cyl (Y),\F_{\cyl})\}$. 
From the restriction Lemma
of the previous subsection, Lemma  \ref{lemma:restriction},
we know that $\partial_3\,|_{\D_m}$ induces a closed derivation
$\overline{\delta}_3$ with domain $\D_m$. This is clearly a closed extension
of the derivation $\delta_3$ considered 
in Subsection \ref{subsect:sigma3}. 

\begin{definition}\label{def:Bp}
If 
$m\geq 1$
 we define 
$\mathcal{D}_{m}(\cyl (Y), \mathcal F_{{\rm cyl}})$ 
as $\mathrm{Dom} \,\overline{\delta}_3$
 endowed with norm 
\begin{equation}\label{Bp-norm}
\|\ell \|_{\mathcal{D}_{m}} := ||| \ell  ||| + \| [\chi_{{\rm cyl}}^0,\ell]  \|_{\J_m} \,.
\end{equation}
We shall often simply write $\mathcal{D}_{m}$ instead of $\mathcal{D}_{m}(\cyl (Y), \mathcal F_{{\rm cyl}})$.
 \end{definition}

\begin{proposition}\label{prop:shatten-is-ideal} Let 
$m\geq 1$, 
then
$\mathcal{D}_{m}$ is a Banach algebra with respect to \eqref{Bp-norm}
and, obviously, a subalgebra of $B^*\equiv B^* (\cyl(Y), \mathcal F_{{\rm cyl}})$. Moreover,
$\mathcal{D}_{m}$ is holomorphically closed  in $B^*$. 
\end{proposition}

\begin{proof}
From the results of the previous  subsection, we know that
$\mathrm{Dom}(\overline{\delta}_3)$, endowed with the graph norm, is a Banach algebra;
since $\mathrm{Dom}(\overline{\delta}_3)$ is {\it by definition} $\mathcal{D}_m$, we have proved 
the first part of the Proposition. 
Finally,  that $\mathcal{D}_m\equiv \mathrm{Dom}(\overline{\delta}_3)$ is holomorphically closed in $\operatorname{OP}^{-1}$ 
is a classic consequence of the fact
that it is equal to the domain of a closed derivation. See \cite{roe-partitioned}, page 197 or 
\cite{Co}, Lemma 2, page 247.
Since $\operatorname{OP}^{-1}$ is in turn holomorphically closed in $B^*$, 
see \cite{MN} Theorem 3.3, we see that $\mathcal{D}_m$ is holomorphically closed
in $B^*$ as required.
The Proposition is proved.
\end{proof}

The Banach algebra we have defined is still too large for the purpose of extending
the eta cocyle. We shall first intersect it with another holomophically closed Banach subalgebra
of $B^*$.

Observe that there exists an action of $\RR$ on $\Psi^{-1}_{c}(G_{\cyl}/\RR_{\Delta})\subset
\operatorname{OP}^{-1} (\cyl (Y),\F_{\cyl})\subset B^*$ defined by
\begin{equation}\label{r-action-alpha}
\alpha_t (\ell):= e^{its} \ell e^{-its}\,,
\end{equation}
with $t\in\RR$, $s$ the variable along the cylinder and $\ell\in
\Psi^{-1}_{c}(G_{\cyl}/\RR_{\Delta})$. 
Note that $\alpha_t (\ell)$ is again $(\RR\times\Gamma)$-equivariant; indeed $e^{its}$ is $\Gamma$-equivariant and moreover 
$$T_\lambda \circ \alpha_t (\ell)\circ T_\lambda^{-1}=\alpha_t (\ell)\,,$$
$T_\lambda$ denoting the action induced by a translation on $\cyl (Y)$ by $\lambda\in\RR$.
It is clear that $||| \alpha_t (\ell) ||| =  ||| \ell |||$; thus, by continuity, 
$\{\alpha_t\}_{t\in\RR}$ yields  a well-defined  action, still denoted  $\{\alpha_t\}_{t\in\RR}$,
of $\RR$ on the Banach algebra $\operatorname{OP}^{-1} (\cyl (Y),\F_{\cyl})$. Note that this action is only strongly continuous.
Let $\partial_\alpha: \operatorname{OP}^{-1}\to \operatorname{OP}^{-1}$ be the  
unbounded derivation associated to $\{\alpha_t\}_{t\in\RR}$
\begin{equation}\label{pa-alpha}
\pa_\alpha (\ell):= \lim_{t\to 0} \frac{(\alpha_t (\ell)-\ell)}{t}
\end{equation}
By definition $${\rm Dom} (\partial_\alpha)=\{\ell\in  \operatorname{OP}^{-1} \;|\:  \partial_\alpha (\ell)\;\text{exists in }
\operatorname{OP}^{-1}\}.
$$

\begin{proposition}\label{prop:partial-alpha-closed}
The derivation $\partial_\alpha$ is  closed.
\end{proposition}

\begin{proof}
Observe preliminary that if $A$ and $A '$ are two closed operator on a Banach space $\B$
then their sum $A+A' $ is also closed (with domain equal to the intersection of the two domains). The proof is elementary. \\ Next we claim that 
if $A$ is a densely defined operator and $A^{-1}: \B\to {\rm Dom} (A)$ exists and is bounded,
then $A$ is closed. Indeed: suppose that $x_j\to x$ and $A x_j\to y$; we want to prove that
$x\in {\rm Dom} (A)$ and $Ax=y$. By hypothesis we know that $x_j\to A^{-1} y$. Thus $x=\lim_j x_j= A^{-1} y$.
Since $A^{-1}$ is bijective, one has $x\in {\rm Dom} (A)$ and $Ax=y$, as required.\\
Finally  for each $\ell\in\operatorname{OP}^{-1}$ we consider the following Laplace transform
$$R (\ell):=\int_0^{+\infty} dt e^{-t} \alpha_t (\ell)\,.$$
Since $||| \alpha_t (\ell) ||| = ||| \ell |||$, we see that the integral converges. Now, 
an elementary computation shows that 
$ (I-\partial_\alpha) R = I$. Thus the previous statement, applied to $(I-\partial_{\alpha})$,
implies that $(I-\partial_{\alpha})$ is a closed operator. Thus, by our first observation we get
that $\partial_{\alpha}$ is closed. The Proposition is proved.
\end{proof}

We endow
${\rm Dom} (\partial_\alpha)$ with the graph norm
\begin{equation}\label{norm-b}
|||\ell |||+ ||| \pa_\alpha (\ell) |||\,.
\end{equation}

\begin{proposition}\label{prop:shatten-is-ideal-b} 
${\rm Dom} (\partial_\alpha)$  is a Banach algebra with respect to \eqref{norm-b}
and, obviously, a subalgebra of $B^*\equiv B^* (\cyl(Y), \mathcal F_{{\rm cyl}})$; moreover 
it is holomorphically closed  in $B^*$. 
\end{proposition}

\begin{proof}
From the results of the previous  subsection, we know that
$\mathrm{Dom}(\pa_{\alpha})$, endowed with the graph norm, is a Banach algebra.
The first part of the Proposition is thus proved. 
That $\mathrm{Dom}(\pa_{\alpha})$ is holomorphically closed in $\operatorname{OP}^{-1}$ 
is as before a  consequence of the fact
that it is equal to the domain of a closed derivation. 
Since, as before, $\operatorname{OP}^{-1}$ is in turn holomorphically closed in $B^*$, 
see \cite{MN} Theorem 3.3, we see that $\mathrm{Dom}(\pa_{\alpha})$ is holomorphically closed
in $B^*$ as required.
The Proposition is proved.
\end{proof}

Let now $p\geq 1$ and consider $\operatorname{OP}^{-p} (\cyl (Y),\F_{\cyl})$. Then $\alpha_t$ on $\operatorname{OP}^{-1} (\cyl (Y),\F_{\cyl})$
preserves the subspaces $\operatorname{OP}^{-p} (\cyl (Y),\F_{\cyl})$ and we therefore get a well-defined stongly continuous
one-parameter group of automorphisms on each Banach algebra $\operatorname{OP}^{-p} (\cyl (Y),\F_{\cyl})$.
Let $\partial_{\alpha,p}$ be the associated derivation. Proceeding as in the proof of Proposition 
\ref{prop:partial-alpha-closed} we can check
that this is a closed derivation with domain
$${\rm Dom} (\partial_{\alpha,p}) = \{ \ell\in \operatorname{OP}^{-p} (\cyl (Y),\F_{\cyl})\;|\;  \lim_{t\to 0} \frac{(\alpha_t (\ell)-\ell)}{t}\;
\text{exists in } \operatorname{OP}^{-p} (\cyl (Y),\F_{\cyl})\}\,.$$
Similarly, proceeding as above, we can check that ${\rm Dom} (\partial_{\alpha,p})$ is a Banach
algebra with respect to the norm $|||\ell |||_p + |||\partial_{\alpha,p} (\ell) |||_p$.

\smallskip
Before going ahead we make a useful remark.

\begin{remark}\label{remark:product-domains}
Multiplication in $B^*$ induces a bounded bilinear map
\begin{equation}\label{product-domains}
{\rm Dom}( \partial_{\alpha,p})\times {\rm Dom} (\partial_{\alpha,q})\longrightarrow 
{\rm Dom}( \partial_{\alpha,p+q}).
\end{equation}
The proof is an easy consequence of the derivation property and 
of the inequality $|||\ell \ell ' |||_{p+q}\leq |||\ell |||_{p} ||| \ell ' |||_{q}$ for $ \ell\in \operatorname{OP}^{-p} (\cyl (Y),\F_{\cyl})$
and  $\ell ' \in \operatorname{OP}^{-q} (\cyl (Y),\F_{\cyl})$.
\end{remark}

We can now take the intersection of the  Banach subalgebras
$\mathcal{D}_{m}(\cyl (Y), \mathcal F_{{\rm cyl}})$ and $ \mathrm{Dom}(\pa_{\alpha})$:
$$\mathcal{D}_{m,\alpha}(\cyl (Y), \mathcal F_{{\rm cyl}}):= \mathcal{D}_{m}(\cyl (Y), \mathcal F_{{\rm cyl}})\cap
 \mathrm{Dom}(\pa_{\alpha})$$
and we  endow it with the norm 
\begin{equation}\label{Bp-norm-bis-0}
\|\ell\|_{m,\alpha}:= ||| \ell  ||| + \| [\chi_{{\rm cyl}}^0,\ell]  \|_{\J_m} +  |||  \partial_\alpha \ell |||\,.
\end{equation}
Being the intersection of two  holomorphically closed dense subalgebras, also $\mathcal{D}_{m,\alpha}(\cyl (Y), \mathcal F_{{\rm cyl}})$ enjoys this property.\\
We are finally ready to define the subalgebra we are interested in. Recall the function 
$f_{\cyl} (s,y)=\sqrt{1+s^2}$.

\begin{definition}\label{def:Bp-bis}
If 
$m\geq 1$
 we define 
\begin{equation}\label{new-Bm}
\mathcal{B}_{m}(\cyl (Y), \mathcal F_{{\rm cyl}}):= 
\{\ell\in \mathcal{D}_{m,\alpha}(\cyl (Y), \mathcal F_{{\rm cyl}}) \;|\; [f,\ell]\;\;\text{and}\;\;[f,[f,\ell]]\;\;\text{are bounded}\}\,.
\end{equation}
This will be  endowed with norm 
\begin{align*}
\|\ell \|_{\mathcal{B}_{m}} := &
\|\ell\|_{m,\alpha} + 2 \|[f,\ell]\|_{B^*} + \| [f,[f,\ell]]\|_{B^*}\\
=&||| \ell  ||| + \| [\chi_{{\rm cyl}}^0,\ell]  \|_{\J_m} +  |||  \partial_\alpha \ell |||
+2 \|[f,\ell]\|_{B^*} + \| [f,[f,\ell]]\|_{B^*}\,.
\end{align*}
The appearance of the factor 2 will be clear from the proof of Lemma \ref{new-ideal}.
Proceeding as in the proof of Proposition \ref{prop:newj} one can 
prove that $\B_m (\cyl (Y), \mathcal F_{{\rm cyl}})$  is {\it a holomorphically closed dense subalgebra of $B^*$.}
We shall often simply write $\mathcal{B}_{m}$ instead of $\mathcal{B}_{m}(\cyl (Y), \mathcal F_{{\rm cyl}})$.
 \end{definition}

\medskip
Let us go back to the foliated bundle with cylindrical end $(X,\F)$.
We now define \begin{equation}
\mathcal{A}_{m}(X,\mathcal F):= \{k\in A^*(X,\mathcal F); 
 \pi(k)\in \mathcal{B}_{m}(\cyl(\pa X), \mathcal F_{{\rm cyl}}), t(k)\in \mathcal{J}_{m}(X,\mathcal F)\}
\end{equation}



Now we observe that, as   vector spaces,
\begin{equation}\label{direct-sum-cal}
\mathcal{A}_m \cong \mathcal{J}_{m} \,\oplus \, s (\mathcal{B}_{m})\,.
\end{equation}
In order to prove  \eqref{direct-sum-cal} we recall the $C^*$-sequence
$0\to C^* (X,\F)\to A^* (X,\F)\xrightarrow{\pi} B^*  (\cyl(\pa X),\F_{{\rm cyl}})\to 0$ and the sections
$s: B^* (\cyl(\pa X),\F_{{\rm cyl}})\to A^* (X,\F)$  and $t: A^* (X,\F) \to C^* (X,\F)$ defined in \eqref{eq:split-2} and \eqref{eq:t} respectively.
Note that $\Ker t = \Im s$ since $t(k)=k-s\circ \pi (k)$ and $\pi\circ s (\ell)=\ell$ for 
$k\in A^* (X,\F)$ and $\ell\in B^* (\cyl(\pa X),\F_{{\rm cyl}})$. Moreover, we obviously  have $\pi (a)=0$ and $t(a)=a$ for
$a\in C^* (X,\F)$.\\

\medskip
\noindent
{\it Proof of \eqref{direct-sum-cal}.} Define $\phi: \mathcal{A}_m\to \mathcal{J}_{m} \,\oplus \, s (\mathcal{B}_{m})$
by $\phi (k)= (t(k),s\circ \pi (k))$. 
Define $\psi: \mathcal{J}_{m} \,\oplus \, s (\mathcal{B}_{m})\to \mathcal{A}_m$
by $\psi (a, s(\ell))=a + s(\ell)$. Note that $\Im \psi\subset  \mathcal{A}_m$ since $t(a+s (\ell))=a\in \mathcal{J}_{m}$
and 
$\pi (a + s(\ell))=\ell\in \mathcal{B}_{m}$. 
The maps $\phi$ and $\psi$ are obviously linear. Then we have 
$$\psi\circ\phi (a,s(\ell))= (t (a+s(\ell)), s\circ\pi (a+s(\ell)))= (a,s(\ell))\,,\quad
\phi\circ \psi (k)=(k-s\circ\pi (k))+ s\circ\pi (k)=k$$
and we are done.

\medskip

We endow $\mathcal{A}_m$ with the direct-sum norm:
\begin{equation}\label{direct-sum-norm}
\| k \|_{\mathcal{A}_m} :=  \| t(k) \|_{\J_m} + \| \pi (k) \|_{\mathcal{B}_{m}}
\end{equation}
 Obviously $s$ induces a bounded linear map
$\mathcal{B}_{m} \to \mathcal{A}_m$ 
of Banach spaces and similarly for $\pi$.
Moreover, note that the restriction of the norm $\| \;\; \|_{\mathcal{A}_m} $
to the subalgebra $ \mathcal{J}_{m}$ is precisely the norm $\|\;\;\|_{\J_m}$.

We shall prove momentarily that these algebras fits into a short exact sequence;
before doing this we prove a useful Lemma. Remark that for a foliation  $(Y,\F_Y)$ without boundary,
 $(\cyl (Y),\F_{\cyl})$ is  a  foliation with cylindrical ends;  for the latter $\J_m (\cyl(Y),\F_{\cyl})$
 makes perfect sense.
 
\begin{lemma}\label{new-ideal}
Recall the function $\chi^0$ on $X$ and $\chi^0_{\cyl}$ on the cylinder $\cyl (\pa X)$. One has:
\item{1)} $\chi^0 \J_m \subset \J_m $ and $ \J_m \chi^0\subset \J_m$;
\item{2)}  $\chi^0 \J_m (\cyl (\pa X),\F_{\cyl}) \chi^0\subset \J_m (X,\F)$
\item{3)} on $\cyl (Y)$, for example on $\cyl (\pa X)$, we have $\J_m \B_m \subset \J_m$ and $\B_m \J_m \subset \J_m$;
\item{4)} $(\chi^0 \B_m \chi^0) \J_m (X,\F)\subset \J_m (X,\F)$ and 
$\J_m (X,\F)(\chi^0 \B_m \chi^0) \subset \J_m (X,\F)$
\item{5)} $(\chi^0 \B_m \chi^0)(\chi^0 \B_m \chi^0)\subset \chi^0 \B_m \chi^0 + \J_m$.
\end{lemma}

\begin{proof}
1) The operators $g\chi^0 k=\chi^0 g k$ and $\chi^0 k g$ are bounded if $k\in \J_m$. Thus one has $\chi^0 \J_m\subset \J_m$. Similarly we proceed for the other inclusion.\\
2) The proof is similar to 1)\\
3) Take $k\in \J_m$ and $\ell\in\B_m$. Obviously one has $k\ell$ and $\ell k $ $\in \I_m$, given that $\J_m\subset \I_m$
and that $\I_m$ is an ideal. Moreover, $g_{\cyl} k\ell$ is bounded and so is
\begin{align*}
k\ell g_{\cyl} &= k\ell f_{\cyl}^2 =k [\ell,f_{\cyl}] f_{\cyl} + k f_{\cyl}\ell f_{\cyl}\\
&= k [[\ell,f_{\cyl}],f_{\cyl}]+ 2 k f_{\cyl} [\ell, f_{\cyl}] + kg_{\cyl} \ell
\end{align*}
given that $[[\ell,f_{\cyl}],f_{\cyl}]$, $kf_{\cyl}$, $[\ell,f_{\cyl}]$ and $kg_{\cyl}$ are all bounded.
Thus $k\ell\in \J_m$. Similarly one proves that $\ell k\in \J_m$.\\
4) The proof is analogous to the one of 3),  let us see the details for the second inclusion:
\begin{align*}
k\chi^0 \ell \chi^0 g &= k\chi^0 \ell g_{\cyl} \chi^0 = k\chi^0 \ell f^2 _{\cyl}\chi^0=
k \chi^0 [\ell,f_{\cyl}] f_{\cyl}\chi^0  + k \chi^0 f_{\cyl}\ell f_{\cyl}\chi^0\\
&= k\chi^0 [[\ell,f_{\cyl}],f_{\cyl}]\chi^0+ 2 k f \chi^0 [\ell, f_{\cyl}]\chi^0  + k g \chi^0  \ell \chi^0
\end{align*}
which is easily seen to be bounded using the definitions of $\J_m$ and $\B_m$. The rest of the
proof is similar but easier.
\\
5) Note that, on the cylinder,  $[\chi^0_{\cyl},\ell]\in \J_m$ if $\ell\in \B_m$. Thus for $\ell,\ell ' \in \B_m$
we have that $\chi^0 \ell (1-\chi^0_{\cyl}) \ell ' \chi^0 = \chi^0 [\chi^0_{\cyl}, \ell] (1-\chi^0_{\cyl}) [\ell' ,\chi^0_{\cyl}]\chi^0$ 
belongs to $\J_m$, due to 1). This implies that 
$$\chi^0 \ell \chi^0 \ell ' \chi^0 = \chi^0 \ell \ell ' \chi^0 -
\chi^0 \ell (1-\chi^0) \ell ' \chi^0\; \in\; \chi^0 \B_m \chi^0 + \J_m.$$
\end{proof}

\begin{proposition}\label{prop:short-ec-seq-call}
($\mathcal{A}_m, \| \;\; \|_{\mathcal{A}_m})$ is a Banach subalgebra
of $A^*$. Moreover, 
 $ \mathcal{J}_{m}$ is an ideal in 
 $\mathcal{A}_m$ and there is a short
 exact sequence of Banach algebras: 
\begin{equation}\label{Shatten exact sequence}
0\rightarrow 
 \mathcal{J}_{m} (X,\mathcal F)\rightarrow \mathcal{A}_m (X;\mathcal F)
 \xrightarrow{\pi}\mathcal{B}_{m} (\cyl(\pa X), \mathcal F_{{\rm cyl}}) \rightarrow 0 \,.
\end{equation}
Finally, 
$t: A^* (X,\F)\to C^*(X,\F)$ restricts to a bounded 
section
$t:  \mathcal{A}_m (X,\F)
 \to \mathcal{J}_{m} (X,\F)
 $
 
\end{proposition}
\begin{proof}
Write $k= a + \chix \ell_k \chix$, with $\pi(k)=\ell_k$. By definition
$t(k)=k-\chix \ell_k\chix=a \in \mathcal{J}_{m}(X,\mathcal F)$. 
Similarly we write $k'= a' + \chix \ell_{k'}\chix$. 
We thus have
$$k k'= (a + \chix \ell_k \chix) (a' + \chix \ell_{k'} \chix).$$
Since $\rho$ is an injective homomorphism we check easily that $\ell_{ k k'}= \ell_k \ell_{k'}$  
 We compute, with $\ell\equiv \ell_k$ and $\ell ' = \ell_{k^\prime}$,
 \begin{align*}
 k k' &= (a+\chix \ell \chix)(a' + \chix \ell' \chix)\\
 &=a a' + a\chix \ell' \chix + \chix \ell \chix a' + \chix\ell\chix \chix \ell' \chix\\
 &= a a' + a\chix \ell' \chix + \chix \ell \chix a' + \chix\ell(\chix_{\cyl}-1) \ell' \chix + \chix \ell \ell ' \chix\\
&= 
 a a' + a\chix \ell' \chix + \chix \ell \chix a' + \chix [\chix_{\cyl},\ell] [\ell ', \chix_{\cyl},] \chix + \chix \ell  \ell ' \chix \,.
 \end{align*}
 The first three terms belong to 
$\mathcal{J}_{m}(X,\mathcal F)$ because $\J_m (X,\mathcal F)$ is an algebra and because of property 4) in the Lemma ; we also know that, by the very definition of $\B_m$,
$[\chix_{{\rm cyl}},\ell]$ and $[\chix_{{\rm cyl}},\ell']$ are in  $\mathcal{J}_{m}(\cyl (\pa X),\mathcal F_{{\rm cyl}})$ so that their product is in $\mathcal{J}_{m}(\cyl (\pa X),\mathcal F_{{\rm cyl}})$.
Using this,  the second item of the Lemma  and the identity $\ell_{ k k'}= \ell_k \ell_{k'}$,
we finally  see that $\A_m$ is a subalgebra.\\
Next we prove that $\A_m$ is a Banach algebra. Recall that if $a\in\A_m$ then $\| a \|_{J_m}
=\|a \|_m + \|ag\|_{C^*} + \| ga \|_{C^*}$; this clearly satisfies $\| a a' \|_{\J_m} \leq \| a  \|_{\J_m} \,\| a' \|_{\J_m} $.
We shall prove that 
$$ \| a \chi^0 \ell \chi^0\|_{\J_m} \leq \| a \|_{\J_m}  \| \ell \|_{\B_m} \quad\text{and}\quad 
 \|  \chi^0 \ell \chi^0 a\|_{\J_m} \leq \| a \|_{\J_m}  \| \ell \|_{\B_m}\,.$$
 Indeed one has 
 \begin{align*}
 \| a \chi^0 \ell \chi^0\|_{\J_m} &=  \| a \chi^0 \ell \chi^0\|_{m} +  \|  a \chi^0 \ell \chi^0 g\|_{C^*}+
  \|g a \chi^0 \ell \chi^0 \|_{C^*}\\
  &=  \| a \chi^0 \ell \chi^0\|_{m} + \|2 a f \chi^0 [\ell,f] \chi^0 + a \chi^0 [[\ell,f],f]\chi^0 
  +a g \chi^0 \ell\chi^0 \|_{C^*} +   \|g a \chi^0 \ell \chi^0 \|_{C^*}\\
  &\leq \| a \|_m \| \ell\|_{B^*} + 2 \| af\|_m \|[\ell, f]\|_m + \| a \|_m \| [[\ell,f],f]\|_m +
  \| ag \|_{C^*} \| \ell\|_{B^*} + \| ga \|_{C^*} \|\ell\|_{B^*}\\
  &\leq \| a \|_{\J_m} \| \ell\|_{\B_m}\,.
\end{align*}
Similarly one proves the second inequality. Then we have

\begin{align*}
\| k\,k^\prime \|_{\mathcal{A}_m} &=  \| a a^\prime + a \chix \ell^\prime\chix
+ \chix \ell \chix a^\prime +  \chix [\chix_{\cyl},\ell] [\chix_{\cyl}, \ell '] \chix \|_{\J_m} + 
\| \ell \ell^\prime\|_{\mathcal{B}_{m}}\\
&\leq   \| a \|_{\J_m}  \|a^\prime \|_{\J_m} + \|a \|_{\J_m} \, \| \ell^\prime\|_{\B_m}+ \| \ell\|_{\B_m}\,\|a^\prime \|_{\J_m} + 
\| [\chix_{\cyl}, \ell]\|_{\J_m} \| [\chix_{\cyl} , \ell ']\|_{\J_m} + \|\ell\|_{\B_m}\|\ell ' \|_{\B^m}\\
&\leq \| k\|_{\mathcal{A}_m}  \, \|k^\prime \|_{\mathcal{A}_m} \,.
\end{align*}
Thus $\A_m$ is a Banach algebra. Since it is clear that the inclusion of $\A_m$ into
$A^*$ is bounded, we see that $\A_m$ is a Banach subalgebra of $A^*$.
The fact that we obtain a short exact sequence of Banach algebras is now clear.
Finally, observe that $t(k)=k-s (\pi (k))$; thus the boundedness of $s$ implies that
of $t$.
\end{proof}

\subsection{Smooth subalgebras defined by the modular automorphisms.}\label{subsect:dense}

The short exact sequence of Banach algebras $0\to \J_{m}\to \A_m\to \B_m\to 0$ does not involve
in any way the modular function $\psi$ and the two derivations $\delta_1$ and $\delta_2$.
Thus we cannot expect the two  Godbillon-Vey cyclic 2-cocycles to extend
to the cyclic cohomology groups of these algebras. For this reason we need to further decrease
the size of these subalgebras, taking into account the derivations $\delta_1$ and $\delta_2$.

\subsubsection{Closable derivations defined by commutators}

Let $k$ be an element either in $J_c (X,\F)$,
$A_c (X,\F)$ or $B_c (\cyl(\pa X),\F_{\cyl})$. We consider $k$ as a $\Gamma$-equivariant family of operators $k=(k(\theta))_{\theta\in T}$
acting on a family of Hilbert spaces $\mathcal{H}_\theta$ as in Sections \ref{sect:operators}
and \ref{sec:algebras}. 

\medskip
We first work on $A_c(X,\F)$
which we endow with a Banach norm $\|\:\:\|_0$ and denote it as $A^0_c$. Next, we consider 
the bimodule $A^1_c$, as in the preceeding subsections, i.e. the bimodule built out of $A_c$ by considering
operators acting from sections of $E$ to sections of 
the bundle with new equivariant structure, $E^\prime$.
We endow the bimodule $A^1_c$ with a norm $\|\:\:\|_1$. We shall assume that both $\|\:\:\|_0$ and  $\|\:\:\|_1$ are stronger
than the $C^*$-norm:
\begin{equation}
\label{stronger-norms}
 \|k\|_i \geq \|k\|_{C^*}\,,\quad i=0,1\,.
 \end{equation}
 
Let $f$ be a smooth function on $\tV\times T$ and consider the bimodule derivation 
$\delta:(A^0_c,\|\:\:\|_0)\to (A^1_c,\|\:\:\|_1)$ given by $\delta k:= [f,k]$. 
We assume that $f$ has been chosen so that $[f,k]$ is a $\Gamma$-equivariant family of operators. Note that, then,
$$(\delta k)(\theta) \xi_\theta = f(x,\theta) k(\theta) \xi_\theta - k(\theta) (f(x,\theta)\xi_\theta)$$
for $\xi_\theta\in\mathcal{H}_\theta$. We don't assume that $f$ is $\Gamma$-invariant, nor we assume that
$f$ is compactly supported or even bounded (this being a basic difference with the case of
$\chi^0$ already considered).

\begin{proposition}\label{prop:closable}
Under the above assumptions we have that  $\delta$ is a closable derivation. 
\end{proposition}
\begin{proof}
Because of the   Lemma above it suffices to show that $\delta$ satisfies the
following property:
$$\text{if } \;\;\|k_i\|_0\to 0\;\;\text{ and }\;\; \|\delta k_i - k \|_1 \to 0\,,\;\;\text{with}\;\;k_i\in A_c\,,\;\;\text{ then }\;\; k=0\,.$$
Take $\xi,\eta\in C^\infty_c (\tV\times T;E)$; these induce elements $\xi_\theta, \eta_\theta\in \mathcal{H}_\theta$
once we restrict them to $\tV\times\{\theta\}$. Since, from \eqref{stronger-norms} the operator norm $\| [f,k_i](\theta)-k(\theta)\|$
is less than or equal to $\|[f,k_i]-k\|_1$, which in turn goes to zero, one has
$$\langle [f,k_i](\theta) \xi_\theta,\eta_\theta \rangle\longrightarrow \langle k(\theta)\xi_\theta,\eta_\theta \rangle$$
where $\langle\:\:\rangle$ denotes the innser product on $\mathcal{H}_\theta$. On the other hand
\begin{align*}
|\langle [f,k_i](\theta) \xi_\theta,\eta_\theta \rangle | &\leq | \langle f(\cdot,\theta) k_i (\theta) \xi_\theta,\eta_\theta \rangle |
+ | \langle k_i (\theta) f(\cdot,\theta) \xi_\theta,\eta_\theta \rangle | \\
 &= | \langle  k_i (\theta) \xi_\theta,\overline{f}(\cdot,\theta)\eta_\theta \rangle |
+ | \langle  f(\cdot,\theta) \xi_\theta, k_i (\theta)^* \eta_\theta \rangle |\\
&\leq \|k_i (\theta)\| \| \xi_\theta \| \| \overline{f}(\cdot,\theta) \eta_\theta\| + \|f(\cdot,\theta)\xi_\theta\|  \|k_i (\theta)\|  \| \eta_\theta\|\\
&\leq C  \|k_i (\theta)\| 
\\
&\leq C  \|k_i (\theta)\|_0
\end{align*}
where $C$ is a constant depending on $\xi,\eta$ and $f$ but independent of $k_i$. Note that $\overline{f}(\cdot,\theta)\eta_\theta$
and $f(\cdot,\theta) \xi_\theta$ are of compact support in $\tV\times\{\theta\}$ and thus their norms are finite. Thus we obtain
$$|\langle [f,k_i](\theta) \xi_\theta,\eta_\theta \rangle | \longrightarrow 0\quad\text{as}\quad i\to \infty, \quad\text{since}\quad \|k_i\|_0\to 0\,.
$$
This implies that $\langle k(\theta)\xi_\theta,\eta_\theta \rangle=0$ for any $\xi,\eta\in C^\infty_c (\tV\times T;E)$ and hence the family 
$(k(\theta))_{\theta\in T}$ is the zero operator. Thus we have proved that $\delta$ is closable.
\end{proof}

\subsubsection{The smooth subalgebra $ \mathbf{ \mathbf{ \mathfrak{J}_m } }\subset C^* (X,\F)$}
\label{subsubsection:gothic-j}

We apply the above  general results to the two derivations $\delta_1$ and $\delta_2$ introduced in Subsection \ref{subsect:absolute-taugv},
 namely 
 $\delta_1:= [\dot{\phi},\:\:]$ and $\delta_2:=[\phi,\:\:]$, with $\phi$ equal to the
   logarithm of the modular function.
   
Recall the $C^*$-algebra $C^*_\Gamma (\H)\supset C^*(X,\F)$; it is obtained, by definition, by
closing up the subalgebra $C_{\Gamma,c} (\H)\subset \End_{\Gamma} (\H)$ consisting of those
elements that preserve the continuous field $C^\infty_c (\tV\times T, E)$.
We set 
$${\rm Dom}\,(\delta_2^{{\rm max}})=\{k\in  C_{\Gamma,c} (\H) \,|\, [\phi,k]\in C^*_\Gamma (\H)\}$$
and 
$$\delta_2^{{\rm max}}: {\rm Dom}\,(\delta_2^{{\rm max}})\to C^*_\Gamma (\H), \;\;\;
\delta_2^{{\rm max}} (k):= [\phi,k]\,.$$
The same proof as above establishes that $\delta_2^{{\rm max}}$
is closable. Similarly,  with self-explanatory notation, the bimodule derivation
$$\delta_1^{{\rm max}}:
{\rm Dom}\,(\delta_1^{{\rm max}})\to C^*_\Gamma (\H,\H^\prime), \;\;\;
\delta_1^{{\rm max}} (k):= [\dot{\phi},k]\,,$$
with ${\rm Dom}\,(\delta_1^{{\rm max}}):= \{k\in  C_{\Gamma,c} (\H) \,|\, [\dot{\phi},k]\in 
C^*_\Gamma (\H,\H^\prime)\}$
 is closable. 
Let $\overline{\delta}^{{\rm max}}_j$ be their respective closures;
thus, for example,
 $$\overline{\delta}^{{\rm max}}_2: {\rm Dom}\,\overline{\delta}^{{\rm max}}_2\subset C^*_{\Gamma} (\H)
 \longrightarrow C^*_\Gamma (\H)$$
and similarly for $\delta_1^{{\rm max}}$. Define now 
$$\mathfrak{D}_2:=\{a\in {\rm Dom}\,\overline{\delta}^{{\rm max}}_2\cap \J_m (X,\F)\;|\;\overline{\delta}^{{\rm max}}_2
a\in \J_m (X,\F)\}$$
and $\overline{\delta}_2: \mathfrak{D}_2\to  \J_m (X,\F)$ 
as the restriction of $\overline{\delta}^{{\rm max}}_2$ to $\mathfrak{D}_2$
with values in  $\J_m (X,\F)$. We know from Lemma \ref{lemma:restriction}
that $\overline{\delta}_2$ is a closed derivation.
Define similarly $\mathfrak{D}_1$ and the closed derivation $\overline{\delta}_1$.

%
We set
 \begin{equation}\label{gothic-j}
\mathbf{ \mathbf{ \mathfrak{J}_m } } := \mathcal{J}_m \cap
{\rm Dom} (\overline{\delta}_1)
\cap {\rm Dom} (\overline{\delta}_2) \equiv  \mathcal{J}_m \cap \mathfrak{D}_1\cap \mathfrak{D}_2\,.
 \end{equation}
 We endow $\mathbf{ \mathbf{ \mathfrak{J}_m } }$ with the norm
 \begin{equation}\label{norm-gothic-j}
 \| a \|_{\mathbf{ \mathbf{ \mathfrak{J}_m } }}:= \| a \|_{m} + \| \overline{\delta}_2 a \|_{m}+
  \| \overline{\delta}_1 a \|_{m}\,.
 \end{equation}
 \begin{proposition}\label{prop:gothic-j-hol-closed}
 $ \mathbf{ \mathbf{ \mathfrak{J}_m } }$ is holomorphically closed in $C^* (X,\F)$.
 \end{proposition}
 \begin{proof}
 We already know that the Banach algebra $\mathcal{J}_m$ is holomorphically closed in the $C^*$-algebra $C^* (X,\F)$.
On the other hand, we know  \cite{roe-partitioned}, page 197 or 
\cite{Co}, Lemma 2, page 247, that 
  ${\rm Dom} (\overline{\delta}_1)$ and $ {\rm Dom} (\overline{\delta}_2)$ are holomorphically closed 
  in $\mathcal{J}_m$ (since they are the domains
  of closed derivations). Thus  $ \mathbf{ \mathbf{ \mathfrak{J}_m } }$ is holomorphically closed in $C^* (X,\F)$
  as required.
  \end{proof}

\subsubsection{The smooth subalgebra $ \mathbf{ \mathbf{ \mathfrak{B}_m } }\subset B^* (\cyl(\pa X),\F_{\cyl})$}\label{subsubsection:gothic-b}
 Consider  $\mathcal{B}_m$; we consider  the  derivations 
 $\delta_1:=[\dot{\phi}_{\pa},\;\;]$,  $\delta_2:= [\phi_{\pa},\;\;]$
 on the cylinder $\RR\times\pa X_0$; we have already encountered these derivations in Subsection \ref{subsect:sigma3},
 see more precisely Definition \ref{def:sigma3}. 
 Consider first $\delta_2$.  Define a  closed derivation $\overline{\partial}_2$
by taking the closure of the closable derivation $ \Psi^{-1}_c (G_{\cyl}/\RR_{\Delta})\xrightarrow{\partial_2} B^*$,
with $\partial_2 (\ell):= [\phi_\partial,\ell]$ and 
with $ \Psi^{-1}_c (G_{\cyl}/\RR_{\Delta}) $ endowed with the norm $|||\cdot|||$. 
Then, from Lemma 
 \ref{lemma:restriction}, we know that 
 $\overline{\partial}_2|_{\mathfrak{D}_{2}}$, with 
 $$
 \mathfrak{D}_{2}=
\{b\in {\rm Dom}(\overline{\partial}_2)\;\;|\;\; \overline{\partial}_2 (b)\in \B_{m}\}
$$
is a closed derivation with values in $\mathcal{B}_m$. We set 
$\overline{\delta}_2:=\overline{\partial}_2|_{\mathfrak{D}_2}$;
thus ${\rm Dom}(\overline{\delta}_2)=\mathfrak{D}_2$ and 
$\overline{\delta}_2:=\overline{\partial}_2|_{\mathfrak{D}_2}$. A similarly definition of 
$\overline{\delta}_1$ and 
 ${\rm Dom} (\overline{\delta}_1)$ can be given.

We set  \begin{equation}\label{gothic-b}
\mathbf{ \mathbf{ \mathfrak{B}_m } } := \mathcal{B}_m \cap
{\rm Dom} (\overline{\delta}_1)
\cap {\rm Dom} (\overline{\delta}_2)\equiv \mathcal{B}_m \cap \mathfrak{D}_1\cap \mathfrak{D}_2\,.
 \end{equation}
 We endow  $\mathbf{ \mathbf{ \mathfrak{B}_m } } $ with the norm
 \begin{equation}\label{norm-gothic-b}
 \| \ell \|_{\mathbf{ \mathbf{ \mathfrak{B}_m } } }:= \| \ell \|_{\B_m} +  \| \overline{\delta}_1 \ell \|_{\B_m} + 
 \|\overline{\delta}_2 \ell \|_{\B_m}
 \end{equation}
 
 \begin{proposition}\label{prop:gothic-b-hol-closed}
 $ \mathbf{ \mathbf{ \mathfrak{B}_{m} } }$ is holomorphically closed in $B^* (\cyl(\pa X),\F_{\cyl})$.
 \end{proposition}
 \begin{proof}
 We already know that the Banach algebra 
 $\mathcal{B}_{m}$ is holomorphically closed in the $C^*$-algebra\\ $B^* (\cyl(\pa X),\F_{\cyl})$.
On the other hand, we know 
that 
  ${\rm Dom} (\overline{\delta}_1)$ and $ {\rm Dom} (\overline{\delta}_2)$ are holomorphically closed 
  in  $\mathcal{B}_{m}$. Thus  
   $ \mathbf{ \mathbf{ \mathfrak{B}_{m} } }$ is holomorphically closed in $B^* (\cyl(\pa X),\F_{\cyl})$
   as required.
       \end{proof}

    \subsubsection{The subalgebra $ \mathbf{ \mathbf{ \mathfrak{A}_m } }\subset A^* (X,\F)$}
    \label{subsubsection:gothic-a}
 
Next we consider the Banach algebra $\A_m (X,\F)$ which is certainly contained
in $C^*_{\Gamma} (\H)$, given that $A_c (X,\F)$ is contained in $C_{\Gamma,c} (\H)$. 
Consider again $\overline{\delta}^{{\rm max}}_j$ and restrict it to a derivation
with values in $\A_m (X,F)$:
$$\overline{\delta}_2: \D_2 \to \A_m (X,F)$$
with $\D_2=\{a\in {\rm Dom}\,\overline{\delta}^{{\rm max}}_2\;|\;\overline{\delta}^{{\rm max}}_2 a\in 
\A_m (X,F)\}$ and similarly for $\overline{\delta}_1$.  We obtain in this way closed derivations
$\overline{\delta}_1$ and $\overline{\delta}_2$ with domains $ {\rm Dom}\overline{\delta}_1=\mathfrak{D}_1$ and 
$ {\rm Dom}\overline{\delta}_2=\mathfrak{D}_2$.
We set
 \begin{equation}\label{gothic-a}
\mathbf{ \mathbf{ \mathfrak{A}_m } } := \mathcal{A}_m \cap
{\rm Dom} (\overline{\delta}_1)
\cap {\rm Dom} (\overline{\delta}_2)\cap \pi^{-1} (  \mathbf{ \mathbf{ \mathfrak{B}_{m} } } ) \,.
 \end{equation}
 We endow the algebra $\mathbf{ \mathbf{ \mathfrak{A}_m } }$, which is a subalgebra
 of $A^*$,
 with the norm
  \begin{equation}\label{norm-gothic-a}
 \| k \|_{\mathbf{ \mathbf{ \mathfrak{A}_m } } }:= \| k \|_{\A_m} +  \| \overline{\delta}_1 k \|_{\A_m} + 
 \|\overline{\delta}_2 k \|_{\A_m} + \| \pi (k) \|_{\mathbf{ \mathbf{ \mathfrak{B}_m } } }
 \end{equation}
 It is an easy exercise to show 
  that $\mathbf{ \mathbf{ \mathfrak{A}_m } }$ is a Banach algebra.


\subsubsection{The modular Shatten extension}
 We can finally state one of the basic results of this whole section:
 
 \begin{proposition}\label{lemma:frak-sequence}
 The map $\pi$ sends 
 $ \mathbf{ \mathfrak{A}_m}$ into  $ \mathbf{\mathfrak{B}_{m}}$;  
 $ \mathbf{ \mathfrak{J}_{m} }$ is an
 ideal in  $\mathbf{ \mathfrak{A}_m}$  and we have a short exact sequence
 of Banach algebras
 \begin{equation}\label{frak-sequence}
 0\rightarrow \mathbf{ \mathfrak{J}_{m} } \rightarrow \mathbf{ \mathfrak{A}_m}\xrightarrow{\pi}   \mathbf{\mathfrak{B}_{m}} \rightarrow 0
 \end{equation}
 The sections $s$ and $t$ restricts to bounded 
sections 
$s:   \mathbf{\mathfrak{B}_{m}}\to  \mathbf{ \mathfrak{A}_m} $ 
and
$t:  \mathbf{ \mathfrak{A}_m} 
 \to \mathbf{ \mathfrak{J}_{m} }
 $.
\end{proposition}
 
 We give a proof of this Proposition in Subsection \ref{subsection:modular-shatten}

 \subsection{Isomorphisms of K-groups}\label{subsect:iso-of-k}

Let $0\to J\to A\xrightarrow{\pi} B\to 0$ a short exact sequence of Banach algebras. Recall that $K_0 (J):=
K_0 (J^+,J)\simeq \Ker (K_0 (J^+)\to \ZZ)$
and that $K(A^+,B^+)= K(A,B)$.
For the definition of relative K-groups we refer, for example, to \cite{bla}, \cite{hr-book}, \cite{lmpflaum}. 
Recall that a relative $K_0$-element
for $ A\xrightarrow{\pi} B$
is represented by a  triple $(P,Q,p_t)$ with $P$ and $Q$ idempotents in $M_{k\times k} (A)$ 
and $p_t\in M_{k\times k} (B)$ a 
path of idempotents connecting $\pi (P)$ to $\pi (Q)$. 
The excision  isomorphism
\begin{equation}\label{excision-general}
\alpha_{{\rm ex}}: K_0 (J)\longrightarrow K_0 (A,B)
\end{equation}
is given by
\begin{equation*}
\alpha_{{\rm ex}}([(P,Q)])=[(P,Q,{\bf c})]
\end{equation*}
with ${\bf c}$ denoting the constant path (this is not necessarily the 0-path, given that
we are taking $J^+$). 
In particular, from the short exact sequence given by the Wiener-Hopf extension of $B^*\equiv B^* (\cyl(\pa X),\F_{{\rm cyl}})$,
see \eqref{eq:short-cstar}, we obtain the isomorphism:
\begin{equation}\label{is-cstar-ex}
\alpha_{{\rm ex}}: K_0 (C^* (X,\F)) \xrightarrow{\simeq} K_0 (A^*, B^*) 
\end{equation}
whereas from the short exact sequence of  subalgebras
\eqref{frak-sequence}
we obtain the "smooth" excision isomorphism
\begin{equation}\label{is-frak-ex}
\alpha_{{\rm ex}}^s : K_0 (\mathbf{ \mathfrak{J}_{m} }) \xrightarrow{\simeq} K_0 (\mathbf{ \mathfrak{A}_m},  \mathbf{\mathfrak{B}_{m}}) \,.
 \end{equation}
  On the other hand, since $\mathbf{ \mathfrak{J}_{m} }$ is a smooth subalgebra of $C^* (X,\F)$
 (i.e. it is dense and holomorphically closed),
  we also have that the inclusion $\iota: \mathbf{ \mathfrak{J}_{m} }
 \hookrightarrow C^* (X,\F)$
induces an isomorphism 
$\iota_* : K_0 (\mathbf{ \mathfrak{J}_{m} }) \xrightarrow{\simeq}  K_0 (C^* (X,\F))$.
Consider the homomorphism $\iota_* : K_0 (\mathbf{ \mathfrak{A}_m},  \mathbf{\mathfrak{B}_{m}})\to K_0 (A^*, B^*)$ induced by the inclusion.
We have a commutative diagram
\begin{equation}\label{diagram}
\xymatrix{
K_0 (\mathbf{ \mathfrak{J}_{m} }) \ar[r]^{\alpha^s_{{\rm ex}}} \ar[d]^{\iota_*} & K_0 (\mathbf{ \mathfrak{A}_m},  \mathbf{\mathfrak{B}_{m}}) \ar[d]^{\iota_*}\\
K_0 (C^* (X,\F)) \ar[r]^{\alpha_{{\rm ex}}} & K_0 (A^*, B^*)}
\end{equation}
and since three of the four arrows are isomorphisms we conclude that 
$\iota_* : K_0 (\mathbf{ \mathfrak{A}_m},  \mathbf{\mathfrak{B}_{m}})\to K_0 (A^*, B^*)$
is also an isomorphism. In particular,
\begin{equation}\label{is-frak}
K_0 (A^*, B^*) \simeq K_0 (C^* (X,\F))\simeq K_0 (\mathbf{ \mathfrak{J}_{m} })\simeq
K_0 (\mathbf{ \mathfrak{A}_m},  \mathbf{\mathfrak{B}_{m}})\,.
\end{equation}

\subsection{Notation}\label{notation-smooth} 
From now on we shall fix the dimension of the leaves, equal to $2n$, and 
set
\begin{equation}\label{algebras-no-subscripts}
\mathbf{\mathfrak{J}}:=\mathbf{\mathfrak{J}_{m}}\,,\quad \mathbf{\mathfrak{A}}:=\mathbf{\mathfrak{A}_{m}}\,\quad
\text{and}\quad \mathbf{\mathfrak{B}}:=\mathbf{\mathfrak{B}_{m}}
\end{equation}
with $m=2n+1$.
The short exact sequence in \eqref{frak-sequence}, for such $m$, is denoted simply as 
\begin{equation}\label{sequence-no-subscripts}
0\to \mathbf{\mathfrak{J}}\to \mathbf{\mathfrak{A}}\to \mathbf{\mathfrak{B}}\to 0
\end{equation}
{\it This is the intermediate subsequence, between $0\to J_c\to A_c \to B_c\to 0$
and $0\to C^* (X,\F)\to A^*(X,\F) \to B^* (\cyl(\pa X),\F_{\cyl}) \to 0$, that we have mentioned
in the  introductory remarks in Subsection \ref{subsection:intr-remarks}}.

\section{{\bf 
$C^*$-index classes. Excision}}\label{sec:index}

\subsection{Geometric set-up and assumptions}\label{subsect:dirac} 
Let $(X_0,\F_0)$,
$X_0=\tM\times_\Gamma T$,  be a
foliated bundle with boundary.
Let 
$(X,\F)$ be the associated foliated bundle with cylindrical ends. We
assume $\tM$ to be of even dimension and we consider the $\Gamma$-equivariant
family of Dirac operators $D\equiv (D_\theta)_{\theta\in T}$ introduced  in \ref{subsect:dirac0}.
We denote as before by $D^\pa\equiv (D^\pa_\theta)_{\theta\in T}$ the boundary family defined by $D^+$
and by $D^{{\rm cyl}}$ the operator induced by  $D^\pa\equiv (D^\pa_\theta)_{\theta\in T}$ 
on the cylindrical foliated manifold $(\cyl(\pa X),\F_{{\rm cyl}})$; $D^{{\rm cyl}}$ is  $\RR\times\Gamma$-equivariant.
From now on we shall make the following fundamental 

\medskip
\noindent
{\bf Assumption.} There exists $\tilde{\epsilon}>0$ such that $\forall \theta \in T$
\begin{equation}\label{assumption}
L^2-{\rm spec} (D^\pa_\theta) \cap (-\tilde{\epsilon},\tilde{\epsilon})= \emptyset
\end{equation}

\medskip
\noindent
For specific examples where this assumption is satisfied, see \cite{LPETALE}.

 \subsection{Index classes in the closed case}\label{pseudo-index-closed}
 Let $(Y,\F)$, $Y=\tN\times_\Gamma T$, be a closed foliated bundle.
We need to recall how in the closed case we can define an index class $\Ind (D)\in K_* (C^*(Y,\F))$.
There are in fact several equivalent descriptions of $\Ind (D)$, each one with its own 
interesting features.  

\subsubsection{The Connes-Skandalis projection.}\label{subsubsec:cs}

First recall that given vector bundles 
$E$ and $F$ on $Y$ with lifts $\widehat{E}$, $\widehat{F}$ on $\tN\times T$,
we can define the space of  $\Gamma$-compactly supported pseudodifferential operators of order $m$,
denoted here 
$\Psi^m_c(G;E,F)$. An element   $P\in \Psi^m_c(G;E,F)$
should be thought of as a $\Gamma$-equivariant family of psedodifferential operators, $(P(\theta))_{\theta\in T} $
with Schwartz kernel $K_P$, a distribution on $G$, of compact support. See \cite{MN}
and \cite{ben-pia} for more details.

The space 
$\Psi^\infty_c (G;E,E):= \bigcup_{m\in \ZZ} \Psi^m_c(G;E,E)$ 
is a filtered algebra.
Moreover, assuming $E$ and $F$ to be hermitian and assigning to $P$ its
formal adjoint $P^*= (P_\theta^*)_{\theta\in T}$ gives $\Psi^\infty_c (G;E,E)$ the structure of an involutive algebra; the formal adjoint of an element  $P \in \Psi^m_c (G;E,F) $ 
 is in general  an alement in $\Psi^m_c (G;F,E)$. 

Consider now a $\ZZ_2$-graded odd Dirac operator $D=(D_\theta)_{\theta\in T}$ 
$D_\theta=\left(\begin{array}{cc} 0 & D^-_\theta\\ D^+_\theta& 0 \end{array} \right)$, $(D^-_\theta)^*=D^+_\theta$.
acting on a
$\ZZ_2$-graded vector bundle $E=E^+ \oplus E^-$. 
Using the pseudodifferential calculus, one can prove that  $D^+$ admits   parametrix $Q
\in \Psi^{-1}_c (G;\widehat{E}^-,\widehat{E}^+)$:
\begin{equation}\label{simbpar}
Q D^+=\Id-S_{+}\,,\quad\quad D^+ Q=\Id - S_{-}
\end{equation}
with remainders $S_{-}$ and $S_{+}$ that are in 
$C^\infty_c (G, (s^*E^\pm)^*\otimes r^* E^\pm)\equiv C^\infty_c (Y,\F;E^\pm)$.

 All of this is carefully explained in \cite{MN}; even more detailes are given in \cite{ben-pia}. 

Consider the projector
\begin{equation}\label{e_Q}
P_Q:= \left(\begin{array}{cc} S_{+}^2 & S_{+}  (I+S_{+}) Q\\ S_{-} D^+ &
I-S_{-}^2 \end{array} \right).
\end{equation}
See, for example,  \cite{Co} (II.9.$\alpha$) and  \cite{CM} (p. 353) for motivation. 
Set 
$
e_0:=\left( \begin{array}{cc} 1 & 0 \\ 0&0 
\end{array} \right)
$
and 
\begin{equation}\label{e1}
e_1:=\left( \begin{array}{cc} 0 & 0 \\ 0&1
\end{array} \right)
\end{equation}
Also denote by  $C^\infty_c (Y,\F; E) ^{++}$  the algebra generated by 
$e_0$, $e_1$ and $C^\infty_c (Y,\F; E) $. It is isomorphic to the direct sum 
$C^\infty_c (Y,\F; E) \oplus \CC e_0 \oplus \CC e_1$ as a linear space.
Note that there exists a splitting short exact sequence:
$0\to C^\infty_c (Y,\F; E)\to C^\infty_c (Y,\F; E)^{++}\xrightarrow{\pi} \CC e_0 \oplus \CC e_1 \to 0$, 
which naturally contains a subsequence
$0\to C^\infty_c (Y,\F; E)\to C^\infty_c (Y,\F; E)^+ \to \CC  \to 0$, 
where $C^\infty_c (Y,\F; E)^+$ is the algebra with unit $1=e_0\oplus e_1$ adjoined. 
Hence, comparing  the induced exact sequences of $K_0$-groups, 
one has the following isomorphism:
$$
K_0(C^\infty_c (Y,\F; E)) :=  \ker [K(C^\infty_c (Y,\F; E)^+) \to K_0(\CC )]
\cong  \ker [K(C^\infty_c (Y,\F; E)^{++}) \to K_0(\CC e_0 \oplus \CC e_1)]
$$
Now it is easy to verify that  $P_Q$ and $e_1$ are idempotents in 
$C^\infty_c (Y,\F; E)^{++}$.
In fact they belong to $C^\infty_c (Y,\F; E) \oplus \CC e_1\subset C^\infty_c (X,\F; E)^{++}$
(but they are not in $C^\infty_c (Y,\F; E)^+$); moreover it is clear that $\pi (P_Q)=e_1=\pi (e_1)$.  
Thus we obtain a class  $[P_Q]-[e_1] \in K_0 (C^\infty_c (Y,\F;E)).$ 
Notice that this class is well defined in 
$K_0 (C^\infty_c (Y,\F;E))$, independent
of the choice of the $\Gamma$-compactly supported parametrix. 
Recall now that there is an inclusion 
$C^\infty_c (Y,\F;E)\hookrightarrow C^* (Y,\F;E)\equiv \KK (\E)$; the Connes-Skandalis index class is the image 
of $[P_Q]-[e_1]$ under the induced homomorphism   $K_0 (C^\infty_c (Y,\F;E))\to K_0 (C^* (Y,\F;E))$.
Unless strictly necessary we don't introduce a new notation for the Connes-Skandalis index class
in $K_0 (C^* (Y,\F;E))$.

\subsubsection{The graph projection.}
 If we give up the requirement that the elements in our projector are  of $\Gamma$-compact support then we have
more representative for the index class. One which is particularly useful in computations of
explicit index formulae is the index class defined by the family $e_D =(e_{D,\theta})_{\theta\in T}$
of projections onto the graph 
(of the closure) of $D^+_\theta$. (With common abuse of notation we 
do not introduce a new symbol for closures.) 
The projector $e_D$ is explicitly given by
\begin{equation}\label{graph-projector}
e_D=\left(\begin{array}{cc} (I+D^- D^+)^{-1} & (I+D^- D^+)^{-1} D^-\\ D^+   (I+D^- D^+)^{-1} &
D^+  (I+D^- D^+)^{-1} D^- \end{array} \right)\,.
\end{equation}
Let $\mathfrak{s}$ be the grading operator on $E$. Define
\begin{equation}\label{e-hat}
\widehat{e}_D:=e_D - \left( \begin{array}{cc} 0 & 0\\0&1
\end{array} \right)
\end{equation}
It is useful to point out, see \cite{MN} page 514, that 
\begin{equation}\label{explicit-e-hat}
\widehat{e}_D = (\mathfrak{s}+ D)^{-1}
\end{equation}
Notice that $(\mathfrak{s}+ D)$ is invertible, indeed
\begin{equation}\label{graph-laplacian}
(\mathfrak{s}+ D)^{-1}=(\mathfrak{s}+ D) (1+D^2)^{-1}\,.
\end{equation}
One proves by finite propagation speed techniques that $\widehat{e}_D$ is
in $C^* (Y,\F;E)$, see \cite{MN} (Section 7) for details; thus the following class is well defined
\begin{equation}\label{graph-projector}
[e_D] -[e_1]\quad\text{with}\quad e_1=\left( \begin{array}{cc} 0 & 0\\0&1
\end{array} \right)  \quad \in \quad  K_0 (C^* (Y,\F;E))
\end{equation}

\begin{proposition}\label{prop:cs=graph}
The Connes-Skandalis index class equals the class defined by the graph projection:
\begin{equation}\label{cs=graph}
[P_Q] -[e_1]
= [e_D] - [e_1]
\quad \text{ in } \quad  K_0 (C^* (Y,\F;E))
\end{equation}
\end{proposition}
For a proof see \cite{MN}, where two elements 
$u, v\in C^* (Y,\F;E)^{++}$ 
 are explicitly
defined such that $uv=P_Q$ and $vu=e_D$. 
Here $C^*(X,\F; E) ^{++}$ denotes as before  the $C^*$-algebra generated by 
$e_0$, $e_1$ and $C^*(X,\F; E) $. 

\subsubsection{The  Wassermann projection.}\label{subsubsection:wass}
This is the self-adjoint projection:
\begin{equation}\label{wassermann}
W_D := \left( \begin{array}{cc} e^{-D^- D^+} & e^{-\frac{1}{2}D^- D^+}\left( \frac{I- e^{-D^- D^+}}{D^- D^+} \right)^{ \frac{1}{2}} D^-
\\
e^{-\frac{1}{2}D^+ D^-}\left( \frac{I- e^{-D^+ D^-}}{D^+ D^-} \right)^{ \frac{1}{2}} D^+& I- e^{-D^+ D^-}
\end{array} \right)
\end{equation}
That $W_D$ is an element in $C^* (Y,\F;E)\oplus \CC e_1$ follows, as usual,  by finite propagation speed techniques.
Thus the class 
\begin{equation}\label{index-wass}
[W_D] - [e_1]
\quad \text{ in } \quad  K_0 (C^* (Y,\F;E))
\end{equation}
is well defined. The following Proposition is proved in \cite{CM}:
\begin{proposition}\label{prop:cs=wass}
The Connes-Skandalis index class equals the class defined by the Wassermann projection:
\begin{equation}\label{cs=wass}
[P_Q] - [e_1]=
[W_D] -[e_1]
\quad \text{ in } \quad  K_0 (C^* (Y,\F;E))
\end{equation}
\end{proposition}

\smallskip
\noindent
{\bf Conclusion:}  we  have 
\begin{equation}\label{all-equal}
[P_Q] -[e_1]
=[e_D] -[e_1]
=[W_D]-[e_1]
\quad \text{ in } \quad  K_0 (C^* (Y,\F;E)),\quad \text{with}\quad
 e_1\equiv \left( \begin{array}{cc} 0 & 0\\0&1
\end{array} \right)   
\end{equation}
and we define the index class associated to $D$, denoted $\Ind (D)$, as this common value.

We remark that 
there exists a path of idempotents $\{H(t)\}_{t\in [0,1]}$ in  $C^* (Y,\F;E)\oplus \CC e_1 $ with 
$H(1)=P_Q$ and $H(0)=e_D$. 
An explicit formula  is given in Subsection \ref{subsection:proofexcision}; this gives an
alternative proof to Proposition \ref{prop:cs=graph}. Using similar techniques we can
construct a path of idempotents 
 $\{P_s (D)\}_{s\in [0,1]}$ in  $C^* (Y,\F;E)\oplus \CC e_1 $
connecting $e_D$ and $W_D$.
This information will be useful in the treatment of the relative index class.

\subsection{The 
index class $\Ind (D)$}\label{subsec:absolute}

We now go back to our foliated bundle with boundary $(X_0,\F_0)$ and associated
foliated bundle  with cylindrical ends $(X,\F)$.
It is proved in \cite{LPETALE}
that given $D^+=(D^+_\theta)_{\theta\in T}$, a $\Gamma$-equivariant family
 with invertible boundary family $(D^{\pa}_{\theta})_{\theta\in T}$, there exists
a parametrix $Q$ for $D^+$ with remainders $S_-$ and $S_+$ in $C^*(X,\F)$:
\begin{equation}\label{par}
Q  D^+=\Id-S_{+}\,,\quad\quad D^+ Q=\Id - S_{-}\,,\quad
S_\pm\in\KK (\E)\equiv C^*(X,\F).
\end{equation}
Thus, there is a well defined index class in $K_0 (C^* (X,\mathcal{F}))$, fixed by the  Connes-Skandalis
projection $P_Q$.
The construction explained in \cite{LPETALE} is an extension to the foliated case
of the parametrix construction of Melrose, using heavily $b$-calculus techniques; needless to say,
all the complications in the foliated context go into dealing with the non-compactness
of the leaves. 

\smallskip
{\it In Subsection \ref{subsec:absolute-elementary}
 we  give an elementary treatment of the parametrix construction
for Dirac operators on manifolds with cyclindrical ends, 
using nothing more than the functional calculus on complete
manifolds. In particular, we do not use any pseudodifferential calculus.}

\smallskip
We shall prove more precisely  that for a Dirac operator on an even
dimensional manifold with cylindrical ends and invertible boundary operator
the following holds:

\begin{theorem}\label{theo:parametrix-elementary}
 Let $G= (I+ D^- D^+)^{-1} D^-$, let $
G^\prime := - \chi ((D^+_{\cyl})^{-1}  (I+D^+_{\cyl} D^-_{\cyl})^{-1} ) \chi$, 
with $\chi$ a smooth approximation of the characteristic
function of $(-\infty,0]\times \pa X_0$.
Then the  operator $Q=G-G^\prime$ is an inverse of $D^+$ modulo $m$-Shatten class operators,
with $m>\dim X$.
\end{theorem}

Similar (elementary) arguments also establish the following basic result:

\begin{theorem}\label{theo:shatten-absolute}
Let $D\equiv (D_\theta)_{\theta\in T}$ be a $\Gamma$-equivariant family  of odd Dirac operators
on a foliated bundle with cylindrical ends $(X,\F)\equiv (\tilde{V}\times_\Gamma T, \F)$.
Assume \eqref{assumption}.
If $\dim \tilde{V}\in 2\NN$ and if $m>\dim \tilde{V}$ then there exists  $Q\in \L (\E)$,
$S_\pm\in \mathcal{I}_m (X,\F)$
such that
\begin{equation}\label{par-bis}
I-Q\,D^+=S_-\,,\quad\quad I-D^+\,Q=S_+\,.
\end{equation}
\end{theorem}

 We give a proof of these two Theorems in Subsection \ref{subsec:absolute-elementary}.

\begin{definition}
The 
index class associated to a Dirac operator $D=(D_\theta)_{\theta\in T}$
 satisfying assumption \eqref{assumption}
is the Connes-Skandalis index class $[P_Q] - [e_1]$
associated to the parametrix $Q$ appearing
in \eqref{par-bis}. It is denoted by $\Ind (D)$ and it is an element
in $K_0 (\mathcal{I}_m (X,\F))\equiv K_0 (C^*(X,\F) )$ ($m$ large).
\end{definition}

\subsection{The relative index class $\Ind(D,D^\pa)$}\label{subsec:relative}
Let $(X,\F)$ be a foliated bundle with cylindrical ends. Let $(\cyl (\pa X),\F_{{\rm cyl}})$
be the associated foliated cylinder and recall the Wiener-Hopf extension 
\begin{equation*} 0\rightarrow C^*(X,\mathcal F)\rightarrow 
A^* (X;\mathcal F)\xrightarrow{\pi} B^* (\cyl (\pa X), \mathcal F_{{\rm cyl}})\rightarrow 0
 \end{equation*}
of the $C^*$-algebra of translation invariant operators  $B^* (\cyl (\pa X),\F_{{\rm cyl}})$.
We shall be concerned with the K-theory group $K_* (C^*(X,\mathcal F))$
and with the relative group $K_* (A^* (X;\mathcal F),B^* (\cyl (\pa X), \mathcal F_{{\rm cyl}}))$,
often denoted simply   $K_* (A^* ,B^*)$,.
Recall that a relative $K_0$-cycle
for $(A^*,B^*)$
is a triple $(P,Q,p_t)$ with $P$ and $Q$ idempotents in $M_{k\times k} (A^*)$ 
and $p_t\in M_{k\times k} (B^*)$
a 
path of idempotents connecting $\pi (P)$ to $\pi (Q)$. 

Denote by $D^{{\rm cyl}}$ the Dirac operator induced by $D^\pa$ on the cylinder.
Consider the triple 
\begin{equation}\label{graph-triple}
(e_D, \begin{pmatrix} 0&0\\0&1 \end{pmatrix},p_t) \,, \;\;t\in [1,+\infty]\,,\;\;\text{ with } p_t:= \begin{cases} e_{(tD^{\cyl})} 
\;\;\quad\text{if}
\;\;\;t\in [1,+\infty)\\
\begin{pmatrix} 0&0\\0&1 \end{pmatrix}\;\;\text{ if }
\;\;t=\infty
 \end{cases}
\end{equation}

Similarly,
we can consider the Wassermann projection and the triple $(W_D, \begin{pmatrix} 0&0\\0&1 \end{pmatrix}, q_t)$,
$t\in [0,+\infty]$,
with 
\begin{equation}\label{wassermann-triple}
(W_D, \begin{pmatrix} 0&0\\0&1 \end{pmatrix}, q_t)
\,, \;\;t\in [1,+\infty]\,,\;\;\text{ with }
q_t:= \begin{cases} W_{(tD^{\cyl})}
\;\;\quad\text{if}
\;\;\;t\in [1,+\infty)\\
\begin{pmatrix} 0&0\\0&1 \end{pmatrix}\;\;\;\,\text{ if }
\;\;t=\infty
 \end{cases}
\end{equation}

\begin{proposition}\label{prop:relative-indeces}
Let $(X,\F)$ be a foliated bundle with cyclindrical ends, as above. Consider the Dirac operator on $X$, 
$D=(D_\theta)_{\theta\in T}$.
Assume    \eqref{assumption}. Then
the graph projection $e_D$ and the Wassermann projection $W_D$  define
through  \eqref{graph-triple} and \eqref{wassermann-triple} two  relative 
classes in $K_0 (A^* ,B^*)$. 
These two classes are equal and fix 
the {\it relative index class} $$\Ind(D,D^\pa)\in K_0 (A^* ,B^*)\,.$$
\end{proposition}

We shall give a proof of this Proposition in Subsection \ref{subsection:proof-relative}

\subsection{Excision for $C^*$-index classes}
The main goal of this subsection is to state the following crucial 

\begin{proposition}\label{proposition:excision}
Let $D=(D_\theta)_{\theta\in T}$ be a $\Gamma$-equivariant family of Dirac operators 
on a foliated manifold with cylindrical ends $X=\tV\times_\Gamma T$. Assume that $\tV$
is even dimensional.
Assume \eqref{assumption}. Let $\alpha_{{\rm ex}}: K_0 (C^* (X,\F))\to K_0 (A^*,B^* )$ be 
the excision isomorphism for the short exact sequence 
$$0\to C^* (X,\F)\to A^*  (X,\F)
\to B^* (\cyl (\pa X),\F_{\cyl}) \to 0.$$
Then
\begin{equation}\label{ex-of-rel-is-ab}
\alpha_{{\rm ex}}\,(\,\Ind (D)\,)= \Ind (D,D^\pa)
\end{equation}
\end{proposition}

We give a  proof of this Proposition in Subsection \ref{subsection:proofexcision}.

\section{{\bf Smooth pairings}}\label{section:smoothpairings}

In the previous Section we have proved the existence of $C^*$-algebraic index classes.
In this Section 
we shall prove that we can extend the cocycles $\tau_{GV}$ and $(\tau_{GV}^r,\sigma_{GV})$ 
from $J_c$ and $A_c\xrightarrow{\pi_c} B_c$
to the smooth  subalgebras $\mathbf{ \mathfrak{J} }$ and $\mathbf{ \mathfrak{A}} \xrightarrow{\pi} 
 \mathbf{\mathfrak{B}}$ 
 and that  we can 
simultaneously  smooth-out our index classes and define them directly in $0\to\mathbf{ \mathfrak{J} }\to\mathbf{ \mathfrak{A}} \xrightarrow{\pi} 
 \mathbf{\mathfrak{B}}\to 0$. Once this will be achieved, we will be able to pair {\it directly}
$[\tau_{GV}]$ with $\Ind (D)$ and $[\tau_{GV}^r,\sigma_{GV}]$ will $\Ind (D,D^\pa)$. This is, as often happens in higher index
theory, a rather crucial point.

\subsection{Smooth index classes}\label{subsect:smooth-index}

\begin{proposition}\label{prop:smooth-cylinder-1}
 Let $D=(D_\theta)_{\theta\in T}$ 
and $X=\tV\times_\Gamma T$ as above;
then the Connes-Skandalis projection $P_Q$
belongs to $\mathbf{ \mathfrak{J}}_{m}\oplus \CC e_1$ with $m>\dim \tV$. 
\end{proposition}
\begin{proposition}\label{prop:smooth-cylinder-2}
Let $e_{(D^{\cyl})}$ be the graph projection for the translation invariant Dirac family $D^{\cyl}= (D^{\cyl}_\theta)_{\theta\in T}$
on the cylinder. 
Then $e_{(D^{\cyl})}\in  \mathbf{\mathfrak{B}}_{m}\oplus \CC e_1$ with $m>\dim \tV$.
More generally, $\forall s\geq 1$ we have $e_{s(D^{\cyl})}\in  \mathbf{\mathfrak{B}}_{m}\oplus \CC e_1$
with $m>\dim \tV$.
\end{proposition}
\begin{proposition}\label{prop:smooth-cylinder-3}
Let $e_{D}$ be the graph projection on $X$. 
Then 
$e_D\in \mathbf{ \mathfrak{A}}_{m}\oplus \CC e_1$ with $m>\dim \tV$.
\end{proposition}

We give a detailed proof of these three Propositions in Subsection \ref{subsection:proof3props}.

As a consequence of these three statements we obtain easily
the first two items of the following 

\begin{theorem}\label{theo:smooth-indeces}
Let $m=2n+1$ with $2n$ equal to the dimension of the leaves of $(X,\F)$. Consider
the modular Shatten extension $0\to \mathbf{ \mathfrak{J}}_{m}\to \mathbf{ \mathfrak{A}}_{m}
\to \mathbf{ \mathfrak{B}}_{m}\to 0$, simply denoted, as in Subsection \ref{notation-smooth},   as
$ 0\to \mathbf{ \mathfrak{J}}\to \mathbf{ \mathfrak{A}}
\to \mathbf{ \mathfrak{B}}\to 0$. 
\item{1)}The Connes-Skandalis projector defines a smooth 
index class $\Ind ^s (D)
\in K_0 ( \mathbf{ \mathfrak{J} })$; moreover, if $\iota_* :  K_0 ( \mathbf{ \mathfrak{J} })
\to K_0 (C^*(X,\F))$ is the isomorphism induced by the inclusion $\iota$, then
$\iota_* (\Ind ^s (D))=\Ind (D)$.
\item{2)}The graph projections on $(X,\F)$ and $(\cyl(\pa X),\F_{\cyl})$ define 
a smooth relative index class $\Ind^s  (D,D^\pa)
\in K_0 (\mathbf{ \mathfrak{A}},  \mathbf{\mathfrak{B}})$; moreover, if 
$\iota_* :  K_0 (\mathbf{ \mathfrak{A}},  \mathbf{\mathfrak{B}})
\to K_0 (A^*,B^*)$ is the isomorphism induced by the inclusion $\iota$, see \eqref{diagram},
then $\iota_* (\Ind ^s (D,D^\pa))=\Ind (D,D^\pa).$
\item{3)}Finally, 
if $\alpha^s_{{\rm ex}}: K_0 ( \mathbf{ \mathfrak{J} })\rightarrow
K_0 (\mathbf{ \mathfrak{A}},  \mathbf{\mathfrak{B}})$ is the smooth excision isomorphism, then
\begin{equation}\label{smooth-ex}\alpha^s_{{\rm ex}}(\Ind^s (D))=\Ind^s (D,D^\pa)\quad\text{in}\quad K_0 (\mathbf{ \mathfrak{A}},  \mathbf{\mathfrak{B}})\,.
\end{equation}
\end{theorem}

\begin{proof}
The fact that the Connes-Skandalis projector $P_Q$ defines an  index class
$\Ind^s (D)\in K_0 (\mathbf{ \mathfrak{J} })$ such that $\iota_* (\Ind^s (D))=\Ind (D)$
in $K_0 (C^* (X,\F))$, is a direct consequence of Proposition \ref{prop:smooth-cylinder-1}.
Similarly, the fact that the graph projections on $(X,\F)$ and $(\cyl(\pa X),\F_{\cyl})$ define
a smooth relative index class $\Ind^s (D,D^{\pa})\in K_0 (\mathbf{ \mathfrak{A}},  \mathbf{\mathfrak{B}})$
such that  $\iota_* (\Ind^s (D,D^{\pa}))=\Ind (D,D^{\pa})$ in $K_0 (A^*, B^*)$
is a direct consequence of Proposition \ref{prop:smooth-cylinder-2} and Proposition \ref{prop:smooth-cylinder-3}.
Regarding the third statement, namely  that $\alpha_{{\rm ex}}^s (\Ind^s (D))=\Ind^s (D,D^{\pa})$,
we argue as follows.
Recall that we have a commutative diagram where all arrows are isomorphism:
\begin{equation}\label{diagram-bis}
\xymatrix{
K_0 (\mathbf{ \mathfrak{J}}) \ar[r]^{\alpha^s_{{\rm ex}}} \ar[d]^{\iota_*} & K_0 (\mathbf{ \mathfrak{A}},  \mathbf{\mathfrak{B}}) \ar[d]^{\iota_*}\\
K_0 (C^* (X,\F)) \ar[r]^{\alpha_{{\rm ex}}} & K_0 (A^*, B^*)}
\end{equation}
Assume, by contradiction, that $\alpha_{{\rm ex}}^s (\Ind^s (D))-\Ind^s (D,D^{\pa})\not= 0$ in 
$K_0 (\mathbf{ \mathfrak{A}},  \mathbf{\mathfrak{B}})$.
Then $\iota_* (\alpha_{{\rm ex}}^s (\Ind^s (D)))-\iota_* (\Ind^s (D,D^{\pa}))\not= 0$, given that
$\iota_* $ is an isomorphism. By the commutativity of the diagram we thus have
$$\alpha_{{\rm ex}}(\iota_*  (\Ind^s (D)))-\iota_* (\Ind^s (D,D^{\pa}))\not= 0.$$
Since we know that
$$\iota_*  (\Ind^s (D))=\Ind (D)\;\;\text{ and } \;\;\iota_* (\Ind^s (D,D^{\pa}))=\Ind (D,D^{\pa})
$$
we conclude that 
$$\alpha_{{\rm ex}}(\Ind (D))- \Ind (D,D^{\pa}))\not= 0$$
and this contradicts the excision formula \eqref{ex-of-rel-is-ab} we have already proved.
\end{proof}

 \subsection{Extended cocycles}\label{subsect:extended-cyclic}

We begin by  recalling the definition of the pairing between $K$-groups and cyclic cohomology
groups. 
First we state it in the absolute case, explaining the pairing  
between the $K_0$-group and the cyclic cohomology group of even degree. 
Here 
we shall follow the definition in \cite{Co} p. 224 rather than the one in \cite{connes-ihes} p.324;
notice that the difference in these two definitions is only in the normalizing constants
(and more precisely in powers of $2\pi i$).
 
Let $A$ be an arbitrary Banach algebra with unit.  
Given a projection $e \in M_{n,n}(A)$  and 
a continuous cyclic cocycle  $\tau: A^{\otimes (2p+1)} \to \CC$  of degree $2p$, 
the  pairing 
$\langle , \rangle :K_0(A) \times HC^{2p}(A) \to \CC$
is defined to be: 
$$
\langle [e], [\tau] \rangle = \frac{1}{p!}\sum_{i_{0}, i_{1}, \cdots , i_{2p}} 
\tau (e_{i_{0} i_{1}}, e_{i_{1} i_{2}}, \cdots , e_{i_{2p} i_{0}} ),
$$
where $e_{ij}$ denotes the $(i, j)$-component 
of the idempotent $e$. 
In the sequel we denote the summation in the right hand side simply by $\tau (e,\dots ,e)$. 
This also  satisfies 
\begin{equation}\label{s-operator}
\langle [e], [\tau] \rangle =\langle [e], [S\tau] \rangle 
\end{equation}
where $S\tau$ is the result of the $S$-operation  in cyclic cohomology, see \cite{Co} p. 193.

If $A$ is not unital, we take the algebra $A^+$ with unit adjoined. 
We then  extend $\tau$ 
to a mulitilinear map $\tau^+: (A^+)^{\otimes (2p+1)} \to \CC$ 
in such a way that  
$\tau^+(a_0, a_1, \dots ,a_{2p})=0$ if $a_i \in \CC 1\subset A^+$ for some $0\leq i \leq 2p$. 
It is easily verified that  $\tau^+$ is again a cyclic cocycle on $A^+$. 
We shall often suppress the $+$ in the notation of $\tau^+$ and denote it simply by $\tau$. 
Given $[e_1]-[e_0] \in  K_0(A)$ 
(note that $e_i$ $(i=0,1)$ is a projection in a  matrix algebra of $A^+$ of a certain size), 
the pairing between $K_0 (A)$ and  $HC^{2p}(A)$ is defined by the following formula:
$$
\langle [e_1]-[e_0], [\tau] \rangle
 = \frac{1}{p!}\left( \tau(e_1, \dots ,e_1) - \tau(e_0, \dots ,e_0)\right)
 :=  \frac{1}{p!} \left [\tau(e_i, \dots ,e_i)\right ]_0^1
 .
$$

Next, recall the definition of relative $K_0$-group:
if $A$ and $B$ are unital Banach algebras  and  
$\pi : A \to B$ denotes  a unital bounded homomorphism, then
the relative group $K_0(A,B)$  is   the abelian group obtained from equivalence classes  
of  triplets
$(e_1, e_0, p_t)$ with
$e_0$ and $e_1$  projections in a matrix algebra of $A$, say $e_0, e_1 \in M_{n,n}(A)$, and
$p_t$  a continuous family of projections in $M_{n,n}(B)$, $t\in [0,1]$, satisfying 
$\pi (e_i)=p_i$ for $i=0,1$. Recall also from Subsection \ref{subsection:rcc} that 
$(\tau, \sigma)$ is a relative cyclic cocycle  of degree $2p$ if 
 $b\tau=\pi^* \sigma$ and $b\sigma=0$ with 
$\tau\in C^{2p}_\lambda (A)$ and  $\sigma\in C^{2p+1}_\lambda (B)$. 
Then the  pairing 
$K_0(A, B) \times HC^{2p}(A, B) \to \CC$
is defined by 
$$
\langle [(e_1, e_0, p_t)], [(\tau, \sigma)] \rangle 
= \frac{1}{p!} \left( 
\left [\tau(e_i, \dots ,e_i)\right ]_0^1
-  (2p+1) \int_0^1\sigma([\dot{p}_t, p_t],p_t,\dots,p_t)dt
\right)
$$
One can prove, thanks to the transgression formula  
of Connes-Moscovici, \cite{CM} p.354, that 
this formula is well defined.

\bigskip
  
   Observe now that $[\tau_{GV}]\in HC^2 (J_c)$ and $[(\tau_{GV}^r, \sigma_{GV})]\in HC^2 (A_c,B_c)$
  can be paired with elements in $K_0 (J_c)$ and $K_0 (A_c, B_c)$ respectively.
  As in \cite{MN}, and with the pairing with the index classes 
  in mind,
  we set
  \begin{equation}\label{higher-cocycles}
  S^{p-1}\tau_{GV}:=\tau_{2p}\quad\text{and}\quad (S^{p-1}\tau_{GV}^r, \frac{3}{2p+1}S^{p-1} \sigma_{GV}):= 
  (\tau_{2p}^r, \sigma_{(2p+1)}).
  \end{equation}
  with $S$ denoting the $S$-operation introduced in \cite{connes-ihes}. Recall the formula
  $b S\phi=\frac{q+1}{q+3} S b\phi$ for a cyclic cochain of degree $q$ (see \cite{connes-ihes}, p. 322). We then have 
  $$b\tau^r_{2p}=b S^{p-1} \tau^r_{GV}=  \frac{3}{2p+1}S^{p-1} b \tau^r_{GV}=\frac{3}{2p+1}S^{p-1}  \pi^* \sigma_{GV}=\pi^* \sigma_{2p+1}\,.$$
   We obtain in this way cyclic cohomology classes
     \begin{equation}\label{higher-cocycles-bis}
   [\tau_{2p}]\in HC^{2p} (J_c)\quad\text{and}\quad [(\tau_{2p}^r, \sigma_{(2p+1)})]\in HC^{2p} (A_c,B_c)
   \,.
  \end{equation}
  
  \begin{proposition}\label{prop:extended-cocycles}
   Let $2n$ be equal the dimension of the leaves in $X=\tilde{V}\times_\Gamma S^1$, i.e.
  the dimension of $\tilde{V}$.
  Let $\mathbf{\mathfrak{J}}:=\mathbf{\mathfrak{J}_{2n+1}}$, 
 Then,  the 
 cocycle $\tau_{2n}$ extends to a bounded 
  cyclic cocycle on $\mathbf{\mathfrak{J}}$.
  \end{proposition}
  \begin{proof}
 By the definition of the $S$ operation in cyclic cohomology, we know that $\tau_{2n} (k_0, \dots, k_{2n})$
is expressed, up to a multiplicative constant,  as the sum of elements of the following type
\begin{align*}
&\omega_{\Gamma} (k_0 \cdots k_{i-1} \,\delta_1 (k_{i}) k_{i+1} \cdots k_{j-1} \,\delta_2 (k_{j}) k_{j+1}\cdots
k_{2n})\\&-\omega_{\Gamma} (k_0 \cdots k_{i-1} \,\delta_2 (k_{i}) k_{i+1} \cdots k_{j-1} \,\delta_1 (k_{j}) k_{j+1}\cdots
k_{2n}) \,;
\end{align*} 
We know, see Proposition \ref{prop:extension-of-weight}, that $\omega_{\Gamma}$ is bounded with respect to the $\mathcal{I}_1$-norm;
moreover, the product appearing in the above formula is bounded from $(\mathbf{\mathfrak{J}_{2n+1}})^{\otimes (2n+1)}$
to $\mathcal{I}_1$. This establishes the Proposition.
\end{proof}
  
  \begin{proposition}\label{prop:extended-cocycles-bis} 
  Let $m=2n+1$ with $2n$ equal to the dimension of leaves.
 Then the  eta cocycle $\sigma_{m}$ 
extends to a  bounded
  cyclic cocycle on $\mathbf{\mathfrak{B}}_m$.
  \end{proposition}

   \begin{proposition}\label{prop:extended-cocycles-tris} 
  Let ${\rm deg} S^{p-1} \tau_{GV}^r=2p>m(m-1)^2-2=m^3-2m^2+m-2$, with $m=2n+1$ and $2n$ equal to the dimension of the leaves
  in $(X,\F)$. 
  Then the regularized Godbillon-Vey cochain $S^{p-1} \tau_{GV}^r$, which is by definition
  $\tau_{2p}^r$, extends to a  bounded 
  cyclic cochain on  $\mathbf{\mathfrak{A}}_{m}$.
    \end{proposition}
 
 \smallskip
 \noindent
 We give a detailed proof of these two propositions in Subsection \ref{subsection:proofs-extension}.
 
 \smallskip
  \noindent
  Fix $m=2n+1$, with $2n$ equal to dimension of the leaves and  set as usual
$$\mathfrak{J} := \mathbf{ \mathfrak{J}_m }\,,\quad  \mathfrak{A}:= \mathbf{ \mathfrak{A}_m} \,,\quad  
\mathfrak{B}: = \mathbf{\mathfrak{B}_{m}}$$
Using the above three Propositions we see that there are well defined classes
   \begin{equation}\label{higher-cocycles-extended}
   [\tau_{2p}]\in HC^{2p} (\mathbf{\mathfrak{J}})\quad\text{for}\quad 2p\geq 2n\quad
   \text{and}\quad [(\tau_{2p}^r, \sigma_{(2p+1)})]\in 
   HC^{2p} (\mathbf{\mathfrak{A}}, \mathbf{\mathfrak{B}})\quad \text{for}\quad 2p>m(m-1)^2-2
   \,.
   \end{equation}
   
\section{{\bf Index theorems}}\label{section:indextheorems}

\subsection{The higher APS index formula for the Godbillon-Vey cocycle}\label{subsection:aps}

We now have all the ingredients to state and prove a APS formula for the Godbillon-Vey cocycle.
Let us summarize our geometric data.
We have a foliated bundle with boundary $(X_0,\F_0)$, $X_0=\tM\times_\Gamma T$.
We assume that the dimension of $\tM$ is even and that
all our geometric structures (metrics, connections, etc) are of product type near the boundary. 
We also
consider $(X,\F)$, the associated foliation with cylindrical ends.
We are given a $\Gamma$-invariant $\ZZ_2$-graded hermitian bundle $\widehat{E}$ on the trivial
fibration $\tM\times T$, endowed with a $\Gamma$-equivariant vertical
Clifford structure. We have a resulting $\Gamma$-equivariant family of Dirac operators $D=(D_\theta)$.

Fix $m=2n+1$, with $2n$ equal to dimension of the leaves and  set
$$\mathfrak{J} := \mathbf{ \mathfrak{J}_m }\,,\quad  \mathfrak{A}:= \mathbf{ \mathfrak{A}_m} \,,\quad  
\mathfrak{B}: = \mathbf{\mathfrak{B}_{m}}$$
We have proved that there are well defined 
smooth index classes
\begin{equation*}
\Ind^s  (D)\in K_0 ( \mathfrak{J})\,,\quad \Ind^s  (D,D^\pa)\in K_0 (\mathfrak{A}, \mathfrak{B})\,,
\end{equation*}
the first given in terms of a  parametrix $Q$ and the second given in term of the graph projections
$e_{D}$ and $e_{D^{\cyl}}$.
 Let $T=S^1$; consider $\tau_{2p}:=S^{p-1} \tau_{GV}$
 and $(\tau_{2p}^r,\sigma_{(2p+1)}):=(S^{p-1} \tau^r_{GV}, \frac{3}{2p+1}S^{p-1} \sigma_{GV})$.
 Then we know that  there are well defined additive maps:
\begin{align}\label{pairings}
& \langle \;\cdot\;, [\tau_{2p}] \rangle\,: K_0 (\mathfrak{J})\rightarrow \CC\,,\quad 2p\geq 2n\\
&\langle \;\cdot\; , [(\tau^r_{2p},\sigma_{(2p+1)})] \rangle\,: K_0 (\mathfrak{A}, \mathfrak{B})
\rightarrow \CC\,,\quad 2p>m(m-1)^2-2\,.
\end{align}

\begin{definition}
Let $(X_0,\F_0)$, $X_0=\tM\times_\Gamma S^1$, as above and assume \eqref{assumption}.
The Godbillon-Vey higher index is the number 
\begin{equation}\label{gvindex}
\Ind_{GV} (D) := \langle \Ind^s (D), [\tau_{2n}] \rangle.
\end{equation}
with $2n$ equal to the dimension of the leaves.
\end{definition}
Notice that, in fact, $\Ind_{GV} (D) := \langle \Ind^s (D), [\tau_{2p}] \rangle$ for each
$p\geq n$, see \eqref{s-operator}.

\medskip
The following theorem is the  main results of this paper:

\begin{theorem}
Let $(X_0,\F_0)$, with $X_0=\tM\times_\Gamma S^1$, be a foliated bundle with boundary 
and let $D:= (D_\theta)_{\theta\in S^1}$ be a $\Gamma$-equivariant
family of  Dirac operators as above. Assume \eqref{assumption} on the boundary family. 
Fix  $2p>m(m-1)^2-2$ with $m=2n+1$ and $2n$ equal to the dimension of the leaves. Then
the following two equalities hold 
\begin{equation}\label{main-dim3}
\Ind_{GV} (D)= \langle \Ind^s (D,D^{\pa}), [(\tau_{2p}^r,\sigma_{2p+1}] \rangle =
\int_{X_0} {\rm AS}\wedge\omega_{GV} - \eta_{GV}
\end{equation}
 with 
 \begin{equation}\label{main-eta}
 \eta_{GV}:= \frac{(2p+1)}{p!}\int_0^{\infty}\sigma_{(2p+1)} ([\dot{p_t},p_t],p_t,\dots,p_t)dt\,,\quad p_t:= e_{tD^{{\rm cyl}}}\,,
 \end{equation}
 defining the {\em Godbillon-Vey eta invariant} of the boundary family and ${\rm AS}$ denoting the form
 induced on $X_0$ by the ($\Gamma$-invariant) Atiyah-Singer form for the fibration $\tM\times S^1\to S^1$
 and the hermitian bundle $\widehat{E}$.
\end{theorem}
Notice that using the Fourier transformation the Godbillon-Vey eta invariant $\eta_{GV}$ does depend only on
 the boundary
family $D^\pa\equiv (D^\pa_\theta)_{\theta\in S^1}$.

\begin{proof}
For notational convenience we set $\tau_{2p}\equiv \tau_{GV}$,
$\tau_{2p}^r \equiv \tau_{GV}^r$
and $\sigma_{(2p+1)}\equiv \sigma_{GV}$. We also write $\alpha_{{\rm ex}}$
instead of $\alpha_{{\rm ex}}^s$.
The left hand side of formula \eqref{main-dim3} is, by definition,
the pairing $\langle [P_Q,e_1], \tau_{GV} \rangle$ with $P_Q$ the Connes-Skandalis
projection and $e_1:= \begin{pmatrix} 0&0\\0&1 \end{pmatrix}$. 
Recall that $\alpha_{{\rm ex}} ( [P_Q,e_1])$
is by definition $[P_Q,e_1,{\bf c}]$, with ${\bf c}$ the constant path with value $e_1$.
Since the derivative of the constant path is equal to zero
and since $\tau^r_{GV}|_{\mathbf{\mathfrak{J}}}=\tau_{GV}$, using the obvious extension
of \eqref{regularized-via-t}, 
we obtain at once the 
crucial relation 
\begin{equation}\label{restriction-of-reg}
\langle \alpha_{{\rm ex}} ([P_Q,e_1]), [(\tau^r_{GV},\sigma_{GV})] \rangle 
= \langle  [P_Q,e_1], [\tau_{GV}] \rangle \,.
\end{equation}
Now we use the excision formula, asserting that  $\alpha_{{\rm ex}} ( [P_Q,e_1])$
is equal, {\it as a relative class},
to $[e_{D}, e_1, p_t]$ with $p_t:= e_{t D^{{\rm cyl}}}$. Thus 
$$\langle [e_{D}, e_1, p_t], [(\tau^r_{GV},\sigma_{GV})] \rangle= \langle  [P_Q,e_1], [\tau_{GV}] \rangle$$
which is the first equality in \eqref{main-dim3} (in reverse order).
Using also the definition of the relative pairing we can summarize our
results so far as follows:
\begin{align*}
\Ind_{GV} (D)&:= \langle \Ind^s (D), [\tau_{GV}] \rangle\\
&\equiv  \langle  [P_Q,e_1], [\tau_{GV}] \rangle\\
&= \langle \alpha_{{\rm ex}} ([P_Q,e_1]), [(\tau^r_{GV},\sigma_{GV})] \rangle \\
&= \langle [e_{D}, e_1, p_t ] , [(\tau^r_{GV},\sigma_{GV})] \rangle \\
&: = \frac{1}{p!} \tau^r_{GV} (e_{D}-e_1)+\frac{(2p+1)}{p!} \int_1^{+\infty}\sigma_{GV} ([\dot{p_t},p_t],p_t,\dots,p_t)dt\\
& \equiv  \frac{1}{p!} \tau^r_{GV} (\widehat{e}_{D})+\frac{(2p+1)}{p!} \int_1^{+\infty}\sigma_{GV} ([\dot{p_t},p_t],p_t,\dots,p_t)dt
\end{align*}
with $\widehat{e}_{D}=(D+\mathfrak{s})^{-1}$.
Notice that the convergence at infinity of 
$\int_1^{+\infty}\sigma_{GV} ([\dot{p_t},p_t],p_t,\dots,p_t)dt$ follows from the fact that the pairing
is well defined. 
Replace $D$ by $uD$, $u>0$.
We obtain, after a simple change of variable in the integral,
$$\frac{(2p+1)}{p!} \int_u^{+\infty}\sigma_{GV} ([\dot{p_t},p_t],p_t,\dots, p_t,p_t)dt=-\langle \Ind^s (uD), [\tau_{GV}] \rangle
+\frac{1}{p!}\tau^r_{GV} (\widehat{e}_{uD})$$
But the absolute pairing  $\langle \Ind^s (uD), [\tau_{GV}] \rangle$ in independent of $u$ and of course equal to $\Ind_{GV} (D)$; thus
$$\frac{(2p+1)}{p!} \int_u^{+\infty}\sigma_{GV} ([\dot{p_t},p_t],p_t,\dots,p_t,p_t)dt=-\Ind_{GV} (D)+
\frac{1}{p!}\tau^r_{GV} (\widehat{e}_{uD})$$
 The second summand
of the right hand side can be proved to converge 
as $u\downarrow 0$ to $\int_{X_0} {\rm AS}\wedge \omega_{GV}$ (this employs Getzler rescaling 
exactly as
in \cite{MN}). 
Thus the limit
$$\frac{(2p+1)}{p!} \lim_{u\downarrow 0} \,\int_s^{+\infty}\sigma_{GV} ([\dot{p_t},p_t],p_t,\dots,p_t,p_t)dt$$
exists \footnote{
the situation here is similar to the one for the eta invariant
in the seminal paper of Atiyah-Patodi-Singer; the regularity there is a consequence of  their index theorem}
and  is equal to $\int_{X_0} {\rm AS}\wedge \omega_{GV}-\Ind_{GV} (D)$.
The theorem is proved
\end{proof}

\subsection{The classic Atiyah-Patodi-Singer index theorem}\label{subsection:classic}
The classic Atiyah-Patodi-Singer index theorem on manifolds with cylindrical ends
 in obtained proceeding as above, but pairing
the 
index class with the 
0-cocycle $\tau_0$ and the relative index class
with the relative  0-cocycle $(\tau_0^r,\sigma_1)$.
(If we use the Wassermann projector we don't need to use the $S$ operation; if we use the graph projection
then we need to consider  $\tau_{2n}:= S^{n}\tau_0$ and $\sigma_{2n+1}:= S^n \sigma_1$ with $2n$ equal
to the dimension of the manifold.) Equating the absolute
and the relative pairing, as above, we obtain an index theorem. It can be proved that this is precisely
the Atiyah-Patodi-Singer  index theorem on manifolds with cylindrical ends; in other words, the eta-term 
we obtain from the relative pairing is precisely the Atiyah-Patodi-Singer  eta invariant for the boundary operator.
As we have pointed out in the Introduction this approach to the classic APS index theorem was announced
by the first author in \cite{AMS}.
This approach to the classic APS index formula is also a Corollary of the main result of the recent preprint of Lesch, Moscovici and Pflaum
\cite{LMP2}, that is,  the computation of the Connes-Chern character of the relative homology
cycle  associated to a Dirac operator on a manifold with boundary in terms of local data 
and a higher eta cochain for the {\em commutative}
algebra of smooth functions on the boundary (see also \cite{Getzler-b} and \cite{WU}). 
Needless to say, the results in \cite{LMP2}
go well beyond the computation of the index; however, they don't have much in common with the
non-commutative results presented in this paper.

\subsection{Gluing formulae for Godbillon-Vey indeces}
 A direct application of our formula is a gluing formula for Godbillon-Vey indeces: 
 if $Y:=\tN\times_\Gamma T$
is a closed foliated bundle and $\tN= \tN^1 \cup_H \tN^2$ with $H$ a $\Gamma$-invariant hypersurfaces, then
we obtain $$\tN\times_\Gamma T =: Y= X^1\cup_Z X^2 :=(\tN^1\times_\Gamma T)\cup_{(H\times_\Gamma T)} (\tN^2\times_\Gamma T).$$
Under the invertibility assumption 
\eqref{assumption} and assuming all geometric structures
to be of product type near $H$, we have, with obvious notation, 
$$\Ind_{GV} (D)=\Ind_{GV} (D^1) + \Ind_{GV} (D^2)$$

\subsection{The Godbillon-Vey eta invariant}
Let $Y=\tN\times_\Gamma T$ be a closed foliated bundle and let $D=(D_\theta)_{\theta\in T}$ be an equivariant Dirac family
satisfying assumption \eqref{assumption}. We do {\it not} assume that $Y$ is the boundary of
a foliated bundle with boundary; in particular, we don't assume that $D$ arises a boundary family. 
Then, thanks to Proposition \ref{prop:extended-cocycles-bis}, we know that 
for $\epsilon>0$ the following  integral is well defined
$$\frac{(2n+1)}{n!} \int_\epsilon^{1/\epsilon}\sigma_{2n+1} ([\dot{p_t},p_t],p_t,\dots, p_t,p_t)dt$$
with $2n-1$ equal to the dimension of the leaves of $Y$.

If the integral converges as $\epsilon\downarrow 0$ then its value defines the Godbillon-Vey eta invariant
of the foliated bundle $\tN\times_\Gamma T$. This is a $C^*$-algebraic
invariant (precisely because we are assuming \eqref{assumption}). 

One might speculate that there is a
corresponding  von Neumann invariant, defined  in the same way, but without the assumption \eqref{assumption}. This is indeed the situation for the von Neumann eta invariant of a {\it measured}
foliation; it exists without any invertibility  assumption on the operator.

\section{{\bf Proofs.}}\label{section:proofs}

In this  Section we have collected all long proofs.
On the one hand  this results in some repetitions
leading to one or two additional pages; on the other hand in this way we were
able to present the main ideas of this paper without long and technical interruptions.

\subsection{Proof of Lemma \ref{lemma:split}.}\label{subsection:split}
Recall that we want to prove that there exists a bounded linear map
$s : B^* \to \mathcal{L}(\mathcal{E})$ extending $s_c : B_c \to 
\mathcal{L} ( \mathcal{E})$,  $s_c (\ell) := \chi^0 \ell \chi^0 $, and that
the composition $\rho =\pi s $ induces an  {\it injective}
$C^*$-homomorphism
$\rho : B^* \to \mathcal{Q}(\mathcal{E}).$
Our first task is to make sense of the operators appearing in the statement of the Lemma.
Thus consider the function $\chi^0$ 
and its lift to 
the covering  $\tilde X :=\tV\times T$, which will be still  denoted by $\chi^0$.
Consider the family of operators induced by the multiplication operator  by $\chi^0$. 
To be precise this consists of the multiplication operators  on 
the Hilbert spaces 
$L^2(\tV\times \{\theta\})$, for 
$\theta\in T$, obtained by  restriction of $\chi^0$ to $\tV\times \{\theta\}$.  
Call the resulting family of operators simply {\it the multiplication operator by $\chi^0$}
and still denote it by  $\chi^0$. 
Similarly, we consider  $\chi^0_{{\rm cyl}}$ and the induced 
 multiplication. 
 Then we are able to consider $\chi^0$ as a family of operators from the Hilbert spaces 
$L^2(\tM\times\{\theta\},E_{\theta})$ to 
 $ L^2( (\RR \times \partial \tM) \times\{\theta\},E_{\theta}))$, 
 identifying the half cylinders 
$(-\infty, 0] \times \partial X_0$  
contained  in both $V$ and $ \cyl  (\partial X)$ 
and hence the images of multiplication operators 
$\chi^0_{\cyl}$ and $\chi^0$.
Thus, given a translation invariant operator $\ell\in B_c$, we can consider the compressed element $\chi^0 \ell \chi^0$ as 
a $\Gamma$-equivariant family of operators  acting  on the Hilbert spaces 
$L^2(\tV\times \{\theta\})$;
in order to define  this element  rigorously we decompose 
the  family of Hilbert spaces $\mathcal H =\{L^2(\tV\times \{\theta\})\}_{\theta\in T}$ as follows:\\
write $\mathcal{H}$ as the direct sum
\begin{equation}\label{decompose}
\mathcal{H}=\mathcal{H}_0\oplus
\mathcal{H}_{{\rm cyl}}^-
\end{equation} 
of families of Hilbert spaces associated to the decomposition 
$ (X,\mathcal F)= (X_0,\mathcal F_0)\cup_{(\partial X_0,\mathcal F_\partial^-)} 
 ((-\infty,0] \times \partial X_0,  \mathcal F_{{\rm cyl}})
$; accordingly  $\chi^0 \ell \chi^0$ is represented by a matrix as
$$
\begin{pmatrix} 0 & 0
\cr 0  & \chi^0 \ell \chi^0 \cr
\end{pmatrix}
.
$$
We shall prove below that $\chi^0 \ell \chi^0$
 belongs to $\mathcal{L}( \mathcal{E})$ 
 and therefore  defines a class in $\mathcal{Q} ( \mathcal{E})$.
Here  observe that 
$\chi^0 \ell \chi^0$  admits a $\Gamma$-equivariant kernel function on 
$\tV\times\tV\times T$ for $\ell\in B_c $.  
Although it is not continuous, it is certainly a measurable function. 

\begin{sublemma}\label{sublemma:compression}
Let $\ell\in B_c $.
Then the  element 
$\chi^0 \ell \chi^0$ belongs to $ \mathcal{L}( \mathcal{E})$.
\end{sublemma}

\begin{proof} 
Let $\chi_{\epsilon}$ be the function 
introduced  in \eqref{smooth-cut-off} and 
set $\sigma_{\epsilon}= \chi^0 -\chi_{\epsilon}$.  
We may assume that  
$\sigma_{\epsilon}(p)$ converges to zero for almost every $p\in X$ 
as $\epsilon \to 0$.
We often suppress $\epsilon$ when it is clear from the context. 
Given $\ell\in B_c $,  
we have 
$$
\chi^0 \ell \chi^0 - \chi \ell \chi
=\sigma \ell \chi + \chi\ell\sigma + \sigma\ell\sigma
.
$$
Note that  $\sigma\ell\chi$,  $\chi\ell\sigma$ and  $\sigma\ell\sigma$
admit  kernel functions  
that have $\Gamma$-compact support (although, again, they are not continuous).   
For such a function $k$ the $\Gamma$-Hilbert-Schmidt norm  $\|\;\;\|_2$
will be defined in Definition \ref{def:shatten} in Subsection \ref{subsect:shatten-ideals}.
We  have
$$
\|\sigma_{\epsilon}\ell \chi\|_{2}^2
\leq
\left( \sup_{\theta\in T} \left( \int_{\tilde{V}_\theta\times \tilde{V}_\theta }
|\chi_\Gamma (p)\sigma_{\epsilon}(p)\ell(p,q,\theta)|^2 dpdq \right) \right)
,\quad\text{with}\quad \tilde{V}_\theta\equiv \tilde{V}\times\{\theta\},
$$
which implies 
$\|\sigma_{\epsilon}\ell \chi\|_{2} \to 0$ as $\epsilon \to 0$ 
due to 
Lebesgue's dominated convergence theorem.  
A similar argument proves that
$\|\chi\ell\sigma_{\epsilon}\|_{2}$ and $\|\sigma_{\epsilon}\ell\sigma_{\epsilon}\|_{2}$ 
also converge to zero. 
Now, if $k\in C_c (G)$ then, see  Proposition
\ref{prop:shatten-properties}, we know that
\begin{eqnarray}\label{eq:operator norm < HS}
\|k \|_{C^*}\leq \|k\|_{2}.
\end{eqnarray}
This implies that 
$
\|\chi^0 \ell \chi^0 - \chi_{\epsilon} \ell \chi_{\epsilon} \|_{C^*} \to 0
$
as $\epsilon \to 0$. 
We thus obtain 
$\chi^0 \ell \chi^0 \in \mathcal{L}(\mathcal E )$ 
for $\ell \in B_c $ 
since  $\chi_{\epsilon}\ell\chi_{\epsilon} \in \mathcal{L} (\mathcal E )$. 

This  completes the proof of Subemma \ref{sublemma:compression}.
\end{proof}

We go on establishing a result on the elements of $B_c$; it will be often used in the sequel.
\begin{sublemma}\label{sublemma:basic1-last}
Let $\ell\in B_c$. Then $\chi^\lambda \ell (1-\chi^\lambda)$, $(1-\chi^\lambda) \ell \chi^\lambda$
and $[\chi^\lambda,\ell]$ are all of $\Gamma$-compact support on $\cyl(\pa X)$.
\end{sublemma}
\begin{proof}
Recall first that by definition of $B_c$
 the support of $\ell$ is compact 
on $(\cyl(\pa X)\times \cyl(\pa X))/\RR\times \Gamma$; observe also that
$$\chi^\lambda \ell- \ell \chi^\lambda= \chi^\lambda \ell (1-\chi^\lambda)- (1-\chi^\lambda)\ell \chi^\lambda\,,\quad \forall \ell\in B_c\,;
$$
We can explicitly write down the kernels  $\kappa_1$, 
$\kappa_2$ and $\kappa$
corresponding to $\chi^\lambda \ell (1-\chi^\lambda)$, $(1-\chi^\lambda)\ell \chi^\lambda$ and $[\chi^\lambda\,,\,\ell]$.
 The first two  are given by:
 \begin{equation}\label{kernel-for-commutator1}
 \kappa_1 (y,s,y^\prime,s^\prime,\theta)= \begin{cases}
\ell (y,y^\prime,s-s^\prime,\theta)
\qquad & \text{ if }\;\; s\leq -\lambda\,,\;\;s^\prime \geq -\lambda
\\
0 \qquad & \text{ otherwise }
\end{cases}
\end{equation}
\begin{equation}\label{kernel-for-commutator2}
\kappa_2 (y,s,y^\prime,s^\prime,\theta)= \begin{cases}
 \ell (y,y^\prime,s-s^\prime,\theta)
\qquad & \text{ if }\;\; s^\prime\leq -\lambda\,,\;\;s \geq -\lambda\\
0 \qquad & \text{ otherwise }
\end{cases}
\end{equation}
whereas the third is obviously given by the relation $\chi^\lambda \ell- \ell \chi^\lambda= \chi^\lambda \ell (1-\chi^\lambda)- (1-\chi^\lambda)\ell \chi^\lambda$, viz.
\begin{equation}\label{kernel-for-commutator}
 \kappa (y,s,y^\prime,s^\prime,\theta) = 
\begin{cases}
\ell (y,y^\prime,s-s^\prime,\theta)
\qquad & \text{ if }\;\; s\leq -\lambda\,,\;\;s^\prime \geq -\lambda
\\
- \ell (y,y^\prime,s-s^\prime,\theta)
\qquad & \text{ if }\;\; s^\prime\leq -\lambda\,,\;\;s \geq -\lambda\\
0 \qquad & \text{ otherwise }
\end{cases}
\end{equation}
In these formulae $y,y^\prime\in \pa \tM$, $s,s^\prime\in \RR$, $\theta\in T$ and we have used the 
translation invariance of $\ell$ in order to write 
$\ell (s,y,s^\prime,y^\prime,\theta)\equiv  \ell (y,y^\prime,s-s^\prime,\theta)$.
These explicit formulae establish the sublemma; indeed since 
$\ell$ is of $\RR\times\Gamma$-compact support it is immediate to check that
the  kernels appearing in  \eqref{kernel-for-commutator1}, 
\eqref{kernel-for-commutator2} and \eqref{kernel-for-commutator}
are all of $\Gamma$-compact support.
 \end{proof}

Consider  now the map 
$s_c : B_c \to 
\mathcal{L} ( \mathcal{E})$,  
  $s_c(\ell ) =\chi^0\ell\chi^0$,
 appearing in the statement of Lemma \ref{lemma:split}.
 The fact that the map $s_c$ extends to a bounded linear map 
$s: B^*\rightarrow \mathcal{L} ( \mathcal{E})$ is clear; indeed
we have
$$
\|s_c(\ell)\|_{C^*}
=
\|\chi^0 \ell \chi^0\|_{C^*}
\leq
\|\ell \|_{C^*}\,.
$$
It remains to show that $\rho:= \pi s$ is a 
injective and a $C^*$-algebra {\it homomorphism}. For the latter property
observe that $\rho_c := \pi s_c$ does satisfy $\rho_c (\ell \ell^\prime)=\rho_c (\ell) \rho_c (\ell^\prime)$:
indeed, if $\ell, \ell^\prime\in B_c$ then 
\begin{align*}
\rho_c (\ell \ell^\prime)&=\pi(\chi^0 \ell \ell^\prime \chi^0)=\pi ((\chi^0 \ell \chi^0 \ell^\prime \chi^0)+(\chi^0 \ell (1-\chi^0) \ell^\prime \chi^0))\\
&= \pi ((\chi^0 \ell \chi^0  \ell^\prime \chi^0))+\pi ((\chi^0 \ell (1-\chi^0) \ell^\prime \chi^0))=  \pi (\chi^0 \ell \chi^0
 \chi^0  \ell^\prime \chi^0)\\
&=  \pi (\chi^0 \ell \chi^0) 
\pi( \chi^0  \ell^\prime \chi^0)=\rho_c (\ell) \rho_c (\ell^\prime)
\end{align*}
since $\pi ((\chi^0 \ell (1-\chi^0) \ell^\prime \chi^0))=0$ given that $\chi^0 \ell (1-\chi^0)$ is 
of $\Gamma$-compact support (we have used   Sublemma \ref{sublemma:basic1} here).
By continuity it follows that $\rho (\ell \ell^\prime)=\rho (\ell) \rho (\ell^\prime)$ for $\ell, \ell^\prime
\in B^*$. The fact that it is a $*$-homomorphism is clear.

Injectiveness is implied at once by  the following:
\begin{equation}\label{injective-rho}
s(B^* (\cyl (\pa X), \mathcal F_{{\rm cyl}}))\cap C^* (X,\F;E)=0\,.
\end{equation}
Let us prove \eqref{injective-rho}. First observe that, because of the translation invariance of
the elements in $B_c$ we immediately have that $s_c (B_c)\cap C_c (X,\F;E)=0$. Next we show that
$s_c (B_c)\cap C^* (X,\F;E)=0$. Suppose the contrary and let $a\in s_c (B_c)\cap C^* (X,\F;E)$, $a\not=0$.
Then $a=\chi^0 \ell \chi^0$ for $\ell\in B_c$ and $\exists$ $a_j\in C_c (X,\F;E)$ such that $\|a_j - a\|_{C^*} \to 0$
as $j\to \infty$. The first information tells us that there exists a $c\in \RR^+$ and $y,y^\prime \in \pa \tM$ such that
$a(y,t,y^\prime,t+c)\not= 0$ for each $t>0$. Take a bump-function $\delta (t)$ at $(y,t,y^\prime,t+c)$ with
$\|\delta (t) \|_{L^2}=1$. Then, keeping the notation $a$ for the operator defined by  $a$, we have that for some
$\epsilon>0$ we have $\| a (\delta (t)) \|_{L^2} > \epsilon >0$ $\forall t >0$.  On the other hand, for each fixed $j$
we also have that $\| a_j (\delta (t)) \|_{L^2} \to 0$ as $t\to +\infty$, given that $a_j$ is an element of $C_c (X,\F;E)$.
Write now $\| a (\delta (t)) \|_{L^2} \leq \| (a-a_j) (\delta (t))\|_{L^2} + \| a_j (\delta (t)) \|_{L^2}\leq  \| (a-a_j)\|_{C^*}+ 
 \| a_j (\delta (t)) \|_{L^2}$.
Then, choosing $j$  big enough we can make the first summand smaller than $\epsilon/2$. For such 
a $j$ we can then choose $t$ big enough so that $\| a_j (\delta (t)) \|_{L^2}$ is also smaller than  $\epsilon/2$.
Summarizing,  $\epsilon< \| a (\delta (t)) \|_{L^2} < \epsilon$, a contradiction. Finally, we show that 
$s(B^*)\cap C^* (X,\F;E) =0$. Assume the contrary and let $\kappa\in s(B^*)\cap C^* (X,\F;E)$, $\kappa\not= 0$.
Then $\exists  \ell\in B^*$ such that $\kappa=s (\ell)$. Choose $\ell_j\in B_c$ such that $\ell_j\to \ell$;
clearly $s(\ell_j)= \chi^0 \ell_j \chi^0\to \kappa$. Set $\kappa_j := s(\ell_j)$, so that $\| \kappa_j - \kappa\|_{C^*}\to 0$.
On the other hand there exists $a_j\in C_c (X,\F;E)$
such that $\| a_j - \kappa \|_{C^*}\to 0$. Proceeding as above we have that there exists an $\epsilon >0$
such that $\|\kappa_j (\delta (t))\|_{L^2} > \epsilon $ for each $t>0$. Observe now that
$\|\kappa_j (\delta (t))\|_{L^2}\leq \|\kappa_j - \kappa \|_{C^*} + \| a_j -\kappa \|_{C^*} + \|a_j (\delta (t)) \|_{L^2}$
and the right hand side  can be made smaller than $\epsilon$ 
by choosing $j$ and $t$ suitably. Thus,  there exists $j$ and $t$ such that
$\epsilon<\|\kappa_j (\delta (t))\|_{L^2} < \epsilon $, a contradiction. \\{\it The proof of Lemma \ref{lemma:split}
is complete.}


\subsection{Proof of Proposition \ref{prop:gv-cocycle}: $(\tau^r_{GV},\sigma_{GV})$ is a relative cyclic 2-cocycle.}\label{subsect:proof-relative}
The proof of  Proposition \ref{prop:sigma3-cyclic} ($\sigma_{GV}$ is cyclic)
is based on the properties of the three derivarions
$\delta_j$ on $\Omega_B$  and of 
the trace $\tau^{\cyl}_{\Gamma}:\Omega_B \to  \CC$, 
where $$\Omega_B:=B_c (\cyl (\pa X),\F_{{\cyl}}; E_{{\cyl}})\oplus 
B_c (\cyl (\pa X),\F_{{\cyl}}; E_{{\cyl}}, E_{{\cyl}}^\prime )$$
and with the algebra structure given as in Lemma \ref{lemma:existence-omega}.
Similarly, we consider the linear space $$\Omega_A:=A_c (X,\F;E)\oplus A_c (X,\F;E,E^\prime)$$
endowed with the algebra structure of Lemma \ref{lemma:existence-omega}. $\Omega_A$ is endowed with {\it two}
commuting derivations.
In the proof to be given below we shall  once again work on 
$\Omega_B$ and $\Omega_A$ rather  than  $B_c (\cyl (\pa X),\F_{{\cyl}}; E_{{\cyl}})$
and $A_c (X,\F:E)$; 
this will allow us to set up an elegant
method in order to prove the fundamental equation $b\tau_{GV}^r= (\pi_c)^* (\sigma_{GV})$.
We remark, preliminary, that the homomorphism $\pi_c: A_c\to B_c$ induces an algebra homomorphism
$\pi_\Omega: \Omega_A\to \Omega_B$.

\medskip
\noindent
We  already know that $b\sigma_{GV}=0$, so we concentrate on the equation $b\tau_{GV}^r= (\pi_c)^* \sigma_{GV}$.

\medskip
\noindent
Let $\{e_1,e_2\}$ be the standard basis of $\RR^2$; consider $\Lambda^* \RR^2$ endowed with
the induced basis. We shall use standard multi-index notation; thus a generic element of the induced basis
in $\Lambda^* \RR^2$
will be denoted by $e_{\bf J}$. 
For notational convenience we set
$
\Omega_B:= \Omega_B$ and $\Omega(G):=\Omega_A.$
Then $\Omega_A\otimes \Lambda^* \RR^2$ becomes a graded algebra  with respect to the multiplication
$(\kappa\otimes e_{\bf J}) (\kappa^\prime \otimes e_{\bf I})=\kappa\kappa^\prime \otimes e_{\bf J}\wedge e_{\bf I}$ 
and the grading in $\Lambda^* \RR^2$. 
Here we forget the grading orininally defined on $\Omega_A$. 
As we have already recalled, there exist derivations 
$
\delta_j : \Omega_A \to \Omega_A
$
for $j=1,2$  with $\delta_1 a = [\dot{\phi},a], \:\:\delta_2 a = [\phi,a]$.
These are defined in the same way as in \eqref{deriv-on-omega}. 

Let $\tau^r_{\Gamma}$ be the functional on $\Omega_A$  
 obtained as a natural extension of $\omega^r_{\Gamma}$. In other words,
  we employ the weight  $\omega^r_\Gamma$ on  the algebra $A_c (X,\F;E)$ in order to define
  a map, still denoted  $\omega^r_\Gamma$, on  the bimodule $A_c (X,\F;E,E^\prime)$; then we 
set 
 \begin{equation}\label{taur}
 \tau^r_\Gamma |_{ A_c (X,\F;E)}:= 0\,,\quad  \tau^r_\Gamma |_{ A_c (X,\F;E,E^\prime)}:=\omega^r_{\Gamma}\,.
 \end{equation}
 Since the regularized trace is not a trace, we remark that $\tau^r_\Gamma$
 is not a trace map on the algebra $\Omega_A$.

\medskip
\noindent
{\bf Notation:} 
\begin{itemize}
\item for the element $\kappa\otimes (e_1\wedge e_2)$ let us set 
\begin{equation}\label{det-r}
\langle \kappa\otimes (e_1\wedge e_2) \rangle_r:=
\tau^r_{\Gamma}(\kappa).
\end{equation} 
\item we set $D: \Omega_A\otimes \Lambda^* \RR^2 \to \Omega_A\otimes \Lambda^* \RR^2$,
$$D(\kappa\otimes e_{\bf J}):=\delta_1 \kappa\otimes e_1\wedge e_{\bf J}+ \delta_2 \kappa \otimes e_2 \wedge e_{\bf J}$$ 
for $\kappa\otimes e_{\bf J}\in \Omega_A\otimes \Lambda^* \RR^2 $.
\end{itemize}

\begin{lemma}\label{sublemma1-relative}
\item{(1)} $D$ is a skew-derivation on $\Omega_A\otimes \Lambda^* \RR^2$ and $D^2=0$; 
\item{(2)} 
 we have:  $\langle D\alpha
 \rangle_r=0$ $\forall \alpha \in \Omega_A\otimes \Lambda^1 \RR^2$.
 \end{lemma}
\begin{proof}
For the first one we compute
\begin{align*}
D^2 (k\otimes e_{\bf J}) &= D(\delta_1 \kappa\otimes e_1\wedge e_{\bf J}+ \delta_2 \kappa \otimes e_2 \wedge e_{\bf J})\\
&= \delta_2 \delta_1 \kappa\otimes e_2\wedge e_1\wedge e_{\bf J} + \delta_1 \delta_2 \kappa \otimes e_1
\wedge e_2 \wedge e_{\bf J} = 0
\end{align*}
since $[\delta_1,\delta_2] \kappa =0 $.\\
Second we consider $\beta= \kappa D \kappa^\prime\in  \Omega_A\otimes \Lambda^1 \RR^2$. Since any
element in $\Omega_A\otimes \Lambda^1 \RR^2$ is a linear combination of elements such as $\kappa D \kappa^\prime$,
it suffices to show that $\langle D\beta
 \rangle_r=0$. We compute:
 \begin{align*}
 D\beta &= D(\kappa\delta_1 \kappa^\prime \otimes e_1 + \kappa \delta_2 \kappa^\prime \otimes e_2)\\
 &= (-\delta_2 \kappa \delta_1 \kappa^\prime + \delta_1\kappa \delta_2 \kappa^\prime)\otimes e_1\wedge e_2\\
 &= \frac{1}{2} \left( \delta_2 (-\kappa \delta_1 \kappa^\prime + (\delta_1 \kappa) \kappa^\prime) +
 \delta_1 (-(\delta_2 \kappa) \kappa^\prime + \kappa \delta_2 \kappa^\prime) \right) \otimes e_1\wedge e_2.
 \end{align*}
Since $\tau^r_{\Gamma}$ is an extension of $\omega_\Gamma^r$, 
 it suffices to show that 
$\omega_\Gamma^r (\delta_1 \kappa)=0=\omega_\Gamma^r (\delta_2 \kappa)\quad \forall \kappa\in \Omega_A\,.$
Recall the definition of $\omega^r_\Gamma$ given in \eqref{regularized-weight}. Remark that $[\phi, \kappa]$,
which is by definition $\delta_1 (\kappa)$, is given explicitly at $(x,x^\prime,\theta)\in \tV\times \tV \times T$
by  $(\phi(x,\theta) - \phi (x^\prime,\theta))\kappa  (x,x^\prime,\theta)$. Next, from the definition of
$\phi$ (it is the logarithm of the Radon-Nykodim derivative of measures that are constant in the normal direction
near the boundary), we see that $\pi_c ([\phi, \kappa])= [\phi_\pa, \ell]$ with $\pi_c (\kappa)=\ell$ and with $\phi_\partial$
the restriction of $\phi$ to $\pa X_0$ (extended to be constant along the cylinder). Thus the value of $ [\phi_\partial, \ell]$ at $(y,t,y^\prime,t^\prime,\theta)$
is equal to $(\phi_\partial (y,\theta) - \phi_\partial (y^\prime,\theta))\ell(y,y^\prime,t-t^\prime,\theta)$. In any case, by applying
 the definition of $\omega^r_\Gamma$ (see again \eqref{regularized-weight}), which involves $ [\phi, \kappa](x,x,\theta)
 $ and $ [\phi_\pa, \ell] (y,t,y,t,\theta)$, we immediately get that $\omega_\Gamma^r (\delta_1 \kappa)=0$. Similarly
 one proves that $\omega_\Gamma^r (\delta_2 \kappa)=0$.
 \end{proof}

\begin{lemma}\label{sublemma2-relative}
Set 
$$\tau_1 (\kappa_0, \kappa_1, \kappa_2):=\langle \kappa_0 D\kappa_1 D \kappa_2\rangle_r\,,\;\;
\tau_2 (\kappa_0, \kappa_1, \kappa_2):=\langle  D\kappa_1 (D \kappa_2) \kappa_0\rangle_r\,\;\;
\tau_3 (\kappa_0, \kappa_1, \kappa_2):=- \langle  (D \kappa_2)\kappa_0 D\kappa_1\rangle_r\,.
$$
Then:
\item{(1)} $6\tau^r_{GV}$ is equal to the restriction of  $( \tau_1 + \tau_2 + \tau_3)$ to $A_c (X,\F;E)$;
\item{(2)} $b\tau_1 (\kappa_0, \kappa_1, \kappa_2,\kappa_3)=\langle [\kappa_0 D\kappa_1 D \kappa_2,\kappa_3]\rangle_r$,
$b\tau_2 (\kappa_0, \kappa_1, \kappa_2, \kappa_3)=\langle  [D\kappa_2 (D \kappa_3) \kappa_0,\kappa_1]\rangle_r$ and\\
$b\tau_3 (\kappa_0, \kappa_1, \kappa_2, \kappa_3)=\langle [ (D \kappa_3)\kappa_0 D\kappa_1,\kappa_2] \rangle_r$.

\end{lemma}

\begin{proof}
Let $\kappa_j\in A_c (X,\F;E)$.
Using Sublemma \ref{sublemma1-relative} (2) we have: $$D(\kappa_1 (D\kappa_2) \kappa_0)=
D\kappa_1 (D\kappa_2) \kappa_0 - \kappa_1 D\kappa_2 D\kappa_0\,,\quad D(\kappa_2\kappa_0 D\kappa_1)=
(D\kappa_2)\kappa_0 D\kappa_1 + \kappa_2 D\kappa_0 D\kappa_1\,.$$
Using Sublemma \ref{sublemma1-relative} (3) we infer that 
$$ \langle \kappa_1 D\kappa_2 D \kappa_0 \rangle_r = \langle D\kappa_1 (D\kappa_2)  \kappa_0 \rangle_r\,,\quad
\langle \kappa_2 D\kappa_0 D \kappa_1 \rangle_r = \langle (D\kappa_2) \kappa_0 D \kappa_1 \rangle_r  \,.$$
Note that, by definition, $2\psi^r_{GV} ( \kappa_0, \kappa_1, \kappa_2)= \langle \kappa_0 D\kappa_1 D \kappa_2 \rangle_r $;
then the equality between $6\tau^r_{GV}$
and the restriction of $( \tau_1 + \tau_2 + \tau_3)$ to $A_c (X,\F;E)$ follows from the above formulae, since 
\begin{align*}
\tau^r_{GV}(\kappa_0,\kappa_1,\kappa_2)&:=\frac{1}{3} \left (\psi^r_{GV} ( \kappa_0, \kappa_1, \kappa_2) + \psi^r_{GV} ( \kappa_1, \kappa_2, \kappa_0)
+\psi^r_{GV} ( \kappa_2, \kappa_0, \kappa_1)\right)\\
&=\frac{1}{6} \left ( \langle \kappa_0 D\kappa_1 D \kappa_2 \rangle_r + \langle \kappa_1 D\kappa_2 D \kappa_0 \rangle_r 
+ \langle \kappa_2 D\kappa_0 D \kappa_1 \rangle_r \right)
\end{align*} 
Next we tackle the second part of the Lemma. By definition and then by the derivation property of $D$ we have:
\begin{align*}
b\tau_1 ( \kappa_0 , \kappa_1 ,  \kappa_2 , \kappa_3 )&=
\langle \kappa_0 \kappa_1 D\kappa_2 D \kappa_3 \rangle_r - \langle \kappa_0 D(\kappa_1 \kappa_2) D \kappa_3 \rangle_r 
+ \langle \kappa_0 D\kappa_1 D (\kappa_2 \kappa_3) \rangle_r  - 
\langle \kappa_3 \kappa_0 D\kappa_1 D \kappa_2 \rangle_r \\
&= \langle \kappa_0 D\kappa_1 (D \kappa_2) \kappa_3 \rangle_r - \langle \kappa_3\kappa_0 D\kappa_1 D \kappa_2 \rangle_r   
\end{align*} 
Similarly, using again the derivation property of $D$ we have:
\begin{align*}
b\tau_2 ( \kappa_0 , \kappa_1 ,  \kappa_2 , \kappa_3 )&=
\langle D\kappa_2 (D\kappa_3) \kappa_0 \kappa_1 \rangle_r - \langle  D(\kappa_1 \kappa_2) (D \kappa_3) \kappa_0 \rangle_r 
+ \langle D\kappa_1 D(\kappa_2 \kappa_3)  \kappa_0) \rangle_r  - 
\langle D\kappa_1 (D\kappa_2) \kappa_3  \kappa_0 \rangle_r \\
&= \langle D\kappa_2 (D\kappa_3) \kappa_0 \kappa_1 \rangle_r - \langle \kappa_1 D\kappa_2 (D \kappa_3) \kappa_0\rangle_r   \\
b\tau_3 ( \kappa_0 , \kappa_1 ,  \kappa_2 , \kappa_3 )&=
-\langle (D\kappa_3)  \kappa_0 \kappa_1 D\kappa_2 \rangle_r + \langle  (D\kappa_3) \kappa_0 (D \kappa_1 \kappa_2) \rangle_r 
- \langle ( D(\kappa_2 \kappa_3)  \kappa_0 D\kappa_1) \rangle_r  +
\langle  (D\kappa_2)  \kappa_3  \kappa_0 D\kappa_1\rangle_r \\
&= \langle (D\kappa_3)\kappa_0 (D\kappa_1) \kappa_2  \rangle_r - \langle \kappa_2 (D \kappa_3) \kappa_0 D \kappa_1
\rangle_r   
\end{align*} 
The proof of the Lemma is now complete.
\end{proof}
Remark now that by using
Melrose' formula for the $b$-trace of a commutator followed by \eqref{equality-with-rbm-bis}, one can show that
\begin{equation}\label{reg-weight-of-comm}
\omega^r_{\Gamma}(k k^\prime-
k^\prime k) = \omega_\Gamma^{{\cyl}} 
(\ell [\chi^0,\ell^\prime])\,,\end{equation}
if $k\in A_c (X,\F;E)$, $k^\prime\in A_c (X,\F;E,E^\prime)$,
$\pi_c (k)=\ell$, $\pi_c (k^\prime)=\ell^\prime$.

Notice that Melrose' proof extend to the regularized weight $\omega_\Gamma^r$ (even though 
 $\omega_\Gamma$ on the boundary  is a weight and not, in general, a trace).
 
 Alternatively, we can simply adapt  the alternative proof 
 of Proposition \ref{prop:c-relative-cocycle},
 which works here for the
 linear functional $\omega^r_\Gamma: A_c (X,\F;E,E^\prime)\to\CC$; namely we write, using Proposition \ref{prop:regularized-via-t},
 and for $k\in A_c (X,\F;E)$, $k^\prime\in A_c (X,\F;E,E^\prime)$,
\begin{align*}
 \omega^r_\Gamma (k k^\prime - k^\prime k)&=\omega_\Gamma (t (k k^\prime - k^\prime k))\\
&= \omega_\Gamma \left( [a,a^\prime]+ [\chi^\mu \ell \chi^\mu, a^\prime] + [a,\chi^\mu \ell^\prime \chi^\mu]- \chi^\mu \ell (1-\chi^\mu)\ell^\prime
\chi^\mu + \chi^\mu\ell^\prime (1-\chi^\mu)\ell \chi^\mu \right)\\
&=\omega_\Gamma (- \chi^\mu \ell (1-\chi^\mu)\ell^\prime
\chi^\mu + \chi^\mu\ell^\prime (1-\chi^\mu)\ell \chi^\mu)\\ &= \omega_\Gamma (\ell [\chi^0,\ell^\prime])\\
 &\equiv - \omega_\Gamma ( [\chi^0,\ell]\ell^\prime)
 \end{align*}
 where we have used the bimudule-trace property for $\omega_{\Gamma}$ in order to justify the third equality.
 Notice that, with obvious notation,
 $$a\in J_c(X, \F;E), k\in A_c (X,\F;E,E^\prime)\Rightarrow ak\in J_c (X,\F;E,E^\prime) \,;$$
 $$a\in J_c(X, \F;E,E^\prime), k\in A_c (X,\F;E)\Rightarrow ak\in J_c (X,\F;E,E^\prime) \,,$$
and similarly for $ka$.

Thus, in any case, using \eqref{reg-weight-of-comm} we obtain immediately that
\begin{equation}\label{tau-trace-formula}
\tau^r_{\Gamma}(\kappa \kappa^\prime-
\kappa^\prime\kappa) = \tau_\Gamma^{{\cyl}} 
(\ell [\chi^0,\ell^\prime])
\end{equation}
where $k, k^\prime \in \Omega_A$, 
$\pi_\Omega (\kappa)=\ell\in \Omega_B$, $\pi_\Omega (\kappa^\prime)=\ell^\prime\in  \Omega_B.$

 From Lemma \ref{sublemma2-relative}, formula \eqref{tau-trace-formula} and recalling the definition $\delta_3 (\ell):= [\chi^0,\ell]$
we finally get, with $\pi_\Omega (\kappa_j)=\ell_j$:
\begin{align*}
6\,b\tau^r_{GV} (\kappa_0,\kappa_1,\kappa_2,\kappa_3) &= 
\langle [\kappa_0 D\kappa_1 D \kappa_2,\kappa_3]\rangle_r
+ \langle  [D\kappa_2 (D \kappa_3) \kappa_0,\kappa_1]\rangle_r
+ \langle [ (D \kappa_3)\kappa_0 D\kappa_1,\kappa_2] \rangle_r\\
&= \tau^{\cyl}_\Gamma (\ell_0 \,\delta_1 \ell_1 \,\delta_2 \ell_2 \,[\chi^0,\ell_3])-
 \tau^{\cyl}_\Gamma (\ell_0 \,\delta_2 \ell_1\, \delta_1 \ell_2\, [\chi^0,\ell_3])\\
 &+ \tau^{\cyl}_\Gamma ( \delta_1 \ell_2 \,\delta_2 \ell_3 \,\ell_0 \,[\chi^0,\ell_1])-
 \tau^{\cyl}_\Gamma ( \delta_2 \ell_2 \,\delta_1 \ell_3 \,\ell_0\, [\chi^0,\ell_1])\\
&+ \tau^{\cyl}_\Gamma ( \delta_1 \ell_3\, \ell_0 \,\delta_2 \ell_1 \, [\chi^0,\ell_2])-
 \tau^{\cyl}_\Gamma ( \delta_2 \ell_3\,\ell_0\,  \delta_1 \ell_1  \,[\chi^0,\ell_2])
\\
&= \sum_{\alpha\in \mathfrak{S}_3} \mathrm{sign}  (\alpha) \tau^{\cyl}_\Gamma (\ell_0 \;\delta_{\alpha (1)} \ell_1\;
\delta_{\alpha (2)} \ell_2\;\delta_{\alpha (3)} \ell_3)
\end{align*}
The proof of the equation $b\tau_{GV}^r=(\pi_c)^* \sigma_{GV}$ is complete.

\subsection{The modular Shatten extension: proof of Proposition  \ref{lemma:frak-sequence}}
\label{subsection:modular-shatten}

Recall that we want to show that there is a short exact sequence
 of Banach algebras
 \begin{equation*} 0\rightarrow \mathbf{ \mathfrak{J}_{m} } \rightarrow \mathbf{ \mathfrak{A}_m}\xrightarrow{\pi}   \mathbf{\mathfrak{B}_{m}} \rightarrow 0
 \end{equation*}
 Moreover, the sections $s$ and $t$ restricts to bounded 
sections 
$s:   \mathbf{\mathfrak{B}_{m}}\to  \mathbf{ \mathfrak{A}_m} $ 
and
$t:  \mathbf{ \mathfrak{A}_m} 
 \to \mathbf{ \mathfrak{J}_{m} }
 $.

 We begin with  two Sublemmas.
 
 \begin{sublemma}
 Let us set $B^\prime_c := \Psi^{-1}_c (G_{\cyl}/\RR_{\Delta})$.
 If $\ell_0$ is an element in $B_c^\prime$, then $\chi^0 \ell_0 \chi^0$ belongs to 
 ${\rm Dom} (\overline{\delta}^{{\rm max}}_j)$ for $j=1,2$ and it follows that
 \begin{equation*}
\overline{ \delta}^{{\rm max}}_2 (\chi^0 \ell_0 \chi^0) 
= \chi^0 [\phi_{\partial},\ell_0] \chi^0\;\text{ and }\;
 \overline{\delta}^{{\rm max}}_1 (\chi^0 \ell_0 \chi^0) = \chi^0 [\dot{\phi}_{\partial},\ell_0] \chi^0\,.
  \end{equation*}
 \end{sublemma}
 \begin{proof}
 We shall work on $\delta_2$ first.
 Let $\chi_\epsilon$ be a smooth approximation of the 
 function induced by $\chi^0$ on $\tilde{V}\times T$.
It is easily verified that
 $\chi^0 \ell_0 \chi^0$ preserves the continuous field $C^\infty_c (\tilde{V}\times T)$
 and that $[\phi,\chi_\epsilon \ell_0 \chi_\epsilon]=\chi_\epsilon [\phi_{\pa},\ell_0] \chi_\epsilon$
 belongs to $C^*_{\Gamma} (\H)$, since $[\phi_{\pa},\ell_0]$ is again a compactly supported pseudodifferential 
 operator of order $-1$. Thus one has 
 $\chi_\epsilon\ell_0 \chi_\epsilon\in {\rm Dom} (\delta^{{\rm max}}_2)$ and
 $ \delta^{{\rm max}}_2 (\chi_\epsilon \ell_0 \chi_\epsilon) = \chi_\epsilon [\phi_{\partial},\ell_0] \chi_\epsilon$. Next we observe 
 that $\| \chi_\epsilon b \chi_\epsilon - \chi^0 b \chi^0 \|_{C^*}\longrightarrow 0$
 as $\epsilon\to 0$ for any $b\in B^\prime_c$. Indeed, according to Lemma \ref{lemma:(-1)-bstar}
  we can choose an approximating sequence $\{b_i\}$ in $B_c$ such that $\|b_i-
 b\|_{C^*}\to 0$; 
then one has
$$
\| \chi_\epsilon b \chi_\epsilon - \chi^0 b \chi^0 \|_{C^*} \leq 
\| \chi_\epsilon (b-b_i)\chi_\epsilon\|_{C^*} +
\|\chi_\epsilon b_i\chi_\epsilon - \chi^0 b_i\chi^0\|_{C^*}
+\|\chi^0 (b_i-b) \chi^0 \|_{C^*}\longrightarrow 0
$$
since 
$\| \chi_\epsilon b_i \chi_\epsilon - \chi^0 b_i \chi^0 \|_{C^*}\longrightarrow 0$ for $b_i\in B_c$ 
due to Sublemma \ref{sublemma:compression}. 
This implies that $\| \chi_\epsilon \ell_0 \chi_\epsilon - \chi^0 \ell_0 \chi^0 \|_{C^*}\longrightarrow 0$ and that
$$ \|\delta^{{\rm max}}_2 (\chi_\epsilon \ell_0 \chi_\epsilon) - \chi^0 [\phi_{\partial},\ell_0] \chi^0\|
= \| \chi_\epsilon [\phi_{\partial},\ell_0] \chi_\epsilon - \chi^0 [\phi_{\partial},\ell_0] \chi^0\|
_{C^*}\to 0$$
as $\epsilon\downarrow 0$. Since $ \chi_\epsilon \ell_0 \chi_\epsilon \in C_{\Gamma,c} (\H)$
this proves that 
$\chi^0 \ell_0 \chi^0$ belongs to 
 ${\rm Dom} (\overline{\delta}^{{\rm max}}_2)$ and  that
$ \overline{\delta}^{{\rm max}}_2 (\chi^0 \ell_0 \chi^0) = \chi^0 [\phi_{\partial},\ell_0] \chi^0$
as required. We can apply a  similar argument to the second derivation and prove
that $\chi^0 \ell_0 \chi^0$ belongs to 
 ${\rm Dom} (\overline{\delta}^{{\rm max}}_1)$ and  that
$ \overline{\delta}^{{\rm max}}_1 (\chi^0 \ell_0 \chi^0) = \chi^0 [\dot{\phi}_{\partial},\ell_0] \chi^0$ .
 \end{proof}
 
  \begin{sublemma}
Assume that  $\ell\in \mathcal{B}_m\cap 
{\rm Dom} (\overline{\delta}_1)\cap  {\rm Dom} (\overline{\delta}_2).$ 
Then $s(\ell)\in  {\rm Dom} (\overline{\delta}_1)\cap  {\rm Dom} (\overline{\delta}_2)$
and $\overline{\delta}_j (s (\ell))=s (\overline{\delta}_j \ell)$
for $j=1,2.$

 \end{sublemma}
 
 \begin{proof}
 Notice that we employ the same notation for the derivations on the cylinder
 $\cyl (\pa X)$ and on $X$; this should not cause confusion here.
 Let $\ell$ be an element  ${\rm Dom} (\overline{\delta}_2)$. Then, by definition,
 there exists a sequence $\{\ell_i\}\in B^\prime_c$ such that $||| \ell_i - \ell |||\to 0$
 and $[\phi_{\pa},\ell_i]$ congerges in $C^*$-norm as $i\to +\infty$. Thus, there exists
 an element $\overline{\delta}_2 \ell\in \B_m$. We then obtain
 $\| s(\ell_i) - s (\ell)\|_{C^*}\to 0$ and  $\|s([\phi_{\pa},\ell_i])- s(\overline{\delta}_2 \ell)\|_{C^*}\to 0$,
 since 
 $$\| s(\ell)\|_{C^*}\leq \| \ell\|_{C^*}\leq |||\ell |||\;\;\text{ for }\; \ell\in B^*\,.$$ Using the previous
 sublemma we have
 $$s([\phi_{\pa},\ell_i]):= \chi^0 [\phi_{\pa},\ell_i] \chi^0 =\overline{\delta}_2^{{\rm max}}
 (\chi^0 \ell_i \chi^0)=\overline{\delta}_2^{{\rm max}}
 s(\ell_i) \,.$$
 Hence we obtain 
 $$\| \overline{\delta}_2^{{\rm max}}
 (s(\ell_i))- s (\overline{\delta}_2 (\ell))\|_{C^*}\to 0\,.
 $$
 Since $\overline{\delta}_2^{{\rm max}}$ is a closed derivation, this proves
 that $\overline{\delta}_2^{{\rm max}}
 (s(\ell))= s (\overline{\delta}_2 (\ell))$. Now recall that that $\A_m\cong \J_{m} \oplus
 s(\B_m)$, see \eqref{direct-sum-cal}. 
Then one has
 $\overline{\delta}_2^{{\rm max}}
 (s(\ell))=s (\overline{\delta}_2 (\ell))\in s(\B_m)\subset \A_m$, since
 $\overline{\delta}_2 (\ell)\in \B_m$. This implies that $s (\ell)\in {\rm Dom} (\overline{\delta}_2)$
 by the definition of  domain for $\overline{\delta_2}$ 
 and thus yields 
 $$\overline{\delta}_2 (s(\ell))=  \overline{\delta}_2^{{\rm max}}
 (s(\ell))= s (\overline{\delta}_2 (\ell))\,.$$
 A similar argument will work for $\overline{\delta_1}$. The proof of this second Sublemma
 is completed.

 \end{proof}
We now go back to the proof of Proposition \ref{lemma:frak-sequence}. First we show
that $\mathbf{ \mathfrak{A}_m}$ is isomorphic as Banach space  to  the direct sum
$ \mathbf{ \mathfrak{J}_{m} }\oplus s(\mathbf{\mathfrak{B}_{m}} ) $, in a way compatible 
with the identification 
$\psi:  \J_{m}\oplus s(\B_m)\to \A_m$
sending $(k,s(\ell))$ to $k+s(\ell)$
 explained in
\eqref{direct-sum-cal}. 
Let $\ell$ be an element in $\mathbf{\mathfrak{B}_{m}} $,
which is by definition $\B_m\cap  {\rm Dom} (\overline{\delta}_1)\cap  {\rm Dom} (\overline{\delta}_2)$. 
Using the last Sublemma 
we then  see that $s (\ell) \in \A_m\cap  {\rm Dom} (\overline{\delta}_1)\cap  {\rm Dom} (\overline{\delta}_2)$ and hence that $s(\ell)\in \mathbf{ \mathfrak{A}_m}$, given
that 
$\pi\circ s (\ell)=\ell\in\mathbf{\mathfrak{B}_{m}} $. Moreover, if 
$a\in  \mathbf{ \mathfrak{J}_{m} }:= \J_{m} \cap  {\rm Dom} (\overline{\delta}_1)\cap  {\rm Dom} (\overline{\delta}_2)$, then we certainly have $a\in \mathbf{ \mathfrak{A}_m}$
since $\pi(a)=0\in \mathbf{\mathfrak{B}_{m}} $. This proves that 
$\mathbf{ \mathfrak{J}_{m} }\oplus s(\mathbf{\mathfrak{B}_{m}} )$ is sent into $\mathbf{ \mathfrak{A}_m}$
by $\psi$. 
Conversely,
given $k\in \A_m$ we can write $k=a+s (\ell)$, with $a\in\J_{m}$ and $\ell\in \B_m$.
If $k\in\mathbf{ \mathfrak{A}_m}$, then $\pi (k)=\pi (a) + \pi (s (\ell))=\ell\in 
\mathbf{\mathfrak{B}_{m}}$ by definition of $\mathbf{ \mathfrak{A}_m}$. This implies in turn
that $a=k- s(\ell)\in \mathbf{ \mathfrak{A}_m}$ because $k$ and $s (\ell)$ belong to
$\mathbf{ \mathfrak{A}_m}$. We have proved 
above that $\ell\in \mathbf{\mathfrak{B}_{m}} \Rightarrow s(\ell)\in 
\mathbf{\mathfrak{A}_{m}}$; thus $a\in \mathbf{\mathfrak{A}_{m}} \cap \J_{m} $
which is nothing but $\mathbf{\mathfrak{J}_{m}} $ by definition. 
This proves that $k=a+s (\ell)$
belongs to the image of $\mathbf{ \mathfrak{J}_{m} }\oplus s(\mathbf{\mathfrak{B}_{m}} ) $ 
through 
\eqref{direct-sum-cal}. Thus we have established that 
$\mathbf{ \mathfrak{A}_m}$ is isomorphic to  the direct sum
$\mathbf{ \mathfrak{J}_{m} }\oplus s(\mathbf{\mathfrak{B}_{m}} ) $. Now it is clear that
the sequence  \eqref{frak-sequence}
 $0\rightarrow \mathbf{ \mathfrak{J}_{m} } \rightarrow \mathbf{ \mathfrak{A}_m}\xrightarrow{\pi}   \mathbf{\mathfrak{B}_{m}} \rightarrow 0
 $ is exact, since $\pi\circ s=\Id$ on $\mathbf{\mathfrak{B}_{m}}$. Moreover, one has
 $$\overline{\delta}_j (k)=  \overline{\delta}_j (a) + \overline{\delta}_j (s (\ell))=
  \overline{\delta}_j (a) + s(\overline{\delta}_j (\ell))\,.$$
  This proves that $ \overline{\delta}_j $ commutes with $\pi: \mathbf{ \mathfrak{A}_m}\rightarrow   \mathbf{\mathfrak{B}_{m}}$ as well as with $s: \mathbf{\mathfrak{B}_{m}}
\to \mathbf{ \mathfrak{A}_m}$. This implies that $\pi$ and $s$ are bounded linear maps.
The boundedness of $t$ follows from that of $s$. Finally, it is obvious that $\pi$ is a
homomophism and that $\mathbf{\mathfrak{J}_{m}}=\Ker\pi$ is an ideal in 
$\mathbf{\mathfrak{A}_{m}}$.

\medskip
\subsection{The index class: an  elementary approach to the parametrix construction}\label{subsec:absolute-elementary}
In this Subsection we give a detailed proof of Theorem \ref{theo:parametrix-elementary}
and Theorem  \ref{theo:shatten-absolute}.
We first collect some elementary results for a Dirac operator $D$
on an even dimensional manifold $X$ with cylindrical end  obtained from
a riemannian manifold $(X_0,g)$ with boundary $\pa X_0 = Y$
and  with $g$ a product metric near the boundary. 
As usual we denote the infinite cylinder $\RR\times \pa X_0\equiv \RR\times Y$
by the simple notation $\cyl (Y)$. Finally, we denote by $\mathfrak{s}$ the grading
operator on the $\ZZ_2$-graded bundle $E$ on which $D$ acts; we shall employ the
same symbol for the grading on the induced bundle on the cylinder.

\begin{lemma}\label{lemma:elementary}
Let $f\in C^\infty (X)$. We assume that $f$ and  $df$ are  bounded. Then we have
the following equality of $L^2$-bounded operators
\begin{equation}\label{bracket-formula}
[(D + \mathfrak{s})^{-1},f]=- (D + \mathfrak{s})^{-1} \cl (df) (\mathfrak{s}+ D)^{-1}\,.
\end{equation}
\end{lemma}
\begin{proof}
We just observe that $(\mathfrak{s}+ D)^{-1} f-f(\mathfrak{s}+ D)^{-1}$ is equal
to $(\mathfrak{s}+ D)^{-1} (f (\mathfrak{s}+ D)- (\mathfrak{s}+ D) f ) (\mathfrak{s}+ D)^{-1}$
and since $f (\mathfrak{s}+ D)- (\mathfrak{s}+ D) f =[f, D]=-\cl(df)$ we are done. 
\end{proof}

\begin{lemma}\label{lemma:elementary2}
Let $\chi$ be a smooth approximation of the characteristic function of $(-\infty,0]\times Y$
in $\cyl (Y)$. Consider $\chi$ as a multiplication operator from $C^\infty_c (\cyl (Y),E_{\cyl})$
to $C^\infty_c (X,E)$. Similarly consider the operator 
$C^\infty_c (\cyl (Y),E_{\cyl})\to C^\infty_c (X,E)$
given by Clifford multiplication $\cl(d\chi)$.
Then 
\begin{equation}
\label{elementary2}
D \chi = \chi D_{\cyl} + \cl (d\chi)
\end{equation}
as operators $C^\infty_c (\cyl (Y),E_{\cyl})\to C^\infty_c (X,E)$
\end{lemma}
\begin{proof}
If $s\in C^\infty_c (\cyl (Y),E_{\cyl})$ then from the locality property for $D$
and the product structure near the boundary
we have $D(\chi s)=D_{\cyl} (\chi s)= \chi (D_{\cyl}  s) + \cl(d\chi) s$
and the lemma is proved.
\end{proof}

\begin{lemma}\label{lemma:rellich-higson}
\item{1).} Let $\phi_1, \phi_2\in C^\infty_c (X)$. Then as a  bounded operator on $L^2 (X,E)$
the operator 
$\phi_1 (D+\mathfrak{s})^{-1} \phi_2$
is compact.
\item{2).} Let $\phi\in C^\infty_c (X)$. Then as bounded operators on $L^2 (X,E)$
the operators $\phi (D+\mathfrak{s})^{-1}$ and $(D+\mathfrak{s})^{-1}\phi $
are compact.
\end{lemma}

\begin{proof}
The proof of 1) is classic (just apply Rellich's lemma).
For 2). Let $\psi\in C^\infty_c (X)$ be  equal to one on the support of
$\phi$. Then, obviously, $\phi (D+\mathfrak{s})^{-1}=\psi (\phi (D+\mathfrak{s})^{-1})$.
The latter term can be rewritten as $\psi (D+\mathfrak{s})^{-1} \phi - \psi (D+\mathfrak{s})^{-1}
(\cl (d\phi)) (D+\mathfrak{s})^{-1}$. Now, by item 1) both 
$\psi (D+\mathfrak{s})^{-1}\phi$ and $\psi (D+\mathfrak{s})^{-1}
(\cl (d\phi))$ are compact operators; since  $(D+\mathfrak{s})^{-1}$ is bounded and the
compact operators are an ideal, we can finish the proof.
\end{proof}

\begin{proposition}\label{proposition:compression}
The difference 
 $(\mathfrak{s}+ D)^{-1} - \chi (\mathfrak{s}+D^{\cyl})^{-1} \chi$ 
 is a compact operator \footnote{The latter property is of course obvious from a $b$-calculus perspective.}.
 \end{proposition}

\begin{proof}
First notice that we have already made sense of $\chi (\mathfrak{s}+D^{\cyl})^{-1} \chi$
as an operator on $L^2 (X,E)$ in Subsection \ref{subsection:split}
For notational convenience we set 
$$A=(\mathfrak{s}+ D)\quad\text{and}\quad B=(\mathfrak{s}+ D_{\cyl})\,.$$
Thus $(\mathfrak{s}+ D)^{-1}=A^{-1}$ and $(\mathfrak{s}+ D_{\cyl})^{-1}=B^{-1}$. We compute
$$A^{-1} - \chi B^{-1} \chi=A^{-1} (1-\chi B \chi B^{-1} \chi)+A^{-1} (\chi B \chi - A \chi) B^{-1} \chi\,.$$
Here we observe that $(1-\chi B\chi B^{-1}\chi)$ is considered as an operator on $C^\infty_c (X,E)$
and $(\chi B\chi - A\chi)$ as an operator from $C^\infty_c (\cyl (Y),E_{\cyl})$ to $C^\infty_c (X,E)$.
Since $[B,\chi]=\cl (d\chi)$ the first term on the right hand side is equal to
\begin{equation}\label{eq}A^{-1} ( 1-\chi (\chi B+\cl (d\chi))B^{-1} \chi)= A^{-1} ( 1-\chi^3) - A^{-1}\chi \cl (d\chi) B^{-1}\chi\,.\end{equation}
Due to Lemma \ref{lemma:elementary2}, for the second term on the right hand side we have
$$A^{-1} (\chi B \chi - A \chi) B^{-1} \chi=A^{-1} ( (A\chi -\cl(d\chi))\chi - A\chi) B^{-1} \chi$$
which is in turn equal to $\chi (\chi-1) B^{-1}\chi - A^{-1} \cl(d\chi) \chi B^{-1}\chi$.
By Lemma \ref{lemma:rellich-higson} both this last term $\chi (\chi-1) B^{-1}\chi - A^{-1} \cl(d\chi) \chi B^{-1}\chi$
and $A^{-1} (1-\chi B \chi B^{-1} \chi)$, which is \eqref{eq}, are compact operators, given that $(1-\chi^3)$, $\chi (\chi- 1)$
and $\cl (d\chi)$ are all compactly supported. Thus $A^{-1} - \chi B^{-1} \chi$, which is nothing but
$(\mathfrak{s}+ D)^{-1} - \chi (\mathfrak{s}+D_{\cyl})^{-1} \chi$, is compact.  The Proposition is proved.
\end{proof}

\begin{remark}\label{remark:from-compact-to-shatten}
It is important to point out that we have in fact
established that 
\begin{equation}\label{from-compact-to-shatten}
(\mathfrak{s}+ D)^{-1} - \chi (\mathfrak{s}+D_{\cyl})^{-1} \chi\;\in \;\mathcal{I}_m\,,\quad\text{with}\quad m>\dim X.
\end{equation}
Indeed, Lemma \ref{lemma:rellich-higson} can be sharpened to the statement that if $ m>\dim X$ then
for each compactly supported function $\phi$
$$\phi(\mathfrak{s}+ D)^{-1}\quad\text{and}\quad (\mathfrak{s}+ D)^{-1} \phi \quad\text{are $m$-Shatten class}$$
For the proof we first observe the useful identity
\begin{equation}\label{useful-identity}
(\mathfrak{s}+ D)^2=(I+D^2)
\end{equation}
and then consider $(\phi(\mathfrak{s}+ D)^{-1})(\phi(\mathfrak{s}+ D)^{-1})^*$
which is then equal to $\phi (1+D^2)^{-1} \phi$ where we remark once again
that $\phi$ is compactly supported. It is a classic result that such an operator is $m/2$-Shatten class.
Similarly we proceed for $(\mathfrak{s}+ D)^{-1} \phi$.
\end{remark}

We shall now construct a parametrix for $D^+$; in fact we shall construct
an inverse of $D^+$ modulo $m$-Shatten class operators, with $m>\dim X$. 
We introduce the following useful notation: if $L$ and $M$ are two bounded operators
on a Hilbert space and if $m\in [1,+\infty)$ then 
\begin{equation}\label{tilde-shatten}
L\;\sim_m\;K\quad\text{if}\quad L-M\in \mathcal{I}_m\,.
\end{equation}
Consider the operator 
\begin{equation}\label{A}
G= (I+ D^- D^+)^{-1} D^-\,.
\end{equation}
Using  elementary properties of 
the functional calculus for Dirac operators on complete manifolds, we certainly have that 
\begin{equation*}
I- G\, D^+ = (I+D^- D^+)^{-1}\,,\quad I-D^+ \, G= (I+D^+ D^-)^{-1}
\,.
\end{equation*}
The operator $G$, as well as the two remainders, do not have Schwartz kernels that are localized 
near the diagonal; still they are perfectly defined and  they are all bounded on $L^2$.
For notational convenience we set  $$(D^\pm)_{\cyl}=:D^\pm_{\cyl}\,.$$ Recall that up to standard
identifications $D^\pm_{\cyl}=\pm \pa_x + D^{\pa}$, acting on the restriction of
$E^+$ to the boundary, extended in the obvious way to the cylinder.
Consider the operator
\begin{equation}\label{gprime}
G^\prime := - \chi ((D^+_{\cyl})^{-1}  (I+D^+_{\cyl} D^-_{\cyl})^{-1} ) \chi \,.
\end{equation}
Then, a simple computation (which employs the same elementary observations we 
 have already used above)
proves that
\begin{align}\label{g-prime-d-plus}
G^\prime D^+ &= -\chi (I+D^-_{\cyl} D^+_{\cyl})^{-1}\chi  + \chi (D^+_{\cyl})^{-1}  (I+D^+_{\cyl} D^-_{\cyl})^{-1}
\cl(d\chi)\\  D^+ G^\prime &= -\chi (I+D^+_{\cyl} D^-_{\cyl})^{-1}\chi - \cl(d\chi) (D^+_{\cyl})^{-1}  (I+D^+_{\cyl} D^-_{\cyl})^{-1}\chi\,.
\end{align}
Let now  $Q:=G-G^\prime$. $Q$ is clearly bounded on $L^2$. We restate for the benefit of the reader
the Theorem we wish to prove (Theorem \ref{theo:parametrix-elementary}):

\begin{theorem}\label{theo:parametrix-elementary-bis}
 The operator $Q$ is an inverse of $D^+$ modulo $m$-Shatten class operators,
with $m>\dim X$.
\end{theorem}

\begin{proof}
First we observe, from \eqref{useful-identity}, that $(I+D^2)^{-1}= (\mathfrak{s}+D)^{-2}$. Using this 
we check that
 $(I+D^2)^{-1}-\chi (I+D_{\cyl}^2)^{-1}\chi$ can be expressed as
  \begin{align*}
  &(\mathfrak{s}+D)^{-1} ((\mathfrak{s}+D)^{-1}  - \chi (\mathfrak{s}+D_{\cyl})^{-1} \chi)+
  ((\mathfrak{s}+D)^{-1}  - \chi (\mathfrak{s}+D_{\cyl})^{-1} \chi)\chi (\mathfrak{s}+D_{\cyl})^{-1}\chi
  + \\&\chi (\mathfrak{s}+D_{\cyl})^{-1} (\chi^2 -1) (\mathfrak{s}+D_{\cyl})^{-1}\chi
  \end{align*}
  Since this term is $m$-Shatten, wee see that $(I+D^2)^{-1}
  \,\sim_m\,\chi (I+D^2_{\cyl})^{-1}\chi$.
  Now, from  \eqref{g-prime-d-plus}, we  have
  $$ G^\prime D^+   \,\sim_m\, -\chi (I+D^-_{\cyl} D^+_{\cyl})^{-1}\chi\,,\quad
   D^+ G^\prime \,\sim_m\, -\chi (I+D^+_{\cyl} D^-_{\cyl})^{-1}\chi$$
   so that, if we define 
   $$S_+ \,:= \,I- Q\, D^+ \,,\quad S_-\,:=\,I-D^+ \, Q\,$$
   and recall that $Q=G-G^\prime$, we obtain 
   $$S_+= I- (G-G^\prime)\, D^+=(I+D^- D^+)^{-1} + G^\prime D^+
   \,\sim_m\,
   (I+D^- D^+)^{-1}  -\chi (I+D^-_{\cyl} D^+_{\cyl})^{-1}\chi
    \,\sim_m\,
0\,.$$
Thus the remainder $S_+$ is $m$-Shatten class. Similarly we proceed for $S_-$.
The theorem is proved.
 \end{proof}
 
We have presented the parametrix construction in the case $T={\rm point}$, $\Gamma=\{1\}$.
However,  a similar  proof applies to the general case of a foliated bundle with cylindrical ends
$(X,\F)\equiv (\tilde{V}\times_\Gamma T, \F)$ with $\tilde{V}$ of even dimension
\footnote{Similar arguments establish the analogues in odd dimension.}. It will suffice
to apply to the $\Gamma$-equivariant family 
$(D_\theta)_{\theta\in T}$ the functional calculus along the fibers of the trivial fibration
$\tilde{V}\times T \to T$ (obtaining, of course, $\Gamma$-equivariant families). 
All our argument apply verbatim once we observe that
given compactly supported smooth functions $\phi$, $\psi$ on $X$,
the family $(\phi (D_\theta+\mathfrak{s})^{-1} \psi)_{\theta}$ defines an element
in $\KK (\E)$, the compacts of the Hilbert module $\E$. In fact, once we observe that
such an element is in fact in $\mathcal{I}_m (X,\F)$, if $m>\dim \tilde{V}$, we can finally conclude
that Theorem \ref{theo:shatten-absolute} holds. 

\subsection{Proof of the existence of the relative index class.}\label{subsection:proof-relative}
In this subsection we give a proof of Proposition \ref{prop:relative-indeces}.
Denote by $D^{{\rm cyl}}$ the Dirac operator induced by $D^\pa$ on the cylinder.
Consider the triple 
\begin{equation}\label{graph-triple-proof}
(e_D, \begin{pmatrix} 0&0\\0&1 \end{pmatrix},p_t) \,, \;\;t\in [1,+\infty]\,,\;\;\text{ with } p_t:= \begin{cases} e_{(tD^{\cyl})} 
\;\;\quad\text{if}
\;\;\;t\in [1,+\infty)\\
e_1\equiv \begin{pmatrix} 0&0\\0&1 \end{pmatrix}\;\;\text{ if }
\;\;t=\infty
 \end{cases}
\end{equation}
First, we need to justify the fact that the relevant elements here are in the right algebras. Thus we need to show that
\begin{itemize}
\item $e_D$ is in $A^* (X;\mathcal F)$.
\item $e_{(tD^{\cyl})}$ is in $B^* (\cyl (\pa X), \mathcal F_{{\rm cyl}})$.
\end{itemize}
We start with the latter. Fix for simplicity $t=1$. 
We need to show that there exists a sequence of elements $k_j\in B_c (\cyl (\pa X),\F_{{\rm cyl}})$
such that $ \| e_{(D^{\cyl})} - k_j \| \longrightarrow 0$ as $j\to 
+\infty$, with the norm denoting the $C^*$-norm of Subsection \ref{subsec:translation}. 
We use  the fact that 
$D^{\cyl}$ is an $\RR\times \Gamma$-equivariant family. 
(Strictly speaking we are taking the closure of the operators in this family.)
Proceeding precisely as in \cite{MN}, Section 7, thus following ideas of Roe,
we are reduced to the following remark: if $f$ is a rapidly decreasing function on $\RR$ with compactly
supported Fourier transform, then $f(D^{\cyl})$ is given by (the  family of integral operators induced by)
an element in $B_c (\cyl (\pa X),\F_{{\rm cyl}})$. The proof of the last assertion 
 is an easy
generalization of the well known results by Roe, see for example 
\cite{roe-foliation} or  the detailed discussion in
\cite{roe-partitioned}. Since the functions as $f$ are dense in $C_0 (\RR)$ the assertion follows.

Next we show that $e_D\in A^* (X;\mathcal F)$. 
First of all, we need to show that $e_D\in \mathcal{L} (\E)$. This  is the same 
proof as in \cite{MN}.

Now we need to show that the image of $e_D$ in $\mathcal{Q}(\mathcal{E})$ is in the image of $\rho$.
Write $e_D= (e_D - \chi^0 e_{(D^{\cyl})} \chi^0) + \chi^0 e_{(D^{\cyl})} \chi^0$. 
Since we have proved that $e_{D^{\cyl}}$ is in $B^* (\cyl (\pa X), \mathcal F_{{\rm cyl}})$, it suffices to show that
\begin{equation}\label{difference-of-graph}
e_D - \chi^0 e_{(D^{\cyl})} \chi^0\in \KK (\E)\,.
\end{equation}
In order to prove \eqref{difference-of-graph} we first show that $e_D - \chi e_{(D^{\cyl})} \chi\in \KK (\E)$,
with $\chi$ a smooth approximation of $\chi^0$. Using 
\eqref{explicit-e-hat}
we reduce ourselves to establishing that  $(\mathfrak{s}+ D)^{-1} - \chi (\mathfrak{s}+D^{\cyl})^{-1} \chi$,
which we have already done.
As far as $(\mathfrak{s}+ D)^{-1} - \chi^0 (\mathfrak{s}+D^{\cyl})^{-1} \chi^0$ is concerned,
we simply choose a sequence of smooth functions $\chi_j$ converging to $\chi^0$ in $L^2$ and we use the fact
that $\KK (\E)$ is closed in $\mathcal{L} (\E)$; we have already used this argument in the proof of Subemma
\ref{sublemma:compression}. 
The proof of \eqref{difference-of-graph} is complete.

Finally, we need to prove that $p_t$ is a continuous path in $B^*$  joining $\pi (e_D)$
to $e_1$
Now, the above argument shows that for $t\in [1,+\infty)$ 
$\pi (e_{tD})=e_{t(D^{\cyl})}=p_t$,
so we only need to show that $e_{(tD^{\cyl})}$ converges to $\begin{pmatrix} 0&0\\0&1 \end{pmatrix}$
in the $C^*$-norm of $B^* (\cyl(\pa X),\F_{{\rm cyl}})$ as $t\to \infty$; however, using  assumption \eqref{assumption}
this  follows easily. 

Next we consider the Wassermann projection and the triple 
$(W_D, e_1, q_t)$
$t\in [0,+\infty]$,
with 
\begin{equation*}
q_t:= \begin{cases} W_{(tD^{\cyl})}
\;\;\quad\text{if}
\;\;\;t\in [1,+\infty)\\
e_1
\;\;\;\,\text{ if }
\;\;t=\infty
 \end{cases}
\end{equation*}
 Exactly the same arguments as above show that 
$W_{(tD^{\cyl})}\in B^* (\cyl (\pa X),\F_{{\rm cyl}})$ and that $W_D\in \mathcal{L} (\E)$. It remains to show that
$W_D\in A^* (X,\F)$, i.e., arguing as above,  that 
\begin{equation}\label{difference-of-wass}
W_D - \chi^0 W_{(D^{\cyl})} \chi^0\in \KK (\E)\,.
\end{equation}
This is proved as for the graph projection.
The fact that $q_t$
joins $\pi (W_D)$ to $e_1$
follows from classic properties
of the heat kernel together with assumption \eqref{assumption}.

Finally, we need to show that 
$$[W_D, e_1
, q_t]=[e_D, e_1
, p_t]\in 
K_0 (A^* (X;\mathcal F),B^* (\cyl (\pa X), \mathcal F_{{\rm cyl}}))\,.$$
This is proved  using the explicit homotopy 
$P_s (D)$ between $e_D$ and $W_D$ we referred to after Proposition
\ref{prop:cs=wass} (but since we don't use this result we shall be
rather short); indeed $P_s (D)$ 
can be explicitly written down and from its form we realize that $P_s (tD)$
joins $e_{tD}$ to  $W_{tD}$. Thus $\{\pi (P_s (tD)\}_{s\in [0,1],t\in [0,+\infty]}$ provides an homotopy between 
$\{p_t\}_{t\in [0,+\infty]}$ and $\{q_t\}_{t\in [0,+\infty]}$. 
Thus $[W_D, e_1, q_t]=[e_D, e_1, p_t]$
as required. The proof of Proposition \ref{prop:relative-indeces} is complete

\subsection{Proof of the excision formula \eqref{ex-of-rel-is-ab}}\label{subsection:proofexcision}
 Let $Q\in \L(\E^-,\E^+)$ be the  parametrix for $D^+$ obtained as in Theorem
 \ref{theo:parametrix-elementary}.
 
 We consider 
 \begin{equation}\label{proj-for exc}
 e(D^+,Q):= 
 \left( \begin{array}{c}  I\\ D^+
\end{array} \right) \left( \begin{array}{cc}  S_+ & Q
\end{array} \right) = \left( \begin{array}{cc}  S_+  & Q\\ D^+ S_+ &  D^+ Q
\end{array} \right)
 \end{equation}
 The following Lemma  is elementary to check 
 \begin{lemma}
 $e(D^+,Q)$ is 
an idempotent
 in $A^*(X,\F)\equiv A^*$ .  Moreover, if 
  $P_Q$ denotes, as usual, the Connes-Skandalis projection
 associated to $Q$, then
  \begin{equation}\label{conj-for exc}
 P_Q =  \left( \begin{array}{cc} I & Q\\0 & I \end{array} \right)^{-1} 
  e(D^+,Q)  \left( \begin{array}{cc} I & Q\\0 & I \end{array} \right)\,.
  \end{equation}
  \end{lemma}
The path obtained substituting $sQ$, $s\in [0,1]$, to $Q$ in the first and third matrix appearing 
on the right hand  side of \eqref{conj-for exc} is a path of projections in $A^*$ and 
connects the projection $P_Q\in 
C^* (X,\F)\subset A^*$ with the projection  $e(D^+,Q)$. 
On the other hand, another  direct computation shows that if $G=(I+D^- D^+)^{-1} D^-$,
then 
$e( D^+, G)=e_D$, the graph projection. Recall that $Q=G-G^\prime$, with 
$G^\prime$ given by \eqref{gprime}; by composing the path 
of projections 
$$\left( \begin{array}{cc} I & sQ\\0 & I \end{array} \right)^{-1} 
  e(D^+,Q)  \left( \begin{array}{cc} I & sQ\\0 & I \end{array} \right)$$
  with the path of projections  $e(D^+, G-\tau G^\prime)$, $\tau\in [0,1]$,  we obtain  a path 
  of projections $H(t)$ in $A^*$
  joining $P_Q=H(1) $ to $e_D=H(0)$.
    Consider now 
      \begin{equation}\label{exc-mu}
    D^+_\mu:=\mu D^+\,, \;\;\;G_\mu := (I+D^-_\mu  D^+_\mu )^{-1} D^-_\mu \,,
    \;\;\; Q(\mu,\tau):= G_\mu -\tau G^\prime_\mu\,,
    \end{equation}
    with $G^\prime_\mu$ as in \eqref{gprime} but defined in terms of
    $ D^+_\mu$. We have then
     \begin{equation}\label{rests-exc-mu}
    D^+_\mu Q(\mu,\tau)=I-S_- (\mu,\tau)\,,\quad Q(\mu,\tau)D^+_\mu = I - S_+ (\mu,\tau)\,.
    \end{equation}
    In this notation the above path, $H(t)$,   first joins $P_{Q(1,1)}$ to $e(D^+,Q(1,1))$
    and then joins $e(D^+,Q(1,1))$ to $e(D^+,Q(1,0))$, which is $e_D$. We write
    $$P_Q\equiv P_{Q(1,1)}\curvearrowright e(D^+,Q(1,1)) \curvearrowright e(D^+,Q(1,0))\equiv e_D
    \,.$$
    Similarly, we can consider
    $$P_{Q(\mu,1)}\curvearrowright e(D^+,Q(\mu,1)) \curvearrowright e(D^+,Q(\mu,0))\equiv e_{\mu D}$$
    with the second homotopy provided by $e(D^+,Q(\mu,\tau))$, $\tau\in [0,1]$.
   Let $H(\mu,t)$ be this homotopy, connecting  $P_{Q(\mu,1)}$ to $e_{\mu D}$.
   We set $p(\mu,t):=\pi (H(\mu,t))$, where $\mu\in [1,+\infty), t\in [0,1]$. We also set
   \begin{equation}\label{p-infty}
   p(\infty,t):= \left( \begin{array}{cc} 0&0\\0&I \end{array} \right)\,,\quad\forall t\in [0,1]\,.
   \end{equation}
   Assume we could prove that
   the above defined function $p(\mu,t)$ is continuous on $[1,+\infty]_\mu
   \times [0,1]_t$. 
     Then from the above discussion we obtain that    
   $$(H(t), \left( \begin{array}{cc} 0&0\\0&I \end{array} \right),p(\mu,\tau))\;\;\text{joins}
   \;\; (H(1), \left( \begin{array}{cc} 0&0\\0&I \end{array} \right),p(\mu,1))
   \;\;\text{to}
   \;\; (H(0), \left( \begin{array}{cc} 0&0\\0&I \end{array} \right),p(\mu,0))$$
But, as already remarked,
$$H(1)=P_Q\quad\text{and}\quad H(0)=e_D;$$
moreover $p(\mu,1)$ is the constant path, indeed
   $$p(\mu,1):=\pi (H(\mu,1))=\pi (P_{Q(\mu,1)})=\left( \begin{array}{cc} 0&0\\0&I \end{array} \right)\,,$$
   given that $P_{Q(\mu,1)}$ is  a true Connes-Skandalis projection, thus with the property that
   $$ P_{Q(\mu,1)}- \left( \begin{array}{cc} 0&0\\0&I \end{array} \right)\in C^* (X,\F)\,;$$
   finally,
   $H(\mu,0)=e_{\mu D}$, so that  $p(\mu,0)=e_{\mu D^{\cyl}}$; thus, taking into account
   \eqref{p-infty}, we see that
   $p(\mu,0)$  is precisely the path of
   projections appearing in the definition of the relative index class.   
   Summarizing, if   we could prove that
    $p(\mu,t)$ is continuous on $[1,+\infty]_\mu
   \times [0,1]_t$  then 
   $$[P_Q,  \left( \begin{array}{cc} 0&0\\0&I \end{array} \right), \mathrm{const}]=[e_D,\left( \begin{array}{cc} 0&0\\0&I \end{array} \right),p_\mu]$$
   which is what we need to prove in order to conclude. Now, $p(\mu,t)$ is certainly continuous
   in $[1,+\infty)\times [0,1]$;
   we end the proof by showing that, in the $C^*$-norm,
   $$\lim_{\mu\to +\infty} p(\mu,t)=  \left( \begin{array}{cc} 0&0\\0&I \end{array} \right)$$
   uniformly in $t\in [0,1]$.\\
   We begin with the projection of the first homotopy, that connecting
   $P_{Q(\mu,1)}$ to $e(\mu D^+,Q(\mu,1))$. This is
   \begin{equation}\label{proj-of-conj}
      \pi \left(  \left( \begin{array}{cc} I & sQ(\mu,1)\\0 & I \end{array} \right)^{-1} 
  e(D^+_\mu,Q(\mu,1))  \left( \begin{array}{cc} I & sQ(\mu,1)\\0 & I \end{array} \right) \right) \,,\quad s\in [0,1],
  \end{equation}
  which is easily seen to be equal to
  $$ \left( \begin{array}{cc} 0& (1-s)\pi (Q(\mu,1))\\0 & 1 \end{array} \right)\,.$$
  Now we write explicitly:
  $$\pi (Q(\mu,1))= (\mu D^{\cyl,-})(I+ D^{\cyl,-} D^{\cyl,+}\mu^2)^{-1}-\frac{1}{\mu} ( D^{\cyl,+})^{-1}
  (I+ D^{\cyl,-} D^{\cyl,+}\mu^2)^{-1}$$
  which does converge to 0 in the $C^*$-norm as $\mu\to +\infty$. Thus \eqref{proj-of-conj}
  converges to $
\left( \begin{array}{cc} 0& 0 \\0 & I \end{array} \right)$ uniformly in $s$, as required. 
Next we look at the second path,
connecting $e(\mu D^+, Q(\mu,1))$ to $e(\mu D^+, Q(\mu,0))$. We need to compute explicitly
$\pi (e(\mu D^+, Q(\mu,\tau))$ and show that it goes to 0 uniformly in $\tau$.
An explicit and elementary computation shows that 
\begin{align*}\pi (e(\mu D^+, Q(\mu,\tau)) &=  \left( \begin{array}{cc} 
(I+\mu^2 D^{\cyl,-} D^{\cyl,+})^{-1} & (I+\mu^2 D^{\cyl,-} D^{\cyl,+})^{-1} \mu D^{\cyl,-})\\
\mu D^{\cyl,+} (I+\mu^2 D^{\cyl,-} D^{\cyl,+})^{-1}  & I- (I+\mu^2 D^{\cyl,+} D^{\cyl,-})^{-1} 
 \end{array} \right) \\&+ (I+\mu^2 (D^{\cyl})^2)^{-1} 
   \left( \begin{array}{cc} -\tau & \tau (\mu D^{\cyl,-})^{-1}\\
   -\tau \mu D^{\cyl,+}&\tau\end{array} \right).
   \end{align*}
   The second summand converges uniformly to 0 in the $C^*$-norm, whereas the first summand
   converges  uniformly to $\left( \begin{array}{cc} 0& 0 \\0 & I \end{array} \right)$
   in the $C^*$-norm. This ends the proof.

\subsection{Proof of the existence of smooth index classes}\label{subsection:proof3props}

\subsubsection{Proof of Proposition \ref{prop:smooth-cylinder-1}}\label{subsection10-7-1}
Recall the Connes-Skandalis projector
\begin{equation*}
P_Q:= \left(\begin{array}{cc} S_{+}^2 & S_{+}  (I+S_{+}) Q\\ S_{-} D^+ &
I-S_{-}^2 \end{array} \right)
\end{equation*}
Let \begin{equation*}
\widehat{P}_Q:= \left(\begin{array}{cc} S_{+}^2 & S_{+}  (I+S_{+}) Q\\ S_{-} D^+ &
-S_{-}^2 \end{array} \right)
\end{equation*}
We want to show that
$$\widehat{P}_Q\in \mathcal{J}_{m}(X,\F)\cap
{\rm Dom}\overline{\delta}_1\cap {\rm Dom}\overline{\delta}_2\,,$$
with $m>2n$ and $2n$ equal to the dimension of the leaves of $(X,\F)$. We fix such an $m$.
We set, as usual, $(D^\pm)^{\cyl}:= D^\pm_{\cyl}$.
We begin by showing that the Connes-Skandalis matrix $\widehat{P}_Q$ is  in 
 $\mathcal{J}_{m}(X,\F)$. First we show that it belongs to $\I_m$.
   Recall our parametrix $Q=G-G^\prime$
with 
$G= (I+ D^- D^+)^{-1} D^-$ and $G^\prime := \chi ((D^+_{\cyl})^{-1}  (I+D^+_{\cyl} D^-_{\cyl})^{-1} ) \chi \,.$
We know that
  $$S_+ \,:= \,I- Q\, D^+ \,,\quad S_-\,:=\,I-D^+ \, Q\,$$
 are elements in $\mathcal{I}_m (X,\F)$ for $m>\dim \tilde{V}$; hence, obviously,
so they are  $(S_\pm)^2 $ and $(S_+ (I+S_+))Q$.
 Thus we only need to show that $S_- D^+$ belongs to $\mathcal{I}_m (X,\F)$ for $m>\dim \tilde{V}$.
 Recall  that
   $S_-= (I+D^+ D^-)^{-1} + D^+ G^\prime $; thus $S_- D^+=(I+D^+ D^-)^{-1}D^+ + D^+ G^\prime D^+$.
   Now, with the usual elementary techniques, we can express the last term as
   \begin{align*}
   &  ((I+D^+ D^-)^{-1}D^+ - \chi  (I+D^+_{\cyl} D^-_{\cyl})^{-1} D^+_{\cyl}\chi)+\\  
   &(\chi  (I+D^+_{\cyl} D^-_{\cyl})^{-1} 
   \cl (d\chi)- \cl(d\chi)  (I+D^+_{\cyl} D^-_{\cyl})^{-1}  \chi + \cl(d\chi) (D^+_{\cyl})^{-1}
    (I+D^+_{\cyl} D^-_{\cyl})^{-1}  \cl(d\chi))\,.
    \end{align*}
    with $d$ denoting the differential along $\tV$ in the product $\tV\times T$.
    See formula \eqref{d-decomposition}.
    Employing the usual reasoning, the first term  is easily seen to be in $\mathcal{I}_m (X,\F)$ for $m>\dim \tilde{V}$; we have already proved that the same is true for the second term. Thus   $S_- D^+$ is
    in $\mathcal{I}_m (X,\F)$ for $m>\dim \tilde{V}$.\\
          Summarizing: we have proved that $\widehat{P}_Q\in \mathcal{I}_m (X,\F)$ for $m>\dim \tilde{V}$.\\ 
           Next we show that $\widehat{P}_Q\in \J_m$. Consider for example 
           $$S_+=(I+D^- D^+)^{-1} 
   -\chi (I+D^-_{\cyl} D^+_{\cyl})^{-1}\chi  + \chi (D^+_{\cyl})^{-1}  (I+D^+_{\cyl} D^-_{\cyl})^{-1}
\cl(d\chi)$$
We want to show that $g S_+$ is bounded. However, from the explicit expression
we have just written this is readily checked by hand using (variants of) the following

\begin{lemma}
The operator $g (1+D^2)^{-1}$ is bounded.
\end{lemma}
\begin{proof}
Write $g (1+D^2)^{-1}= f f (D+\mathfrak{s})^{-1} (D+\mathfrak{s})^{-1} $ and write the last term as
$f [f, (D+\mathfrak{s})^{-1}]  (D+\mathfrak{s})^{-1} + f (D+\mathfrak{s})^{-1} f  (D+\mathfrak{s})^{-1}$
which is in turn equal to $f  (D+\mathfrak{s})^{-1} \cl (df)  (D+\mathfrak{s})^{-1}  (D+\mathfrak{s})^{-1}
+  f (D+\mathfrak{s})^{-1} f  (D+\mathfrak{s})^{-1}$. Thus it suffices to show that $f   (D+\mathfrak{s})^{-1} $
and $  (D+\mathfrak{s})^{-1} f  $ are bounded. This is easily proved using the Sublemma \ref{the-sublemma} below.
The Lemma is proved.
\end{proof}

             Next we show that $\widehat{P}_Q\in {\rm Dom}\overline{\delta}_1\cap {\rm Dom}\overline{\delta}_2$. First of all, we have the following 
 
 \begin{lemma}\label{lemma:inverse}
 Under assumption \eqref{assumption} we have that
 $$D_{\cyl}^{-1}\in {\rm Dom} (\overline{\delta}^{{\rm max}}_{\cyl,1} )\cap 
 {\rm Dom} (\overline{\delta}^{{\rm max}}_{\cyl,2} )$$
 with $\delta_{\cyl,2} := [\phi_\pa , \;]$ and $\delta_{\cyl,1} := [\dot{\phi}_{\pa} , \;]$
 \end{lemma}
 
 \begin{proof}
Consider a smooth function $h\in C^\infty (\RR)$ such that $h(x)=1/x$ for $|x|>\tilde{\epsilon}$,
with $\tilde{\epsilon}$ as in our invertibility assumption $\eqref{assumption}$.
Clearly $h(D_{\cyl})=D^{-1}_{\cyl}$. We can
find a sequence of functions $\{\beta_\lambda\}_{\lambda >0}$ with the following properties:
\begin{enumerate}
\item $\widehat{\beta}_\lambda$ is compactly supported;
\item $\{\beta_\lambda\}_{\lambda >0}$  is a Cauchy sequence in $W^2 (\RR)$-norm;
\item $\beta_\lambda\longrightarrow h$ in sup-norm as $\lambda\to +\infty$.
\end{enumerate}
The function $\beta_\lambda$ such that  $\widehat{\beta}_\lambda= \rho_\lambda \widehat{h}$, with $\rho_\lambda$
as in \cite{MN} p. 515, does satisfy these three properties. We assume this for the time being
and we conclude the proof of the Lemma.
First, from  the very definition of $\beta_{\lambda}$ and from finite propagation
techniques
we have that $\beta_{\lambda} (D_{\cyl})$ is 
a $(-1)$-order pseudodifferential operator of compact $\RR\times\Gamma$-support.
Next, from property (3), we see that $\beta_\lambda (D_{\cyl})\longrightarrow 
h(D_{\cyl})=D^{-1}_{\cyl}$ in $C^*$-norm when $\lambda\to +\infty$. Finally, from Duhamel
formula we have:
$$\delta_2 (\beta_{\lambda} (D_{\cyl})= [\phi_\pa , \beta_{\lambda} (D_{\cyl})]=\int_{\RR}ds \int_0^1
dt \,
\sqrt{-1}s \widehat{\beta}_\lambda (s) e^{\sqrt{-1}st D_{\cyl}} [\phi_\pa , D_{\cyl}] 
e^{\sqrt{-1}s(1-t) D_{\cyl}}$$
Moreover, as explained in \cite{MN}, p. 520, we have
$$\| [\phi_\pa , \beta_{\lambda} (D_{\cyl})] - [\phi_\pa , \beta_{\mu} (D_{\cyl})]\|\leq C\,\int_{\RR}
|\widehat{\beta}_\lambda (s) - \widehat{\beta}_\mu (s)| |s| ds\,.$$
Now:
\begin{align*}
\int_{\RR}
|\widehat{\beta}_\lambda (s) - \widehat{\beta}_\mu (s)| |s| ds &= \int_{\RR}
|\widehat{\beta}_\lambda (s) - \widehat{\beta}_\mu (s)| |s \sqrt{1+s^2}| \frac{1}{\sqrt{1+s^2}} ds\\
&\leq \| (\widehat{\beta}_\lambda - \widehat{\beta}_\mu)  |s \sqrt{1+s^2}|\|_{L^2 (\RR)}\,
\| \frac{1}{\sqrt{1+s^2}} \|_{L^2 (\RR)}\\
&\leq C \| D_{\cyl} (1+D_{\cyl}^2)^{\frac{1}{2}} (\beta_\lambda - \beta_\mu)\|_{L^2 (\RR)}\\
&\leq C' \| \beta_\lambda - \beta_\mu\|_{W^2 (\RR)}\,.
\end{align*}
Thus, from property (2), we infer that $[\phi_\pa ,\beta_\lambda (D_{\cyl})]$ is a Cauchy sequence
in $C^*$-norm. This means that $h (D_{\cyl})$, which is $D^{-1}_{\cyl}$, is in the domain of 
the closure $\overline{\delta}^{{\rm max}}_{\cyl,2} $. Similarly we proceed for $\delta_1$.\\
It remains to prove that with our definition of $\beta_{\lambda}$ we can satisfy the three properties.
The first one is obvious from the definition. For the second property we estimate, with
$D:=\frac{1}{\sqrt{-1}}\frac{d}{dx}$ on $\RR$:
\begin{align*}
 \| \beta_\lambda - \beta_\mu\|_{W^2 (\RR)} &= \|(1+D^2)  (\beta_\lambda - \beta_\mu)\|_{L^2 (\RR)}\\
 &=\|(1+s^2) (\widehat{\beta}_\lambda  - \widehat{\beta}_\mu)\|_{L^2 (\RR)} \,=\,
 \|(1+s^2) (\rho_\lambda\widehat{h}  - \rho_\mu\widehat{h})\|_{L^2 (\RR)}\\
 &= \|(1+s^2)^2  \widehat{h} \, (\rho_\lambda  - \rho_\mu) \frac{1}{1+s^2}\|_{L^2 (\RR)} \\
 &\leq  \|(1+s^2)^2  \widehat{h}\|_{L^2 (\RR)} \, 
 \|(\rho_\lambda  - \rho_\mu) \frac{1}{1+s^2}\|_{L^\infty (\RR)} .
 \end{align*}
 In the last term, the first factor can be estimated directly and shown to be finite, using the equality
 $$\|(1+s^2)^2  \widehat{h}\|_{L^2 (\RR)}= \|(1+D^2)^2 h \|_{L^2 (\RR)}$$
 and the very definition of $h$ (namely, that it is equal to $1/x$ for $|x|> \tilde{\epsilon}$); the second factor, on the other hand,
  is clearly Cauchy (from the definition
 of $\rho_\lambda$). Thus we have established (2).
 Finally, we tackle (3). Recall that the Fourier transformation extends to a bounded
 map from $L^1 (\RR)$ to $C_0 (\RR)$.
 Thus 
 \begin{align*}
 \| \beta_\lambda - h\|_{C_0 (\RR)} &\leq \|\widehat{\beta}_\lambda -\widehat{h}\|_{L^1 (\RR)}=
 \| (\rho_\lambda -1) \widehat{h} \|_{L^1 (\RR)}\\&=  \| (\rho_\lambda -1) \frac{1}{1+s^2}\,
 (1+s^2)\widehat{h} \|_{L^1 (\RR)}\\
 &\leq   \| (\rho_\lambda -1) \frac{1}{1+s^2}\|_{L^2 (\RR)} \,
\| (1+s^2)\widehat{h} \|_{L^2 (\RR)}
\end{align*}
 The second factor can be estimated as above and shown to be finite; the first factor 
 goes to zero using Lebesgue dominated convergence theorem. The Lemma is now completely
 proved.\end{proof}
 We go back to our goal, proving that $\widehat{P}_Q$ is in
 ${\rm Dom} \overline{\delta}_1 \cap {\rm Dom} \overline{\delta}_2$. This means that for $j=1,2$
 and $m>\dim \tilde{V}$ we have:
 $$\widehat{P}_Q \in {\rm Dom} \overline{\delta}^{{\rm max}}_j\cap 
 \mathcal{J}_m (X,\F)\;\;\;\text{and}\;\;\;
  \overline{\delta}^{{\rm max}}_j (\widehat{P}_Q )\in \mathcal{J}_m (X,\F)\,.$$
  First, we establish the fact that $\widehat{P}_Q \in {\rm Dom} \overline{\delta}^{{\rm max}}_j$
  (we already proved
   that $\widehat{P}_Q  \in \mathcal{J}_m (X,\F))$. We concentrate
   on $ \overline{\delta}^{{\rm max}}_2$; similar arguments will work for
   $ \overline{\delta}^{{\rm max}}_1$. Recall that
   \begin{equation*}
\widehat{P}_Q:= \left(\begin{array}{cc} S_{+}^2 & S_{+}  (I+S_{+}) Q\\ S_{-} D^+ &
-S_{-}^2 \end{array} \right)
\end{equation*}
Let us concentrate on each single entry of this matrix. For the sake of brevity, let
us give all the details for the $(1,1)$-entry, $S^2_+$. It suffices to show that 
$S_+ \in {\rm Dom} \overline{\delta}^{{\rm max}}_2$ and that $\overline{\delta}^{{\rm max}}_2
S_+ \in \mathcal{J}_m (X,\F)$.

For notational convenience we set, for this proof only,
$$\overline{\delta}^{{\rm max}}_2 : = \Theta \,,\quad \overline{\delta}^{{\rm max}}_{\cyl,2}:=
\Theta_{\cyl}$$

  We observe preliminarily that proceeding exactly
   as in \cite{MN} we can prove that $(\mathfrak{s}+D)^{-1}$ is in 
   ${\rm Dom} \Theta$;
   hence so is $(\mathfrak{s}+D)^{-2}$
   which is equal to $(1+D^2)^{-1}$. The same proof establishes the corresponding
   result on the cylinder, 
   for  $(\mathfrak{s}+D_{\cyl})^{-1}$ and $(\mathfrak{s}+D_{\cyl})^{-2}=(1+D_{\cyl}^2)^{-1}$.
   This, together with the last Lemma, shows also that $D_{\cyl}^{-1}(1+D_{\cyl}^2)^{-1}$
   belongs to the domain of $\Theta_{\cyl} $.
   Recall now that $$S_+=(I+D^- D^+)^{-1} 
   -\chi (I+D^-_{\cyl} D^+_{\cyl})^{-1}\chi  + \chi (D^+_{\cyl})^{-1}  (I+D^+_{\cyl} D^-_{\cyl})^{-1}
\cl(d\chi)$$
The first summand is in  ${\rm Dom} 
\Theta$,
as we have
already remarked. The second summand, $- \chi (I+D^+_{\cyl} D^-_{\cyl})^{-1}\chi $,
is obtained by grafting through pre-multiplication and post-multiplication by $\chi$
an element which is the domain of $ \Theta_{\cyl} $.
Such a grafted element is easily seen to belong to  ${\rm Dom} \Theta$,
since we can simply choose as an approximating sequence the one obtained by
grafting the approximating sequence for $(I+D^+_{\cyl} D^-_{\cyl})^{-1}$. 
In the  (easy) proof we use $$ \phi\chi=\chi\phi_{\pa}\,,\quad \chi\phi=\chi\phi_{\partial}\,,\quad [\phi_{\partial},\chi]=0\,.$$
(They all  follow from the fact that the modular function is independent of the
normal variable in a neighbourhood of the boundary of $X_0$.)
Similarly, the third summand is
in ${\rm Dom} \Theta$,
given that, as we have observed above, $D_{\cyl}^{-1}(1+D_{\cyl}^2)^{-1}$
   belongs to the domain of $ \Theta_{\cyl}$.  
   Summarizing: $S_+$ is an element
in ${\rm Dom} \Theta$. 
Next we need to show that 
$\Theta (S_+)$ belongs to $\J_m (X,\F)$. First we prove that it is in $\I_m$.
We 
first observe that 
$S_+$ is the $(1,1)$-entry of the $2\times 2$-matrix
$$(\mathfrak{s}+D)^{-2} - \begin{pmatrix} \chi & 0\\0 & \chi \end{pmatrix}
(\mathfrak{s}+D_{\cyl})^{-2} \begin{pmatrix} \chi & 0\\0 & \chi \end{pmatrix}
+ \begin{pmatrix} \chi & 0\\0 & \chi \end{pmatrix}  (\mathfrak{s}+D_{\cyl})^{-2} 
D_{\cyl}^{-1} \begin{pmatrix} 0& \cl(d\chi) \\\cl(d\chi)& 0 \end{pmatrix}$$
We  compute $\Theta$ of this term, finding 
$$\Theta ((\mathfrak{s}+D)^{-2} ) - \begin{pmatrix} \chi & 0\\0 & \chi \end{pmatrix}
\Theta_{\cyl} ((\mathfrak{s}+D_{\cyl})^{-2}) \begin{pmatrix} \chi & 0\\0 & \chi \end{pmatrix}
+ \begin{pmatrix} \chi & 0\\0 & \chi \end{pmatrix} \Theta_{\cyl}((\mathfrak{s}+D_{\cyl})^{-2} 
D_{\cyl}^{-1} )\begin{pmatrix} 0& \cl(d\chi) \\\cl(d\chi)& 0 \end{pmatrix}$$
The last summand is certainly in $\I_m (X,\F)$, since $d\chi$ is of compact support. 
It is clear that this last term is also in $\J_m$,  i.e. it is bounded
if it is  multiplied on the right and on the left by the function $g$.
Thus we are left with the task of proving that
$$\Theta ((\mathfrak{s}+D)^{-2} ) - \chi
\Theta_{\cyl} ((\mathfrak{s}+D_{\cyl})^{-2}) \chi $$
is in  $\J_m (X,\F)$. We first show that this term is in $\I_m$. Remark that 
the above difference can be computed explicitly, using \cite{MN}; we get
\begin{align*}
& (\mathfrak{s}+D)^{-1} \cl (d\phi) (\mathfrak{s}+D)^{-2}  -
 \chi  ( (\mathfrak{s}+D_{\cyl})^{-1} \cl (d\phi_{\pa}) (\mathfrak{s}+D_{\cyl})^{-2}  )\chi\\
&+ (\mathfrak{s}+D)^{-2}  \cl (d\phi)
(\mathfrak{s}+D)^{-1}-
 \chi ((\mathfrak{s}+D_{\cyl})^{-2}  \cl (d\phi_{\pa})
(\mathfrak{s}+D_{\cyl})^{-1})\chi
\end{align*}
Now, proceeding as in the discussion on the parametrix given in Subsection 
\ref{subsec:absolute}, we can prove that each of these two differences is in
$\I_m (X,\F)$ \footnote{Once again, this is clear from a $b$-calculus
perspective}. 
Let us see the details; we concentrate on the first difference
\begin{equation}\label{firstdifference}
(\mathfrak{s}+D)^{-1} \cl (d\phi) (\mathfrak{s}+D)^{-2}  -
 \chi  ( (\mathfrak{s}+D_{\cyl})^{-1} \cl (d\phi_{\pa}) (\mathfrak{s}+D_{\cyl})^{-2}  )\chi \,;
\end{equation}
we shall analyze the second difference, namely
\begin{equation}\label{seconddifference}
(\mathfrak{s}+D)^{-2}  \cl (d\phi)
(\mathfrak{s}+D)^{-1}-
 \chi ((\mathfrak{s}+D_{\cyl})^{-2}  \cl (d\phi_{\pa})
(\mathfrak{s}+D_{\cyl})^{-1})\chi
\end{equation}
later.

For notational convenience we set $A=(\mathfrak{s}
+D)$ and $B=(\mathfrak{s}+D_{\cyl})$. Recall that $A^{-1}\sim_m \chi B^{-1} \chi$, see
Remark \ref{remark:from-compact-to-shatten}. There we also remarked that 
$fA^{-1}\sim_m 0$, $A^{-1} f\sim_m 0$, $g B^{-1}\sim_m 0$ and $B^{-1} g \sim_m 0$ if
$f$ and $g$ are compactly supported. Rewrite 
the difference \eqref{firstdifference} as $A^{-1}\cl (d\phi) A^{-2} - \chi
B^{-1}\cl (d\phi_{\pa}) B^{-2}\chi$. Using $A^{-1}\sim_m \chi B^{-1} \chi$ we see that the 
difference is $\sim_m$-equivalent to 
$$\chi B^{-1} \chi \cl (d\phi) \chi B^{-1} \chi^2 B^{-1}
\chi - \chi B^{-1}\cl (d\phi_{\pa}) B^{-2}\chi\,.$$
Now add and substract $\chi B^{-1} \chi \cl (d\phi) \chi B^{-2}\chi$ to the first summand;
 obtaining, for this first summand,
$$\chi B^{-1} \chi \cl (d\phi) \chi B^{-2} \chi +
\chi B^{-1} \chi \cl (d\phi) \chi B^{-1} (\chi^2-1) B^{-1}$$
which, by our second remark, is  $\sim_m$-equivalent to 
$ \chi B^{-1} \chi \cl (d\phi) \chi B^{-2} \chi \,.$
Here we have used that $\chi^2-1$ is compactly supported.
Thus \eqref{firstdifference} is $\sim_m$-equivalent to 
$\chi B^{-1} (\chi \cl (d\phi) \chi - \cl(d\phi_{\pa}) ) B^{-2} \chi$
which is equal to $\chi B^{-1} (\chi \cl (d\phi_{\pa}) \chi - \cl(d\phi_{\pa})) B^{-2} \chi$,
given that $\chi \cl (d\phi) \chi=\chi \cl (d\phi_{\pa}) \chi$. We can rewrite this last term
as 
$$\chi B^{-1} (\chi \cl (d\phi_{\pa}) \chi - (\chi+(1-\chi))\cl(d\phi_{\pa})(\chi+(1-\chi))) B^{-2} \chi$$
which is in turn equal to 
$$-\chi B^{-1} (1-\chi)\cl(d\phi_{\pa})(1-\chi) B^{-2} \chi - \chi B^{-1} \chi \cl (d\phi_{\pa}) (1-\chi) B^{-2} \chi-  \chi B^{-1}  (1-\chi)\cl(d\phi_{\pa})\chi B^{-2} \chi\,.$$
The last two summands are $\sim_m$-equivalent to 0 because $\chi (1-\chi)$ has compact support.
Regarding the term $\chi B^{-1} (1-\chi)\cl(d\phi_{\pa})(1-\chi) B^{-2} \chi$; we rewrite it as
$$\chi  (1-\chi) B^{-1}\cl(d\phi_{\pa})(1-\chi) B^{-2} \chi + \chi [B^{-1}, (1-\chi)]\cl(d\phi_{\pa})(1-\chi) B^{-2} \chi$$ and this is certainly $\sim_m$-equivalent to
$\chi [B^{-1}, (1-\chi)]\cl(d\phi_{\pa})(1-\chi) B^{-2} \chi$. The latter term is in turn equal, up to a sign,
to
$$\chi B^{-1} \cl (d\chi) B^{-1} \cl(d\phi_{\pa})(1-\chi) B^{-2} \chi$$
which is $\sim_m$-equivalent to 0 ($d\chi$ is compactly supported).
Thus \eqref{firstdifference} is in $\I_m (X,\F)$; similarly one proves
that the second difference \eqref{seconddifference}, viz.
$(\mathfrak{s}+D)^{-2}  \cl (d\phi)
(\mathfrak{s}+D)^{-1}-
 \chi ((\mathfrak{s}+D_{\cyl})^{-2}  \cl (d\phi_{\pa})
(\mathfrak{s}+D_{\cyl})^{-1})\chi$ is  in $\I_m (X,\F)$.
Now, by direct inspection we also see that both the first difference \eqref{firstdifference} 
and the second difference  \eqref{seconddifference} are in $\J_m$, i.e. they are bounded
if they are multiplied on the right and on the left by the function $g$.
Summarizing: we have proved that  $S_+$ is in $\J_m (X,\F)\cap {\rm Dom}\overline{\delta}_2$.
Similarly one proves that $S_+\in  {\rm Dom}\overline{\delta}_1$, proving finally that
$$S_+ \in \J_m (X,\F)\cap{\rm Dom}\overline{\delta}_1\cap {\rm Dom}\overline{\delta}_2\equiv
\mathbf{\mathfrak{J}_m}\,,\quad m>\dim \tV$$
The reasoning for the other entries in the Connes-Skandalis projection is analogous
and hence omitted.
The proof of Proposition \ref{prop:smooth-cylinder-1} is now complete.

\subsubsection{Proof of Proposition \ref{prop:smooth-cylinder-2}}
We shall first concentrate on the 
larger algebra $\mathcal{B}_{m}$; thus we begin by establishing Proposition  \ref{prop:smooth-cylinder-2}.
in this context, namely, we prove  that $e_{(D^{\cyl})}\in  \mathcal{B}_{m}\oplus \CC e_1$ with $m$ greater 
than $2n$, which is the dimension of the leaves of $(X,\F)$.
\begin{lemma}\label{lemma:commutator-in-b}
For the translation invariant Dirac family $D^{\cyl}= (D^{\cyl}_\theta)_{\theta\in T}$
on the cylinder we have:
\begin{equation}\label{lemma:crucial}
[\chi^0, (D^{\cyl}+\mathfrak{s})^{-1}]\in \mathcal{I}_{m}(\cyl(\pa X), \F_{\cyl})
\end{equation}
with $\chi^0$ denoting as usual the function induced on the cylinder by
$\chi^0_\RR$, the
characteristic function of $(-\infty,0]$, and $\mathfrak{s}$ the grading operator.
\end{lemma}

\begin{proof}
We shall prove  that $[\chi^0, (D^{\cyl}+\mathfrak{s})^{-1}]$ has finite Shatten
$m$-norm. \
We shall denote by $t$ the variable along the $\RR$-factor in the cylinder; we shall omit the vector bundles from the notation. 
First we observe that  $\chi^0$ is bounded and only depends on the cylindrical variable. 
Observe next that $[D^{\cyl},\chi^0 ]$ defines a family of bounded operators from
 $W^1 ( (\pa\tM\times\{\theta\})\times\RR)\to W^{-1} ( (\pa \tM\times\{\theta\})\times\RR)$
 and the same is true for $[D^{\pa},\chi^0]$; it is then elementary to check that 
 $[D^{\pa}, \chi^0]=0$, as an operator from $W^1$ to $W^{-1}$. 
 Similarly, the operator $[\pa_t,\chi^0 ]$
induces  bounded maps $W^1 ( (\pa \tM\times\{\theta\})\times\RR)\to W^{-1} ( (\pa \tM\times\{\theta\})\times\RR)$.
 We then have the following equality of bounded operators:
 \begin{align*}
[\chi^0,(D^{\cyl}+\mathfrak{s})^{-1}]&= (D^{\cyl}+\mathfrak{s})^{-1} \,[D^{\cyl},\chi^0 ]\,(D^{\cyl}+\mathfrak{s})^{-1}\\
&=(D^{\cyl}+\mathfrak{s})^{-1} \,[\pa_t,\chi^0 ]\,(D^{\cyl}+\mathfrak{s})^{-1}\end{align*}
Thus we can write
\begin{equation*}
[\chi^0,(D^{\cyl}+\mathfrak{s})^{-1}]=
(D^{\cyl}+\mathfrak{s}) \,(I+(D^{\cyl} )^2)^{-1} \,[\pa_t,\chi^0 ]\,(I+(D^{\cyl} )^2)^{-1}\,(D^{\cyl}+\mathfrak{s})
\end{equation*}
This means that it suffices to prove that $(I+(D^{\cyl} )^2)^{-1/2} \,[\pa_t,\chi^0 ]\,(I+(D^{\cyl} )^2)^{-1/2}\in \mathcal{J}_{m}(\cyl(\pa X), \F_{\cyl})$.
We conjugate this operator with Fourier transform and obtain the operator
	$$T:=(I+t^2+(D^\partial)^2)^{-1/2} [\mathcal{H},it] (I+t^2+(D^\partial)^2)^{-1/2} $$
	with $\mathcal{H}$ denoting the Hilbert transform on $L^2 (\RR)$.\\
				 Note that $ [\mathcal{H},t]\xi (t)=i/\pi\, \int_{\RR}  \xi (s)ds$.   Thus,                
	$T=(T_\theta)_{\theta\in T}$ and each $T_\theta$ is the composite    $\iota_\theta \circ \pi_\theta$, with
	$$\iota_\theta: L^2 (\partial \tM\ \times \{\theta\}) \to L^2 ((\partial\tM \times \{\theta\})\times\RR)\,,\;\; 
	\iota_\theta (\eta)= (I+t^2+(D^\partial_\theta)^2)^{-1/2}\eta$$
	$$\pi_\theta:  L^2 ((\partial\tM \times \{\theta\})\times\RR)\to L^2 (\partial \tM\ \times \{\theta\})\,, \;\;
	\pi_\theta (\xi) (y) = \int_{\RR} (I+t^2+(D^\partial_\theta)^2)^{-1/2} \xi (y,t) dt$$ 
	where, as before,  we are omitting the vector bundles from the notation. Thus
	$$T^{m}_\theta = \iota_\theta \circ (\pi_\theta \circ \iota_\theta)^{m-1}  \circ\pi_\theta\,.$$
	On the other hand, 
	\begin{align*}\pi_\theta \circ \iota_\theta (\eta) &= \int_{\RR} dt (I+t^2+  (D^\partial_\theta)^2)^{-1/2} (I+t^2+  (D^\partial_\theta)^2)^{-1/2}
	\eta(y) \\
	&= \int_{\RR} dt (I+t^2+  (D^\partial_\theta)^2)^{-1} \eta(y)\\
	&= C (I+ (D^\partial_\theta)^2)^{-1/2} \eta(y) \,, \text{ with } C=\pi
	\end{align*}
	The last step can be justified as follows. Observe that $\forall a>0$ and $k\geq 0$
	$$\int_{\RR} \frac{dt}{(t^2+a^2)^{k+1}}= \frac{\pi}{2^{2k}} \frac{(2k)!}{(k!)^2}a^{-2k-1}\,.
	$$
	Using this we can show that, in the strong topology, 
	$$\int_{\RR} (I+t^2+(D^\partial_\theta)^2)^{-\frac{p+1}{2}}= \frac{\pi}{2^{p-1}}\frac{(p-1)!}{(\frac{p-1}{2} !)^2} (1+(D^\partial_\theta)^2)
	^{-\frac{p}{2}}\,,$$ where $p=2k+1$. Thus, for $p=1$ we have, in the strong topology,
	$$\int_{\RR}  (I+t^2+(D^\partial_\theta)^2)^{-1}=C (I+(D^\partial_\theta)^2)^{-\frac{1}{2}}\,, \text{ with } C=\pi\,,$$
	which is what we had to justify. We thus obtain:
	$T^{m}_\theta \xi = C^p (\iota_\theta \circ  (I+(D^\partial_\theta)^2)^{-\frac{m-1}{2}} \circ \pi_\theta \xi)$ and we are left with the task of proving
	that $\pi_\theta$ and $\iota_\theta$ are bounded on $L^2$ (indeed, it is well known, see \cite{MN}, that
	 $(I+(D^\partial_\theta)^2)^{-\frac{1}{2}}$ has  finite Shatten $(m-1)$-norm for $m-1>\dim \partial{\tM}$, which is the case here
	 since $m>\dim \tM$).


	\begin{sublemma}
	The map $\iota_\theta:  L^2 (\partial \tM\ \times \{\theta\}) \to L^2 ((\partial\tM \times \{\theta\})\times\RR)$,	$\iota_\theta (\eta)= (I+t^2+(D^\partial_\theta)^2)^{-1/2}\eta$
	 is bounded.
\end{sublemma}
\begin{proof}
We set $\Delta_\theta:= (D^\partial_\theta)^2$ and compute:
\begin{align*}
\|\iota_\theta (\eta)\|^2_{L^2_\theta} &= \int_{(\partial\tM \times \{\theta\})\times\RR}dy dt | (I+t^2+\Delta_\theta)^{-\frac{1}{2}}\eta (y) |^2\\
&= \int_{\RR} dt \|(I+t^2 + \Delta_\theta)^{-\frac{1}{2}} \eta \|^2_{L^2 (\pa \tM\times\{\theta\})}\\
&\leq \|\eta\|_{L^2 (\pa \tM\times\{\theta\})}\int_{\RR} dt \|(I+t^2+ \Delta_\theta)^{-\frac{1}{2}}\|^2
\end{align*}
where $ \|(I+t^2+ \Delta_\theta)^{-\frac{1}{2}}\|^2$ is the operator norm and it is considered as a function of $t$.
We are left with the task of proving that $ \|(I+t^2+ \Delta_\theta)^{-\frac{1}{2}}\|$ is in $L^2 (\RR_t)$, namely
\begin{equation}\label{subsublemma}
\int_{\RR} \|(I+t^2+ \Delta_\theta)^{-\frac{1}{2}}\|^2 \,< \infty.
\end{equation}
In order to establish \eqref{subsublemma} we write 
$$\|(I+t^2+ \Delta_\theta)^{-\frac{1}{2}}\|= (1+t^2)^{-\frac{1}{2}}\| (I+\frac{\Delta_\theta}{1+t^2})^{-\frac{1}{2}} \|=
(1+t^2)^{-\frac{1}{2}}\| f(D^\partial_\theta)\|$$
with $f(x):+ (1+\frac{x^2}{1+t^2})^{-\frac{1}{2}}$. Now, the sup-norm of $f(x)$ is equal to 1: thus
$\| f(D^\partial_\theta)\|\leq 1$ from which \eqref{subsublemma} follows.
\end{proof}

\begin{sublemma}
	The map $\pi_\theta: L^2 ((\partial\tM \times \{\theta\})\times\RR)\to   L^2 (\partial \tM \times \{\theta\}) $,	
	$\pi_\theta (\xi) (y) = \int_{\RR} (I+t^2+(D^\partial_\theta)^2)^{-1/2} \xi (y,t) dt$ is bounded.
\end{sublemma}

\begin{proof}
We can consider  a decomposable element $\xi(y,t)=\eta(y) f(t)$, with $\eta\in L^2 (\partial \tM\times \{\theta\})$
and $f(t)\in L^2 (\RR_t)$. Then
\begin{align*}
\|\pi(\xi)\|_{L^2(\partial \tM \times \{\theta\})}&= \int_{\partial \tM \times \{\theta\}} dy |\int_{\RR} dt 
(I+t^2+(D^\partial_\theta)^2)^{-1/2} \eta(y) f(t) |^2\\
&\leq  \int_{\partial \tM \times \{\theta\}} dy  \left( \int_{\RR} dt | (I+t^2+(D^\partial_\theta)^2)^{-1/2} \eta(y) f(t) |^2 \right)^2\\
&\leq \int_{\partial \tM \times \{\theta\}} dy\left( \int_{\RR} dt |f(t)|^2 \cdot \int_{\RR} dt |(1+t^2+(D^\partial_\theta)^2)^{-1/2}
\eta(y) |^2 \right)\\
&= \|f\|^2_{L^2(\RR)} \int_{(\partial\tM \times \{\theta\})\times\RR} dy\,dt | (I+t^2+ (D^\partial_\theta)^2)^{-1/2}
\eta(y) |^2 \\
&=\|f\|^2_{L^2(\RR)} \int_{\RR} dt \|(I+t^2+ (D^\partial_\theta)^2)^{-1/2}
\eta(y)\|_{L^2 (\partial \tM \times \{\theta\} )}^2 \\
&\leq \|f\|^2_{L^2(\RR)}  \|\eta\|^2_{L^2 (\partial \tM \times \{\theta\} )}\,\int_{\RR}dt \|(I+t^2+ \Delta_\theta)^{-\frac{1}{2}}\|^2\\
&\leq \|\xi\|^2_{L^2 ((\partial \tM \times \{\theta\} )\times\RR)}\int_{\RR}dt \|(I+t^2+ \Delta_\theta)^{-\frac{1}{2}}\|^2
\end{align*}
where in the last term we are taking the operator norm considered as a function of $t$. Using  \eqref{subsublemma} 
we finish the proof.
\end{proof}
The proof of Lemma \ref{lemma:commutator-in-b} is now complete.

\end{proof}
Going back to the proof of  Proposition  \ref{prop:smooth-cylinder-2}, we  observe that
 $\widehat{e}_{(D^{\cyl})}\in  \operatorname{OP}^{-1}$ by the results of \cite{MN}. Thus, using the Lemma
we have  just proved,  we conclude
 that
$$\widehat{e}_{(D^{\cyl})}\in  \operatorname{OP}^{-1}\quad\text{and}\quad [\chi^0,\widehat{e}_{(D^{\cyl})}]\in \mathcal{I}_{m}$$ with $m$ greater than the dimension of the leaves of $(X,\F)$.
Now we need to prove that, in fact,  $[\chi^0,\widehat{e}_{(D^{\cyl})}]$ is in $\J_m$, that is
 $g_{\cyl}[\chi^0,\widehat{e}_{(D^{\cyl})}]$ and $[\chi^0,\widehat{e}_{(D^{\cyl})}]g_{\cyl}$
are bounded; this will ensure that $\widehat{e}_{(D^{\cyl})}\in \mathcal{D}_{m}$.

\begin{lemma} 
The operators $g_{\cyl}[\chi^0,\widehat{e}_{(D^{\cyl})}]$ and $[\chi^0,\widehat{e}_{(D^{\cyl})}]g_{\cyl}$
are bounded.
\end{lemma}

\begin{proof}
Consider  $g_{\cyl}[\chi^0,\widehat{e}_{(D^{\cyl})}]$; this is equal to $$[g_{\cyl}, (D^{\cyl}+\mathfrak{s})^{-1}]
[\pa_t,\chi^0] (D^{\cyl}+\mathfrak{s})^{-1}+ (D^{\cyl}+\mathfrak{s})^{-1} g_{\cyl} [\pa_t,\chi^0] (D^{\cyl}+\mathfrak{s})^{-1}$$
which is in turn equal to 
$$ (D^{\cyl}+\mathfrak{s})^{-1} 2f_{\cyl}\cl (df_{\cyl})  (D^{\cyl}+\mathfrak{s})^{-1}[\pa_t,\chi^0] (D^{\cyl}+\mathfrak{s})^{-1}
+ (D^{\cyl}+\mathfrak{s})^{-1} g_{\cyl} [\pa_t,\chi^0] (D^{\cyl}+\mathfrak{s})^{-1}$$
\begin{sublemma}\label{the-sublemma}
The operator $(D^{\cyl}+\mathfrak{s})^{-1} f_{\cyl}$ and $ f_{\cyl}(D^{\cyl}+\mathfrak{s})^{-1}$ are bounded.
\end{sublemma}
Assuming the sublemma for a moment we see that in the last displayed formula the term
$ (D^{\cyl}+\mathfrak{s})^{-1} 2f_{\cyl}$ appearing in the first summand
is bounded; thus so is
$ (D^{\cyl}+\mathfrak{s})^{-1} 2f_{\cyl}\cl (df_{\cyl}) $
since $\cl (df_{\cyl})$ is itself bounded (here we use the very definition of $f_{\cyl}$);   the term  $(D^{\cyl}+\mathfrak{s})^{-1}[\pa_t,\chi^0] (D^{\cyl}+\mathfrak{s})^{-1}$
is known to be bounded (just see the proof of Lemma \ref{lemma:commutator-in-b}); next, 
we consider the term $(D^{\cyl}+\mathfrak{s})^{-1} g_{\cyl} [\pa_t,\chi^0] (D^{\cyl}+\mathfrak{s})^{-1}$.
We shall prove that with  with $\Lambda:= (1+(D^{\cyl})^2)^{-\frac{1}{2}}$
$$ \Lambda g_{\cyl} [\pa_t,\chi^0] \Lambda :=T_g=T:=
\Lambda  [\pa_t,\chi^0]  \Lambda $$
and since the latter is bounded by Lemma \ref{lemma:commutator-in-b}, we will be able to conclude that
 $(D^{\cyl}+\mathfrak{s})^{-1} g_{\cyl} [\pa_t,\chi^0] (D^{\cyl}+\mathfrak{s})^{-1}$ is bounded and that
the Lemma holds.\\
For $\xi$, $\eta$ in $L^2$ we have:
\begin{align*}
\langle T_{\theta}\xi,\eta \rangle_{L^2} &= \int_{\RR\times\tN} dtdy\left( \Lambda_{\theta} g_{\cyl} [\pa_t ,\chi^0]
  \Lambda_{\theta} \xi\right) (t,y)\overline{\eta} (t,y)\\
 &= \int_{\tN} dy \left( \int_{-\infty}^0 dt (  \Lambda_{\theta} \xi)(t,y) \pa_t (g  \Lambda_{\theta} \overline{\eta})(t,y)
 -\int_{-\infty}^0 dt (\pa_t  \Lambda_{\theta} \xi)(t,y) (g \Lambda_{\theta}\overline{\eta})(t,y) \right)\\
 & = - \int_{\tN} \big[ ( \Lambda_{\theta} \xi )(t,y) g(t) ( \Lambda_{\theta} \overline{\eta})(t,y)\big]_{t=0}\\
 &= \int_{\tN}  ( \Lambda_{\theta} \xi )(0,y) ( \Lambda_{\theta} \overline{\eta})(0,y)\quad\text{since}\quad g(0)=1\\
 &= \langle (1+(D_{\theta}^{\cyl})^2)^{-\frac{1}{2}} [\pa_t,\chi^0]  (1+(D_{\theta}^{\cyl})^2)^{-\frac{1}{2}} \xi,\eta \rangle_{L^2}
 \end{align*}
 where for the last equality we use again the computation done in the preceding four equalities.\\
 We are left with the task of proving the Sublemma. To this end we observe that, with $\pa_t:=\frac{1}{i}\frac{d}{dt}$
 we have  $(D^{\cyl})^2=
\pa_t^2 +D^2_{\pa X}$; we also know that $\pa_t^2$ and $D^2_{\pa X}$ commute. 
It is easy to see that the (unique) self-adjoint extensions of $(1+\pa^2_t)$ and $(1+D^2_{\pa X})$
are invertible and that the following two equalities hold:
$$ (1+\pa_t^2)^{-1} - (1+(D^{\cyl})^2)^{-1}=(1+\pa_t^2)^{-1} ( (1+(D^{\cyl})^2) -  (1+\pa_t^2) )  
(1+(D^{\cyl})^2)^{-1} 
= (1+\pa_t^2)^{-1} D^2_{\pa X} (1+(D^{\cyl})^2)^{-1} \,.$$
Moreover the last operator is non-negative; thus $ (1+\pa_t^2)^{-1} \geq (1+(D^{\cyl})^2)^{-1}$
and thus $ (1+\pa_t^2)^{-1/2} \geq (1+(D^{\cyl})^2))^{-1/2}$.
Then with $f_{\cyl}=\sqrt{1+t^2}$ as usual, we have
$\| (1+(D^{\cyl})^2)^{-1/2} f_{\cyl} \xi\|_{L^2}\leq \| (1+\pa_t^2)^{-1/2}  f_{\cyl} \xi\|$ with $\xi\in\C^\infty_c$. This means that
if $(1+\pa_t^2)^{-1/2}  f_{\cyl}$ is bounded, then $ (1+(D^{\cyl})^2)^{-1/2} f_{\cyl}$ is also bounded.
Now remark that $(1+\pa_t^2)^{-1/2}  f_{\cyl}$ has Schwartz kernel equal to 
$k(s,t)= e^{-|s-t|} f(s)\equiv e^{-|s-t|} \sqrt{1+s^2}$ and that this is an $L^2$-function on $\RR\times \RR$.
Thus the integral operator defined on $L^2 (\RR)$ by $k(s,t)$ is Hilbert-Schmidt and thus, in particular,
bounded. This implies that our operator, which is really $(1+\pa_t^2)^{-1/2}  f_{\cyl}\otimes {\rm Id}_{\pa X}$
is also bounded. Summarizing, $(1+(D^{\cyl})^2)^{-1/2} f_{\cyl}$ is bounded. Thus $(D^{\cyl}+\mathfrak{s})^{-1} f_{\cyl}$
is also bounded, since it can be written as $(D^{\cyl}+\mathfrak{s})^{-1} (1+(D^{\cyl})^2)^{1/2} (1+(D^{\cyl})^2)^{-1/2}f_{\cyl}$.
Finally notice that $f_{\cyl}(D^{\cyl}+\mathfrak{s})^{-1} = [f_{\cyl}, (D^{\cyl}+\mathfrak{s})^{-1}] + (D^{\cyl}+\mathfrak{s})^{-1} f_{\cyl}$
and we know that both summands on the right hand sides are bounded. The Sublemma (and thus the Lemma) is proved.

\end{proof}

Thus we have proved  that 
$\widehat{e}_{(D^{\cyl})}\in \mathcal{D}_{m}$. On the other hand, see Definition 
\ref{def:Bp-bis}, 
$$\B_m:= \{\ell\in \D_m\cap {\rm Dom} (\pa_\alpha)\;|\; [f_{\cyl},\ell]\,,\, [f_{\cyl},[f_{\cyl},\ell]]\;\;\text{ is bounded }\}.$$ 
Thus we
need to show, first of all, that it is also true that $\widehat{e}_{(D^{\cyl})}\in {\rm Dom} (\pa_\alpha)$. We need to prove that the limit
$$\lim_{t\to 0} \frac{1}{t} (\alpha_t ( \widehat{e}_{(D^{\cyl})})- \widehat{e}_{(D^{\cyl})})$$
exists in $\operatorname{OP}^{-1}$.
We compute
\begin{align*}
 &\frac{1}{t} (\alpha_t ( \widehat{e}_{(D^{\cyl})})- \widehat{e}_{(D^{\cyl})})= \frac{1}{t}
 [e^{its},(D^{\cyl}+\mathfrak{s})^{-1}] e^{-its}
 =\frac{1}{t} (D^{\cyl}+\mathfrak{s})^{-1} [D^{\cyl},e^{its}] (D^{\cyl}+\mathfrak{s})^{-1} e^{-its}\\
 =&  (D^{\cyl}+\mathfrak{s})^{-1} i\cl (ds) e^{its} (D^{\cyl}+\mathfrak{s})^{-1} e^{-its}\\
 =&  (D^{\cyl}+\mathfrak{s})^{-1} i\cl (ds) [e^{its}, (D^{\cyl}+\mathfrak{s})^{-1}] e^{-its}
 +  (D^{\cyl}+\mathfrak{s})^{-1} i\cl (ds)  (D^{\cyl}+\mathfrak{s})^{-1} \\
 =& (D^{\cyl}+\mathfrak{s})^{-1} i\cl (ds)(D^{\cyl}+\mathfrak{s})^{-1} [D^{\cyl}, e^{its}](D^{\cyl}+\mathfrak{s})^{-1} e^{-its}
 +  (D^{\cyl}+\mathfrak{s})^{-1} i\cl (ds)  (D^{\cyl}+\mathfrak{s})^{-1} \\
 =& (D^{\cyl}+\mathfrak{s})^{-1}i\cl (ds)(D^{\cyl}+\mathfrak{s})^{-1} i\cl(ds) it e^{its}
(D^{\cyl}+\mathfrak{s})^{-1} e^{-its}+  (D^{\cyl}+\mathfrak{s})^{-1} i\cl (ds)  (D^{\cyl}+\mathfrak{s})^{-1} 
\end{align*}
and the last term converges to $(D^{\cyl}+\mathfrak{s})^{-1} i\cl (ds)  (D^{\cyl}+\mathfrak{s})^{-1} $
as $t\to 0$. Thus $\widehat{e}_{(D^{\cyl})}\in {\rm Dom} (\pa_\alpha)$. Summarizing, we have proved that
$\widehat{e}_{(D^{\cyl})}\in \mathcal{D}_{m,\alpha}:= \D_m\cap {\rm Dom} (\pa_\alpha)$.\\
Next we need to show that $ [f_{\cyl},\widehat{e}_{(D^{\cyl}})]$ and $ [f_{\cyl},[f_{\cyl},\widehat{e}_{(D^{\cyl}})]]$ are bounded. However, this is elementary at this point: for example 
$ [f_{\cyl},\widehat{e}_{(D^{\cyl}})]$ is nothing but  $(D^{\cyl}+\mathfrak{s})^{-1} i\cl (df_{\cyl})  (D^{\cyl}+\mathfrak{s})^{-1} $ which is indeed bounded. Similarly we proceed for $ [f_{\cyl},[f_{\cyl},\widehat{e}_{(D^{\cyl}})]]$.

\medskip
Next we prove that $\widehat{e}_{(D^{\cyl})}\equiv (D^{\cyl}+\mathfrak{s})^{-1}\in {\rm Dom} (\overline{\delta}_j)$, $j=1,2$. 
We only do it for $\overline{\delta}_2$, the arguments for $ \overline{\delta}_1$
are similar.
It suffices to show that
$(D^{\cyl}+\mathfrak{s})^{-1}\in {\rm Dom}(\overline{\partial}_2)$
and $\overline{\partial}_2 ((D^{\cyl}+\mathfrak{s})^{-1})\in \B_{m}$. Using \cite{MN}
Proposition 7.17,
we know that
$(D^{\cyl}+\mathfrak{s})^{-1}$ does belong to ${\rm Dom}(\overline{\partial}_2)$ and moreover
$$\overline{\partial}_2 (D^{\cyl}+\mathfrak{s})^{-1}\,=\, (D^{\cyl}+\mathfrak{s})^{-1} [D^{\cyl},\phi_{\partial}]
(D^{\cyl}+\mathfrak{s})^{-1}\,.$$
In order to see that the right hand side of this formula belongs to $\mathcal{B}_{m}$ we 
show separately that it belongs to $\D_m$ and ${\rm Dom}(\pa_{\alpha})$ and that, in addition,
it is such that its commutator and its double-commutator with $f_{\cyl}$ is bounded.
First of all
we employ
Lemma \ref{lemma:commutator-in-b}, which shows that 
$(D^{\cyl}+\mathfrak{s})^{-1}\in \mathcal{D}_{m}={\rm Dom}(\overline{\delta}_3)$.
Since, by the same arguments, $[D^{\cyl},\phi_{\partial}](D^{\cyl}+\mathfrak{s})^{-1}$, which is 
the composition of Clifford multiplication by $d\phi_{\partial}$  with $(D^{\cyl}+\mathfrak{s})^{-1}$,
also belongs to ${\rm Dom}(\overline{\delta}_3)$
 we conclude that their product is in  ${\rm Dom}(\overline{\delta}_3)$
 i.e. in $\mathcal{D}_{m}$. Here we have used the fact that the domain of a closed derivation
 is a Banach algebra. Exactly the same argument, together with the above proof that 
 $(D^{\cyl}+\mathfrak{s})^{-1}$ belongs to ${\rm Dom}(\pa_{\alpha})$, establishes that
$$(D^{\cyl}+\mathfrak{s})^{-1} [D^{\cyl},\phi_{\partial}]
(D^{\cyl}+\mathfrak{s})^{-1}\in {\rm Dom}(\pa_{\alpha})\,.$$ Finally 
it is clear that the above term is such that its commutator with $f_{\cyl}$ is bounded; similarly
we proceed with its double-commutator.
This completes the proof of the Proposition \ref{prop:smooth-cylinder-2} .



\subsubsection{Proof of Proposition  \ref{prop:smooth-cylinder-3}}
We need to show that $\widehat{e}_{D}\in \mathbf{ \mathbf{ \mathfrak{A}_m } } := \mathcal{A}_m \cap
{\rm Dom} (\overline{\delta}_1)
\cap {\rm Dom} (\overline{\delta}_2) \cap \pi^{-1} (\mathbf{ \mathbf{ \mathfrak{B}_m } } )\,,$ for $m>\dim \tV$.
First of all we prove that $\widehat{e}_{D}\in \mathcal{A}_m$.  Thus we need to show that $t(\widehat{e}_{D})\in \J_{m}$ and that $\pi (\widehat{e}_{D})\in \B_m$.
However,
this is clear from our previous results. Indeed, $\pi (\widehat{e}_{D})=\widehat{e}_{D_{\cyl}}$ and we 
have proved in the previous proposition that the right hand side is in $\mathcal{B}_m$. Similarly,
if $\chi$ is a smooth approximation
of $\chi^0$, we can  write:
\begin{equation}\label{trick-chi}
t(\widehat{e}_{D}):=\widehat{e}_{D}- \chi^0 \widehat{e}_{D_{\cyl}} \chi^0=(\widehat{e}_{D}- \chi \widehat{e}_{D_{\cyl}} \chi)+ (\chi \widehat{e}_{D_{\cyl}} \chi - 
\chi^0 \widehat{e}_{D_{\cyl}} \chi^0)\,.
\end{equation}
We have already proved that the first summand $\widehat{e}_{D}- \chi \widehat{e}_{D_{\cyl}} \chi$ is in $\I_{m}$.
We now proceed to show that it is indeed in $\J_m$. By the argument given in Remark 
\ref{remark:from-compact-to-shatten} we need to show that $\phi (D+\mathfrak{s})^{-1}$
and $ (D+\mathfrak{s})^{-1}\phi$ are not only in $\I_m$ but in fact in $\J_m$
provided that $\phi$ is compactly supported. However, this is can be proved as follows.
First, $g\phi  (D+\mathfrak{s})^{-1}$ is clearly bounded, given that it is in $\I_m$ (indeed
$g\phi$ is compactly supported, so we can apply Remark \ref{remark:from-compact-to-shatten}).
As far as $\phi  (D+\mathfrak{s})^{-1} g$, we rewrite it as $\phi [(D+\mathfrak{s})^{-1},g] + \phi g (D+\mathfrak{s})^{-1}$. The latter is equal to $- \phi (D+\mathfrak{s})^{-1} 2 \cl (df) f (D+\mathfrak{s})^{-1} + \phi g (D+\mathfrak{s})^{-1}$. This is bounded using the above reasoning and Sublemma \ref{the-sublemma}. Summarizing,
we have proved that $\widehat{e}_{D}- \chi \widehat{e}_{D_{\cyl}} \chi\in\J_m$.\\
As far as the second term in \eqref{trick-chi} is concerned  we simply observe that it can be rewritten
as $(\chi \widehat{e}_{D} (\chi-\chi^0))+ (\chi-\chi^0) \widehat{e}_{D_{\cyl}} \chi^0$ and both these terms
are $m$-Shatten class if $m>\dim \tV$, given that $(\chi-\chi^0)$ is compactly
supported. This trick can be also used in order to show that
if we multiply by $g$ on the left and on the right we get a bounded operator, according to what
we have obseved above.
Notice that since $\pi (\widehat{e}_{D})=\widehat{e}_{D_{\cyl}}$ we also have, directly from the previous Subsubsection,
 that $\widehat{e}_{D}\in \pi^{-1} (\mathbf{ \mathbf{ \mathfrak{B}_m } } )$.
Next we need to show that $\widehat{e}_{D}$ is in the domain of the two closed derivations
introduced in Subsubsection \ref{subsubsection:gothic-a}. Consider, for example,
$\overline{\delta}_2$.
Recall that $\overline{\delta}_2: {\rm Dom}\overline{\delta}_2 \to \A_m (X,F)$
with $${\rm Dom}\overline{\delta}_2=\{a\in {\rm Dom}\,\overline{\delta}^{{\rm max}}_2\;|\;\overline{\delta}^{{\rm max}}_2 a\in 
\A_m (X,F)\}$$ 
The fact 
that $\widehat{e}_D\equiv (\mathfrak{s}+D)^{-1}$ is in ${\rm Dom}\,\overline{\delta}^{{\rm max}}_2$ is proved 
 in \cite{MN} where it is also proved that 
\begin{equation}\label{usualformula}
\overline{\delta}^{{\rm max}}_2 ((\mathfrak{s}+D)^{-1})=
(\mathfrak{s}+D)^{-1} [D,\phi] (\mathfrak{s}+D)^{-1}\,.
\end{equation}
Thus we only need to show that the right hand side belongs to $\A_m (X,F)$, where we recall
that $\A_m (X,F)=\{a\in A^* \text{ such that } \pi(a)\in \B_m (\cyl(\pa X),F_{\cyl})\text{ and }
t(a)\in \J_m (X,F)\}$. But the image under $\pi$ of the right hand side of \eqref{usualformula}
is precisely $(D^{\cyl}+\mathfrak{s})^{-1} [D^{\cyl},\phi_{\partial}]
(D^{\cyl}+\mathfrak{s})^{-1}$ which was shown to belong to $\B_m$
at the end of the proof of the previous Proposition. Thus we are left with the task
of proving that $$t((\mathfrak{s}+D)^{-1} [D,\phi] (\mathfrak{s}+D)^{-1} )\in \J_m (X,F)\,.$$
By definition of $t$ this means that
$$(\mathfrak{s}+D)^{-1} [D,\phi] (\mathfrak{s}+D)^{-1} - \chi^0
(D^{\cyl}+\mathfrak{s})^{-1} [D^{\cyl},\phi_{\partial}]
(D^{\cyl}+\mathfrak{s})^{-1} \chi^0 \;\in\;\J_m (X,F)\,.$$
However, this can be proved by first reducing to $\chi$, a smooth approximation
of $\chi^0$, using the same reasoning as in \eqref{trick-chi}; then,
in order to show that  
$$(\mathfrak{s}+D)^{-1} [D,\phi] (\mathfrak{s}+D)^{-1} - \chi
(D^{\cyl}+\mathfrak{s})^{-1} [D^{\cyl},\phi_{\partial}]
(D^{\cyl}+\mathfrak{s})^{-1} \chi \;\in\;\J_m (X,F)$$
we employ the same arguments 
used in order to
establish that \eqref{firstdifference} is in $\J_m (X,F)$. The proof of  Proposition  \ref{prop:smooth-cylinder-3} is now complete.

\subsection{Proofs for the extension of the 
 relative GV cyclic cocycle}\label{subsection:proofs-extension}

\subsubsection{Further properties of the modular Shatten extension}\label{subsect:further}

This Subsection is devoted to the statement of some technical properties
of the algebras $ \mathbf{\mathfrak{B}}_m $.
In the next Subsection we shall use these properties in order to show that the  Godbillon-Vey eta cocycle extends 
from $B_c$ to $ \mathbf{\mathfrak{B}}_{2n+1} $ with $2n$ equal to the dimension
of the leaves.

We begin with the Banach algebra $\operatorname{OP}^{-1} (\cyl (Y),\F_{\cyl})$.
Here $Y=\tN\times_\Gamma T$ is a foliated
 bundle without boundary.
 
 We have proved in Proposition \ref{prop:b-o-sub-of-b-star} that $\operatorname{OP}^{-1} (\cyl (Y),\F_{\cyl})$ is a subalgebra
 of $B^*$. The latter was proved to be isomorphic to $C^* (\RR)\otimes C^* (Y,\F)$, see
 Proposition \ref{prop:bstar-structure}. Thus, the Fourier transformation
 applied to $B^*$  has values in
 $C_0 (\RR, C^* (Y,\F))$.  
 In particular, the image of $\operatorname{OP}^{-1} (\cyl (Y),\F_{\cyl})$ under Fourier transformation
 is a subalgebra of $C_0 (\RR, C^* (Y,\F))$. Summarizing, the Fourier transform
 $\hat{\ell}$ of an element $\ell\in \operatorname{OP}^{-1} (\cyl (Y),\F_{\cyl})$ is a continuous
 function $\hat{\ell}:\RR\to C^* (Y,\F)$ vanishing at infinity.
 
 Recall also the Banach algebras $\operatorname{OP}^{-p} (\cyl (Y),\F)$,
 see Proposition \ref{prop:b-o-sub-of-b-star}. 
 Similarly, we can define $\operatorname{OP}^{-p} (Y, \F)$ for a closed foliated bundle $(Y,\F)$.
 Thinking to  $\operatorname{OP}^{-p} (Y,\F)$
 as a subalgebra of the von Neumann algebra  of the foliation $(Y,\F)$, we
 see that an element $b$ in $\operatorname{OP}^{-p}(Y,\F)$
 is given by a family of operators $(b_\theta)_{\theta\in T}$, with $b_\theta$ acting on Sobolev
  spaces along $N\times\{\theta\}$. Let $f:\RR\to \operatorname{OP}^{-p} (Y, \F)$ be a measurable 
  $\operatorname{OP}^{-p} (Y, \F)$-valued function. This means that $f=(f_\theta)_{\theta\in T}$
  with $f_\theta$ a measurable function with values in the bounded operators of
  order $-p$ on Sobolev spaces on $N\times\{\theta\}$. We define  a norm 
  $\| f\|_{L^2 (\RR,\operatorname{OP}^{-p}(Y,\F))}$ by setting
  \begin{equation}\label{normell2}
  \| f\|_{L^2 (\RR,\operatorname{OP}^{-p}(Y,\F))}=\sup_{\theta\in T} \left( \int_{\RR} ||| f_\theta (x) |||^2_p dx
  \right) ^{\frac{1}{2}}\,.
  \end{equation}
  Moreover, let $g:\RR\times\RR \to \operatorname{OP}^{-p} (Y,\F)$ be a measurable $\operatorname{OP}^{-p} (Y,\F)$-valued
  function; it is also considered as a family $(g_\theta)_{\theta\in T}$, with $g_\theta$ a 
  measurable function
  on $\RR\times\RR$ with values in the bounded operators of
  order $-p$ on Sobolev spaces on $N\times\{\theta\}$.
  We define a norm $  \| g\|_{L^2 (\RR\times\RR,\operatorname{OP}^{-p}(Y,\F))}$ by the analogue of 
  \eqref{normell2}.
  It is easily verified that $L^2 (\RR\times\RR,\operatorname{OP}^{-p}(Y,\F))$ is a Banach algebra with the convolution
  product $g \ast h $ given by the family $(g_\theta \ast h_\theta)_{\theta\in T} $,
  with $$g_\theta \ast h_\theta (x,z)=\int_{\RR} dy g_{\theta} (x,y) h_{\theta} (y,z)$$
  for $g=(g_{\theta})_{\theta\in T}$, $h=(h_\theta)_{\theta\in T}\in L^2 (\RR\times\RR,\operatorname{OP}^{-p}(Y,\F))$.
 Notice that the product in the integrand involves the algebra structure of $\operatorname{OP}^{-p}(Y,\F)$.
 Remark that if $p+q>\dim N$ then there exists a  bounded pairing 
 \begin{equation}\label{ell2-pairing}
  \langle\,,\,\rangle: L^2 (\RR\times\RR,\operatorname{OP}^{-p}(Y,\F))\times L^2 (\RR\times\RR,\operatorname{OP}^{-q}(Y,\F))\to\CC
  \end{equation}
  with $$\langle g,h\rangle =\int_{\RR\times\RR} dx dy\,\omega^Y_\Gamma (g(x,y) h(x,y)^*) \,.$$
  Indeed, given  $g, h\in C^\infty_c (G_{\cyl})$, considered as elements in $L^2 (\RR\times\RR,\operatorname{OP}^{-p}(Y,\F))$ and 
  $L^2 (\RR\times\RR,\operatorname{OP}^{-q}(Y,\F))$ respectively, we have
   \begin{align*} 
 \langle\,g,h\rangle  &\leq \int_{\RR\times\RR} dx dy  | \omega^{Y}_\Gamma (g(x,y)h(x,y)^*)| \\
   &\leq C \int_{\RR\times\RR} dx dy ||| g(x,y) h(x,y)^* |||_{p+q}\\
   & \leq  C \int_{\RR\times\RR} dx dy ||| g(x,y)  |||_{p}  ||| h(x,y)  |||_{q}\\
   &\leq  \| g\|_{L^2 (\RR\times\RR,\operatorname{OP}^{-p}(Y,\F))}  \| h\|_{L^2 (\RR\times\RR,\operatorname{OP}^{-q}(Y,\F))}
  \end{align*}
  In the second step  we have used \cite{MN}, p. 508, Cor. 6.11.
  Thus \eqref{ell2-pairing} is bounded.
 Observe now that 
$ \omega^{\cyl}_\Gamma (gh)= \langle\,g,h^*\rangle$; thus the above inequality also implies that 
\begin{equation}\label{pairing-bounded} L^2 (\RR\times\RR,\operatorname{OP}^{-p}(Y,\F))\times L^2 (\RR\times\RR,\operatorname{OP}^{-q}(Y,\F))\ni (g,h)\rightarrow 
 \omega^{\cyl}_\Gamma (gh)\in\CC
 \end{equation}
 is bounded.
 
 \medskip
  The following Lemma will play a crucial role:
  \begin{lemma}\label{lemma:keylemma} (Key Lemma)
  \item{1)} If  $\ell\in \operatorname{OP}^{-p} (\cyl (Y),\F_{\cyl})$ and  $0\leq q < p-1/2$ then for each $x\in\RR$
    \begin{equation}\label{ell-inequality}
  ||| \hat{\ell} (x) |||_{q} \leq ||| \ell |||_p  \end{equation}
so that, in particular,   $\hat{\ell} (x)$ is an element
  of $B^{-q} (Y,\F)$ for each $x\in\RR$ and for each $0\leq q < p-1/2$.
  Moreover there is a constant $C>0$ such that 
  \begin{equation}\label{ellhatestimate}
  \|\hat{\ell}\|_{L^2 (\RR,\operatorname{OP}^{-q}(Y,\F))}\,\leq\, C ||| \ell |||_p \quad\text{for}\quad 0\leq q < p-1/2\,.
  \end{equation}
\item{2)}  If  $\ell\in  \operatorname{OP}^{-p} (\cyl (Y),\F_{\cyl})\cap {\rm Dom}\partial_{\alpha,p}$ 
then $\hat{\ell}$ is differentiable as a  function from $\RR$  with values in the Banach
algebra $\operatorname{OP}^{-q}(Y,\F)$, $0\leq q < p-1/2$. Moreover  there is a constant $C>0$ such that 
 $$
  \|\frac{d\hat{\ell}}{dx}\|_{L^2 (\RR,\operatorname{OP}^{-q}(Y,\F))}\,\leq\, C ||| \partial_\alpha \ell |||_p \quad\text{for}\quad 0\leq q < p-1/2\,.
$$
 \item{3)} Given $\ell\in \operatorname{OP}^{-p} (\cyl (Y),\F_{\cyl})\cap {\rm Dom}\partial_{\alpha,p}$, the commutator $[\chi^0,\ell]$
  admits a kernel function $k:\RR\times\RR\to \operatorname{OP}^{-q}(Y,\F)$ and there exists a constant
  $C$ such that 
   \begin{equation}\label{commutator-estimate}
    \| k\|_{L^2 (\RR\times\RR,\operatorname{OP}^{-q}(Y,\F))} \,\leq\, C \left( ||| \ell |||_p + ||| \partial_\alpha \ell |||_p \right)
    \quad\text{for}\quad 0\leq q < p-1/2\,.
  \end{equation}
\end{lemma}

\begin{proof}

For the first item we need to show that given $\ell\in \operatorname{OP}^{-p} (\cyl (Y),\F_{\cyl})$,
$\hat{\ell} (x)$ is in $\operatorname{OP}^{-q}(Y,\F)$ and   there is a constant $C>0$ such that 
$$  \|\hat{\ell}\|_{L^2 (\RR,\operatorname{OP}^{-q}(Y,\F))}\,\leq\, C ||| \ell |||_p \quad\text{for}\quad 0\leq q < p-1/2\,.$$
 We consider $\ell\in\operatorname{OP}^{-p} (\cyl (Y))$ as a family of operators 
$(\ell_\theta)_{\theta\in T}$ with $\ell_\theta: L^2(\cyl (\tN)\times\{\theta\})\to L^2(\cyl (\tN)\times\{\theta\})$.
By definition one has
$$\| \ell_\theta \|_{n+p,n}=\| (1+\Delta)^{(n+p)/2} \ell_\theta (1+\Delta)^{-n/2} \|_{C^*}$$
with $\Delta$ denoting  the Laplacian on the cylinder. Applying the Fourier transformation
along the cylindrical variable we obtain:
\begin{align*}
\|\ell\|_{n+p,n}=&\sup_{x\in\RR}\|(1+x^2 +\Delta_{N})^{(n+p)/2} \hat{\ell_\theta} (x) (1+x^2+\Delta_N)^{-n/2}\|_{C^*}\\
& \geq \sup_{x\in\RR}\|(1+x^2)^{(p-q)/2} (1+\Delta_{N})^{(n+q)/2} \hat{\ell_\theta} (x) (1+x^2+\Delta_N)^{-n/2}\|_{C^*}\\
& \geq (1+x^2)^{(p-q)/2}\| (1+\Delta_{N})^{(n+q)/2} \hat{\ell_\theta} (x) (1+x^2+\Delta_N)^{-n/2}\|_{C^*}
\end{align*}
where we have used the fact that 
$$\| (1+x^2 +\Delta_N)^r \xi \|\geq \|(1+x^2)^r \xi\|\quad\text{and}\quad \| (1+x^2 +\Delta_N)^r \xi \|\geq 
\| (1 +\Delta_N)^r \xi \|\,.$$
Taking adjoints we also obtain:
\begin{align*}
\|\ell_{\theta}\|_{n+p,n}& \geq (1+x^2)^{(p-q)/2}\| (1+x^2+\Delta_N)^{-n/2} (\hat{\ell_{\theta}} (x))^*  (1+\Delta_{N})^{(n+q)/2}  \|_{C^*}\\
&\geq (1+x^2)^{(p-q)/2}\| (1+\Delta_N)^{-n/2} (\hat{\ell_\theta} (x))^*  (1+\Delta_{N})^{(n+q)/2}  \|_{C^*}\\
&= (1+x^2)^{(p-q)/2}\| (1+\Delta_N)^{(n+q)/2} \hat{\ell_\theta} (x)  (1+\Delta_{N})^{-n/2}  \|_{C^*}\\
&= (1+x^2)^{(p-q)/2}\|\hat{\ell_\theta} (x) \|_{n+q,n}
\end{align*}
Applying the same argument we prove also that $\|\ell_\theta\|_{-n,-n-p}\geq (1+x^2)^{(p-q)/2} \|\hat{\ell_{\theta}} (x)\|_{-n,-n-q}$.
These inequalities imply that
$$||| \ell_\theta |||_p \geq (1+x^2)^{(p-q)/2} ||| \hat{\ell_{\theta}} (x) ||| _{q}$$
Note now that $(1+x^2)^{- (p-q)/2}$ is a $L^2$-function if $q<p-1/2$; let $C$ be the $L^2$-norm of $(1+x^2)^{- (p-q)/2}$.
 We thus obtain 
$$C^2 ||| \ell_\theta |||^2_p \geq \int_{\RR} ||| \hat{\ell_{\theta}} (x) |||^2 _{q}$$
which implies that $$C ||| \ell |||_p \geq \sup_{\theta\in T} \left( \int_{\RR} ||| \hat{\ell_{\theta}} (x) |||^2 _{q}
\right)^{1/2} = \| \hat{\ell}\|_{L^2 (\RR,\B^q (Y))}\,.$$
The first part of the Lemma is proved.

Next we tackle the second item. We first show that if, in addition, $\ell\in {\rm Dom}\pa_{\alpha,p}$,
then $\hat{\ell}$ is differentiable as a function $\RR\to \operatorname{OP}^{-q}$ .  
Consider $\pa_{\alpha,p} (\ell)$, an element in $\operatorname{OP}^{-p}$ by hypothesis.
Remark that  the automorphism $\alpha_t$
appearing in the definition of $\partial_{\alpha}$, see \eqref{pa-alpha}, corresponds to the translation operator by $t$ under Fourier transformation.
Thus we have, using item 1),
$$\| \frac{\hat{\ell} (x+t)- \hat{\ell} (x)}{t}-\widehat{\pa_{\alpha,p} (\ell)} (x) \|_{\operatorname{OP}^{-q}}\leq \| \frac{\alpha_t (\ell)-\ell}{t}-
\pa_{\alpha,p} (\ell)\|_{\operatorname{OP}^{-p}}$$
and we know by hypothesis that the right hand side converges to 0 as $t\to 0$. Thus the limit
$$\lim_{t\to 0}\frac{\hat{\ell} (x+t)- \hat{\ell} (x)}{t}$$
exists in $\operatorname{OP}^{-q} (Y,\F)$ for each $x\in\RR$ and it is equal to $\widehat{\pa_{\alpha,p} (\ell)} (x)$. This proves the differentiability 
of $\hat{\ell}$. The estimate in this second item is now a consequence of the one in the first item.

Finally, we tackle the third item of the Lemma.
We must show that given $\ell\in \operatorname{OP}^{-p} (\cyl (Y),\F_{\cyl})\cap {\rm Dom}\pa_{\alpha,p}$, 
the commutator $[\chi^0,\ell]$
  admits a kernel function $k:\RR\times\RR\to \operatorname{OP}^{-q}(Y,\F)$ and there exists a constant
  $C$ such that 
 $$    \| k\|_{L^2 (\RR\times\RR,\operatorname{OP}^{-q}(Y,\F))} \,\leq\, C\left( ||| \ell |||_p + ||| \partial_\alpha \ell |||_p \right)
    \quad\text{for}\quad 0\leq q < p-1/2\,.
$$
Let $\ell$ be an element in $\operatorname{OP}^{-p} (\cyl (Y),\F_{\cyl})\cap {\rm Dom}\pa_{\alpha,p}$. We  know that 
it exists in $\operatorname{OP}^{-q}$ the limit
$$\lim_{t\to 0} (\hat{\ell} (x+t) - \hat{\ell} (x))/t$$ for each $x\in\RR.$ Set 
$$\omega (u,v)=\frac{\hat{\ell} (u) - \hat{\ell} (v)}{u-v}\,.$$
The above argument shows that $\omega$ is a continuous function from $\RR\times\RR$ into $\operatorname{OP}^{-q} (Y,\F)$.
Recall the Hilbert transformation $\H: L^2 (\RR)\to L^2 (\RR)$, see the proof of Proposition \ref{prop:hilbert-transf}. It can also be defined on $L^2(\RR,\operatorname{OP}^{-q} (Y,\F))$.
Here we recall that for  $\ell\in
\operatorname{OP}^{-p}(\cyl (Y),\F_{\cyl})$ we have proved that $\hat{\ell} \in L^2 (\RR,\operatorname{OP}^{-q} (Y,\F))$. 
We know that the Hilbert transformation corresponds to the multiplication operator by $2\chi^0-1$ under Fourier transformation $F$ (by this we mean that $F (2\chi^0-1)F^{-1}=\H$). 
Thus $[\chi^0,\ell]$ corresponds to $[\H,\hat{\ell}]/2$ under Fourier transform. 
  As already remarked from the very  definition of $\H$ we know that
  $[\H,\hat{\ell}]$ is the integral operator with kernel function equal to $-i/\pi\, \omega(u,v)$. This proves the first part of the statement
  in item 3) but for the operator  $F\circ [\chi^0,\ell]\circ F^{-1}$ . 
  We now establish the estimate claimed in item 3) but  for $\omega$; we thus estimate
  $\| \omega \|_{L^2 (\RR\times\RR,\operatorname{OP}^{-q}(Y,\F))}$,
  with $0\leq q < p-1/2$.

 Let $(u,v)$ be a point in $\RR\times\RR$, with $|u-v|\geq 1$. Setting $t=u-v$ we get 
 $$||| \omega (u,v) |||_{q}\leq \left( ||| \hat{\ell} (u) |||_q +  ||| \hat{\ell} (v) |||_q \right) /|t|$$
 which implies that
 \begin{align*}
 \int_{|u-v|\geq 1} du dv  ||| \omega (u,v) |||^2_{q} & \leq \int_{|t|\geq 1} \int_{\RR}dtdv \left(  ||| \hat{\ell} (v+t) |||_q +  ||| \hat{\ell} (v) |||_q \right)^2/t^2\\
 &= \int_{|t|\geq 1}\frac{dt}{t^2} \int_{\RR}dv (  ||| \hat{\ell} (v+t) |||^2_q   + ||| \hat{\ell} (v) |||^2_q  + 2 ||| \hat{\ell} (v) |||_q  ||| \hat{\ell} (v+t) |||_q ) 
 \end{align*}
 which is bounded by a constant times $||| \ell |||^2_p$ given that  $\hat{\ell} \in L^2 (\RR,\operatorname{OP}^{-q} (Y,\F))$. In the region $|u-v|<1$ we have
 $$\hat{\ell}(u)-\hat{\ell}(v)=\int_u^v ds \frac{d\hat{\ell}}{ds} (s)$$
 which gives $\omega (u,v)=\frac{1}{u-v}\int_u^v ds \frac{d\hat{\ell}}{ds}(s)$. It then follows that
 \begin{align*}
 \int_{|u-v|<1} dudv  ||| \omega (u,v) |||^2_{q} & \leq \int_{|t|<1} dt \int_{\RR} dv \frac{1}{t^2} \left( \int_v^{v+t} ds \,|||  \frac{d\hat{\ell}}{ds} (s) |||_q
 \right)^2 \\
 & \leq  \int_{|t|<1} dt \int_{\RR} dv \frac{1}{t^2} \big|  \int_v^{v+t} ds \big| \,\big| \int_v^{v+t} ds\, |||  \frac{d\hat{\ell}}{ds} (s)|||^2_q \big|\\
 &=  \int_{|t|<1} \frac{dt}{|t|}\int_{\RR}dv \big| \int_0^t dr  |||  \frac{d\hat{\ell}}{ds} (v+r) |||^2_q \big|\\
  &=  \int_{|t|<1} \frac{dt}{|t|} \big| \int_0^t dr \int_{\RR}dv\, |||  \frac{d\hat{\ell}}{ds} (v+r) |||^2_q \big|\\
  &=  \int_{|t|<1}dt \int_{\RR}dv\, |||  \frac{d\hat{\ell}}{ds} (v) |||^2_q
   \end{align*}
   which is bounded by a constant times $||| \partial_\alpha \ell |||_p^2$ (by item 2)). 
   The third item of the Key Lemma is 
   thus proved for the operator $F \circ [\chi^0,\ell]\circ F^{-1}$ .
Observe now that conjugation by Fourier transformation $F$ defines an isometry on  $L^2 (\RR\times\RR,
\operatorname{OP}^{-q})$; thus $ [\chi^0,\ell]$ admits a kernel function which is in $L^2 (\RR\times\RR, B^{-q} (Y,\F))$.
This means that 
$$  \int_{\RR\times\RR} dudv  ||| \omega (u,v) |||^2_{q} = \int_{\RR\times\RR} dudv  ||| k  (u,v) |||^2_{q} $$
and this implies  the estimate we wanted to prove.

\end{proof}

\subsubsection{Proof of Proposition \ref{prop:extended-cocycles-bis} (extension of the eta cocycle)}
We want to show that if  $m=2n+1$, with $2n$ equal to the dimension of leaves,
 then the  eta cocycle $\sigma_{m}$ 
extends to a  bounded
  cyclic cocycle on $\mathbf{\mathfrak{B}}_m$.
   
  \begin{proof}
  We begin by observing that from its very definition $\sigma_{2n+1}$ is the sum
  of elements of the following type $\omega_{\Gamma} (b [\chi^0,\ell] b')$ where
  
  - $b$ is a product of $p$ elements in $\operatorname{OP}^{-1}$;
  
  - $b'$ is  a product of $s$ elements in $\operatorname{OP}^{-1}$;
  
  - $m=2n+1$ is equal to $p+s$.
  
  \noindent
  Decompose $b$ according to the analogue of the direct sum decomposition explained around
  formula \eqref{decompose}. Then $b=\begin{pmatrix} b_{00} & b_{01}
\cr b_{10}  & b_{11} \cr
\end{pmatrix}$ with $b_{00}=\chi^0 b \chi^0$, $b_{01}= \chi^0 b (1-\chi^0)$, $b_{10}=
(1-\chi^0) b \chi^0$, $b_{11}= (1-\chi^0) b (1-\chi^0)$. Remark for later use that 
$b_{01}= \chi^0 [\chi^0, b] (1-\chi^0)$, $b_{10}=
(1-\chi^0) [b,\chi^0] \chi^0$. Remark also that $[\chi^0,b]= \begin{pmatrix} 0 & b_{01}
\cr -b_{10}  & 0 \cr
\end{pmatrix}$. 
Thus $$(b[\chi^0,\ell]b')_{00}= b_{00} [\chi^0,\ell]_{01} b' _{10} + b_{01} [\chi^0,\ell]_{10} b_{00}'$$
and similarly $$(b[\chi^0,\ell]b')_{11}= b_{10} [\chi^0,\ell]_{01} b' _{11} + b_{11} [\chi^0,\ell]_{10} b_{01}'$$
We then have that
$$\omega_{\Gamma} (b [\chi^0,\ell] b')=  \omega_{\Gamma} ((b [\chi^0,\ell] b')_{00}) + 
\omega_{\Gamma} ((b [\chi^0,\ell] b')_{11})\,.$$
Here we are using the fact that 
the intersection of the diagonal and the support of the kernels
defined by  the off-diagonal terms in the above decomposition have 
 measure zero. 
  We shall work on the term  $\omega_{\Gamma} (b_{00} [\chi^0,\ell]_{01} b' _{10})=
 \omega_{\Gamma} (b_{00} [\chi^0,\ell] [b', \chi^0]\chi^0)$ that appears in 
$ \omega_{\Gamma} (b [\chi^0,\ell] b' )$ (it is the first term in the first summand on the right hand side). Due to the key Lemma \ref{lemma:keylemma} one has
$[\chi^0,\ell]\in L^2 (\RR\times\RR,\operatorname{OP}^{-u} (Y,\F))$ and $[\chi^0,b' ]\in L^2 (\RR\times\RR,\operatorname{OP}^{-t} (Y,\F))$
with $u<1/2$ and $t<s-1/2$. 
Given $b\in 
\operatorname{OP}^{-p}(\cyl (Y),\F_{\cyl})$ and $k\in L^2 (\RR\times\RR,\operatorname{OP}^{-u} (Y,\F))$, we observe that the product $b k$
induces a bounded linear map
$$ \operatorname{OP}^{-p}(\cyl (Y),\F_{\cyl})\times L^2 (\RR\times\RR,\operatorname{OP}^{-u} (Y,\F))\longrightarrow 
L^2 (\RR\times\RR,\operatorname{OP}^{-r} (Y,\F))$$
with $r<p$. This is proved as follows.
Let $F$ denote the Fourier tranformation with respect to $\RR$ on the family of Hilbert spaces
$(L^2 (\RR\times N\times \{\theta\}))_{\theta\in T}$ and consider $F\circ k\circ F^{-1}$. 
It is obvious that $F\circ k\circ F^{-1} \in L^2 (\RR\times\RR,\operatorname{OP}^{-u} (Y,\F))$.
It is easy to see that $\hat{b}=F\circ b\circ F^{-1}$ with $\hat{b}$ equal to the Fourier transform
of $b$ already defined before the key Lemma; 
thus one has $ F \circ (b k) \circ F^{-1}=\hat{b} \circ  F  \circ k  \circ F^{-1}$. 
Now we apply the key Lemma and see that $\hat{b}$ belongs to 
$L^2 (\RR,\operatorname{OP}^{-q} (Y,\F))$ with $q< p-1/2$;
moreover $\hat{b}  \circ  F  \circ k  \circ F^{-1}\in L^2 (\RR\times\RR,\B^{-(q+u)} (Y,\F))$
since $||| \hat{b}( x)|||_q < +\infty$. Thus, thanks to the above formulas,
 $F \circ (b k) \circ F^{-1}$
is an  element of $L^2 (\RR\times\RR,\operatorname{OP}^{-r} (Y,\F))$, with $r:=q+u<p$,
which implies that $bk$ also belongs to $L^2 (\RR\times\RR,\operatorname{OP}^{-r} (Y,\F))$ with $r<p$.
Now, using the  key Lemma again we have a bounded linear map
$$\operatorname{OP}^{-p}(\cyl (Y),\F_{\cyl}) \otimes \operatorname{OP}^{-1}(\cyl (Y),\F_{\cyl}) \otimes \operatorname{OP}^{-s}(\cyl (Y),\F_{\cyl})
\rightarrow L^2 (\RR\times\RR, B^{-r} (Y,\F))\otimes L^2 (\RR\times\RR,\operatorname{OP}^{-t} (Y,\F))$$
defined by 
$$b\otimes\ell\otimes b' \rightarrow b_{00}[\chi^0,\ell]\otimes[b' , \chi^0] \chi^0$$
with $r<p$, $t<s-1/2$. 
Thus one has $r+t<p+s-1/2$ and hence can take $r+t>\dim N$.
Thus we conclude, see \eqref{pairing-bounded}, that $\omega_\Gamma (b_{00} [\chi^0,\ell]
[b' ,\chi^0]\chi^0)$ is a bounded linear functional with respect to $b$, $\ell$ and $b' $. A similar
argument can be applied to the remaining terms
\begin{align*}\omega_\Gamma (b_{01} [\chi^0,\ell]_{10} b' _{00}) = &\,
\omega_\Gamma (\chi^0 [\chi^0,b]
[\chi^0,\ell]b'_{00})\\
\omega_\Gamma (b_{10} [\chi^0,\ell]_{01} b' _{11}) = &\,
\omega_\Gamma ((1-\chi^0) [\chi^0,b]
[\chi^0,\ell]b'_{11})\\
\omega_\Gamma (b_{11} [\chi^0,\ell]_{10} b' _{01}) = &\,
\omega_\Gamma (b_{11} [\chi^0,\ell]
[\chi^0,b'](1-\chi^0))
\end{align*}
to conclude that they are also bounded with respect to $b$. $\ell$ and $b^{\prime}$. 
Thus we have proved that $\omega_{\Gamma} (b [\chi^0,\ell] b')$ 
is bounded with respect $b$, $\ell$ and $b' $.
Observing that  the derivations $\overline{\delta}_j$, $j=1,2$, are bounded from $\mathbf{\mathfrak{B}}_m$ to $\operatorname{OP}^{-1}$ we finally see that the eta cocycle
$\sigma_m$ extends to a continuous cyclic cocycle on $\mathbf{\mathfrak{B}}_m$.
This completes the proof.
  \end{proof}

\subsubsection{Proof of Proposition \ref{prop:extended-cocycles-tris} (extension of the regularized GV cyclic cochain)}
  Recall that we want to show that if  ${\rm deg} S^{p-1} \tau_{GV}^r=2p>m(m-1)^2-2=m^3-2m^2+m-2$, with $m=2n+1$ and $2n$ equal to the dimension of the leaves
  in $(X,\F)$,
  then the regularized Godbillon-Vey cochain $S^{p-1} \tau_{GV}^r$ extends to a  bounded 
  cyclic cochain on  $\mathbf{\mathfrak{A}}_{m}$.

  \begin{proof}
Recall the Banach space decomposition $\A_m = \J_m \oplus \chi^0
\B_m \chi^0 $. We consider elements in $\A_m$ of the the following type:

- $k^\alpha=k_1\cdots k_{n_\alpha}$,  the product of $n_\alpha$ elements in $\J_m$ 
  
  - $b^\beta= \chi^0 \ell_1 \chi^0 \ell_2\chi^0 \cdots \chi^0 \ell_{n_\beta}\chi^0$, the
  product of $n_{\beta}$ elements in $\chi^0
\B_m \chi^0$.

\noindent
We call $n_\alpha$ and $n_{\beta}$ the length of $k^\alpha$ and $b^{\beta}$ respectively.

Let $a_j = k_j + \chi^0 \ell_j \chi^0 \in \A_m$, $j=1,\dots,r$ and consider the product $a:= a_1 \cdots a_r$. We write
$a=\sum_{\gamma} a^\gamma$ with $a^\gamma$ a product of a certain number
of elements of type $k^\alpha$ and of type $b^\beta$.

\begin{lemma}\label{lemma:claim1}
Suppose that $r> s( t-1)+ s-1$. Then for $a= a_1 \cdots a_r$, $a=\sum_{\gamma} a^\gamma$,
at least one of the following will occur  for each $a^\gamma$. 

1) $a^\gamma$ contains at least $s$ elements in $\J_m$;

2)  $a^\gamma$ contains one element of the form $b^\beta$ whose length is at least $t$.
  \end{lemma}
  
  \begin{proof}
  The proof of the Lemma is elementary. Fix $r=s( t-1)+ s-1$. Then the generic  element $a^\gamma$ in the statement
  of the Lemma
   will satisfy at least one of the two above conditions or will be of the form
  $$b^{\gamma_1} k_1 b^{\gamma_2} k_2 b^{\gamma_3} \cdots b^{\gamma_{s-1}} k_{s-1} b^{\gamma_s}$$
  where the length of each $b^{\gamma_i}$ is $t-1$ and the total length is $r$. It is then easy to see
  that if now $r$
  is strictly larger than $s( t-1)+ s-1$ then one of the above two conditions must necessarily occur.
   \end{proof}
Observe now that $\chi^0 \ell_1\chi^0 \ell_2\chi^0 - \chi^0 \ell_1 \ell_2
\chi^0= \chi^0 [\chi^0,\ell_1][\chi^0,\ell_2]\chi^0$. This simple observation is at the basis
of the following  
  \begin{lemma}\label{lemma:claim2}
 Let $b^\beta$ be an element of length $t$, namely $b^\beta= \prod_{j=1}^t \chi^0 \ell_j \chi^0$
 with $\ell_j\in\B_m$. Then one has 
  \begin{equation}\label{claim2}
b^\beta=\chi^0 \left( \prod_{j=1}^t  \ell_j \right) \chi^0 + \chi^0 c \chi^0
 \end{equation}
where $c$ is a linear combination of $c_k$ and $c_k$  is the
product of  $t_1$ elements of type 
$[\chi^0, \ell_i]$ and $t_2$ elements of type $\ell_i$ with $t_1 + t_2= t$ and $t_1\geq 1$.
Moreover the number of such $c_k$ is at most $2^{t-1}-1$
    \end{lemma}
The proof of Lemma \ref{lemma:claim2} is based on an elementary induction argument.\\
 Consider now the product, $a$, of $r$ elements $a_i\in\A_m$
and write $a=\sum_\gamma a^\gamma$ as above. 
Then, obviously, either one of the following will apply to each $a^\gamma$:
\begin{itemize}
\item[a)] $a^\gamma$ is of the form $b^\beta$ introduced after the statement of the Proposition, namely 
$a^\gamma=\prod_{i=1}^r \chi^0 \ell_i \chi^0$;
\item[b)] $a^\gamma$ contains at least one $k_j\in\J_m$.
\end{itemize}
Suppose now that $r> m (m-1) + m-1=m^2 - 2m$. Recall the definition of the map
$t:\A_m\to \J_m$, see \eqref{eq:t}. Clearly, by definition, in case b) we have that $t(a^\gamma)=a^\gamma$,
since $a^\gamma\in\J_m$ given that  $\J_m$ is an ideal in $\A_m$.
In case a) we can write
$$t(a^\gamma)= \prod_{i=1}^r \chi^0 \ell_i \chi^0 - \chi^0 (\prod_{i=1}^r \ell_i)\chi^0=\sum_j \chi^0 c_j \chi^0$$
according to Lemma \ref{lemma:claim2}. Here $c_j$ is a product of $r$ elements out of  $[\chi^0,\ell_i]$
and $\ell_i$. Then at least one  of the following will occur:\\

a-1) $c_j$ contains at least $m$ elements of the form $[\chi^0,\ell_j]$;

a-2) $c_j$ contains a consecutive product of at least $m$ elements in $\B_m$.\\

The latter claim is proved by the same reasoning in the proof of Lemma \ref{lemma:claim1}.
Now, in the case a-1) one has $c_j\in \I_1$, given that $[\chi^0,\ell]\in \J_m$. In case
a-2) we apply the following Lemma, Lemma \ref{lemma:lemmatwo}, in order to see that  $c_j\in\I_1$,
observing that $c_j$ contains a  consecutive product of at least $m$ elements $\ell_j$ and, according to
Lemma \ref{lemma:claim2},
at least one $[\chi^0,\ell_i]$ (which belongs to $\J_m$ by definition).
All things considered we have shown that in case a) the element $c_j$ and thus $t(a^\gamma)$ belongs to 
$\I_1$ for $a=a_1 \cdots a_r$ and $r> m^2-2m$.

\begin{lemma}\label{lemma:lemmatwo}
Recall the Banach algebra $\operatorname{OP}^{-p}$ on $(\cyl (Y),\F_{\cyl})$, defined as the closure
of $\Psi^{-p}_c (G_{\cyl}/\RR_{\Delta})$ with respect to the norm $|||\;\;|||_p$. If $p$
is greater than the dimension of the leaves, then for each natural number $\nu\geq 1$
\begin{equation}\label{lemmatwo}
\J_\nu \operatorname{OP}^{-p}
\subset \I_1 \quad \text{and}\quad 
 \operatorname{OP}^{-p} \J_\nu \subset \I_1\,.
\end{equation}
Moreover if $k\in\J_\nu$ and  $\ell\in\operatorname{OP}^{-p}$ then 
\begin{equation}\label{lemmatwo-estimate}
\| kb\|_{\I_1}\leq C \| k\|_{\J_\nu} ||| b |||_p\quad \text{and}\quad \| bk\|_{\I_1}\leq C ||| b |||_p  \| k\|_{\J_\nu}
\end{equation}
with $C$ is a constant depending only on the Dirac operator $D$ on $(Y,\F_Y)$.
\end{lemma}
We remark that it is precisely for the validity of this Lemma that the extra condition
involving $g(s,y)=1+s^2$ was added in the definition of $\B_k$ and  $\J_k$.
\begin{proof}
Let $k\in\J_\nu$ and let $\ell\in\operatorname{OP}^{-p}$. One 
 can write
$$k  \ell  = k g g^{-1} (1+D^2)^{-p/2} (1+D^2)^{p/2} \ell\,.$$
Note that $k g$ and $(1+D^2)^{p/2}\ell$ are bounded since $k\in\J_\nu$
and $\ell\in \operatorname{OP}^{-p}$. Next we prove that $g^{-1} (1+D^2)^{-p/2}\in \I_1$. 
It suffices to show that $g^{-1/2} (1+D^2)^{-p/4}\in \I_2$;  equivalently, using that $g=(s+i)(s-i)$
we can prove that  $(s\pm i)^{-1} (1+D^2)^{-p/4}\in \I_2$. Let us fix the plus sign, for example.
We want to show that  $A:=(s+ i)^{-1} (1+D^2)^{-q/2}\in \I_2$ if $2q$ is greater than the dimension 
of the leaves. First, we conjugate $A$ with the Fourier transformation in the cylindrical
direction, obtaining $F\circ A\circ F^{-1}=(i+ (\frac{1}{i}\frac{d}{dt}))^{-1} (1+t^2+ D_Y)^{-q/2}$.
For fixed $t\in\RR$ let $K_{\theta,t}$ the Schwartz kernel of  $(1+t^2+ D_Y)^{-q/2}$
along $\tN\times\{\theta\}$. Using elementary
properties of the Fourier transformation one can check that $F\circ A_\theta \circ F^{-1}$
has Schwartz kernel $L_\theta (t,s,y,y')$ given, up to a multiplicative constant, by
$$ L_\theta (s,t,y,y')=u(s-t)e^{-|s-t|} K_{\theta,t} (y,y')\,,$$
with $u(x)=\chi_{[0,+\infty)}$.
Now we estimate
\begin{align*}
\| F\circ A \circ F^{-1}\|_{\I_2 (\cyl(Y),\F_{\cyl})} &= \sup_{\theta\in T} \left( \int_{\RR\times\RR} ds dt 
\int_{ \tN\times\tN} dy dy' \chi_{\Gamma} | L_\theta (s,t,
y,y')|^2 \right) \\
&\leq  \sup_{\theta\in T} \left( \int_{\RR\times\RR} ds dt e^{-2|s-t|} \int_{ \tN\times\tN} dy dy' \chi_{\Gamma} | K_{\theta,t} (y,y')|^2 \right)\\
&\leq  \int ds dt e^{-2|s-t|} \frac{1}{1+t^2} \| (1+ D_{Y}^2 )^{-(q-1)/2}\|_{\I_2 (Y,\F_Y)}^2<+\infty
\end{align*}
Here we have used the characteristic function $\chi_\Gamma$ for a fundamental domain for the action
of $\Gamma$ on $\tN$. We have also used  the inequality  of positive self-adjoint elements  
$(1+t^2+D^2_Y)^{-q/2} \leq (1+t^2)^{-1/2} (1+D^2_Y)^{-(q-1)/2}$ (we have already used this inequality
in the proof of the key Lemma). This implies that
$$ \sup_{\theta\in T} \left(   \int_{ \tN\times\tN} dy dy' \chi_{\Gamma} | K_{\theta,t} (y,y')|^2\right)
 =\| (1+t^2+D^2_Y)^{-q/2}\|^2_{\I_2 (Y,\F_Y)}\leq \frac{1}{1+t^2} 
\|(1+D^2_Y)^{-(q-1)/2}\|_{\I_2 (Y,\F_Y)}$$
which is what is used above.
Since we have proved that $\| F\circ A \circ F^{-1}\|_{\I_2 (\cyl(Y),\F_{\cyl})}$ is finite,
we conclude that $\|  A \|_{\I_2 (\cyl(Y),\F_{\cyl})}$ is also finite. Thus $g^{-1}(1+D^2)^{-p/2}$
is in $\I_1$ if $p$ is greater than the dimension of the leaves. 
Now it is obvious that $k\chi^0 \ell \chi^0\in \I_1$ from the ideal property of $\I_1$.
Similarly $\chi^0 \ell \chi^0 k\in\I_1$.
In order to get the estimate in \ref{lemmatwo-estimate} we use standard
properties:
\begin{align*}
\| kb\|_{\I_1}&= \| (kg) (g^{-1}   (1+D^2)^{-p/2}) ((1+D^2)^{p/2} b)\|_{\I_1}\\& \leq 
\| g^{-1}   (1+D^2)^{-p/2}\|_{\I_1}  \| kg\|_{C^*}  \|(1+D^2)^{p/2} b\|_{C^*}\\
& \leq C \| kg\|_{C^*}  \|(1+D^2)^{p/2} b\|_{C^*}\\
& \leq C \|k\|_{\J_\nu} ||| b |||_p
\end{align*}

The Lemma is proved.
\end{proof}
Now we consider the case b), namely $a^\gamma$ contains at least one $k_j\in\J_m$.
Applying the same argument as in Lemma \ref{lemma:claim1} we see that at least one of the following will
occur if $r>m(t-1)+m-1$:

b-1) $a^\gamma$ contains at least $m$ elements in $\J_m$;

b-2) $a^\gamma$ contains a 
$b^\beta$ of length $t$.

\noindent
In the case b-1) one has $a^\gamma\in\I_1$ in an obvious way. In case b-2), 
we apply Lemma \ref{lemma:claim2} in order to see that $b^\beta$ has the form $\chi^0 (\prod_{i=1}^t \ell_i)\chi^0+
\sum \chi^0 c_j\chi^0$. The first term belongs to $\chi^0 \operatorname{OP}^{-t}\chi^0$. Then,
the corresponding term in 
 $a^\gamma$ will be in $\I_1$ if $t\geq m$, since we can apply Lemma \ref{lemma:lemmatwo} once we recall that
 $a^\gamma$ contains at least one $k_j\in \J_m$. Here we are using a small extension of 
 Lemma \ref{lemma:lemmatwo}:\\
{\it if $p$ is greater than the dimension of the leaves, then 
$$\J_\nu (X,\F) (\chi^0 \operatorname{OP}^{-p} (\cyl{\pa X},\F_{\cyl})\chi^0) \subset \I_1 (X,\F)\quad\text{and}\quad 
 (\chi^0 \operatorname{OP}^{-p} (\cyl{\pa X},\F_{\cyl})\chi^0) \J_\nu (X,\F) \subset \I_1 (X,\F)\,.$$}
Now recall that $t (a^\gamma)\in\I_1$ for $a=a_1 \cdots a_r$ if $r>m^2 - 2m$. Applying the same reasoning to 
$b^\beta$
with length $t$ we obtain $c_j\in\I_1 (\cyl(\pa X),\F_{\cyl})$ if $t>m^2 - 2m$. Then the corresponding term
in $a^\gamma$ also belongs to $\I_1 (X,\F)$. Thus, if $r> m (t-1) + (m-1)$, with $t-1=m^2 -2m$, namely
$r> m(m-1)^2 -1$, then we can conclude that $c_j\in\I_1 (\cyl(\pa X),\F_{\cyl})$ in both cases
b-1) and b-2) ; consequently $a^\gamma$, which in this case is $t(a^\gamma)$, belongs to
$\I_1 (X,\F)$.

Now we put everything together and we show that the regularized cochain $\tau^r_{2p}$ extends
to a continuous cochain on $\mathbf{\mathfrak{A}}_{m}$, with $m=2n+1$.
Recall that the regularized cochain is defined through the regularized weight which
was shown to be equal to $\omega_\Gamma \circ t$ on $A_c\subset A^*$.
See Proposition \ref{prop:regularized-via-t}. Here we recall for later use  that $\omega_\Gamma$
extends continuously to $\I_1$. For an element such as
$$a_0 \cdots a_{i-1} \,\delta_1 (a_{i}) a_{i+1} \cdots a_{j-1} \,\delta_2 (a_{j}) a_{j+1}\cdots
a_{2p}$$
with $a_k\in A_c\subset \mathbf{\mathfrak{A}}_{m}$, we need to prove that
\begin{equation}\label{estimate-for-regularized}
|\omega_\Gamma (t (a_0 \cdots a_{i-1} \,\delta_1 (a_{i}) a_{i+1} \cdots a_{j-1} \,\delta_2 (a_{j}) a_{j+1}\cdots
a_{2p}))|\leq C \prod_{j=1}^{2p} \|a_j\|_{\mathbf{\mathfrak{A}}_{m}}.
\end{equation}
We shall prove a stronger statement, namely that the left hand side of \eqref{estimate-for-regularized}
makes already sense for $a_j\in \mathbf{\mathfrak{A}}_{m}$ and for the
closures $\overline{\delta}_j$ and that the estimate
in \eqref{estimate-for-regularized} holds. 

Let $a=a_1 a_2 \cdots a_r$, with $a_j\in\A_m$,
and write, as above, $a=\sum_\gamma a^\gamma$. Suppose that $r=2p+1$, $2p>m(m-1)^2-2$, so that
 $r>m(m-1)^2 -1$. We have proved that $t(a^\gamma)\in \I_1$ for each $\gamma$. We will now estimate the norm
$\| t(a^\gamma)\|$ in terms of the norms $\|a_j\|_{\A_m}$.
We shall analyze one by one the terms appearing in cases a-1), a-2), b-1), b-2).
To this end recall that if $q>\dim V$ then for 
$k\in\J_\nu$ and  $\ell\in\operatorname{OP}^{-q}$ the following estimate holds
\begin{equation}\label{lemmatwo-estimate-bis}
\| kb\|_{\I_1}\leq C \| k\|_{\J_\nu} ||| b |||_q\quad \text{and}\quad \| bk\|_{\I_1}\leq C ||| b |||_q  \| k\|_{\J_\nu}
\end{equation}
with $C$ depending only on the Dirac operator on $(Y,\F_Y)$.

Consider first the case  a); 
then $t(a^\gamma)=\sum_j \chi^0 c_j \chi^0$. 
In case a-1)   $c_j$ contains at least $m$ elements in $\J_m$, say $k_1=[\chi^0,\ell_1], \dots ,k_t=[\chi^0,\ell_t]$
with $t\geq m$;
thus, without loss of generality, we can assume that $c_j = [\chi^0,\ell_1] \cdots [\chi^0,\ell_t] \ell_{t+1} \cdots \ell_{r}$.
Then
\begin{align*}
\| \chi^0 c_j \chi^0\|_{\I_1 (X,\F)} & \leq \| c_j \|_{\I_1 (\cyl \pa X,\F_{\cyl}) }\leq \prod_{i=1}^t \| [\chi^0,\ell_i] \|_{\I_m}
\prod_{i=t+1}^r \| \ell_i\|_{B^*}\\
&\leq \prod_{i=1}^t \| [\chi^0,\ell_i] \|_{\J_m}
\prod_{i=t+1}^r \| \ell_i\|_{B^*}
\leq \prod_{i=1}^r \| a_i\|_{\A_m}
\end{align*}
where we recall that if $a=\chi^0\ell\chi^0 + k$ then $\| a\|_{\A_m}:= \| \ell\|_{\B_m} + \| k \|_{\J_m}$
and that 
$$\|\ell \|_{\mathcal{B}_{m}}:= ||| \ell  ||| + \| [\chi_{{\rm cyl}}^0,\ell]  \|_{\J_m} +  |||  \partial_\alpha \ell |||
+2 \|[f,\ell]\|_{B^*} + \| [f,[f,\ell]]\|_{B^*}$$
so that, clearly, 
$$\| a \|_{\A_m}\geq \| \ell\|_{\B_m}\geq ||| \ell |||+ \|[\chi^0,\ell]\|_{\J_m} \geq \| \ell\|_{B^*}+ \|[\chi^0,\ell]\|_{\J_m}\,.$$

Next we tackle the case a-2). Then we can assume without loss of generality
that $c_j$ is of the form $b(\ell_1\cdots \ell_t) b'$ with  $t\geq m$ and $b$ and $b'$ 
are certain products of $\ell_i $ and $[\chi^0,\ell_j]$ and either $b$ or $b'$ contains  at least one
$[\chi^0,\ell_k]$. Say that it is $b$ that contains $[\chi^0,\ell_k]$. Then, using \eqref{lemmatwo-estimate-bis}
for $q=t$ and $\nu=m$ we get
\begin{align*}
\| \chi^0 c_j \chi^0\|_{\I_1 (X,\F)} & \leq \| c_j \|_{\I_1 (\cyl \pa X,\F_{\cyl}) }\leq 
C \|b\|_{\J_m} \| b '\|_{B^*} ||| \ell_1\cdots \ell_t |||_t\\
&\leq C \|b\|_{\J_m} \| b ' \|_{B^*} \prod_{i=1}^t ||| \ell_i |||\leq C  \prod_{i=1}^r \| a_i\|_{\A_m}
\end{align*}
Thus, in case a) we have proved that $\|t(a^\gamma)\|_{\I_1}\leq C 2^r \prod_{i=1}^r \| a_i\|_{\A_m}$
since the number of $c_j$ is at most $2^r$.

Now we consider the case b). In the case b-1) we have that $a^\gamma$ is a product of
$[\chi^0,\ell_1],\dots,[\chi^0,\ell_t]$, $t\geq m$ and of $\ell_{t+1},\dots,\ell_{r}$. Then, 
as already remarked, $a^\gamma\in\I_1$, $t (a^\gamma)=a^\gamma$,  and moreover,
from standard estimates we have
$$\| a^\gamma \|_{\I_1 (X,\F)} \leq \prod_{i=1}^t \| [\chi^0,\ell_i]\|_{\I_m} \prod_{i=t+1}^r \| \ell_i \|_{B^*}
\leq \prod_{i=1}^t \| [\chi^0,\ell_i]\|_{\J_m} \prod_{i=t+1}^r \| \ell_i \|_{B^*}\leq  \prod_{i=1}^r \| a_i\|_{\A_m}\,.$$
In case b-2) we can write $a^\gamma= c b^\beta c'$ with
$b^\beta=\prod_{i=1}^t \chi^0 \ell_i \chi^0$ and $t\geq m$, and $c$, $c'$ are certain products of $k_i\in\J_m$
and $\chi^0\ell_i\chi^0$ for $a_i=k_i + \chi^0 \ell_i \chi^0$, $i=1,\dots,r$. 
Here we know that at least one $k_i\in\J_m$ will appear in  $c$ or $c'$. Say that it is $c$ that contains such $k_i$.
Then we can apply the
same argument in case a) and conclude that 
$$\| a^\gamma \|_{\I_1 (X,\F)} \leq \| cb^\beta\|_{\I_1} \|c '\|_{C^*}\leq C  \|c\|_{\J_m} |||\prod_{i=1}^t  \ell_i |||_t
\|c '\|_{C^*}\leq
  C \|c\|_{\J_m} \prod_{i=1}^t ||| \ell_i ||| \|c '\|_{C^*}\leq C \prod_{i=1}^r \| a_i\|_{\A_m}$$
  with \eqref{lemmatwo-estimate-bis} used in order to justify the second estimate.
  We finish the proof by observing that the two closed derivations $\overline{\delta}_1$ and $\overline{\delta}_2$  are bounded from $\mathbf{\mathfrak{A}}_{m}$ to $\A_m$ and that
  the inclusion $\mathbf{\mathfrak{A}}_{m}\subset \A_m$ is bounded; this proves that
 \eqref{estimate-for-regularized} holds for $a_j\in \mathbf{\mathfrak{A}}_{m}$ which is what
 is needed in order to conclude.


  \end{proof}

\begin{appendix}

\section{More on the Godbillon-Vey cocycle $\tau_{GV}$}\label{appendix:absolute-taugv}
Let $(Y,\F)$, $Y=\tN\times_\Gamma S^1$, be a foliated bundle {\it without} boundary. 
We adopt the notations of Subsection \ref{subsect:absolute-taugv}.


{\it The goal of this appendix is to recall the basic definitions and results of \cite{MN}
leading to the definition of Godbillon-Vey cyclic 2-cocycle, $\tau_{GV}$, defined on the 
algebra $\Psi^{-\infty}_c (G,E)$ of equivariant smoothing 
families with $\Gamma$-compact support.  }

\medskip
First, recall that for any $\gamma\in \Gamma$ we can define the positive real function $\lambda_\gamma$
on $T$  through
the formula
$\lambda_\gamma d\theta= \gamma (d\theta)$ with $\gamma$ acting by pull-back by the corresponding
orientation preserving diffeomorphism. 
Recall from the Subsection 
on the Godbillon-Vey class, Subsection \ref{subsect:gv-form},
the modular function  $\psi$ on $\tN\times T$ defined by $\tilde{\omega}\wedge d\theta=\psi \tilde{\Omega}$,
$\tilde{\omega}$ and $ \tilde{\Omega}$ denoting $\Gamma$-invariant volume forms on $\tN$ and $\tN\times T$ respectively.
The operator $\Delta^{it}$, $\Delta^{it}(\xi):= \psi^{-it}\xi$, sends $C^\infty_c (\tN\times T,\widehat{E})$ into itself and extend
to a linear operator from the $C(T)\rtimes \Gamma$-module $\E$ into itself which is continuous. 
The operator $\Delta^{it}$, however,
is not $C(T)\rtimes \Gamma$-linear
\footnote{indeed $\Delta^{it}(\xi f)= \Delta^{it}(\xi)\sigma_t ( f)$ if $f\in C(T)\rtimes \Gamma$,
with $(\sigma_t)$ the modular automorphism group defined by the weight $f\to \int_T f(e) d\theta$; 
$\sigma_t (\sum a_\gamma \gamma)= \sum (\lambda_\gamma^{-it} a_\gamma) \gamma$.}.
Nevertheless, the following does hold: 

if $P\in \L(\E)$ (in particular, it is $C(T)\rtimes \Gamma$-linear) then 
$\Delta^{it} P \Delta^{-it}\in \L(\E)$,

if $P\in \KK(\E)$ then 
$\Delta^{it} P \Delta^{-it}\in \KK(\E)$, 

if $P\in \Psi^{-\infty}_c (G,E)$ then $\Delta^{it} P \Delta^{-it}\in  \Psi^{-\infty}_c (G,E)$.\\
Set $\hat{\sigma}_t (P):= \Delta^{it} P \Delta^{-it}$;
denote by $\delta_2$ the generator for the corresponding $\RR$-action:
\begin{equation}\label{limit-of-sigma}
\delta_2 (P):= \lim_{t\to 0} \frac{ \hat{\sigma}_t (P) - P}{t}, \quad P\in \Psi^{-\infty}_c (G,E)\,.
\end{equation}
Then one proves that
\begin{equation}\label{a-delta_2}
\delta_2 (P)=[\phi,P]\quad\text{with}\quad  \phi=\log \psi\,.
\end{equation}
 $\delta_2$ is a derivation
on the algebra $\Psi^{-\infty}_c (G,E)$.
We should observe here that $\phi$ is not $\Gamma$-invariant nor it is compactly supported.
In fact $\phi$ is not even bounded.

Recall next the bundle $\widehat{E}^\prime$ on $\tN\times T$ introduced in \cite{MN}: this is the same as $\widehat{E}$
but with the new $\Gamma$-equivariant structure given by 
$v\gamma:= \lambda_\gamma (\theta)^{-1} v \gamma$ if $v\in  \widehat{E}_{(y,\theta)}$, $\gamma\in \Gamma\,.$
One can prove that 
\begin{equation}\label{a-iff}
P\in \Psi^{-\infty}_c (G,E)\quad \text{if and only if}\quad P\in \Psi^{-\infty}_c (G,E^\prime).
\end{equation}
We shall  freely use this identification without further comments.
 We shall be interested
in $\Psi^{-\infty}_c (G;E,E^\prime)$, which we will consider as a bimodule over $\Psi^{-\infty}_c (G,E)$ 
through the above identification.
Consider $\dot{\phi}$, the partial derivative of $\phi$ in the direction of $S^1$.
We consider the 
multiplication operator $\dot{\phi}: C^\infty_c (\tN\times T,\widehat{E})\to C^\infty_c (\tN\times T,\widehat{E}^\prime)$
and the operator $[\dot{\phi},\,\,]$. One can prove that if $P\in \Psi^{-\infty}_c (G,E)$,
then $[\dot{\phi},P]\in  \Psi^{-\infty}_c (G;E,E^\prime)$, so that
there is a well defined bimodule
derivation $\delta_1 : \Psi^{-\infty}_c (G,E)\to \Psi^{-\infty}_c (G;E,E^\prime)$:
\begin{equation}\label{a-delta_1}
\delta_1 (P)=[\dot{\phi},P]\quad\text{with}\quad  \phi=\log \psi\,.
\end{equation}
We shall also consider the modified map  
\begin{equation}\label{a-delta2prime}
\delta^\prime_2: \Psi^{-\infty}_c (G,E,E^\prime)\to \Psi^{-\infty}_c (G;E,E^\prime)
\,,\quad \delta^\prime_2 (P)=[\phi,P]\,.
\end{equation}
 with
the 
multiplication operator acting  as
an operator  $C^\infty_c (\tN\times T,\widehat{E})\to C^\infty_c (\tN\times T,\widehat{E}^\prime)$
and as an operator $C^\infty_c (\tN\times T,\widehat{E}^\prime)\to  C^\infty_c (\tN\times T,\widehat{E})$.
 If we consider $\Delta^{-it}$ acting on $C^\infty_c (\tN\times T,\widehat{E})$
and $\Delta^{it}$ acting on $C^\infty_c (\tN\times T,\widehat{E}^\prime)$
then one can show that $P\in \Psi^{-\infty}_c (G;E,E^\prime) \Rightarrow 
\Delta^{it}
P \Delta^{-it}\in \Psi^{-\infty}_c (G;E,E^\prime)$ and $\delta_2^\prime$ is associated to the 
$\RR$-action $\hat{\sigma}^\prime (P):= \Delta^{it}
P \Delta^{-it}$ through the analogue of the limit \eqref{limit-of-sigma}.\\
{\bf Summarizing:} we have defined a derivation $\delta_2$, a bimodule derivation $\delta_1$,
\begin{equation}\label{deltas}
\delta_2:  \Psi^{-\infty}_c (G,E)\to \Psi^{-\infty}_c (G;E)\,,\quad \delta_1 : \Psi^{-\infty}_c (G,E)\to \Psi^{-\infty}_c (G;E,E^\prime)\,,
\end{equation}
and a linear map $\delta_2^\prime : \Psi^{-\infty}_c (G;E,E^\prime)\to \Psi^{-\infty}_c (G;E,E^\prime).$
If $P\in \Psi^{-\infty}_c (G,E)$ and  $Q\in \Psi^{-\infty}_c (G;E,E^\prime)$ then, with
respect to the bimodule $\Psi^{-\infty}_c (G,E)$-structure of $\Psi^{-\infty}_c (G;E,E^\prime)$
we have 
\begin{equation}\label{a-delta2-prime-deriv}
\delta_2^\prime (PQ)= \delta_2 (P) Q +P \delta_2^\prime (Q)\,,\quad \delta_2^\prime (QP)= \delta_2^\prime (Q) P +
Q \delta_2 (P) \,.
\end{equation}
One can also check that
\begin{equation}\label{a-deriv-commute}
\delta_1 (\delta_2 (P))= \delta^\prime_2 (\delta_1 (P))
\end{equation}
Recall, see \eqref{weight-again}, the weight $\omega_{\Gamma}$
defined on the
algebra $\Psi^{-\infty}_c (G;E)$,
\begin{equation}\label{a-weight-again}
\omega_{\Gamma}(k)=
 \int_{Y(\Gamma)} \mathrm{Tr}_{(\tilde{n},\theta)} k(\tilde{n},\tilde{n},\theta)d\tilde{n}\,d\theta\,.
\end{equation}
with $Y(\Gamma)$  the fundamental domain in $\tN\times T$ for the free diagonal action
of $\Gamma$  on $\tN\times T$.
We shall be interested in the linear functional defined on the bimodule
 $\Psi^{-\infty}_c (G;E,E^\prime)$
by the analogue of \eqref{a-weight-again}. To be quite explicit
\begin{equation}\label{a-weight-bimodule}
\omega_{\Gamma}(k)=
 \int_{Y(\Gamma)} \mathrm{Tr}_{(\tilde{n},\theta)} k(\tilde{n},\tilde{n},\theta)d\tilde{n}\,d\theta\,.
\end{equation}
where we now identify $\widehat{E}_{(\tilde{n},\theta)}$ and $\widehat{E}_{(\tilde{n},\theta)}^\prime$
given that thet are identical vector spaces (it is only the $\Gamma$-actions that are different).
For reasons that will be clear in a moment, we name \eqref{a-weight-bimodule} the  {\it bimodule trace}.

One can prove the following alternative expression for the bimodule trace $\omega_\Gamma$:
\begin{equation}\label{weight-again.bis}
\omega_{\Gamma}(k)=
 \int_{S^1} \Tr (\sigma_\theta k(\theta) \sigma_\theta)d\theta\,
\end{equation}
with $\sigma$ a compactly supported smooth function on $\tN\times S^1$ such that $\sum_{\gamma\in \Gamma} \gamma(\sigma)^2=1$,
$\sigma_\theta:= \sigma |_{\tV\times\{\theta\}}$
and $\Tr$ denoting the usual trace functional on the Hilbert space $L^2(\tN\times\{\theta\},\widehat{E}_{\tN\times\{\theta\}})$. Here,
once again, we identify 
$L^2(\tN\times\{\theta\},\widehat{E}_{\tN\times\{\theta\}})$
with $L^2(\tN\times\{\theta\},\widehat{E}_{\tN\times\{\theta\}}^\prime)$.
One also establishes that \eqref{weight-again.bis}
does not depend on the choice of $\sigma$; this also proves that the bimodule trace
\eqref{a-weight-bimodule} does not depend on the choice of the fundamental domain.
For the sake of brevity we omit the proof of formula \eqref{weight-again.bis}.



With $1=\dim T$, 
the Godbillon-Vey cyclic $2$-cocycle 
on  $\Psi^{-\infty}_c (G;E)$  is defined to be 
$
\tau_{GV} (a_0, a_1, a_2)= \frac{1}{2!}\sum_{\alpha\in \mathfrak{S}_2} {\rm sign}(\alpha)
(a_0 \delta_{\alpha (1)} a_1 \delta_{\alpha (2)} a_2)$. We have proved in Proposition
\ref{taugv-cyclic-cocycle}
that $
\tau_{GV}$ defines a cyclic 2-cocycle on the algebra $\Psi^{-\infty}_c (G;E)$

\end{appendix}

{\small \bibliographystyle{plain}
\bibliography{ecrpgvit}
}

\end{document}